\documentclass[draft]{amsart}

\usepackage{amsmath,amssymb,amsthm,enumerate} 
\usepackage{amsfonts} 
\usepackage[top=30mm,bottom=30mm,left=25mm,right=25mm]{geometry}
\usepackage[dvipdfmx]{graphicx}
\usepackage{color}

\theoremstyle{plain}
\newtheorem{pretheo}{Theorem}[section]
\newtheorem{preassu}[pretheo]{Assumption}
\newtheorem{precoro}[pretheo]{Corollary}
\newtheorem{predefi}[pretheo]{Definition}
\newtheorem{preexam}[pretheo]{Example}
\newtheorem{prelemm}[pretheo]{Lemma}
\newtheorem{preprop}[pretheo]{Proposition}
\newtheorem{prerema}[pretheo]{Remark}

\newenvironment{theo}{\begin{pretheo}}{\end{pretheo}}

\newenvironment{defi}{\begin{predefi}}{\end{predefi}}

\newenvironment{lemm}{\begin{prelemm}}{\end{prelemm}}
\newenvironment{prop}{\begin{preprop}}{\end{preprop}}
\newenvironment{rema}{\begin{prerema}\rm}{\end{prerema}}

\DeclareMathOperator{\di}{div}

\DeclareMathOperator{\Di}{Div}

\newcommand{\intd}{\,d}
\newcommand{\supp}{{\rm supp}\,}

\newcommand{\loc}{{\rm loc}}
\newcommand{\pa}{\partial}

\newcommand{\wh}[1]{\widehat{#1}}
\newcommand{\wt}[1]{\widetilde{#1}}
\newcommand{\tp}[1]{{}^{\mathsf{T}}#1}

\newcommand{\BBA}{\mathbb{A}}
\newcommand{\BBB}{\mathbb{B}}

\newcommand{\BBF}{\mathbb{F}}
\newcommand{\BBG}{\mathbb{G}}
\newcommand{\BBH}{\mathbb{H}}
\newcommand{\BBI}{\mathbb{I}}
\newcommand{\BBJ}{\mathbb{J}}
\newcommand{\BBK}{\mathbb{K}}

\newcommand{\BBM}{\mathbb{M}}

\newcommand{\Ba}{\mathbf{a}}
\newcommand{\Bb}{\mathbf{b}}

\newcommand{\Bd}{\mathbf{d}}
\newcommand{\Be}{\mathbf{e}}
\newcommand{\Bf}{\mathbf{f}}
\newcommand{\Bg}{\mathbf{g}}
\newcommand{\Bh}{\mathbf{h}}

\newcommand{\Bn}{\mathbf{n}}

\newcommand{\Bu}{\mathbf{u}}
\newcommand{\Bv}{\mathbf{v}}
\newcommand{\Bw}{\mathbf{w}}
\newcommand{\Bx}{\mathbf{x}}
\newcommand{\By}{\mathbf{y}}
\newcommand{\Bz}{\mathbf{z}}

\newcommand{\BC}{\mathbf{C}}
\newcommand{\BD}{\mathbf{D}}

\newcommand{\BF}{\mathbf{F}}
\newcommand{\BG}{\mathbf{G}}
\newcommand{\BH}{\mathbf{H}}
\newcommand{\BI}{\mathbf{I}}

\newcommand{\BK}{\mathbf{K}}

\newcommand{\BM}{\mathbf{M}}
\newcommand{\BN}{\mathbf{N}}

\newcommand{\BR}{\mathbf{R}}
\newcommand{\BS}{\mathbf{S}}
\newcommand{\BT}{\mathbf{T}}
\newcommand{\BU}{\mathbf{U}}
\newcommand{\BV}{\mathbf{V}}

\newcommand{\CA}{\mathcal{A}}
\newcommand{\CB}{\mathcal{B}}
\newcommand{\CC}{\mathcal{C}}
\newcommand{\CD}{\mathcal{D}}
\newcommand{\CE}{\mathcal{E}}
\newcommand{\CF}{\mathcal{F}}
\newcommand{\CG}{\mathcal{G}}
\newcommand{\CH}{\mathcal{H}}
\newcommand{\CI}{\mathcal{I}}
\newcommand{\CJ}{\mathcal{J}}
\newcommand{\CK}{\mathcal{K}}
\newcommand{\CL}{\mathcal{L}}
\newcommand{\CM}{\mathcal{M}}
\newcommand{\CN}{\mathcal{N}}

\newcommand{\CP}{\mathcal{P}}

\newcommand{\CR}{\mathcal{R}}
\newcommand{\CS}{\mathcal{S}}
\newcommand{\CT}{\mathcal{T}}
\newcommand{\CU}{\mathcal{U}}

\newcommand{\Fa}{\mathfrak{a}}
\newcommand{\Fb}{\mathfrak{b}}
\newcommand{\Fc}{\mathfrak{c}}
\newcommand{\Fd}{\mathfrak{d}}
\newcommand{\Fe}{\mathfrak{e}}

\newcommand{\Fg}{\mathfrak{g}}

\newcommand{\Fm}{\mathfrak{m}}
\newcommand{\Fn}{\mathfrak{n}}

\newcommand{\Fp}{\mathfrak{p}}
\newcommand{\Fq}{\mathfrak{q}}
\newcommand{\Fr}{\mathfrak{r}}

\newcommand{\SSF}{\mathsf{F}}
\newcommand{\SSG}{\mathsf{G}}
\newcommand{\SSH}{\mathsf{H}}

\newcommand{\SSK}{\mathsf{K}}

\newcommand{\SSN}{\mathsf{N}}

\newcommand{\SSR}{\mathsf{R}}

\newcommand{\bdry}{{\BR_0^3}}

\newcommand{\lhs}{{\BR_-^3}}

\newcommand{\al}{\alpha}

\newcommand{\ga}{\gamma}
\newcommand{\de}{\delta}
\newcommand{\ep}{\varepsilon}
\newcommand{\te}{\theta}
\newcommand{\ka}{\kappa}
\newcommand{\la}{\lambda}
\newcommand{\si}{\sigma}
\newcommand{\ph}{\varphi}

\newcommand{\Ga}{\Gamma}
\newcommand{\De}{\Delta}
\newcommand{\Te}{\Theta}
\newcommand{\La}{\Lambda}
\newcommand{\Si}{\Sigma}
\newcommand{\Om}{\Omega}

\numberwithin{equation}{section} 
\allowdisplaybreaks[4]

\newcommand{\pr}{\prime}

\begin{document}
\title[Free boundary problem for the Navier-Stokes system]
{On the global wellposedness of free boundary 
problem for the Navier-Stokes system with surface tension}

\author[H. Saito]{Hirokazu Saito}
\address[H. Saito]{Graduate School of Informatics and Engineering,
The University of Electro-Communications,
5-1 Chofugaoka 1-chome, Chofu, Tokyo 182-8585, Japan}
\email{hsaito@uec.ac.jp}


\author[Y. Shibata]{Yoshihiro Shibata}
\address[Y. Shibata]{Professor Emeritus of Waseda University;
adjunct faculty member
in the Department of Mechanical Engineering and Materials Science, University of Pittsburgh}
\email{yshibata325@gmail.com}




\thanks{This work was supported by JSPS KAKENHI Grant Numbers JP17K14224,  JP17H01097.}



\begin{abstract}
The aim of this paper is to show the global wellposedness of the Navier-Stokes equations,
including surface tension and gravity, with a free surface in an unbounded domain such as bottomless ocean.
In addition, it is proved that the solution decays polynomially as time $t$ tends to infinity.
To show these results, we first use the Hanzawa transformation
in order to reduce the problem in a time-dependent domain $\Omega_t\subset\mathbf{R}^3$, $t>0$,
to a problem in the lower half-space $\mathbf{R}_-^3$.
We then establish some time-weighted estimate of solutions, 
in an $L_p$-in-time and $L_q$-in-space setting, for the linearized problem around the trivial steady state
with the help of $L_r\text{-}L_s$ time decay estimates of semigroup.
Next, the time-weighted estimate,
combined with the contraction mapping principle, shows that
the transformed problem in $\mathbf{R}_-^3$ admits 
a global-in-time solution in the $L_p\text{-}L_q$ setting and that the solution decays polynomially as time $t$ tends to infinity
under the assumption that $p$, $q$ satisfy the conditions:
$2<p<\infty$, $3<q<16/5$, and $(2/p)+(3/q)<1$.
Finally, we apply the inverse transformation of Hanzawa's one to the solution in $\mathbf{R}_-^3$
to prove our main results mentioned above for the original problem in $\Omega_t$.
Here we want to emphasize that
it is not allowed to take $p=q$ in the above assumption about $p$, $q$,
which means that the different exponents $p$, $q$
of $L_p\text{-}L_q$ setting play an essential role in our approach.
\end{abstract}

\maketitle



\section{Introduction}\label{sec1}
\subsection{Introduction}\label{subsec1_1}
This paper is concerned with the global wellposedness and large-time behavior
of solutions for the Navier-Stokes equations,
including surface tension and gravity,
with a free surface in an unbounded domain such as bottomless ocean.
This system describes the motion of a viscous incompressible fluid 
in $\Om_t=\{x=(x_1,x_2,x_3) \mid x' = (x_1,x_2)\in\BR^2, x_3<h(x',t)\}$
for $t>0$ and for an unknown scalar function $h=h(x',t)$.
To be precise, the problem is stated as follows:
We are given a scalar function $h_0=h_0(x')$ as well as 
an initial velocity field $\Bu_0=\Bu_0(x)=\tp(u_{01}(x),u_{02}(x),u_{03}(x))$ of the fluid,
where $\tp\BM$ denotes the transpose of matrix $\BM$. 
Let $\Om_0=\{x=(x_1,x_2,x_3) \mid x'=(x_1,x_2)\in\BR^2, x_3<h_0(x') \}$,
which is a region filled with the fluid at $t=0$.
We wish to find for each $t\in(0,\infty)$
a velocity field $\Bu=\Bu(x,t)=\tp(u_1(x,t),u_2(x,t),u_3(x,t))$ of the fluid,
a pressure field $\Fp=\Fp(x,t)$ of the fluid,
and a scalar function $h=h(x',t)$ so that
\begin{equation}\label{NS1}
\left\{\begin{aligned}
\rho\left(\pa_t\Bu+(\Bu\cdot\nabla)\Bu\right)
&= \Di\BT(\Bu,\Fp), && x\in\Om_t, \\
\di\Bu &= 0, && x\in\Om_t, \\
\BT(\Bu,\Fp)\Bn_t &= (c_\si \ka_t-c_g h -\Fp_0)\Bn_t, && x\in\Ga_t, \\
\pa_t h - u_3 &= -u_1\pa_1 h-u_2\pa_2 h, && x\in\Ga_t, \\
\Bu|_{t=0} &= \Bu_0, && x\in\Om_0,  \\
h|_{t=0} &= h_0, && x'\in \BR^2, 
\end{aligned}\right.
\end{equation}
together with
\begin{align}
\Om_t &=\{x=(x_1,x_2,x_3) \mid x'=(x_1,x_2)\in\BR^2, x_3<h(x',t)\}, \label{Omega_t} \\
\Ga_t &= \{x=(x_1,x_2,x_3) \mid x'=(x_1,x_2)\in\BR^2, x_3=h(x',t)\}. \label{Gamma_t}
\end{align}

Here, $\rho$, $c_g$, and $c_\si$ are positive constants
describing the density of the fluid, the gravitational acceleration,
and the surface tension coefficient of the fluid, respectively;
$\Fp_0$ is a constant that denotes the atmospheric pressure; 
$\BT(\Bu,\Fp)=\mu\BD(\Bu) -\Fp\BI$ is the stress tensor,
where $\mu$ is a positive constant describing the viscosity coefficient of the fluid,
$\BI$ the $N\times N$ identity matrix,
and $\BD(\Bu)=\nabla\Bu+\tp\nabla\Bu$ the doubled deformation tensor;
$\Bn_t$ is the unit outer normal to $\Ga_t$;
$\ka_t$ denotes the doubled mean curvature of $\Ga_t$ that is negative when $\Om_t$ is convex
in a neighborhood of $x\in\Ga_t$.
Here and subsequently, we use the following symbols for differentiations:
Let $f=f(x)$, $\Bg=\tp(g_1(x),g_2(x),g_3(x))$, and $\BM=(M_{ij}(x))$ be
a scalar-, a vector-, and a $3\times 3$ matrix-valued function 
defined in a domain of $\BR^3$.
Then, for $\pa_j = \pa/\pa x_j$, 
\begin{align*}
&\nabla f = \tp(\pa_1 f,\pa_2 f,\pa_3 f), \quad
\De f = \sum_{j=1}^3\pa_j^2 f, \quad
\De \Bg = \tp(\De g_1,\De g_2, \De g_3),  \\
&\di\Bg = \sum_{j=1}^3\pa_j g_j,
~\nabla\Bg = 
\begin{pmatrix}
\pa_1 g_1 & \pa_2 g_1 & \pa_3 g_1 \\
\pa_1 g_2 & \pa_2 g_2 & \pa_3 g_2 \\
\pa_1 g_3 & \pa_2 g_3 & \pa_3 g_3
\end{pmatrix},
~\nabla^2 \Bg = \{\pa_j\pa_k g_l \mid j,k,l=1,2,3\}, \\
&(\Bg\cdot\nabla)\Bg = \tp\bigg(\sum_{j=1}^3g_j\pa_j g_1,
\sum_{j=1}^3g_j\pa_j g_2,\sum_{j=1}^3g_j\pa_j g_3\bigg), \\
&\Di\BM = \tp\bigg(\sum_{j=1}^3 \pa_j M_{1j},
\sum_{j=1}^3 \pa_j M_{2j},\sum_{j=1}^3 \pa_j M_{3j}\bigg).
\end{align*}
It especially holds that
\begin{equation}\label{DivT}
	\Di\BT(\Bu,\Fp) = \mu (\De\Bu + \nabla\di\Bu)-\nabla\Fp.
\end{equation}

The key idea of this paper is to use an $L_p$-in-time and $L_q$-in-space setting
with different exponents $p$, $q$ in order to show the global wellposedness
and large-time behavior of solutions to System \eqref{NS1}.
The advantage of the different exponents is that
we can choose a $p$ large enough for a fixed $q>3$
to guarantee the condition $(2/p)+(3/q)<1$,
which yields some important properties such as
embeddings of function spaces (cf. Lemma \ref{lemm:embed} below).
In fact, we choose $p$, $q$ as follows:
First, we take a $q\in (3,16/5)$ so that
\begin{equation}\label{160518_1}
	\al> 1, \quad
	\|\Bu(t)\|_q\|\nabla\Bu(t)\|_q = O(t^{-\al}) \quad\text{as $t\to\infty$,}
\end{equation}
which enables us to control the convection term $(\Bu\cdot\nabla)\Bu$
in the fixed point argument.
Next, we choose a $p\in(2,\infty)$ large enough to satisfy $(2/p)+(3/q)<1$.
Consequently, we have the global $L_p$-integrability
of e.g. $\|\Bu(t)\|_q$, $\|\nabla\Bu(t)\|_q$ with respect to time $t$,
because they decay as time $t$ tends to infinity
and $p$ is a sufficiently large positive number.

The above choice of $q$ in \eqref{160518_1}
plays an essential role when solutions decay {\it polynomially}
although it is not so important when solutions decay {\it exponentially}.
In other words, a suitable choice of $p$, $q$ is crucial
in the case of {\it unbounded domains},
because we can only expect polynomial decay of solutions in such cases.

One of main tasks in the $L_p\text{-}L_q$ approach
is to establish a time-weighted estimate of solutions 
for the linearized problem, associated with System \eqref{NS1},
around the trivial steady state $(\Bu,\Fp, h)=(0,\Fp_0,0)$
under the assumption that
\begin{equation}\label{pq:lin}
	2<p<\infty, \quad 3<q<4, \quad p\left(\frac{2}{q}-\frac{1}{2}\right)>1, \quad \frac{2}{p}+\frac{3}{q}<1
\end{equation}
(cf. Theorem \ref{theo:main1} below for details).
The restriction $q<4$ arises from $L_r\text{-}L_s$ estimates
of analytic semigroup,
with $1 < s \leq 2 \leq r < \infty$ and $(r,s)\neq (2,2)$,
associated with the linearized problem\footnote{For more details of the $L_r\text{-}L_s$ estimates,
we refer to \cite[Chapter 3]{Saito15}, \cite{SaS1}.}.
In fact, we make use of them for $(r,s)=(q,q/2)$ and for $(r,s)=(2,q/2)$ in the present paper,
which requires the above restriction $q<4$.
On the other hand, the condition $p(2/q-1/2)>1$ guarantees 
the global $L_p$-integrability of $\|h(t)\|_{L_2(\BR^2)}$ with respect to time $t$. 
Then, by the time-weighted estimate and
the contraction mapping principle,
we have 
a global-in-time solution of System \eqref{NS1} in the $L_p\text{-}L_q$ setting
under the assumption\footnote{The condition \eqref{pq} below implies \eqref{pq:lin}.} that
\begin{equation}\label{pq}
	2<p<\infty, \quad 3<q<\frac{16}{5}, \quad \frac{2}{p}+\frac{3}{q}<1
\end{equation}
(cf. 
Theorems \ref{theo:main2} and \ref{theo:main3} below for details),
where the further restriction $q<16/5$ is imposed to 
guarantee the condition \eqref{160518_1}.  
Furthermore, the solution decays polynomially as time $t$ tends to infinity
(cf. Theorem \ref{theo:main4} below for details).

Here, we want to emphasize that
it is not allowed to take $p=q$ in \eqref{pq} (and also in \eqref{pq:lin}),
which suggests to us that
$L_p$-settings in both time and space are not appropriate
to proving the global wellposedness for System \eqref{NS1}.

\subsection{Summary of main results}\label{subsec1_2}
In this subsection, we sketch out main results of this paper.

The symbol $\BN$ stands for the set of all natural numbers,
and $\BN_0=\BN\cup\{0\}$.
Let $1 \leq q \leq \infty$ and $m\in\BN$,
and let $\Om$ be a domain of $\BR^n$ for $n=2$ or $n=3$.
Then, $L_q(\Om)$ and $W_q^m(\Om)$ denote
the Lebesgue spaces and the Sobolev spaces on $\Om$, respectively,
and also their norms are written as
$\|\cdot\|_{L_q(\Om)}$ and $\|\cdot\|_{W_q^m(\Om)}$, respectively.
We set $W_q^0(\Om) = L_q(\Om)$.
In addition, $W_q^s(\Om)$ denotes the Sobolev-Slobodecki\u\i{} spaces  on $\Om$,
for $1 \leq q < \infty$ and $s\in(0,\infty)\setminus\BN$, with norm
\begin{equation*}
	\|f\|_{W_q^s(\Om)} =
		\|f\|_{W_q^{[s]}(\Om)} + \sum_{|\al|=[s]}
		\left\{\int_\Om\int_\Om
		\left(\frac{|\pa_x^\al f(x)-\pa_y^\al f(y)|}{|x-y|^{s-[s]}}\right)^q
		\frac{\intd x \intd y}{|x-y|^n}\right\}^{1/q},
\end{equation*}
where $[s]$ is the largest integer less than $s$ and
\begin{equation*}
	\pa_x^\al f(x) =
		\frac{\pa^{|\al|}}{\pa x_1^{\al_1}\dots\pa x_n^{\al_n}}f(x)
		\quad \text{for $\al=(\al_1,\dots,\al_n)\in \BN_0^n$, $|\al|=\al_1+\dots+\al_n$.}
\end{equation*}
Let $(\cdot\,,\cdot)_{\te,p}$ be the real interpolation functor
with $0<\te<1$ and $1<p<\infty$ (cf. e.g. \cite[Definition 1.37]{DK13}).
For $1<p,q<\infty$ and $s>0$,
following \cite[Section 34]{Tartar07},
we define the Besov spaces $B_{q,p}^s(\Om)$ as
\begin{equation*}
	B_{q,p}^s(\Om) =
		(W_q^{s_1}(\Om),W_q^{s_2}(\Om))_{\te,p}
		\quad \text{with $s=(1-\te)s_1+\te s_2$,}
\end{equation*}
where $s_1$, $s_2$, and $\te$ are non-negative real numbers
satisfying $s_1\neq s_2$ and $0<\te<1$.
To treat the initial data $\Bu_0$, $h_0$ of \eqref{NS1},
we introduce the following function spaces:
For $p$, $q$ satisfying \eqref{pq} and $\bar{q}=q/2$,
\begin{align*}
	&\BBI_{q,p} =
		B_{q,p}^{3-(1/p)-(1/q)}(\BR^2)
		\cap B_{2,p}^{3-(1/p)-(1/2)}(\BR^2) \cap L_{\bar{q}}(\BR^2), \\
	&\BBJ_{q,p,\te}(h_0) =  
		B_{q,p}^{2-(2/p)}(\Om_0)^3 \cap B_{q(\te),p}^{2-(2/p)}(\Om_0)^3 
		\quad \text{for $h_0\in\BBI_{q,p}$,}
\end{align*}
where $q(\te)$ is the dual exponent of $2(1+1/(1-\te))$, i.e.
\begin{equation}\label{q_theta}
	q(\te) = 2\left(1-\frac{1}{3-\te}\right), \quad 0\leq \te\leq 1.
\end{equation}
Here and subsequently, $X^m$ stands for the $m$-product of Banach space $X$ with norm $\|\cdot\|_X$,
and we also denote the norm of $X^m$ by $\|\cdot\|_X$ for short, i.e. 
\begin{equation*}
	\|\Bf\|_X = \sum_{j=1}^m\|f_j\|_X \quad \text{for $\Bf=\tp(f_1,\dots,f_m)\in X^m$.}
\end{equation*}
In what follows,
we denote $(a/p)\pm(b/q)$ by $a/p\pm b/q$, respectively,
for any real numbers $a,b,p,q$ if there is no confusion.

Then, the main results of this paper are summarized as follows
(cf. Subsection \ref{subsec2_4} below for details):

\begin{theo}[global wellposedness]\label{theo:sec1_1}
Let $p$, $q$ be exponents satisfying \eqref{pq} and $0<\te< 1$.
For any initial data $h_0\in\BBI_{q,p}$ and $\Bu_0\in\BBJ_{q,p,\te}(h_0)$,
satisfying a smallness condition and compatibility conditions,
System \eqref{NS1} admits a unique global-in-time  solution $(\Bu,\Fp,h)$.
\end{theo}

\begin{theo}[large-time behavior]\label{theo:sec1_2}
Let $p$, $q$ be exponents satisfying \eqref{pq} and $\bar{q}=q/2$,
and let $0<\te< 1$.
Suppose that $(\Bu,\Fp,h)$ is the solution of System \eqref{NS1},
with the initial data $h_0\in \BBI_{q,p}$ and $\Bu_0\in\BBJ_{q,p,\te}(h_0)$,
obtained in Theorem $\ref{theo:sec1_1}$.
For $1 \leq s_1\leq 2 \leq s_2 <\infty$, we set
\begin{equation}\label{mn}
	 \Fm(s_1,s_2) = 
	 	\left(\frac{1}{s_1}-\frac{1}{s_2}\right)
	 	+ \frac{1}{2}\left(\frac{1}{2}-\frac{1}{s_2}\right), \quad
	\Fn(s_1,s_2) =
		\frac{3}{2}\left(\frac{1}{s_1}-\frac{1}{s_2}\right).
\end{equation}
It then holds that, for $2\leq r\leq q$,
\begin{alignat*}{3}
	&\|\Bu(t)\|_{L_r(\Om_t)}
		= O\left(t^{-\Fm(\bar{q},r)}\right) , \quad
	&&\|\nabla\Bu(t)\|_{L_r(\Om_t)}
		= O\left(t^{-\Fn\left(\bar{q},r\right)-\frac{1}{8}}\right), \quad
	&&\|h(t)\|_{L_r(\BR^2)}
		= O\left(t^{-\left(\frac{1}{\bar{q}}-\frac{1}{r}\right)}\right), \\
	&\|\nabla' h(t)\|_{L_r(\BR^2)}
		= O\left(t^{-\Fm\left(\bar{q},r\right)-\frac{1}{4}}\right), \quad
	&&\|\pa_t h(t)\|_{L_r(\BR^2)}
		= O\left(t^{-\Fm\left(\bar{q},r\right)}\right),
\end{alignat*}
as time $t$ tends to infinity, where $\nabla'h = \tp(\pa_1h,\pa_2h)$.
\end{theo}

\subsection{Previous researches and our approach}\label{subsec1_3}
In what follows, an $L_p\text{-}L_q$ setting means
a function space based on the $L_p$-in-time and $L_q$-in-space framework.

Beale \cite{Beale81} is the pioneer who treats mathematically
the Navier-Stokes equations with a free surface
in an unbounded domain such as ocean.
To be precise, he considered, for a given smooth function $b=b(x')$,
the case where
\begin{equation*}
	\Om_t = \{(x',x_3) \mid x'\in\BR^2, -b(x')<x_3<h(x',t)\}
\end{equation*}
instead of \eqref{Omega_t},
together with a further boundary condition: $\Bu = 0$
on the solid surface
$S = \{(x',x_3) \mid x'\in\BR^2, x_3=-b(x')\}$.
Here note in \cite{Beale81} that surface tension is not taken into account
although gravity works.
He first introduced the so-called Lagrangian transformation
to transform the time-dependent domain $\Om_t$
into a fixed domain $\Om$ independent of time $t$.
Then, for the transformed problem in $\Om$,
resolvent estimates of the Stokes operator
were investigated in the $L_2$-framework.
Next, the resolvent estimates, combined with
a Helmholtz-type decomposition and
the Fourier transform with respect to time $t$,
yield estimates of solutions to a linear time-dependent problem
with homogeneous (boundary and initial) data
and with divergence free condition.
Furthermore, the last estimates were extended to
the case where the problem is fully inhomogeneous.
Thus, the contraction mapping principle,
together with the above linear analysis,
proved the local existence of solutions
to the transformed problem in $\Om$.
Finally, applying the inverse transformation of Lagrangian one to
the solutions showed in an $L_2\text{-}L_2$ setting 
the local existence of solutions to the original problem in $\Om_t$.

Another paper \cite{Beale83} due to Beale proved
the global existence and uniqueness of solutions
to the Navier-Stokes equations,
including both surface tension and gravity, with a free surface
in the same domain as in \cite{Beale81}. 
The result is also in an $L_2\text{-}L_2$ setting.
He pointed out in the paper that
surface tension plays an essential role to show the global existence of solutions 
(cf. also \cite[page 363]{Beale81}),
and used a transformation different from \cite{Beale81}
in order to derive the fixed domain problem.

There is a study due to Sylvester \cite{Sylvester90}
in a manner similar to Beale's work \cite{Beale83}.
The paper treated the same problem as in \cite{Beale81},
i.e. without surface tension,
and proved the global existence and uniqueness of solutions 
in an $L_2\text{-}L_2$ setting.

Allain \cite{Allain87} proved in an $L_2\text{-}L_2$ setting
the local existence of solutions to the same problem as in \cite{Beale83},
but it is in two space dimensions.
He used a technique different from Beale's one in order to analyze linearized problems.
More precisely,
he constructed a weak solution of the linearized problem 
by using a Galerkin approximation,
and it was proved that
the weak solution satisfies in a strong sense the linearized problem 
if given data are sufficiently regular.
After a while, the result of local existence 
was extended to three space dimensions by Tani \cite{Tani96}.

Tani and Tanaka \cite{TT95} considered both situations (with and without surface tension),
and proved in the Lagrangian coordinates 
the global wellposedness 
by using a continuation argument of local solutions based on a priori estimates
in an $L_2\text{-}L_2$ setting.
They relaxed regularity assumptions about the solid surface $S$ and
initial data $\Bu_0$, $h_0$
compared to \cite{Beale83}, \cite{Sylvester90}.

Bae \cite{Bae11} proved the global existence of unique solutions to
the same problems as in \cite{Beale83} with flat bottom surface
(i.e. $b$ is a positive constant) 
by introducing a priori estimates on the time-dependent domain $\Om_t$.
The result is in an $L_2\text{-}L_2$ setting,
whereas his solutions preserve initial regularity in contrast to \cite{Beale83}, \cite{TT95}.

Beale and Nishida \cite{BN85} showed
large-time behavior of the solutions obtained in \cite{Beale83}
in the case where the bottom surface $S$ is flat. 
In such a case, by using the partial Fourier transform
with respect to the tangential variables $x'=(x_1,x_2)$,
one can calculate in the Fourier space
asymptotic expansion of eigenvalues of the Stokes operator
with respect to the dual variables $\xi'=(\xi_1,\xi_2)$ of $x'$.
By the expansion formula, they obtained decay properties of the Stokes semigroup
generated by the Stokes operator (cf. also Hataya \cite{Hataya11} for details),
and proved the large-time behavior mentioned above.
On the other hand, Hataya and Kawashima \cite{HK09} proved
large-time behavior of solutions to the same problem as in \cite{Beale81}
with flat bottom surface
in addition to the global existence of solutions in an $L_2\text{-}L_2$ setting.

Concerning large-time behavior of solutions,
there is also a series of papers due to Guo and Tice \cite{GT13a, GT13b}.
They concerned the case without surface tension
and proved polynomial decay of solutions in an $L_2$-setting
by introducing a two-tier energy method
that couples the boundedness of high-order energy to the decay of low-order energy.
In addition, this technique is applied to the viscous surface internal wave problem,
i.e. the two-phase problem of Navier-Stokes equations, in \cite{WTK14}.
We can find in \cite{WT12}
a nice table that summarizes stability characterization for the two-phase problem.

We next introduce results in $L_p\text{-}L_p$ settings 
or in $L_p\text{-}L_q$ settings more generally.
To this end, we here summarize one of typical approaches in
such settings as follows:
First, by some transformation,
the problem on the time-dependent domain $\Om_t$
is reduced to a fixed domain problem.
Next, we establish the maximal regularity theorem
for the linearized problem associated with the fixed domain problem.
Then, the contraction mapping principle,
combined with the maximal regularity theorem,
yields the local wellposedness 
of the fixed domain problem
in the maximal regularity framework.
Finally, applying the inverse transformation to the solutions
solves the original problem in $\Om_t$.
The following papers are also based on this approach.

Abels \cite{Abels05} proved, in the Lagrangian coordinates,
the local wellposedness 
of the same problem as in \cite{Beale81}
in the $L_p\text{-}L_p$ setting for $N<p<\infty$,
where $N\geq 2$ denotes the space dimension.
Especially, the maximal regularity theorem is stated in \cite[Theorem 3.2]{Abels05},
and is extended to the case of different exponents $p$, $q$ by Saito \cite{Saito15b}.

Pr$\ddot{{\rm u}}$ss and Simonett \cite{PS11} 
considered two-phase problems of the Navier-Stokes equations
with surface tension and gravity.
As a special case, the paper contains our one-phase problem \eqref{NS1}.
They proved the local wellposedness 
in the $L_p\text{-}L_p$ setting for $p>N+2$,
and proved the solutions are real analytic.
The result is extended to the case of non-Newtonian fluids by Hieber and Saito \cite{HS16}.
In \cite{PS11}, the Hanzawa transformation (cf. e.g. \cite{Hanzawa81}, \cite[Subsection 1.3.2]{PS16})
is introduced in order to
transform the time-dependent domain $\Om_t$ to a fixed domain independent of time $t$.
We also use the Hanzawa transformation in the present paper
to obtain a fixed domain problem.

Shibata \cite{Shibata16} proved, in the Lagrangian coordinates,
the local wellposedness 
of the Navier-Stokes equations with a free surface,  
including both surface tension and gravity,
in general domains 
in the $L_p\text{-}L_q$ setting
for $2<p<\infty$ and $N<q<\infty$.
The definition of general domains is given by \cite[Definitions 1.1-1.3]{Shibata16},
and our situation of this paper gives an example of them.

As long as we know, 
there has been no study of the global wellposedness of solutions to System \eqref{NS1}.
This motivated us to research
\eqref{NS1} from the viewpoint of 
the global wellposedness and large-time behavior of solutions.
Our approach is also based on the standard theory of
$L_p\text{-}L_q$ setting mentioned above,
but we additionally prove a time-weighted estimate of solutions
with the help of 
$L_r\text{-}L_s$ estimates of semigroup
associated with the linearized problem.
That estimate enables us to
show the global wellposedness 
of System \eqref{NS1}
in the $L_p\text{-}L_q$ setting with exponents $p$, $q$
satisfying \eqref{pq}
and to show large-time behavior of the solutions,
as was discussed in Subsection \ref{subsec1_1}.
This is completely different from all of the above approaches.

The next section first introduces the notation and function spaces 
that are used throughout this paper.
Secondly, we reduce System \eqref{NS1} to a problem in
the lower half-space $\lhs=\{x_3<0\}$
by using the Hanzawa transformation.
Thirdly, we introduce the definition of solutions System \eqref{NS1}. 
Fourthly, the main results of this paper are stated.
The last part of Section 2 gives the outline of this paper.

\section{Notation and main results}\label{sec2}
\subsection{Notation and function spaces}\label{subsec2_1}
In this subsection, we introduce the notation and function spaces that used throughout this paper.

Let $m,n\in\BN$ and $\Om$ be a domain of $\BR^n$.
For any $m$-vectors $\Ba = \tp(a_1,\dots,a_m)$, $\Bb=\tp(b_1,\dots,b_m)$,
we set $\Ba\cdot\Bb=\sum_{j=1}^m a_j b_j$ and
\begin{equation*}
	\Ba\otimes\Bb =
		\begin{pmatrix}
			a_1b_1 & \dots & a_1b_m \\
			\vdots & \ddots & \vdots \\
			a_m b_1 & \dots & a_m b_m 
		\end{pmatrix}.
\end{equation*}
On the other hand, for any $m$-vector functions $\Bf=\Bf(x)$ and $\Bg=\Bg(x)$ on $\Om$,
we set $(\Bf,\Bg)_\Om=\int_\Om \Bf(x)\cdot\Bg(x)\intd x$.

Let $X$ and $Y$ be Banach spaces with norms $\|\cdot\|_X$ and $\|\cdot\|_Y$, respectively.
Then, $\CL(X,Y)$ is the Banach space of all bounded linear operators from $X$ to $Y$,
and $\CL(X)=\CL(X,X)$. 
The symbol $X\hookrightarrow Y$ means that $X$ is continuously embedded into $Y$,
i.e. there is a positive constant $C$ such that
	$\|f\|_Y\leq C \|f\|_X$ for all $f\in X$.
For $1\leq p\leq \infty$, $L_p(\Om,X)$ and $W_p^m(\Om,X)$ 
denote the standard $X$-valued Lebesgue and Sobolev spaces on $\Om$, respectively.
Let $C(\Om,X)$ be the vector space of all $X$-valued continuous functions on $\Om$,
and let $BUC(\Om,X)$ be the Banach space of all $X$-valued uniformly continuous and bounded functions on $\Om$. 
In addition,
\begin{align*}
	C^m(\Om,X) 
		&=\{f\in C(\Om,X) \mid \pa_x^\al f(x) \in C(\Om,X) \text{ for $1\leq |\al| \leq m$}\}, \\
	BUC^m(\Om,X) 
		&=\{f\in BUC(\Om,X) \mid \pa_x^\al f(x) \in BUC(\Om,X) \text{ for $1\leq |\al| \leq m$}\}.
\end{align*}
If $X=\BR$ or $X=\BC$, then we denote $C(\Om,X)$, $BUC(\Om,X)$, $C^m(\Om,X)$, and $BUC^m(\Om,X)$
by $C(\Om)$, $BUC(\Om)$, $C^m(\Om)$, and $BUC(\Om)$, respectively, for simplicity.
On the other hand, 
$C_0^\infty(\Om)$ stands for the vector space of all $C^\infty$-functions $f$
such that $\supp f$ is compact and $\supp f\subset\Om$.

For $1<q<\infty$, let $L_{q,\loc}(\overline{\Om})$ be the vector space of all measurable functions $f$ on $\Om$
such that $f\in L_q(\Om\cap B)$ for any ball $B$ of $\BR^n$.
We then define $\wh W_q^1(\Om)$ as
\begin{equation*}
	\wh W_q^1(\Om) =
		\{f\in L_{q,\loc}(\overline{\Om}) \mid \|\nabla f\|_{L_q(\Om)}<\infty\}
\end{equation*}
with semi-norm $\|\cdot\|_{\wh W_q^1(\Om)}=\|\nabla\cdot\|_{L_q(\Om)}$.
Inductively, we set, for $l=2,3$, 
\begin{align*}
	&\wh W_q^l(\Om) = \{f\in \wh W_q^{l-1}(\Om) \mid \|\nabla^l f\|_{L_q(\Om)}<\infty\}, \quad
	\|\cdot\|_{\wh W_q^l(\Om)} 
		= \|\nabla \cdot\|_{L_q(\Om)} + \dots +\|\nabla^l\cdot\|_{L_q(\Om)}.
\end{align*}
In addition, $W_{q,0}^1(\lhs)$, $\wh {W}_{q,0}^1(\lhs)$, and $J_q(\lhs)$ are defined as
\begin{align*}
	&W_{q,0}^1(\lhs) = \{f\in W_q^1(\lhs) \mid f =0 \text{ on $\bdry$}\}, \quad
	\wh W_{q,0}^1(\lhs) = \{f\in \wh W_q^1(\lhs) \mid f =0 \text{ on $\bdry$}\}, \\
	&J_q(\lhs) = \{\Bf\in L_q(\lhs)^3 \mid (\Bf,\nabla\ph)_\lhs=0 \text{ for all $\ph\in \wh W_{q',0}^1(\lhs)$}\},
\end{align*}
where $q'=q/(q-1)$ and $\bdry=\{x_3=0\}$.
Furthermore,  we denote the dual space of $\wh W_{q',0}^1(\lhs)$ by $\wh W_q^{-1}(\lhs)$.

As solution spaces, we set, for an interval $I\subset \BR$,
\begin{equation*}
		W_{q,p}^{2,1}(\Om\times I) = W_p^1(I,L_q(\Om))\cap L_p(I, W_q^2(\Om)).
\end{equation*}
In addition, Bessel potential spaces are defined as follows:
Let 
\begin{equation}\label{160920_2}
	\CF_t[u](\tau) = \int_\BR e^{-i t \tau} u(t)\intd t, \quad
	\CF_\tau^{-1}[v](t)=\frac{1}{2\pi}\int_\BR e^{i t\tau} v(\tau)\intd \tau,
\end{equation}
and let $[\cdot,\cdot]_\te$ be the complex interpolation functor with $\te\in(0,1)$ (cf. e.g. \cite[Definition 1.38]{DK13}).
Then, for $1<p<\infty$ and $s\in\BR$,
\begin{align*}
	&H_p^s(\BR,X) = \{f\in \CS'(\BR,X) \mid \|f\|_{H_p^s(\BR,X)}<\infty\}, \\
	&\|f\|_{H_p^s(\BR,X)} = \|D_t^s f\|_{L_p(\BR,X)}, \quad (D_t^s f)(t) =  \CF_\tau^{-1}[(1+|\tau|^2)^{s/2}\CF_t[f](\tau)](t), 
\end{align*}
where $\CS'(\BR,X)=\CL(\CS(\BR),X)$ for the Schwartz space $S(\BR)$ of rapidly decreasing functions.
Furthermore, for $s\in(0,\infty)\setminus\BN$,
\begin{equation*}
	H_p^s(I,X) = [W_p^{[s]}(I,X), W_p^{[s]+1}(I,X)]_{s-[s]},
\end{equation*}
where $[s]$ is the largest integer smaller than $s$.
Note that the last two spaces coincide for $I=\BR$ (cf. e.g. \cite[Remark 1.57 (ii)]{DK13}).
We set
\begin{equation*}
	H_{q,p}^{1,1/2}(\Om\times I) = H_p^{1/2}(I,L_q(\Om))\cap L_p(I,W_q^1(\Om)).
\end{equation*}

Throughout this paper, the letter $C$ denotes generic constants and
$C_{a,b,c,\dots}$ means that the constant depends on the quantities $a,b,c,\dots$.
The values of $C$ and $C_{a,b,c,\dots}$ may change from line to line.

\subsection{Transformed problem in $\BR_-^3$}\label{subsec2_2}
In this subsection, we reduce System \eqref{NS1} to a problem in $\BR_-^3$
by using the Hanzawa transformation
under the assumption that there is a sufficiently regular solution $(\Bu,\Fp,h)$
of \eqref{NS1}-\eqref{Gamma_t} with a suitable initial data $(\Bu_0,h_0)$.
In what follows, we suppose that $\Fp_0=0$ and $\rho=1$ without loss of generality.

First, we define the harmonic extension of function $f=f(x')$
for $x'=(x_1,x_2)\in\BR^2$.
Let us denote the Fourier transform with respect to $x'$
and its inverse by
\begin{equation}\label{PFT}
	\wh{u}(\xi') 
		= \int_{\BR^2} e^{-i x'\cdot\xi'} u(x')\intd x', \quad
	\CF_{\xi'}^{-1}[v](\xi')
		= \frac{1}{(2\pi)^2}\int_{\BR^2}e^{ix'\cdot\xi'}v(\xi')\intd\xi',
\end{equation}
respectively. We then set, for $y=(y',y_3)=(y_1,y_2,y_3)\in\BR_-^3$,
\begin{equation}\label{har_ext}
	\CE(f)=[\CE f](y)
		= \CF_{\xi'}^{-1}\big[e^{|\xi'|y_3}\wh{f}(\xi')\big](y'),
\end{equation}
which is called the harmonic extension of $f$.
Then, $\Ga_t$ is represented as
\begin{align}\label{160601_1}
	\Ga_t 
		&= \{(y',h(y',t)) \mid y'=(y_1,y_2)\in\BR^2\}\\
		&= \{(y_1,y_2,y_3+[\CE h](y,t)) \mid y=(y_1,y_2,y_3)\in\BR_0^3\}
		 \quad (t>0) \notag.
\end{align}

Secondly, by using the harmonic extension,
we construct a suitable transformation. Let
\begin{equation}\label{Te_t}
	\Te_t(y)=(y_1,y_2,y_3+[\CE h](y,t)) \quad \text{for $y=(y_1,y_2,y_3)\in\BR_-^3$ and $t>0$.}
\end{equation}
If it holds that
\begin{equation*}
	\|\nabla\CE(h)\|_{L_\infty(\BR_+,L_\infty(\BR_-^3))}<c_0
\end{equation*}
for some positive constant $c_0\in(0,1/2)$, then
$\Te_t:\BR_-^3\to\Om_t$ is bijective for each $t>0$.
In fact, the existence of $c_0$ furnishes that
\begin{equation*}
	\frac{\pa x_3}{\pa y_3}
		= 1 + \frac{\pa}{\pa y_3} [\CE h](y,t)
		\geq 1-c_0 \geq \frac{1}{2},
\end{equation*}
which implies that
$x_3$ is an increasing function with respect to $y_3\in(-\infty,0)$.
This means that $\Te_t:\BR_-^3\to\Om_t$ is bijective
by the second representation of \eqref{160601_1}.
In what follows, we assume the existence  of $c_0$ in this subsection.
Thus, there is the inverse mapping $\Te_t^{-1}$ of $\Te_t$ from $\Om_t$ onto $\BR_-^3$.
Since the regularity of $h$ yields the smoothness of $\CE(h)$,
we see that $\Te_t^{-1}$ is smooth by the inverse function theorem.
Consequently, $\Te_t:\BR_-^3\to\Om_t$ is a diffeomorphism for each $t>0$.

Thirdly, we calculate derivatives associated with the change of variables $x=\Te_t(y)$.
Let $D_j = \pa/\pa y_j$ and $\pa_j=\pa/\pa x_j$ for $j=1,2,3$,
and let $\eta=\CE(h)$ for simplicity.
It then holds that
\begin{equation}\label{Jacobian}
	\frac{\pa x}{\pa y} = 
		\begin{pmatrix}
			1 & 0 & 0 \\
			0 & 1 & 0 \\
			D_1 \eta & D_2 \eta & 1+D_3 \eta  
		\end{pmatrix},  \quad
	\frac{\pa y}{\pa x}
		= \left(\frac{\pa x}{\pa y}\right)^{-1} = 
		\begin{pmatrix}
			1 & 0 & 0 \\
			0 & 1 & 0 \\
			-\frac{D_1\eta}{1+D_3\eta}
			& -\frac{D_2\eta}{1+D_3\eta}
			& \frac{1}{1+D_3\eta}
		\end{pmatrix}, 
\end{equation}
where $\pa x/\pa y$ and $\pa y/\pa x$ are Jacobian matrices.
For a sufficiently regular function $f(x,t)$, $x\in\Om_t$ and $t>0$,
we set $\bar{f}(y,t)=f(\Te_t(y),t)$.
Since $f(x,t)=\bar{f}(\Te_t^{-1}(x),t)$,
we see by the chain rule and \eqref{Jacobian} that
\begin{equation}\label{deriv}
	\pa_j f = 
		\left(D_j-\left(\frac{D_j \eta}{1+D_3  \eta}\right)D_3\right)\bar{f}, \quad
	\pa_j \pa_k f =
		\left(D_j D_k - \CD_{jk}(\eta)\right)\bar{f},
\end{equation}
where $\CD_{jk}(\eta)$ are second order differential operators defined as
\begin{align*}
	\CD_{jk}(\eta) =& 
		\frac{1}{(1+D_3 \eta)^3}
		\Big\{(D_j D_k\eta)(1+D_3\eta)^{2}-(D_k\eta)(D_j D_3\eta)(1+D_3\eta) \\
		&-(D_j\eta)(D_3 D_k\eta)(1+D_3\eta)+(D_j\eta)(D_k\eta)(D_3^2\eta)\Big\}D_3 \\
		&+\Big(\frac{D_k\eta}{1+D_3\eta}\Big)D_j D_3
		+\Big(\frac{D_j\eta}{1+D_3\eta}\Big)D_3 D_k
		-\frac{(D_j\eta)(D_k\eta)}{(1+D_3\eta)^2}D_3^2.
\end{align*}
In addition, it holds that
\begin{align*}
	\pa_t \bar{f}(y,t) &= \pa_t\{f(\Te_t(y),t)\}
		=  (\pa_t f)(\Te_t(y),t) + (\nabla f)(\Te_t(y),t) \cdot \pa_t\Te_t(y) \\
		&= (\pa_t f)(x,t) + (\pa_3 f)(x,t)\pa_t\eta(y,t),
\end{align*}
which, combined with the first identity of \eqref{deriv}, furnishes that
\begin{equation}\label{deriv:t}
	\pa_t f(x,t) = \pa_t \bar{f}(y,t)
		- \left(\frac{D_3\bar{f}(y,t)}{1+D_3\eta(y,t)}\right)\pa_t\eta(y,t).
\end{equation}

Fourthly, we derive a system of $\Bv=\Bv(y,t)=\Bu(\Te_t(y),t)$
and $\Fq=\Fq(y,t)=\Fp(\Te_t(y),t)$ from the system \eqref{NS1}.
Note that $h(x',t)=h(y',t)=\eta(y',0,t)$ and that,
by \eqref{har_ext}, \eqref{deriv}, and \eqref{deriv:t},
\begin{align}\label{deriv:h}
	&\pa_t h (x',t)=\pa_t h(y',t), \quad
	\pa_{x'}^{\al'}h(x',t) =D_{y'}^{\al'}h(y',t)=(D_{y'}^{\al'}\eta)(y',0,t)=D_{y'}^{\al'}(\eta(y',0,t)), \\
	&\text{where} \quad \pa_{x'}^{\al'}=\frac{\pa^{|\al'|}}{\pa x_1^{\al_1}\pa x_2^{\al_2}}, \quad
	D_{y'}^{\al'} = \frac{\pa^{|\al'|}}{\pa y_1^{\al_1}\pa y_2^{\al_2}} \notag
\end{align}
for any multi-index $\al'=(\al_1,\al_2)\in\BN_0^2$.
Set matrices $\BM_i(\eta)$ $(i=1,2,3)$ as 
\begin{equation*}
	\BM_1(\eta) =
		\begin{pmatrix}
			D_3 \eta & 0 & 0 \\
			0 & D_3 \eta & 0 \\
			-D_1 \eta & -D_2 \eta & 0
		\end{pmatrix}, \quad
	\BM_2(\eta) =
		\begin{pmatrix}
			0&0&-D_1 \eta \\
			0&0&-D_2 \eta \\
			0&0&0
		\end{pmatrix}, \quad
	\BM_3(\eta) =
		\begin{pmatrix}
			0 & 0 & D_1 \eta \\
			0 & 0 & D_2 \eta \\
			0 & 0 & D_3 \eta
		\end{pmatrix}.
\end{equation*}
We then have, by \eqref{deriv},
\begin{align}
	&\nabla_x =\frac{1}{1+D_3\eta}(\BI+\tp\BM_1(\eta))\nabla_y, \label{nabla} \\
	&{\rm div}_x\Bu
		= {\rm div}_y\Bv -\frac{\nabla_y\eta\cdot D_3 \Bv}{1+D_3\eta}
		= \frac{{\rm div}_y\left\{(\BI+\BM_1(\eta))\Bv\right\}}{1+D_3\eta}, \label{div} \\
	&\BD_x(\Bu)
		=\BD_y(\Bv)-\frac{1}{1+D_3\eta}\left\{(\nabla_y\eta\otimes D_3\Bv)
		+\tp(\nabla_y\eta\otimes D_3\Bv)\right\}, \label{Deformation}
\end{align}
where the subscripts $x,y$ denote derivatives of their coordinates.
Here the second identity of \eqref{div} is in fact derived as follows:
For any $\ph=\ph(x)\in C_0^\infty(\Om_t)$,
together with $\bar{\ph} = \bar{\ph}(y,t):=\ph(\Te_t(y))$ and \eqref{nabla},
\begin{align*}
	&(\di_x\Bu,\ph)_{\Om_t} = -(\Bu,\nabla_x\ph)_{\Om_t} 
		=-(\Bv,(\BI+\tp\BM_1(\eta))\nabla_y\bar{\ph})_{\BR_-^3} \\
		&=({\rm div}_y\{(\BI+\BM_1(\eta))\Bv\},\bar{\ph})_{\BR_-^3}
		=((1+D_3\eta)^{-1}{\rm div}_y\{(\BI+\BM_1(\eta))\Bv\},\ph)_{\Om_t}.
\end{align*}

{\bf Step 1: The first equation of \eqref{NS1}.}
The term $(\Bv,\Fq)$ is inserted into the first equation of \eqref{NS1},
and the resultant formula multiplied by
$\BI+\BM_3(\eta)=(1+D_3\eta)(\BI+\tp\BM_1(\eta))^{-1}$
from the left-hand side yields that,
by \eqref{DivT}, \eqref{deriv:t}, and \eqref{nabla},
\begin{equation*}
	\pa_t\Bv-{\rm Div}_y\BT(\Bv,\Fq) = \BF(\Bv,\eta),\quad \text{$y\in\BR_-^3$,}
\end{equation*}
where $\BT(\Bv,\Fq)=\mu\BD_y(\Bv)-\Fq\BI$ and we have set
\begin{align*}
	&\BF(\Bv,\eta) = \pa_t\Bv -(\BI+\BM_3(\eta))
		\left(\pa_t\Bv-\left(\frac{\pa_t\eta}{1+D_3\eta}\right)D_3\Bv\right) 
		-(\BI+\BM_3(\eta))\left((\Bv\cdot\nabla_y)\Bv 
		-\left(\frac{\Bv\cdot\nabla_y\eta}{1+D_3\eta}\right)D_3\Bv\right) \notag \\
		&-{\rm Div}_y(\mu\BD_y(\Bv))
		+\mu(\BI+\BM_3(\eta))\left(\De_y\Bv - \sum_{j=1}^3\CD_{jj}(\eta)\Bv\right) 
		-\mu\nabla_y{\rm div}_y\Bv+\mu\nabla_y\left({\rm div}_y\Bv - \frac{\nabla_y\eta\cdot D_3 \Bv}{1+D_3\eta}\right). \notag
\end{align*}

{\bf Step 2: The second equation of \eqref{NS1}.}
We insert $\Bv$ into the second equation of \eqref{NS1},
together with \eqref{div}, in order to obtain
\begin{equation*}
	{\rm div}_y\Bv = G(\Bv,\eta) = {\rm div}_y \BG(\Bv,\eta), \quad \text{$y\in\BR_-^3$,}
\end{equation*} 
where we have set
\begin{equation*}
	\BG(\Bv,\eta) = -\BM_1(\eta)\Bv, \quad
	G(\Bv,\eta) = 
		\nabla_y\eta\cdot D_3\Bv-(D_3\eta){\rm div}_y\Bv.
\end{equation*}

{\bf Step 3: The third equation of \eqref{NS1}.}
We first calculate $\Bn_t$ and $\ka_t$
based on the first representation of \eqref{160601_1}.
Let $\nabla' h(y',t) =  \tp(D_1 h(y',t), D_2 h(y',t))$,
$\De' h(y',t) =  \sum_{j=1}^2 D_j^2 h(y',t)$, and 
\begin{align*}
	\CK(h) &=
		\frac{|\nabla' h(y',t)|^2\De' h(y',t)}
		{(1+\sqrt{1+|\nabla' h(y',t)|^2})\sqrt{1+|\nabla' h(y',t)|^{2}}} 
		+\sum_{j,k=1}^{2}\frac{D_j h(y',t)D_k h(y',t)
		D_j D_k h(y',t)}{(1+|\nabla' h(y',t)|^2)^{3/2}}.
\end{align*}
It then holds by \eqref{deriv:h} that, for $y=(y',y_3)\in\BR_0^3$ and $\Be_3=\tp(0,0,1)$,
\begin{align}\label{n-ka}
	\Bn_t &= \frac{1}{\sqrt{1+|\nabla' h(y',t)|^2}}
		\begin{pmatrix}
			-\nabla' h(y',t) \\ 1
		\end{pmatrix}
		= \frac{1}{\sqrt{1+|\nabla' \eta(y,t)|^2}}
		\begin{pmatrix}
			-\nabla' \eta(y,t) \\ 1
		\end{pmatrix} \\
		&=\frac{1}{\sqrt{1+|\nabla' \eta(y,t)|^2}}
		\left(\BI+\BM_2(\eta)\right)\Be_3, \notag \\
	\ka_t &= \sum_{j=1}^2 D_ j\left(\frac{D_j h(y',t)}{\sqrt{1+|\nabla' h(y',t)|^2}}\right)
		=\De' h(y',t)-\CK(h)
		=\De' h(y',t) -\CK(\eta). \notag
\end{align}
The term $(\Bv,\Fq)$ is inserted into the third equation of \eqref{NS1},
and the resultant formula multiplied by
\begin{equation*}
	\sqrt{1+|\nabla' \eta(y,t)|^2}(\BI-\BM_2(\eta))
		=\sqrt{1+|\nabla' \eta(y,t)|^2}(\BI+\BM_2(\eta))^{-1}
\end{equation*}
from the left-hand side
furnishes that, by \eqref{Deformation} and \eqref{n-ka},
\begin{align*}
	\BT(\Bv,\Fq)\Be_3 + (c_g -c_\si\De')h\Be_3 = \BH(\Bv,\eta), \quad \text{$y\in \BR_0^3$},
\end{align*}
where we have set
\begin{align*}
	\BH(\Bv,\eta) &= -c_\si\CK(\eta)\Be_3 +\mu\BD_y(\Bv)\Be_3 
		 - \mu(\BI-\BM_2(\eta))\bigg[\BD_y(\Bv) \\
		&-\frac{1}{1+D_3\eta}\left\{(\nabla_y\eta\otimes D_3\Bv)
		+\tp (\nabla_y\eta\otimes D_3\Bv)\right\}\bigg](\BI+\BM_2(\eta))\Be_3. \notag
\end{align*}

{\bf Step 4: The fourth equation of \eqref{NS1}.}
Writing $\Bv$ as $\Bv=\tp(v_1,v_2,v_3)$,
we insert it into the fourth equation of \eqref{NS1},
together with \eqref{deriv:h}, in order to obtain
\begin{equation*}
	\pa_t h -v_3 = K(\Bv,\eta), \quad \text{$y\in\BR_0^3$,}
\end{equation*}
where $K(\Bv,\eta) = -v_1D_1\eta-v_2D_2\eta$.

Summing up Step 1-Step 4,
we have achieved the following set of equations: 
For each $t>0$,
\begin{alignat}{2}
	\pa_t\Bv-\Di\BT(\Bv,\Fq) = \BF(\Bv,\CE(h)),&
		&\quad & \text{$y\in\BR_-^3$,} \label{NS3_1} \\
	\di\Bv =G(\Bv,\CE(h)) = \di\BG(\Bv,\CE(h)),&
		&\quad & \text{$y\in \BR_-^3$,} \label{NS3_2} \\
	\BT(\Bv,\Fq)\Be_3+(c_g-c_\si\De')h\Be_3 =\BH(\Bv,\CE(h)),&
		&\quad & \text{$y\in\BR_0^3$,} \label{NS3_3} \\
	\pa_t h-v_3=K(\Bv,\CE(h)), &
		&\quad & \text{$y\in \BR_0^3$,} \label{NS3_4} \\
	\Bv|_{t=0} = \Bv_0, \quad \text{$y\in \BR_-^3$},
		\quad h|_{t=0} =h_0,& &\quad & \text{$y'\in \BR^2$}, \label{NS3_5}
\end{alignat}
where $\Bv_0$, $h_0$ are given initial data.
Here, 
\begin{align*}
	\BF(\Bv,\CE(h)) 
		&= \BF_1(\Bv,\CE(h))+\BF_2(\Bv,\CE(h))+\BF_3(\Bv,\CE(h)), \\
	\BF_1(\Bv,\CE(h)) 
		&= (\BI+\BM_3(\CE(h)))
		\bigg(-(\Bv\cdot\nabla)\Bv
		+\frac{(\Bv\cdot\nabla\CE(h))D_3\Bv}{1+D_3\CE(h)}\bigg), \\
	\BF_2(\Bv,\CE(h))
		&= (\BI+\BM_3(\CE(h)))\frac{\pa_t\CE(h) D_3\Bv}{1+D_3\CE(h)}, \\
	\BF_3(\Bv,\CE(h))
		&= (-\pa_t v_3+\mu\De v_3)\nabla\CE(h) 
		-\mu(\BI+\BM_3(\CE(h)))
		\bigg(\sum_{j=1}^{3}\CD_{jj}(\CE(h))\Bv\bigg) 
		-\mu\nabla\bigg(\frac{\nabla\CE(h)\cdot D_3\Bv}{1+D_3\CE(h)}\bigg), \\
	\BG(\Bv,\CE(h))
		&= -\BM_1(\CE(h))\Bv, \quad
	G(\Bv,\CE(h)) 
		= \nabla\CE(h)\cdot D_3\Bv -(D_3\CE(h))\di\Bv, \\
	\BH(\Bv,\CE(h))
		&= -c_\si\CK(\CE(h))\Be_3 -\mu\BD(\Bv)\BM_2(\CE(h))\Be_3 
		+\mu\BM_2(\CE(h))\BD(\Bv)(\BI+\BM_2(\CE(h)))\Be_3\\
		&+\frac{\mu}{1+D_3\CE(h)}(\BI-\BM_2(\CE(h))) 
		\Big\{(\nabla\CE(h)\otimes D_3\Bv)
		+\tp(\nabla\CE(h)\otimes D_3\Bv)\Big\}(\BI+\BM_2(\CE(h)))\Be_3, \\
	K(\Bv,\CE(h))
		&= -v_1D_1\CE(h)-v_2D_2\CE(h).
\end{align*}

\subsection{Definition of solutions}\label{subsec2_3}
Following \cite{EPS03},
we introduce the definition of solutions of System \eqref{NS1} as follows:

\begin{defi}\label{def:sol2}
We say that System \eqref{NS1} is solvable globally in time
if the following assertions hold for some $1<p,q<\infty$
and for given initial data $h_0\in B_{q,p}^{3-1/p-1/q}(\BR^2)$, $\Bu_0\in B_{q,p}^{2-2/p}(\Om_0)^3$ with
$\Om_0=\{x_3<h_0(x')\}$.
\begin{enumerate}[$(1)$]
\item
Let $\Bv_0=\Bu_0\circ \Te_0$ for some diffeomophism $\Te_0$ from $\BR_-^3$ onto $\Om_0$.
Then \eqref{NS3_1}-\eqref{NS3_5}
admit a global-in-time solution $(\Bv,\Fq,h)$ satisfying
\begin{equation*}
\Bv\in W_{q,p}^{2,1}(\BR_-^3\times\BR_+)^3, \quad 
\Fq\in L_p(\BR_+,\wh W_q^1(\BR_-^3)), \quad
h\in W_p^1(\BR_+,W_q^{2-1/q}(\BR^2))\cap L_p(\BR_+,W_q^{3-1/q}(\BR^2)).
\end{equation*}
	\item
		Let $\Om_t$ be given by \eqref{Omega_t}. 
		Then the mapping $\Te_t$, defined as \eqref{Te_t},
		is a diffeomorphism from $\BR_-^3$ onto $\Om_t$ for each $t>0$.
\end{enumerate}
In this case, setting $\Bu = \Bu(x,t) = \Bv(\Te_t^{-1}(x),t)$ and $\Fp =\Fp(x,t) = \Fq(\Te_t^{-1}(x),t)$ for $x\in\BR_-^3$ and $t>0$,
we call $(\Bu,\Fp,h)$ a solution to System \eqref{NS1}.
\end{defi}

\subsection{Main results}\label{subsec2_4}
The main results of this paper are stated in this subsection.

First, we introduce function spaces to state our main results precisely.
Let $X$ be a Banach space and its norm $\|\cdot\|_X$,
and let $I$ be an interval of $\BR_+=(0,\infty)$.
Then, for $s>0$ and $1\leq p,q\leq \infty$, we set
\begin{align*}
	&L_p^s(I,X)
		=\{u \in L_p(I,X) \mid \|u\|_{L_p^s(I,X)}<\infty\}, \quad
	\|u\|_{L_p^s(I,X)}
		=\|(t+2)^s u\|_{L_p(I,X)}; \\
	&W_p^{1,s}(I,X)
		=\{u\in W_p^1(I,X) \mid  \|u\|_{W_p^{1,s}(I,X)}<\infty\}, \\
	&\|u\|_{W_p^{1,s}(I,X)}
		=\|(t+2)^s\pa_t u\|_{L_p(I,X)}+\|(t+2)^su\|_{L_p(I,X)};
\end{align*}
and furthermore,
\begin{align*}
	H_{q,p}^{1,1/2,s}(\BR_-^3\times\BR_+)
		&=\{u\in H_{q,p}^{1,1/2}(\BR_-^3\times\BR_+) \mid \|u\|_{H_{q,p}^{1,1/2,s}(\BR_-^3\times\BR_+)}<\infty\}, \\
	\|u\|_{H_{q,p}^{1,1/2,s}(\BR_-^3\times\BR_+)} 
		&= \|(t+2)^s u\|_{H_{q,p}^{1,1/2}(\BR_-^3\times\BR_+)}.
\end{align*}
For the right members of the equations \eqref{NS3_1}-\eqref{NS3_4}, we set
\begin{align*}
	&\BBF_1 = \BBF_2 = \BBF_3
		= \bigcap_{r\in\{q,2\}} L_{p}(\BR_+,L_r(\BR_-^3))^3, \quad
	\CG
		= \bigcap_{r\in\{q,2\}} W_p^1(\BR_+,L_r(\BR_-^3))^3, \\
	&\BBG
		= \bigcap_{r\in\{q,2\}} H_{r,p}^{1,1/2}(\lhs\times\BR_+), \quad
	\BBH
		= \bigcap_{r\in\{q,2\}} H_{r,p}^{1,1/2}(\BR_-^3\times\BR_+)^3, \quad
	\BBK
		=\bigcap_{r\in\{q,2\}}L_p(\BR_+,W_r^2(\BR_-^3)).
\end{align*}
In addition, for positive numbers $\Fa$, $\Fb$
and for $\bar{q}=q/2$ with $q\geq 2$, we set
\begin{alignat*}{2}
	&\wt\BBF_i(\Fa,\Fb)
		\,&=&~L_p^\Fa(\BR_+,L_q(\BR_-^3))^3\cap L_\infty^\Fb(\BR_+,L_{\bar{q}}(\BR_-^3))^3, 
		\quad i=1,2,\\
	&\wt\BBF_3(\Fa,\Fb)
		\,&=&~L_p^\Fa(\BR_+,L_q(\BR_-^3))^3\cap L_p^\Fb(\BR_+,L_{\bar{q}}(\BR_-^3))^3, \\
	&\wt\CG(\Fa,\Fb)
		&=&~W_p^{1,\Fa}(\BR_+,L_q(\BR_-^3))^3\cap W_p^{1,\Fb}(\BR_+,L_{\bar{q}}(\BR_-^3))^3, \\
	&\wt\BBG(\Fa,\Fb)
		&=&~L_p^\Fa(\BR_+,W_q^1(\BR_-^3))\cap L_p^\Fb(\BR_+,W_{\bar{q}}^1(\BR_-^3)), \\
	&\wt\BBH(\Fa,\Fb)
		&=&~H_{q,p}^{1,1/2,\Fa}(\BR_-^3\times\BR_+)^3 \cap H_{\bar{q},p}^{1,1/2,\Fb}(\BR_-^3\times\BR_+)^3, \\
	&\wt\BBK(\Fa,\Fb)
		&=&~L_p^\Fa(\BR_+,W_q^2(\BR_-^3))\cap L_p^\Fb(\BR_+,W_{\bar{q}}^2(\BR_-^3)).
\end{alignat*}
By \eqref{mn}, additional function spaces $\BBA_1$, $\BBA_2$, $\BBA_3$, $\BBA_4$, $\BBB_1$, and $\BBB_2(\te)$ are defined as
\begin{align*}
	&\BBA_1
		= L_\infty^{\Fm(\bar{q},q)}(\BR_+,L_q(\lhs))
		\cap L_\infty^{\Fm(\bar{q},2)}(\BR_+,L_2(\lhs)), \\
	&\BBA_2
		= L_\infty^{\Fm(\bar{q},q)+1/2}(\BR_+,L_q(\BR_-^3))
		\cap L_\infty^{\Fm(\bar{q},2)+1/2}(\BR_+,L_2(\BR_-^3)), \\
	&\BBA_3
		= L_p^{\Fm(\bar{q},q)+1/2}(\BR_+,W_q^1(\BR_-^3))
		\cap L_p^{\Fm(\bar{q},2)+1/2}(\BR_+,W_2^1(\BR_-^3)), \\
	&\BBA_4
		= L_\infty^{\Fm(\bar{q},q)+1/2}(\BR_+,\wh{W}_q^1(\BR_-^3))
		\cap L_\infty^{\Fm(\bar{q},2)+1/2}(\BR_+,\wh{W}_2^1(\BR_-^3)), \\
	&\BBB_1
		=L_\infty^{\Fm(\bar{q},q)}(\BR_+,L_q(\BR_0^3))
		\cap L_\infty^{\Fm(\bar{q},2)}(\BR_+,L_2(\BR_0^3)), \\
	&\BBB_2(\te)
		=L_\infty(\BR_+,L_{q(\te)}(\BR_0^3))\cap L_\infty(\BR_+,L_2(\BR_0^3)) \quad \text{for $0\leq \te\leq 1$,}
\end{align*}
where $q(\te)\in[1,4/3]$ is given in \eqref{q_theta}.
For the initial data, we set
\begin{align*}
	&\BBI_1(\te)
		=B_{q,p}^{2-2/p}(\BR_-^3)^3\cap B_{q(\te),p}^{2-2/p}(\BR_-^3)^3 \quad (0\leq \te \leq 1), \quad 
	\BBI_2
		=B_{q,p}^{3-1/p-1/q}(\BR^2)\cap B_{2,p}^{3-1/p-1/2}(\BR^2) \cap L_{\bar{q}}(\BR^2).
\end{align*}

Next, we define several norms for solutions. Let $\Bz = (\Bu,\Fp,h,\eta)$, and then
\begin{align*}
	\CM_{q,p}(\Bz) 
		&=\|(\pa_t\Bu,\Bu,\nabla\Bu,\nabla^2\Bu,\nabla\Fp)\|_{L_p(\BR_+,L_q(\BR_-^3))}
		+\|\pa_t h\|_{L_p(\BR_+,W_q^{2-1/q}(\BR^2))} \\
		&+\|h\|_{L_p(\BR_+,W_q^{3-1/q}(\BR^2))}
		+\|\pa_t \eta\|_{L_p(\BR_+,\wh W_q^2(\BR_-^3))}
		+\|\eta\|_{L_p(\BR_+,\wh W_q^3(\BR_-^3))}, \\
	\CN_q(\Bz)
		&=\|\Bu\|_{L_\infty^{\Fm(\bar{q},q)}(\BR_+,L_q(\BR_-^3))}
		+\|\nabla\Bu(t)\|_{L_\infty^{\Fn(\bar{q},q)+1/8}(\BR_+,L_q(\BR_-^3))} 
		+\|h\|_{L_\infty^{1/\bar{q}-1/q}(\BR_+,L_q(\BR^2))} \\
		&+\|\pa_t h\|_{L_\infty^{\Fm(\bar{q},q)}(\BR_+,L_q(\BR^2))} 
		+\|\nabla\eta\|_{L_\infty^{\Fm(\bar{q},q)+1/4}(\BR_+,W_q^1(\BR_-^3))}
		+\|\nabla\pa_t \eta\|_{L_\infty^{\Fm(\bar{q},q)+1/2}(\BR_+,L_q(\BR_-^3))}
\end{align*}
that are the norm of maximal $L_p\text{-}L_q$ regularity class appearing in Definition \ref{def:sol2}
and the norm of decay properties of lower order terms, respectively. 
In addition, a time-weighted norm $\CN_{q,p}(\Bz;\Fa_1,\Fa_2)$
with positive numbers $\Fa_1$, $\Fa_2$ is defined as
\begin{equation*}
	\CN_{q,p}(\Bz;\Fa_1,\Fa_2)
		=\|(\pa_t\Bu,\nabla^2\Bu)\|_{L_p^{\Fa_1}(\BR_+,L_q(\BR_-^3))}
		+\|(\nabla^2\pa_t\eta,\nabla^3\eta)\|_{L_p^{\Fa_2}(\BR_+,L_q(\BR_-^3))}.
\end{equation*}

The equations \eqref{NS3_1}-\eqref{NS3_5} lead us to a linearized problem as follows:

\begin{equation}\label{linear1}
\left\{\begin{aligned}
	\pa_t \Bu -\Di\BT(\Bu,\Fp) = \Bf& && \text{in $\BR_-^3$, $t>0$,} \\
	\di\Bu =g=\di \Bg& && \text{in $\BR_-^3$, $t>0$,} \\
	\BT(\Bu,\Fp)\Be_3 +(c_g -c_\si\De')h\Be_3 = \Bh& && \text{on $\BR_0^3$, $t>0$,} \\
	\pa_t h - u_3 = k& && \text{on $\BR_0^3$, $t>0$,} \\
	\Bu|_{t=0} = 0& && \text{in $\BR_-^3$,} \\
	h|_{t=0} = 0 & && \text{on $\BR^2$,}
\end{aligned}\right.
\end{equation}
where $\Bu=\tp(u_1,u_2,u_3)$ and $\BT(\Bu,\Fp)=\mu\BD(\Bu)-\Fp\BI$.
Concerning System \eqref{linear1}, we have the following time-weighted estimate of solutions: 

\begin{theo}\label{theo:main1}
Let $p$, $q$ be exponents satisfying \eqref{pq:lin} and $\bar{q}=q/2$.
Suppose that $\Fa_1$, $\Fa_2$, $\Fb_1$, $\Fb_2$, and $\Fb_3$ are positive numbers
satisfying the following conditions:
\begin{align}\label{ab}
	& 
	\Fb_1,\Fb_2>1, \quad  \Fb_3,\Fb_4\geq1, \\
	&p(\min(\Fb_1,\Fb_2,\Fb_3,\Fb_4)-\Fa_1)>1, \quad
	p(\Fm\left(\bar{q},q\right)+1/4-\Fa_1)>1, \notag \\
	&p(\min(\Fb_1,\Fb_2,\Fb_3,\Fb_4)-\Fa_2)>1, \quad
	p(1/2+2/q-\Fa_2)>1. \notag
\end{align}
Let $\Fa_0=\max(\Fa_1,\Fa_2)$, and let
the right members $\Bf = \Bf_1+\Bf_2+\Bf_3$, $\Bg$, $g$, $\Bh$, and $k$
of \eqref{linear1} satisfy the following conditions:
\begin{enumerate}[$(1)$]
	\item
		$\Bf_1\in\BBF_1\cap \wt\BBF_1(\Fa_0,\Fb_1)$,
		$\Bf_2\in\BBF_2\cap \wt\BBF_2(\Fa_0,\Fb_2)$,
		$\Bf_3\in\BBF_3\cap\wt\BBF_3(\Fa_0,\Fb_3);$
	\item
		$\Bg\in\CG\cap\wt\CG(\Fa_0,\Fb_3)\cap\wt\CG(\Fa_0,\Fb_4)\cap\BBA_1;$
	\item
		$g\in\BBG\cap\wt\BBG(\Fa_0,\Fb_3)\cap \wt\BBG(\Fa_0,\Fb_4)\cap \BBA_2\cap\BBA_3;$
	\item
		$\Bh\in\BBH\cap\wt\BBH(\Fa_0,\Fb_3)\cap\wt\BBH(\Fa_0,\Fb_4)\cap\BBA_3;$
	\item
		$k\in\BBK\cap\wt\BBK(\Fa_0,\Fb_4)\cap\BBA_4\cap\BBB_1\cap\BBB_2(\te)$ for some $0<\te< 1;$
	\item
		$g$ and $\Bh$ satisfy additionally
		\begin{equation*}
				g\in L_p^{\Fc_1}(\BR_+,W_q^1(\BR_-^3))\cap L_p^{\Fd_1}(\BR_+,W_{\bar{q}}^1(\BR_-^3)), \quad
				\Bh\in L_p^{\Fc_1}(\BR_+,W_q^1(\BR_-^3))^3 \cap L_p^{\Fd_1}(\BR_+,W_{\bar{q}}^1(\BR_-^3))^3
		\end{equation*}
		for non-negative real numbers $\Fc_1$, $\Fd_1$ satisfying
		\begin{equation*}
			p\left(1+\Fc_1-\Fa_0\right)>1, \quad
			p\left(1+\Fd_1-\max\left(\Fb_3,\Fb_4\right)\right)>1,
		\end{equation*}
\end{enumerate}
and the following compatibility conditions:
\begin{enumerate}[$(7)$]
	\item
		$g|_{t=0}=0$ in $\BR_-^3$ and $[\Bh]_\tau|_{t=0}=0$ on $\BR_0^3$,
		where $[\Bh]_{\tau}=\Bh-(\Bh\cdot\Be_3)\Be_3$.
\end{enumerate}
Then System \eqref{linear1} admits a unique solution $(\Bu,\Fp,h)$ with
\begin{align*}
	&\Bu \in W_{q,p}^{2,1}(\BR_-^3\times \BR_+)^3, \quad
	\Fp \in L_{p}(\BR_+,\wh{W}_q^1(\BR_-^3)), \\
	& h\in W_{p}^1(\BR_+,W_q^{2-1/q}(\BR^2))\cap L_{p}(\BR_+,W_q^{3-1/q}(\BR^2)), \\
	&\CE(h)\in W_{p}^1(\BR_+,\wh W_q^2(\BR_-^3))\cap L_{p}(\BR_+,\wh W_q^3(\BR_-^3)),
\end{align*}
and also $\Bz=(\Bu,\Fp,h,\CE(h))$ satisfies the estimate:
\begin{equation*}
	\CN_{q,p}(\Bz;\Fa_1,\Fa_2)+\sum_{r\in\{q,2\}}\left(\CM_{r,p}(\Bz)+\CN_r(\Bz)\right)
		+\|\pa_t\CE(h)\|_{L_\infty(\BR_+,L_2(\lhs))}\leq C_{p,q}\SSN
\end{equation*}
for some positive constant $C_{p,q}$,
where 
\begin{align*}
	\SSN &=
			\sum_{i=1}^3 \|\Bf_i\|_{\BBF_i\cap \wt\BBF_i(\Fa_0,\Fb_i)}	
			+\|\Bg\|_{\CG\cap\wt\CG(\Fa_0,\Fb_3)\cap\wt\CG(\Fa_0,\Fb_4)\cap\BBA_1} \\
			&+\|g\|_{\BBG\cap\wt\BBG(\Fa_0,\Fb_3)\cap \wt\BBG(\Fa_0,\Fb_4)\cap \BBA_2\cap\BBA_3} 
			+\|\Bh\|_{\BBH\cap\wt\BBH(\Fa_0,\Fb_3)\cap\wt\BBH(\Fa_0,\Fb_4)\cap\BBA_3} \\
			&+\|k\|_{\BBK\cap\wt\BBK(\Fa_0,\Fb_4)\cap\BBA_4\cap\BBB_1\cap \BBB_2(\te)}
			+\|(g,\Bh)\|_{L_p^{\Fc_1}(\BR_+,W_q^1(\BR_-^3))\cap L_p^{\Fd_1}(\BR_+,W_{\bar{q}}(\BR_-^3))}.
\end{align*}
\end{theo}


To construct solutions of \eqref{NS3_1}-\eqref{NS3_5}, we set
\begin{align}\label{160920_1}
	X_{q,p}(r;\Fa_1,\Fa_2) &= \{\Bz = (\Bu,\Fp,h,\eta) \mid 
		\|\Bz\|_{X_{q,p}(\Fa_1,\Fa_2)}\leq r\}, \\
	\|\Bz\|_{X_{q,p}(\Fa_1,\Fa_2)} &=\CN_{q,p}(\Bz;\Fa_1,\Fa_2)
		+\sum_{r\in\{q,2\}}\left(\CM_{r,p}(\Bz)+\CN_r(\Bz)\right)  
		+\|\pa_t\eta\|_{L_\infty(\BR_+,L_2(\lhs))}. \notag
\end{align}
By combining Theorem \ref{theo:main1} with the contraction mapping principle, we have

\begin{theo}\label{theo:main2}
Let $p,q$ be exponents satisfying \eqref{pq}, and let $0<\te< 1$.
Then there exist positive constants $\de_0,\ep_0\in(0,1)$ such that,
for any initial data $(\Bv_0,h_0)\in\BBI_1(\te)\times\BBI_2$ satisfying
a smallness condition: $\|(\Bv_0,h_0)\|_{\BBI_1(\te) \times\BBI_2}\leq \de_0$
and compatibility conditions:
\begin{equation}\label{comp:1}
	\di\Bv_0 =
		G(\Bv_0,\CE(h_0)) \text{ {\rm in}  $\BR_-^3$}, \quad
	[\mu\BD(\Bv_0)\Be_3]_{\tau}=
		[\BH(\Bv_0,\CE(h_0))]_{\tau} \text{ {\rm on} $\BR_0^3$,}
\end{equation}
the equations \eqref{NS3_1}-\eqref{NS3_5} 
admit a unique global-in-time solution $(\Bv,\Fq,h)$
satisfying $\|\Bz\|_{X_{q,p}(1/2,3/4)}\leq \ep_0$ with $\Bz=(\Bv,\Fq,h,\CE(h))$.
%
%
\end{theo}


Recall in System \eqref{NS1} that $\Om_0 =\{(x',x_3) \mid x'\in\BR^2,x_3<h_0(x')\}$,
and let $\Bn_0$ be the unit outer normal to $\Ga_0=\{(x',x_3) \mid x'\in\BR^2,x_3=h_0(x')\}$.
Then Theorem \ref{theo:main2} enables us to prove the following theorem:

\begin{theo}\label{theo:main3} 
Let $p$, $q$ be exponents satisfying \eqref{pq}, and let $0<\te< 1$.
Suppose that $\de_0$, $\ep_0$ are the positive constants given in Theorem $\ref{theo:main2}$. 
Then there exists a positive number $r_0$ such that, for any initial data
$h_0\in \BBI_{q,p}$, $\Bu_0\in \BBJ_{q,p,\te}(h_0)$ satisfying
a smallness condition:
$\|h_0\|_{\BBI_{q,p}}+\|\Bu_0\|_{\BBJ_{q,p,\te}(h_0)}\leq r_0$
and compatibility conditions:
\begin{equation}\label{comp:2}
	\di\Bu_0 = 0 \text{ {\rm in} $\Om_0$},\quad
	\mu\{\BD(\Bu_0)\Bn_0-(\Bn_0\cdot\BD(\Bu_0)\Bn_0)\Bn_0\}=0 \text{ {\rm on} $\Ga_0$},
\end{equation}
the following assertions hold.
\begin{enumerate}[$(1)$]
\item
	$\Te_0(y)=(y_1,y_2,y_3+\CE(h_0)(y))$ is a $C^2$-diffeomophism from $\lhs$ onto $\Om_0$. 
\item
	$\Bv_0=\Bu_0\circ\Te_0$ and $h_0$ satisfy the smallness condition $\|(\Bv_0,h_0)\|_{\BBI_1(\te)\times\BBI_2}\leq \de_0$ 
	and the compatibility condition \eqref{comp:1}, and thus
	the equations \eqref{NS3_1}-\eqref{NS3_5} admits a unique global-in-time solution $(\Bv,\Fq,h)$
	satisfying $\|\Bz\|_{X_{q,p}(1/2,3/4)}\leq \ep_0$ with $\Bz=(\Bv,\Fq,h,\CE(h))$.
\item
	The mapping $\Te_t$, defined as \eqref{Te_t},
	is a $C^2$-diffeomorphism from $\BR_-^3$ onto $\Om_t$ for each $t>0$.
\end{enumerate}
Namely, System \eqref{NS1} admits a unique global-in-time solution $(\Bu,\Fp,h)$ given by
\begin{equation*}
	(\Bu,\Fp,h)=(\Bv(\Te_t^{-1}(x),t),\Fq(\Te_t^{-1}(x),t),h(x',t)).
\end{equation*}
\end{theo}

Large-time behavior of the solution obtained in Theorem \ref{theo:main3} is stated as follows:

\begin{theo}\label{theo:main4} 
Let $p$, $q$ be exponents satisfying \eqref{pq} and $0<\te< 1$,
and let $r_0$ is the positive number given in Theorem $\ref{theo:main3}$.
Suppose that $h_0\in \BBI_{q,p}$, $\Bu_0\in \BBJ_{q,p,\te}(h_0)$ satisfy
the smallness condition: $\|h_0\|_{\BBI_{q,p}}+\|\Bu_0\|_{\BBJ_{q,p,\te}(h_0)}\leq r_0$
and the compatibility condition \eqref{comp:2}.
Then the solution $(\Bu,\Fp,h)$ of System \eqref{NS1}, obtained in Theorem $\ref{theo:main3}$,
satisfies the following large-time behavior:
\begin{align*}
	&\|\Bu(t)\|_{L_r(\Om_t)}
		= O\left(t^{-\Fm(\bar{q},r)}\right) , \quad
	\|\nabla\Bu(t)\|_{L_r(\Om_t)}
		= O\left(t^{-\Fn\left(\bar{q},r\right)-\frac{1}{8}}\right),
	\|h(t)\|_{L_r(\BR^2)}
		= O\left(t^{-\left(\frac{1}{\bar{q}}-\frac{1}{r}\right)}\right), \\
	&\|\nabla' h(t)\|_{L_r(\BR^2)}
		= O\left(t^{-\Fm\left(\bar{q},r\right)-\frac{1}{4}}\right), 
	\|\pa_t h(t)\|_{L_r(\BR^2)}
		= O\left(t^{-\Fm\left(\bar{q},r\right)}\right)
\end{align*}
for $2\leq r \leq q$ as time $t$ tends to infinity.
\end{theo}


This paper consists of seven sections as follows: 
Section 3 is concerned with an inhomogeneous boundary value problem in $\lhs$,
and proves decay properties and time-weighted estimates of solutions
to the inhomogeneous boundary value problem.
In Section 4, we deal with two linear time-dependent problems in $\lhs$, with homogeneous boundary condition,
associated with the linearized problem of \eqref{NS3_1}-\eqref{NS3_5}.
Then we show time-weighted estimates of the solutions
by $L_r\text{-}L_s$ estimates  proved in \cite{SaS1}. 
Section 5 proves Theorem \ref{theo:main1} by means of results obtained in Sections 3 and 4.
In Section 6, we first construct an initial flow. 
Next, after subtracting the initial flow from the equations \eqref{NS3_1}-\eqref{NS3_5},  
we combine Theorem \ref{theo:main1} with the contraction mapping principle in order to show Theorem \ref{theo:main2}.
Section 7 proves Theorem \ref{theo:main3} and Theorem \ref{theo:main4}
by the solution obtained in Theorem \ref{theo:main2}.

\section{An inhomogeneous boundary value problem}\label{sec3}
In this section, we consider an inhomogeneous boundary value problem as follows:
\begin{equation}\label{linear2}
\left\{\begin{aligned}
	\pa_t\Bu+\Bu-\Di\BT(\Bu,\Fp) &=0 && \text{in $\BR_-^3$, $t>0$,} \\
	\di\Bu &=0 && \text{in $\BR_-^3$, $t>0$,} \\
	\BT(\Bu,\Fp)\Be_3 &=\Bh && \text{on $\BR_0^3$, $t>0$,} \\
	\Bu|_{t=0} &=0 && \text{in $\BR_-^3$,}
\end{aligned}\right.
\end{equation}
where $\BT(\Bu,\Fp) =\mu\BD(\Bu)-\Fp\BI$.
Our aim here is to prove the following two theorems.

\begin{theo}\label{theo:sec3_1}
Let $2<p<\infty$, $1<q<\infty$, and $2/p+1/q<1$, and suppose that
\begin{equation*}
	\Bh \in H_{q,p}^{1,1/2}(\BR_-^3\times\BR_+)^3 \quad \text{with $\Bh|_{t=0}$ on $\BR_0^3$.}
\end{equation*}
Then System \eqref{linear2} admits
a unique solution 
$$(\Bu,\Fp)\in W_{q,p}^{2,1}(\BR_-^3\times\BR_+)^3\times L_{p}(\BR_+,\wh{W}_q^1(\BR_-^3)),$$
and also 
\begin{equation}\label{140825_3}
	\|(\pa_t\Bu,\Bu,\nabla\Bu,\nabla^2\Bu,\nabla\Fp)\|_{L_p(\BR_+,L_q(\BR_-^3))}
		\leq C_{p,q}\|\Bh\|_{H_{q,p}^{1,1/2}(\BR_-^3\times\BR_+)}
\end{equation}
for a positive constant $C_{p,q}$.
In addition, the following assertions hold.
\begin{enumerate}[$(1)$]
\item\label{theo:sec3_1_1}
	The solution $\Bu=\Bu(x,t)$ is represented as
	\begin{equation*}
		\Bu(x,t)
			=\int_0^t\left[\CB(t-s)\Bh(\cdot,0,s)\right](x)\,ds\quad(x\in\BR_-^3,\,t>0)
	\end{equation*}
	with an operator $\CB(\tau)\in\CL(L_q(\BR^2)^3,W_q^1(\BR_-^3)^3)$, $\tau>0$, satisfying
	\begin{equation*}
		\|\nabla^l\CB(\tau)\Bf\|_{L_q(\BR_-^3)} 	\leq
			C_q\tau^{-\frac{1+l}{2}+\frac{1}{2q}}
			e^{-\tau}\|\Bf\|_{L_q(\BR^2)}\quad(\tau>0,\ l=0,1)
	\end{equation*}
	for any $\Bf\in L_q(\BR^2)^3$ and a positive constant $C_q$
	independent of $\tau$ and $\Bf$.
\item\label{theo:sec3_1_2}
	There exists a positive constant $C_{p,q}$ such that
	\begin{equation*}
		\|\CE(u_3|_{\BR_0^3})\|_{W_{q,p}^{2,1}(\BR_-^3\times\BR_+)}
			\leq C_{p,q}\|\Bh\|_{H_{q,p}^{1,1/2}(\BR_-^3\times\BR_+)},
	\end{equation*}
	where $u_3$ is the third component of $\Bu$
	and $\cdot\,|_{\BR_0^3}$  the boundary trace. 
\end{enumerate}
\end{theo}

\begin{theo}\label{theo:sec3_2}
Let $p$, $q$ be exponents satisfying 
\begin{equation}\label{pq:lin2}
	2<p<\infty, \quad 3<q<4, \quad \frac{2}{p}+\frac{3}{q}<1,
\end{equation}
and let $r$ be another exponent satisfying $(1-2/p)^{-1}<r\leq q$.
For $\Bh\in H_{r,p}^{1,1/2}(\BR_-^3\times \BR_+)^3$
with $\Bh|_{t=0}=0$ on $\BR_0^3$,
let $(\Bu,\Fp)$ be the solution of \eqref{linear2}
obtained in Theorem $\ref{theo:sec3_1}$.
If we additionally assume that,
for non-negative real numbers $\Fc_2$, $\Fd_2$,
\begin{equation*}
	\Bh \in L_{p}^{\Fc_2}(\BR_+,W_r^1(\BR_-^3))^3, \quad
	\Bh\in H_{r,p}^{1,1/2,\Fd_2}(\BR_-^3\times\BR_+)^3,
\end{equation*}
then the following assertions hold.
\begin{enumerate}[$(1)$]
	\item\label{theo:sec3_2_1}
		There exists a positive constant $C_{p,r}$,
		independent of $\Bu$, $\Fp$, and $\Bh$, such that
		\begin{equation*}
			\|\Bu(t)\|_{W_r^1(\BR_-^3)} \leq
				C_{p,r}(t+2)^{-\Fc_2}
				\|\Bh\|_{L_p^{\Fc_2}(\BR_+,W_r^1(\BR_-^3))}
				\quad \text{for any $t>0$.}
		\end{equation*}
	\item\label{theo:sec3_2_2}
		There exists a constant $C_{p,r}>0$,
		independent of $\Bu$, $\Fp$, and $\Bh$, such that
		\begin{equation*}
			\|(\pa_t\Bu,\Bu,\nabla\Bu,\nabla^2\Bu,\nabla\Fp)
				\|_{L_p^{\Fd_2}(\BR_+,L_r(\BR_-^3))}  
			\leq C_{p,r}
				\left(\|\Bh\|_{L_p^{\Fc_2}(\BR_+,W_r^1(\BR_-^3))}
				+\|\Bh\|_{H_{r,p}^{1,1/2,\Fd_2}(\BR_-^3\times\BR_+)}\right),
		\end{equation*}
		provided that $p(1+\Fc_2-\Fd_2)>1$.
\end{enumerate}
\end{theo}

\subsection{Preliminaries}\label{subsec3_1}

Let $\CL[f](\la)$ and $\CL_\la^{-1}[g](t)$ be
the Laplace transform of $f=f(t)$ and the inverse Laplace transform of $g=g(\la)$, respectively, i.e.
\begin{equation}\label{la:trans}
	\CL[f](\la) = \int_\BR e^{-\la t} f(t)\intd t, \quad
	\CL_\la^{-1}[g](t) = \frac{1}{2\pi}\int_\BR e^{\la t}g(\la) \intd \tau
\end{equation}
for $\la=\ga+i\tau\in\BC$.
Note that, by \eqref{160920_2},
\begin{equation*}
	\CL[f](\la) =\CF_t[e^{-\ga t} f](\tau), \quad
	\CL_\la^{-1}[g](t) = e^{\ga t}\CF_\tau^{-1}[g(\ga+i\,\cdot)](t).
\end{equation*}
In addition, let $\La_\ga^{1/2} f=\CL_\la^{-1}[|\la|^{1/2}\CL[f](\la)](t)$ $(\la=\ga+i\tau)$,
and let $\wt{h}(\xi',\la)$ be the Fourier-Laplace transform
of $h=h(x',t)$ defined on $\BR^2\times\BR$, that is,
\begin{equation*}
	\wt{h}(\xi',\la) = \int_{\BR^2\times\BR}e^{-(ix'\cdot\xi'+\la t)}h(x',t)\intd x'dt
		=\CL[\wh{h}(\xi',\cdot)](\la)
\end{equation*}
with $\xi'=(\xi_1,\xi_2)\in\BR^2$ and $\la\in\BC$,
where the last identity follows from \eqref{PFT}, \eqref{la:trans}.

Next, we define several function spaces with exponential weight.
Let $X$ be a Banach space and $\de\geq 0$.
Then, for $1<p,q<\infty$ and an interval $I$ of $\BR$, we set
\begin{align*}
	&L_{p,\de}(I,X) = \{f:I\to X \mid e^{-\de t} f(t)\in L_p(I,X)\}, \\
	&W_{p,\de}^1(I,X) = \{f\in L_{p,\de}(I,X) \mid e^{-\de t}\pa_t f(t) \in L_p(I,X)\}, \\
	&W_{q,p,\de}^{2,1}(\BR_-^3\times I) = W_{p,\de}^1(I,L_q(\BR_-^3))\cap  L_{p,\de}(I,W_q^2(\BR_-^3)),
\end{align*}
and also
\begin{align*}
	&{}_0L_{p,\de}(\BR,X) = \{f\in L_{p,\de}(\BR,X) \mid f(t)=0 \text{ for $t<0$}\}, \\
	&{}_0W_{p,\de}^1(\BR,X) = \{f\in {}_0L_{p,\de}(\BR,X) \mid e^{-\de t}\pa_t f(t) \in L_p(\BR,X)\}, \\
	&{}_0W_{q,p,\de}^{2,1}(\BR_-^3\times\BR) = {}_0W_{p,\de}^1(\BR,L_q(\BR_-^3))\cap {}_0L_{p,\de}(\BR,W_q^2(\BR_-^3)).
\end{align*}
In addition, 
Bessel potential spaces are defined as
\begin{align*}
		&H_{p,\de}^{1/2}(\BR,X) =\{f\in L_p(\BR,X) \mid e^{-\de t}\La_\de^{1/2} f\in L_p(\BR,X)\}, \\
		&H_{q,p,\de}^{1,1/2}(\BR_-^3\times\BR) = H_{p,\de}^{1/2}(\BR,L_q(\BR_-^3))\cap L_{p,\de}(\BR, W_q^1(\BR_-^3)). 
\end{align*}

Here, we introduce two classes of multipliers.
Let $0<\ep<\pi/2$, $\ga_0\geq0$, and
\begin{equation*}
	\Si_{\ep,\ga_0} = \{\la\in\BC \mid |\arg\la|<\pi-\ep,\,|\la|>\ga_0\}, \quad \Si_\ep = \Si_{\ep,0}.
\end{equation*}
Let $m(\xi',\la)$ be a function,
defined on $(\BR^2\setminus\{0\})\times\Si_{\ep,\ga_0}$,
that is infinitely many times differentiable
with respect to $\xi'=(\xi_1,\xi_2)\in\BR^2\setminus\{0\}$
and is holomorphic with respect to $\la=\ga+i\tau\in\Si_{\ep,\ga_0}$.
If there exists a real number $s$ such that,
for any multi-index $\al'=(\al_1,\al_2)\in\BN_0^2$ and
$(\xi',\la)\in(\BR^2\setminus\{0\})\times\Si_{\ep,\ga_0}$,
\begin{equation*}
	|\pa_{\xi'}^{\al'}m(\xi',\la)|
		\leq C_{s,\al',\ga_0,\ep}(|\la|^{1/2}+|\xi'|)^{s-|\al'|}, \quad
	\left|\pa_{\xi'}^{\al'}\left(\tau\frac{\pa}{\pa \tau}m(\xi',\la)\right)\right|
		\leq C_{s,\al',\ga_0,\ep}(|\la|^{1/2}+|\xi'|)^{s-|\al'|},
\end{equation*}
with a positive constant $C_{s,\al',\ga_0,\ep}$,
then $m(\xi',\la)$ is called a multiplier of order $s$ with type $1$.
If there exists a real number $s$ such that,
for any multi-index $\al'=(\al_1,\al_2)\in\BN_0^2$ and
$(\xi',\la)\in(\BR^2\setminus\{0\})\times\Si_{\ep,\ga_0}$,
\begin{equation*}
	|\pa_{\xi'}^{\al'}m(\xi',\la)|
		\leq C_{s,\al',\ga_0,\ep}(|\la|^{1/2}+|\xi'|)^s|\xi'|^{-|\al'|}, \\\
	\left|\pa_{\xi'}^{\al'}\left(\tau\frac{\pa}{\pa \tau}m(\xi',\la)\right)\right|
		\leq C_{s,\al',\ga_0,\ep}(|\la|^{1/2}+|\xi'|)^s|\xi'|^{-|\al'|},
\end{equation*}
with a positive constant $C_{s,\al',\ga_0,\ep}$,
then $m(\xi',\la)$ is called a multiplier of order $s$ with type $2$.
In what follows, we denote the set of all multiplies,
defined on $(\BR^2\setminus\{0\})\times \Si_{\ep,\ga_0}$,
of order $s$ with type $l$ $(l=1,2)$ by $\BBM_{s,l}(\Sigma_{\ep,\ga_0})$.
Typical examples of such multipliers are as follows:
the Riesz kernel $\xi_j/|\xi'|$ $(j=1,2)$ is a multiplier of order $0$ with type $2$.
Functions $\xi_j$ are multipliers of order $1$ with type $1$.
We also introduce the following fundamental lemma (cf. \cite[Lemma 5.1]{SS12}).

\begin{lemm}\label{lemm:multi1}
Let $s_1,s_2\in\BR$, $0<\ep<\pi/2$, and $\ga_0\geq 0$.
Then the following assertions hold:
\begin{enumerate}[$(1)$]
	\item
		Given $m_i\in \BBM_{s_i,1}(\Sigma_{\ep,\ga_0})$ $(i=1,2)$,
		we have $m_1m_2\in\BBM_{s_1+s_2,1}(\Sigma_{\ep,\ga_0})$.
	\item
		Given $l_i\in\BBM_{s_i,i}(\Sigma_{\ep,\ga_0})$ $(i=1,2)$,
		we have $l_1l_2\in \BBM_{s_1+s_2,2}(\Sigma_{\ep,\ga_0})$.
	\item
		Given $n_i\in\BBM_{s_i,2}(\Sigma_{\ep,\ga_0})$ $(i=1,2)$,
		we have $n_1 n_2\in \BBM_{s_1+s_2,2}(\Sigma_{\ep,\ga_0})$.
\end{enumerate}
\end{lemm}

We here set, for $\xi'=(\xi_1,\xi_2)\in\BR^2$ and
$\la\in\BC\setminus(-\infty,-\mu|\xi'|^2]$,
\begin{align}\label{symbols}
	&A=|\xi'|, \quad
	B=\sqrt{\frac{\la}{\mu}+|\xi'|^2}\quad({\rm Re}\,B\geq0), \quad
	\CM(a) = \frac{e^{Ba}-e^{Aa}}{B-A} \quad (a\in\BR), \\
	&D(A,B) = B^3 + AB^2 + 3A^2B -A^3.
	\notag
\end{align}
Then, we have  (cf. \cite[Lemma 5.2, Lemma 5.3]{SS12})

\begin{lemm}\label{lemm:multi2}
Let $s\in\BR$ and $0<\ep<\pi/2$.
Then, the following assertions hold.
\begin{enumerate}[$(1)$]
	\item
		$A^s\in \BBM_{s,2,}(\Sigma_\ep)$, provided that $s\geq 0$.
	\item
		$B^s\in\BBM_{s,1}(\Sigma_\ep)$.
	\item
		$D(A,B)^s\in\BBM_{3s,2}(\Sigma_\ep)$.
	\item
		Let $a<0$. 
		For any multi-index $\al'\in\BN_0^2$, there exists a positive constant $C_{\al'}$, independent of $a$,
		such that
		\begin{equation*}
			\left|\pa_{\xi'}^{\al'} e^{Aa}\right|\leq C_{\al'}A^{-|\al'|}e^{-(1/2)A|a|} \quad (\xi'\in\BR^2\setminus\{0\}).
		\end{equation*}
	\item
		Let $a<0$, $l=0,1$, and $0<\ep<\pi/2$.
		For any multi-index $\al'\in\BN_0^2$, there exist constants $0<b_{\ep,\mu}\leq 1$, $C_{\al',\ep,\mu}>0$ such that,
		for any $\la\in\Si_\ep$ and $\xi'\in\BR^2\setminus\{0\}$,
		\begin{equation*}
			\left|\pa_{\xi'}^{\al'}\left\{(\tau\pa_\tau)^l e^{Ba}\right\}\right|\leq C_{\al',\ep,\mu}(|\la|^{1/2}+A)^{-|\al'|}
			e^{-b_{\ep,\mu}(|\la|^{1/2}+A)|a|}.
		\end{equation*}
\end{enumerate}
\end{lemm}

The following lemma plays an essential role
to prove Theorem \ref{theo:sec3_1}.

\begin{lemm}\label{lemm:IJ}
Let $1<p,q<\infty$, $0<\ep<\pi/2$,  and
\begin{equation*}
	f\in L_{p,\de}(\BR,L_q(\BR^2)) \text{ with $f=0$ $(t<0)$}  
\end{equation*}
for some $\de\geq 0$.
Suppose that $m(\xi^\pr,\la)\in\BBM_{-1,2}(\Sigma_\ep)$
and set, for $\la=\ga+i\tau$ $(\ga\geq0,\,\tau\in\BR\setminus\{0\})$,
\begin{equation*}
\begin{aligned}
	I(x,t) &= \CL_\la^{-1}\CF_{\xi'}^{-1}
		\left[m(\xi',\la)e^{Bx_3}\wt{f}(\xi',\la)\right](x',t)
		 && (x\in\BR_-^3,\,t>0), \\
	J(x,t) &= \CL_\la^{-1}\CF_{\xi'}^{-1}
		\left[m(\xi',\la)A
		\CM(x_3)\wt{f}(\xi',\la)\right](x',t)
		 && (x\in\BR_-^3,\,t>0).
\end{aligned}
\end{equation*}
Then there exist operators $\CI(t),\CJ(t)\in\CL(L_q(\BR^2),W_q^1(\BR_-^3))$,
$t\in\BR\setminus\{0\}$, such that
\begin{equation}\label{160418_2}
	\|(\nabla^l\CI(t)g,\nabla^l\CJ(t)g)\|_{L_q(\BR_-^3)}
		\leq C_q\,|t|^{-\frac{1+l}{2}+\frac{1}{2q}}\|g\|_{L_q(\BR^2)} \quad (l=0,1)
\end{equation}
for any $g\in L_q(\BR^2)$ with a positive constant $C_q$,
independent of $t$ and $g$, and that
\begin{alignat*}{2}
	I(x,t) &= \int_0^t[\CI(t-s)f(\cdot,s)](x)\,ds\quad && (x\in\BR_-^3,\,t>0), \\
	J(x,t) &= \int_0^t[\CJ(t-s)f(\cdot,s)](x)\,ds \quad && (x\in\BR_-^3,\,t>0).
\end{alignat*}
\end{lemm}

\begin{proof}
Since the function $f$ in Lemma \ref{lemm:IJ}
can be approximated by a function of $C_0^\infty(\BR^2\times\BR_+)$,
it suffices to consider the case where $f\in C_0^\infty(\BR^2\times\BR_+)$.


For 
$(x,t)\in\BR_-^3\times(\BR\setminus\{0\})$, we set
\begin{align}
	&[\CI(t)g](x) =
		\CF_{\xi'}^{-1}\left[\CL_\la^{-1}
		\left[m(\xi',\la)e^{Bx_3}\right](t)
		\,\wh{g}(\xi')\right](x'), \label{op:I} \\
	&[\CJ(t)g](x) =
		\int_0^1\CF_{\xi'}^{-1}
		\left[\CL_\la^{-1}\left[m(\xi',\la)x_3A
		e^{(B\te+A(1-\te))x_3}\right](t)
		\,\wh{g}(\xi')\right](x')\intd\te \label{op:J}.
\end{align}
Then, by using
\begin{equation*}
	\CM(a) = a \int_0^1 e^{(B\te + A(1-\te))a}\intd\te \quad (a\in\BR),
\end{equation*}
we can write $I(x,t)$, $J(x,t)$ as follows: for $(x,t)\in\BR_-^3\times\BR_+$,
\begin{equation}\label{IJ}
	I(x,t) = \int_0^t [\CI(t-s)f(\cdot,s)](x)\,ds,\quad
	J(x,t) = \int_0^t [\CJ(t-s)f(\cdot,s)](x)\,ds.
\end{equation}
In fact, since $\CL_\la^{-1}[m(\xi',\la)e^{Bx_3}](t)=0$ for $x_3<0$ and $t>0$
by Cauchy's integral theorem
and since $\wh f(\xi',t)=0$ for $t<0$,
we see by Fubini's theorem that 
\begin{align*}
	I(x,t) =& \CF_{\xi'}^{-1}\left[\int_\BR\CL_{\la}^{-1}
			\left[m(\xi',\la)e^{B x_3}\right](t-s)\wh{f}(\xi',s)\intd s\right](x') \\
		=&\int_0^t\CF_{\xi'}^{-1}\left[\CL_{\la}^{-1}
			\left[m(\xi',\la)e^{B x_3}\right](t-s)\wh{f}(\xi',s)\right](x')\intd s,
\end{align*}
which implies the first identity of \eqref{IJ}.
Analogously, we can obtain the second identity of \eqref{IJ}.

We next show estimates of $\CI(t)g$ in \eqref{160418_2}
by using the following proposition
that was proved in \cite[Theorem 2.3]{SS01}
(cf. also \cite[Lemma 3.6]{SS12}).

\begin{prop}\label{prop:SS01}
Let $X$ be a Banach space and $\|\cdot\|_X$ its norm.
Suppose that $L$ and $n$ are a non-negative integer and
a positive integer, respectively.
Let $\si\in(0,1]$ and $s=L+\si-n$, and set
\begin{equation*}
	l(\si) =
		\left\{\begin{aligned}
			&1, && \si=1, \\
			&0, && \si\in(0,1).
		\end{aligned}\right.
\end{equation*}
Let $f=f(\xi)$ be a function
of $C^{L+l(\si)+1}(\BR^n\setminus\{0\},X)$
that satisfies the following two conditions:
\begin{enumerate}[$(1)$]
	\item
		$\pa_\xi^\al f\in L_1(\BR^n,X)$ for any multi-index
		$\al\in\BN_0^n$ with $|\al|\leq L$.
	\item
		For any multi-index $\al\in\BN_0^n$
		with $|\al|\leq L+l(\si)+1$,
		there exists a positive constant $M_\al$ such that
		\begin{equation*}
			\|\pa_\xi^\al f(\xi)\|_X \leq
				M_\al |\xi|^{s-|\al|}
				\quad	(\xi\in\BR^n\setminus\{0\}).
		\end{equation*}
\end{enumerate}
Then there exists a positive constant $C_{n,s}$ such that
\begin{equation*}
	\|\CF_\xi^{-1}[f](x)\|_X \leq
		C_{n,s} \left(\max_{|\al|\leq L+l(\si)+1}M_\al\right)
		|x|^{-(n+s)} \quad
		(x\in\BR^n\setminus\{0\}).
\end{equation*}
\end{prop}

Let $\la=i\tau$ for $\tau\in\BR\setminus\{0\}$, and then
\begin{equation}\label{140910_1}
	\pa_{\xi'}^{\al'}\CL_\la^{-1}[m(\xi',\la)e^{Bx_3}](t)
		=\CF_\tau^{-1}[\pa_{\xi'}^{\al'}(m(\xi',i\tau)e^{Bx_3})](t)
\end{equation}
for any $\al'\in\BN_0^2$. 
By Lemma \ref{lemm:multi2}, $m\in\BBM_{-1,2}(\Sigma_\ep)$,
and Leibniz's rule, we have
\begin{equation}\label{140825_1}
	|(\tau\pa_\tau)^{l}\pa_{\xi'}^{\al'}(m(\xi',i\tau)e^{Bx_3})| 
		\leq C_{\al'}(|\tau|^{1/2}+A)^{-1}
			e^{-b_{\ep,\mu}(|\tau|^{1/2}+A)|x_3|}A^{-|\al'|} 
		\leq C_{\al'}|\tau|^{-1/2}A^{-|\al'|} 
\end{equation}
for $l=0,1$ with positive constants $C_{\al'}$,
which, combined with Proposition \ref{prop:SS01}
with $X=\BR$, $L=0$, $n=1$, and $\si=1/2$,
furnishes that 
\begin{equation}\label{160418_1}
	|\CF_\tau^{-1}[\pa_{\xi'}^{\al'}(m(\xi',i\tau)e^{Bx_3})](t)|
		\leq C_{\al'}|t|^{-1/2}A^{-|\al'|} \quad (t\in\BR\setminus\{0\}).
\end{equation}
On the other hand,
by using \eqref{140910_1} and \eqref{140825_1} with $l=0$,
we have
\begin{equation*}
	|\CF_\tau^{-1}[\pa_{\xi'}^{\al'}(m(\xi',i\tau)e^{Bx_3})](t)|
		\leq C_{\al'}A^{-|\al'|}
			\int_\BR|\tau|^{-1/2}
			e^{-b_{\ep,\mu}|\tau|^{1/2}|x_3|}\intd\tau 
		\leq C_{\al'}|x_3|^{-1}A^{-|\al'|} ,
\end{equation*}
which, combined with \eqref{160418_1}, furnishes that
\begin{equation*}
	\left|\CF_\tau^{-1}\left[\pa_{\xi'}^{\al'}
			\left(m(\xi',i\tau)e^{Bx_3}\right)\right](t)\right|
		\leq C_{\al'}\left(\frac{1}{|t|^{1/2}+|x_3|}\right)
			A^{-|\al'|}.
\end{equation*}
Thus, applying the Fourier multiplier theorem of
H\"olmander-Mikhlin type (cf. \cite[Appendix Theorem 2]{Mikhlin65}) to \eqref{op:I}, we have
\begin{equation*}
	\|[\CI(t)g](\cdot,x_3)\|_{L_q(\BR^2)}
		\leq C_{q}\left(\frac{1}{|t|^{1/2}+|x_3|}\right)\|g\|_{L_q(\BR^2)},
\end{equation*}
and therefore
$\|\CI(t)g\|_{L_q(\BR_-^3)}\leq C_q |t|^{-1/2+1/(2q)}\|g\|_{L_q(\BR^2)}$.
Analogously, we can obtain the estimate of $\nabla \CI(t)g$ in \eqref{160418_2},
which completes the case of $\CI(t)g$.

We finally show the estimates of $\CJ(t)g$ given by \eqref{op:J}.
Let $\la=i\tau$ with $\tau\in\BR\setminus\{0\}$.
Then, Young's inequality yields that
\begin{equation*}
	\|[\CJ(t)g](\cdot,x_3)\|_{L_q(\BR^2)} 
		\leq \int_0^1
			\left\|\CF_\tau^{-1}\left[\CF_{\xi'}^{-1}
			\left[m(\xi',i\tau) x_3A
			e^{(B\te+A(1-\te))x_3}\right](\cdot)\right](t)
			\right\|_{L_1(\BR^2)}\intd\te\|g\|_{L_q(\BR^2)}.
\end{equation*}
By taking $\|\cdot\|_{L_q((-\infty,0))}$ with respect to $x_3$
on both sides of this inequality,
we have for $X=L_q(\BR_-,L_1(\BR^2))$ 
\begin{equation}\label{141213_1}
	\|\CJ(t)g\|_{L_q(\BR_-^3)}  \\
		\leq
		\int_0^1\left\|\CF_\tau^{-1}\left[\CF_{\xi'}^{-1}
		\left[m(\xi',i\tau)x_3A e^{(B\te+A(1-\te))x_3}
		\right](x')\right](t)
		\right\|_X\intd\te\|g\|_{L_q(\BR^2)}. 
\end{equation}
To continue the proof,
we apply Proposition \ref{prop:SS01}
with $X=L_q(\BR_-,L_1(\BR^2))$, $L=0$, $n=1$, and $\si=1/2-1/(2q)$
to the right-hand side of \eqref{141213_1}.
For $\al'\in\BN_0^2$, $l=0,1$, and $0<\de<1$, we have,
by Lemma \ref{lemm:multi2}, $m\in\BBM_{-1,2}(\Sigma_\ep)$, and Leibniz's rule,
\begin{align}\label{140825_2}
	&\left|\pa_{\xi'}^{\al'}\left\{(\tau\pa_\tau)^l
			\left(m(\xi',i\tau)x_3 A 
			e^{(B\te+A(1-\te))x_3}\right)\right\}\right|  
		\leq C_{\al'}\frac{|x_3|A}{|\tau|^{1/2}+A}
			e^{-b_{\ep,\mu}(|\tau|^{1/2}\te+A)|x_3|}A^{-|\al'|}  \\
		&\leq
			C_{\al'}|x_3|^\de|\tau|^{-1/2}
			e^{-b_{\ep,\mu}|\tau|^{1/2}\theta |x_3|}(|x_3|A)^{1-\de}
			e^{- b_{\ep,\mu} A|x_3|}
			A^{\de-|\al'|} \notag \\
		&\leq
			C_{\al',\de}|x_3|^\de|\tau|^{-1/2}
			e^{-b_{\ep,\mu}|\tau|^{1/2}\theta|x_3|}
			e^{-(b_{\ep,\mu}/2)A|x_3|}
			A^{\de-|\alpha'|} \notag
\end{align}
for some positive constants $C_{\al',\de}$,
which, combined with Proposition \ref{prop:SS01}
with $X=\BR$, $L=2$, $n=2$, and $\sigma=\de$,
furnishes that
\begin{equation*}
	\left|\CF_{\xi'}^{-1}\left[(\tau\pa_\tau)^l
			\left(m(\xi',i\tau)x_3A
			e^{(B\te+A(1-\te))x_3}\right)\right](x')\right| 
		\leq C_{\de}|x_3|^\de|\tau|^{-1/2}
			e^{-b_{\ep,\mu}|\tau|^{1/2}\te |x_3|}|x'|^{-(2+\de)}.
\end{equation*}
On the other hand,
we use \eqref{140825_2} again with $\al'=0$,
and then
\begin{align*}
	&\left|\CF_{\xi'}^{-1}
			\left[(\tau\pa_\tau)^l\left(m(\xi',i\tau)x_3A 
			e^{(B\te+A(1-\te))x_3}\right)\right](x')\right| \\
		&\leq
			C_{\de}|x_3|^\de|\tau|^{-1/2}
			e^{-b_{\ep,\mu}|\tau|^{1/2}\te|x_3|}
			\int_{\BR^2}|\xi'|^\de e^{-(b_{\ep,\mu}/2)|\xi'||x_3|}
			\intd\xi' \\
		&\leq
			C_{\de}|x_3|^\de|\tau|^{-1/2}
			e^{-b_{\ep,\mu}|\tau|^{1/2}\te|x_3|}|x_3|^{-(2+\de)}
\end{align*}
for $l=0,1$ with a positive constant $C_{\de}$.
We thus obtain 
\begin{equation*}
	\left|\CF_{\xi'}^{-1}\left[(\tau\pa_\tau)^l
			\left(m(\xi',i\tau)x_3A
			e^{(B\te+A(1-\te)x_3)}\right)\right](x')
			\right| \\
		\leq
			C_{\de} \frac{|x_3|^{\de}|\tau|^{-1/2}
			e^{-b_{\ep,\mu}|\tau|^{1/2}\te|x_3|}}
			{|x'|^{2+\de}+|x_3|^{2+\de}},
\end{equation*}	
which furnishes that
\begin{equation*}
	\left\|(\tau\pa_\tau)^{l}\CF_{\xi'}^{-1}
		\left[m(\xi',i\tau)x_3A e^{(B\te+A(1-\te))x_3}\right]\right\|_X
		\leq C_{q,\de}\te^{-\frac{1}{q}}|\tau|^{-\frac{1}{2}-\frac{1}{2q}}
\end{equation*}
for $l=0,1$ with a positive constant $C_{q,\de}$.
Then, Proposition \ref{prop:SS01} and \eqref{141213_1}
yields that, for any $t\in\BR\setminus\{0\}$,
\begin{equation*}
	\|\CJ(t)g\|_{L_q(\BR_-^3)}
		\leq C_{q,\de}\int_0^1\te^{-\frac{1}{q}}\intd\te
			|t|^{-\frac{1}{2}+\frac{1}{2q}}\|g\|_{L_q(\BR^2)} 
		\leq C_{q,\de}|t|^{-\frac{1}{2}+\frac{1}{2q}}\|g\|_{L_q(\BR^2)}
\end{equation*}
with some positive constant $C_{q,\de}$ independent of $t$ and $g$.
Analogously, we can prove the estimate of $\nabla\CJ(t)g$ in \eqref{160418_2}.
This completes the proof of the lemma.
\end{proof}

\subsection{Proof of Theorem \ref{theo:sec3_1}}\label{subsec:3_2}


We here prove Theorem \ref{theo:sec3_1} \eqref{theo:sec3_1_1} only.
For the other assertions of Theorem \ref{theo:sec3_1}, we refer e.g. to \cite{SS08,SS12,Saito15b}.
By Lemma \ref{lemm:A1} in the appendix below, there exists an extension $\BH$ of $\Bh$,
defined for $t\in\BR$, such that
$\BH=\Bh$ when $(x,t)\in\BR_-^3\times\BR_+$ and
\begin{equation*}
	\BH \in H_{q,p}^{1,1/2}(\BR_-^3\times\BR)^3 \quad \text{with $\BH=0$ on $\BR_0^3$ $(t<0)$,} \quad 
	\|\BH\|_{H_{q,p}^{1,1/2}(\BR_-^3\times\BR)} \leq C\|\Bh\|_{H_{q,p}^{1,1/2}(\BR_-^3\times\BR_+)}. \notag
\end{equation*}
Since $\SSH=e^t\BH \in H_{q,p,1}^{1,1/2}(\BR_-^3\times\BR)^3$ with $\SSH=0$ on $\BR_0^3$ $(t<0)$, 
we have, by \cite[Theorem 1.2]{SS12}\footnote{
The result, stated in \cite[Theorem 1.2]{SS12}, requires the assumption $\SSH=0$ in $\lhs$ $(t<0)$.
We can, however, relax it to the assumption $\SSH=0$ on $\BR_0^3$ $(t<0)$.}
a  solution $(\Bv,\Fq)$ with
\begin{equation*}
	(\Bv,\Fq) \in W_{q,p,1}^{2,1}(\BR_-^3\times\BR_+)^3 \times
		L_{p,1}(\BR_+,\wh{W}_q^1(\BR_-^3))
\end{equation*}
to the following system: 
\begin{equation*}
	\left\{\begin{aligned}
		\pa_t\Bv-\Di\BT(\Bv,\Fq) &=0
			&& \text{in $\BR_-^3$, $t>0$}, \\
		\di\Bv &=0
			&& \text{in $\BR_-^3$, $t>0$}, \\
		\BT(\Bv,\Fq)\Be_3 &=\SSH
			&& \text{on $\BR_0^3$, $t>0$}, \\
		\Bv|_{t=0}&=0
			&& \text{in $\BR_-^3$.}
	\end{aligned}\right.
\end{equation*}
%
%

Let
	$\Bv =\tp(v_1(x,t),v_2(x,t),v_3(x,t))$ and
	$\SSH =\tp (\SSH_1(x,t),\SSH_2(x,t),\SSH_3(x,t))$, and then
the representation formula of $\Bv$
is given by\footnote{We refer  to \cite[Section 4]{SS12} for the details.}
\begin{align*}
	v_j(x,t) &=
		\frac{2}{\mu}\sum_{k=1}^2\CL_\la^{-1}\CF_{\xi'}^{-1}
			\left[\frac{\xi_j\xi_k B}{A D(A,B)}
			A\CM(x_3)\,\wt{\SSH}_k(\xi',0,\la)\right](x',t) \\
		&-\frac{1}{\mu}\CL_\la^{-1}\CF_{\xi'}^{-1}
			\left[\frac{i\xi_j(B^2+A^2)}{AD(A,B)}
			A\CM(x_3)\,\wt{\SSH}_3(\xi',0,\la)\right](x',t) \\
		&-\frac{1}{\mu}\sum_{k=1}^2\CL_\la^{-1}\CF_{\xi'}^{-1}
			\left[\frac{\xi_j\xi_k(3B-A)}{B D(A,B)}
			e^{Bx_3}\,\wt{\SSH}_k(\xi',0,\la)\right](x',t) \\
		&+\frac{1}{\mu}\CL_\la^{-1}\CF_{\xi'}^{-1}
			\left[\frac{i\xi_j(B-A)}{D(A,B)}
			e^{Bx_3}\,\wt{\SSH}_3(\xi',0,\la)\right](x',t) \\
		&+\frac{1}{\mu}\CL_\la^{-1}\CF_{\xi'}^{-1}
			\left[\frac{1}{B}e^{Bx_3}\,\wt{\SSH}_j(\xi',0,\la)\right](x',t), 
			\quad j=1,2, \\
	v_3(x,t) &=
		-\frac{2}{\mu}\sum_{k=1}^2\CL_\la^{-1}\CF_{\xi'}^{-1}
			\left[\frac{i\xi_kB}{D(A,B)}
			A\CM(x_3)\,\wt{\SSH}_k(\xi',0,\la)\right](x',t) \\
		&-\frac{1}{\mu}\CL_\la^{-1}\CF_{\xi'}^{-1}
			\left[\frac{(B^2+A^2)}{D(A,B)}
			A\CM(x_3)\,\wt{\SSH}_3(\xi',0,\la)\right](x',t) \\
		&-\frac{1}{\mu}\sum_{k=1}^2\CL_\la^{-1}\CF_{\xi'}^{-1}
			\left[\frac{i\xi_k(B-A)}{D(A,B)}
			e^{Bx_3}\,\wt{\SSH}_k(\xi',0,\la)\right](x',t) \\
		&+\frac{1}{\mu}\CL_\la^{-1}\CF_{\xi'}^{-1}
			\left[\frac{A(B+A)}{D(A,B)}
			e^{Bx_3}\,\wt{\SSH}_3(\xi',0,\la)\right](x',t). 
\end{align*}
Since the symbols:
\begin{align*}
	&\frac{\xi_j\xi_k B}{A D(A,B)}, \quad
	\frac{i\xi_j(B^2+A^2)}{AD(A,B)}, \quad
	\frac{\xi_j\xi_k(3B-A)}{B D(A,B)}, \quad
	\frac{i\xi_j(B-A)}{D(A,B)}, \quad
	\frac{1}{B}, \\
	&\frac{i\xi_kB}{D(A,B)}, \quad
	\frac{B^2+A^2}{D(A,B)}, \quad
	\frac{i\xi_k(B-A)}{D(A,B)}, \quad
	\frac{A(B+A)}{D(A,B)}\quad(j,k=1,2)
\end{align*}
belong to $\BBM_{-1,2}(\Sigma_\ep)$ for any $0<\ep<\pi/2$
by Lemmas \ref{lemm:multi1} and \ref{lemm:multi2},
it follows from Lemma \ref{lemm:IJ} that
there exists an operator
$\CC(\tau)\in\CL(L_q(\BR^2)^3,W_q^1(\BR_-^3)^3)$,
$\tau\in\BR$, such that the solution $\Bv$ is represented as
\begin{equation}\label{161216_3}
	\Bv(x,t)
		=\int_0^t[\CC(t-s)\SSH(\cdot,0,s)](x)\,ds \quad (x\in\BR_-^3,\,t>0).
\end{equation}
In addition, for any $\Bf \in L_q(\BR^2)^3$, we have
\begin{equation}\label{141012_10}
	\|\nabla^l\CC(\tau)\Bf\|_{L_q(\BR_-^3)}
		\leq C_{q}|\tau|^{-\frac{1+l}{2}+\frac{1}{2q}}
			\|\Bf\|_{L_q(\BR^2)} \quad (\tau\in\BR\setminus\{0\},\,l=0,1)
\end{equation}
with some positive constant $C_{q}$ independent of $\tau$ and $\Bf$.

Let $(\Bu,\Fp)=(e^{-t}\Bv,e^{-t}\Fq)$, and then $(\Bu,\Fp)\in W_{q,p}^{2,1}(\lhs\times\BR_+)^3\times L_p(\BR_+,\wh W_q^1(\lhs))$
satisfies the following system:
\begin{equation*}
\left\{\begin{aligned}
	\pa_t\Bu+\Bu-\Di\BT(\Bu,\Fp) &=0
		&& \text{in $\BR_-^3$, $t>0$}, \\
	\di\Bu&=0
		&& \text{in $\BR_-^3$, $t>0$}, \\
	\BT(\Bu,\Fp)\Be_3&=e^{-t}\SSH
		&& \text{on $\BR_0^3$, $t>0$}, \\
	\Bu|_{t=0}	&=0 && \text{in $\BR_-^3$}.
\end{aligned}\right.
\end{equation*}
Since $e^{-t}\SSH=\Bh$ in $\BR_-^3$ for $t>0$,
we observe that $(\Bu,\Fp)$ is a solution to \eqref{linear2}.
In addition, it holds by \eqref{161216_3} that
\begin{equation*}
	\Bu(x,t) = \int_0^t [e^{-(t-s)}\CC(t-s)\Bh(\cdot,0,s)](x)\intd s \quad (x\in\BR_-^3,\,t>0).
\end{equation*}
%
Setting $\CB(\tau)=e^{-\tau}\CC(\tau)$, together with \eqref{141012_10},
completes the proof of Theorem \ref{theo:sec3_1} \eqref{theo:sec3_1_1}.

\subsection{Proof of Theorem \ref{theo:sec3_2}}\label{subsec:3_3}
In this subsection, we prove Theorem \ref{theo:sec3_2} by using Theorem \ref{theo:sec3_1}.
We first show Theorem \ref{theo:sec3_2} \eqref{theo:sec3_2_1}. 
For $t>0$,
\begin{align*}
	\Bu(x,t)
		= \left(\int_0^{t/2}+\int_{t/2}^{t}\right)[\CB(t-s)\Bh(\cdot,0,s)](x)\intd s
		=: \Bu_1(x,t)+\Bu_2(x,t).
\end{align*}
Concerning $\Bu_1(x,t)$, it follows from the trace theorem that,
for $p'=p/(p-1)$ and $l=0,1$,
\begin{align*}
	&\|\nabla^l \Bu_1(t)\|_{L_r(\BR_-^3)}
		\leq
			C\int_0^{t/2}e^{-(t-s)}(t-s)^{-\frac{1+l}{2}+\frac{1}{2r}}
			\|\Bh(s)\|_{L_r(\BR_0^3)}\intd s \\
		&\leq
			Ce^{-t/2}t^{-\frac{1+l}{2}+\frac{1}{2r}}
			\left(\int_0^{t/2}\,ds\right)^{1/p'}
			\|\Bh\|_{L_p(\BR_+,W_r^1(\BR_-^3))} \\
		&\leq
			C e^{-t/4}\|\Bh\|_{L_p(\BR_+,W_r^1(\BR_-^3))}
\end{align*}
with some positive constant $C $, where we note by \eqref{pq:lin2} that
\begin{equation}\label{141103_1}
	\frac{1}{p'}-\frac{1+l}{2}+\frac{1}{2r}
		\geq \frac{1}{p'}-1+\frac{1}{2q}
		= -\frac{1}{p}+\frac{1}{2q}
		> \frac{2}{q}-\frac{1}{2}>0.
\end{equation}
On the other hand, noting \eqref{141103_1}, we observe that
\begin{align*}
	&\|\nabla^l \Bu_2(t)\|_{L_r(\BR_-^3)}
		\leq
			C\int_{t/2}^t e^{-(t-s)}(t-s)^{-\frac{1+l}{2}+\frac{1}{2r}}
			\|\Bh(s)\|_{W_r^1(\BR_-^3)}\,ds \\
			&\leq
				C (t+2)^{-\Fc_2}
				\int_{t/2}^t e^{-(t-s)}(t-s)^{-\frac{l+1}{2}+\frac{1}{2r}}
				(s+2)^{\Fc_2}\|\Bh(s)\|_{W_r^1(\BR_-^3)}\intd s \\
			&\leq
				C (t+2)^{-\Fc_2}\left(\int_{t/2}^t
				e^{-p'(t-s)}(t-s)^{-p'\left(\frac{l+1}{2}-\frac{1}{2r}\right)}
				\intd s\right)^{1/p'}
				\|\Bh\|_{L_p^{\Fc_2}(\BR_+,W_r^1(\BR_-^3))}\\
			&\leq
				C (t+2)^{-{\Fc_2}}
				\|\Bh\|_{L_p^{\Fc_2}(\BR_+,W_r^1(\BR_-^3))}
\end{align*}
with a positive constant $C$,
which, combined with the estimates of $\Bu_1(x,t)$,
completes the proof of Theorem \ref{theo:sec3_2} \eqref{theo:sec3_2_1}. 

We next prove Theorem \ref{theo:sec3_2} \eqref{theo:sec3_2_2}. 
By Lemma \ref{lemm:A1} in the appendix below,
there exists an extension $\BG$ of $(t+2)^{\Fd_2}\Bh$, defined for $t\in\BR$,
such that $\BG = (t+2)^{\Fd_2}\Bh$ for $(x,t)\in\BR_-^3\times\BR_+$ and that
\begin{align}\label{161216_13}
	&\BG\in H_{q,p}^{1,1/2}(\BR_-^3\times\BR)^3 \quad \text{with $\BG=0$ on $\BR_0^3$ ($t<0$),} \\
	&\|\BG\|_{H_{q,p}^{1,1/2}(\BR_-^3\times\BR)}
		\leq C\|(t+2)^{\Fd_2}\Bh\|_{H_{q,p}^{1,1/2}(\BR_-^3\times\BR_+)}. \notag
\end{align}
We then see that $\BU = (t+2)^{\Fd_2}\Bu$ and $P=(t+2)^{\Fd_2}\Fp$ satisfy
\begin{equation*}
\left\{\begin{aligned}
	\pa_t\BU+\BU-\Di\BT(\BU,P)
		&=-\Fd_2(t+2)^{-1+\Fd_2}\Bu
		&& \text{in $\BR_-^3$, $t>0$,} \\
	\di\BU
		&=0 && \text{in $\BR_-^3$, $t>0$,} \\
	\BT(\BU,P)\Be_3
		&=\BG && \text{on $\BR_0^3$, $t>0$}, \\
	\BU|_{t=0}
		&=0 && \text{in $\BR_-^3$}.
\end{aligned}\right.
\end{equation*}
It follows from $p(1+\Fc_2-\Fd_2)>1$ and Theorem \ref{theo:sec3_2} \eqref{theo:sec3_2_1} that
\begin{align}\label{141014_3}
	\|(t+2)^{-1+\Fd_2}\Bu\|_{L_p(\BR_+,L_r(\BR_-^3))}
	&\leq C\|(t+2)^{-(1+\Fc_2-\Fd_2)}\|_{L_p(\BR_+)}
			\|\Bh\|_{L_p^{\Fc_2}(\BR_+,W_r^1(\BR_-^3))}  \\
		&\leq C\|\Bh\|_{L_p^{\Fc_2}(\BR_+,W_r^1(\BR_-^3))}. \notag
\end{align}
We thus have, by \cite[Theorem 5.1]{SS08}, \eqref{161216_13}, and \eqref{141014_3},
\begin{align}\label{141014_2}
	&\|(\pa_t\BU,\BU,\nabla\BU,\nabla^2\BU,\nabla P)\|_{L_p(\BR_+,L_r(\BR_-^3))} \\
		&\leq
			C\left(\|(t+2)^{-1+\Fd_2}\Bu\|_{L_p(\BR_+,L_r(\BR_-^3))}
			+\|\BG\|_{H_{r,p}^{1,1/2}(\BR_-^3\times\BR)}\right) \notag \\
		&\leq
			C\left(\|\Bh\|_{L_p^{\Fc_2}(\BR_+,W_r^1(\BR_-^3))}
			+\|(t+2)^{\Fd_2}\Bh\|_{H_{r,p}^{1,1/2}(\BR_-^3\times\BR_+)}\right) \notag
\end{align}
with some positive constant $C$. Noting that
	$(t+2)^{\Fd_2}\pa_t\Bu = \pa_t\BU-\Fd_2(t+2)^{-1+\Fd_2}\Bu$,	
we obtain, by \eqref{141014_3} and \eqref{141014_2},
\begin{equation*}
	\|(t+2)^{\Fd_2}\pa_t\Bu\|_{L_p(\BR_+,L_r(\BR_-^3))} 
		\leq
			C\left(\|\Bh\|_{L_p^{\Fc_2}(\BR_+,W_r^1(\BR_-^3))}
			+\|(t+2)^{\Fd_2}\Bh\|_{H_{r,p}^{1,1/2}(\BR_-^3\times\BR_+)}\right),
\end{equation*}
which, combined with \eqref{141014_2}, completes the proof of Theorem \ref{theo:sec3_2} \eqref{theo:sec3_2_2}.

\section{Time-weighted estimates: homogeneous boundary data}\label{sec4}
In this section, we prove time-weighted estimates of solutions 
for the following two linear time-dependent problems: 
\begin{align}
	&\left\{\begin{aligned}
		\pa_t\Bu-\Di\BT(\Bu,\Fp)
			&=0 && \text{in $\BR_-^3$, $t>0$,} \\
		\di\Bu
			&=0 && \text{in $\BR_-^3$, $t>0$,} \\
		\BT(\Bu,\Fp)\Be_3 + (c_g-c_\si\De')h\Be_3
			&=0 && \text{on $\BR_0^3$, $t>0$,} \\
		\pa_t h - u_3
			&= 0 && \text{on $\BR_0^3$, $t>0$,} \\
		\Bu|_{t=0}
			&= \Bu_0 && \text{in $\BR_-^3$,} \\
		h|_{t=0}
			&= h_0 && \text{on $\BR^2$,}
	\end{aligned}\right. \label{linear3} \\
	&\left\{\begin{aligned}
		\pa_t\Bu-\Di\BT(\Bu,\Fp)
			&=\Bf && \text{in $\BR_-^3$, $t>0$,} \\
		\di\Bu 
			&=0 && \text{in $\BR_-^3$, $t>0$,} \\
		\BT(\Bu,\Fp)\Be_3+(c_g-c_\si\De')h\Be_3
			&=0 && \text{on $\BR_0^3$, $t>0$,} \\
		\pa_t h-u_3
			&=k && \text{on $\BR_0^3$, $t>0$,} \\
		\Bu|_{t=0}
			&=0 && \text{in $\BR_-^3$,} \\
		h|_{t=0}
			&=0 && \text{on $\BR^2$,}
	\end{aligned}\right.	\label{linear4}
\end{align}
associated with the following resolvent problem:
\begin{equation}\label{resolvent}
	\left\{\begin{aligned}
		\la\BU-\Di\BT(\BU,P)
			&= \BF && \text{in $\BR_-^3$,} \\
		\di\BU
			&=0 && \text{in $\BR_-^3$,} \\
		\BT(\BU,P)\Be_3 + (c_g-c_\si\De')H\Be_3
			&= 0 && \text{on $\BR_0^3$,} \\
		\la H - U_3 
			&= K && \text{on $\BR_0^3$.}
	\end{aligned}\right.
\end{equation}

First, we recall resolvent estimates of System \eqref{resolvent}
that were proved in \cite[Theorem 1.1]{SS09} and \cite[Theorem 1.3]{SS12}.

\begin{lemm}\label{lemm:SS09}
Let $0<\ep<\pi/2$ and $1<q<\infty$.
Then there is a constant $\ga_0\geq 1$,
depending only on $\ep$, such that,
for any $\la\in\Si_{\ep,\ga_0}$, $\BF\in L_q(\BR_-^3)^3$,
and $K\in W_q^{2-1/q}(\BR^2)$,
System \eqref{resolvent} admits a unique solution
$(\BU,P,H)\in W_q^2(\BR_-^3)^3\times \wh{W}_q^1(\BR_-^3)\times W_q^{3-1/q}(\BR^2)$.
Furthermore, the solution $(\BU,P,H)$ satisfies 
\begin{align}\label{160609_1}
	&\|(\la\BU,\la^{1/2}\nabla\BU,\nabla^2\BU,\nabla P)\|_{L_q(\BR_-^3)}
	+\|\la H\|_{W_q^{2-1/q}(\BR^2)}+\|H\|_{W_q^{3-1/q}(\BR^2)} \\
		&\leq
			C_{q,\ga_0,\ep}
			\left(\|\BF\|_{L_q(\BR_-^3)}+\|K\|_{W_q^{2-1/q}(\BR^2)}\right)
			\notag
\end{align}
for any $\la\in \Si_{\ep,\ga_0}$ with some positive constant 
$C_{q,\ga_0,\ep}$ independent of $\la$.
\end{lemm}

Next, we formulate System \eqref{linear3} in the semigroup setting.
By \cite[Proposition 3.3.2]{Saito15}, we have

\begin{lemm}\label{lemm:weakDN}
Let $1<q<\infty$ and $q'=q/(q-1)$.
Then for any $\Bf\in L_q(\BR_-^3)^3$
there is a unique solution $\Fp\in\wh{W}_{q,0}^1(\BR_-^3)$
to the variational equation:
\begin{equation*}
	(\nabla\Fp,\nabla\ph)_{\BR_-^3} =
		(\Bf,\nabla\ph)_{\BR_-^3}
		\quad\text{for all $\ph\in \wh{W}_{q',0}^1(\BR_-^3)$.}
\end{equation*}
Furthermore, the solution $\Fp$ satisfies the estimate:
$\|\nabla\Fp\|_{L_q(\BR_-^3)}\leq C_q\|\Bf\|_{L_q(\BR_-^3)}$
with a positive constant $C_q$
independent of $\Fp$, $\Bf$, and $\ph$.
\end{lemm}

By Lemma \ref{lemm:weakDN},
we see that, for $\Bf\in L_q(\BR_-^3)^3$ and $g\in W_q^{1-1/q}(\BR^2)$,
there is a unique solution $\Fp\in\wh{W}_{q,0}^1(\BR_-^3)$
to the variational equation:
\begin{equation*}
	(\nabla\Fp,\nabla\ph)_{\BR_-^3}
		= (\Bf-\nabla\wt{g},\nabla\ph)_{\BR_-^3}
		\quad\text{for all $\ph\in\wh{W}_{q',0}^1(\BR_-^3)$,}
\end{equation*}
where $\wt g$, defined on $\BR_-^3$, is an extension of $g$ such that
$g=\wt g$ on $\BR_0^3$ and
$\|\wt{g}\|_{W_q^1(\BR_-^3)}\leq C_q\|g\|_{W_q^{1-1/q}(\BR^2)}$
with some positive constant $C_q$ independent of $g$, $\wt g$.
The solution $\Fp$ then satisfies the estimate:
\begin{equation*}
	\|\nabla\Fp\|_{L_q(\BR_-^3)}
		\leq C_q\left(\|\Bf\|_{L_q(\BR_-^3)}+\|g\|_{W_q^{1-1/q}(\BR^2)}\right).
\end{equation*}
Let $\wh{W}_{q,0}^1(\BR_-^3)+W_q^1(\BR_-^3) 
=\{\Fp_1+\Fp_2 \mid \Fp_1\in \wh{W}_{q,0}^1(\BR_-^3),\Fp_2\in W_q^1(\BR_-^3)\}$.
Setting $\Fq=\Fp+\wt{g}\in\wh{W}_{q,0}^1(\BR_-^3)+W_q^1(\BR_-^3)$,
we observe that $\Fq$ satisfies 
\begin{equation*}
	(\nabla\Fq,\nabla\ph)_{\BR_-^3}
		= (\Bf,\nabla\ph)_{\BR_-^3}	
		\quad\text{for all $\ph\in\wh{W}_{q',0}^1(\BR_-^3)$,}
\end{equation*}
subject to $\Fq=g$ on $\BR_0^3$,
and that the last inequality yields 
\begin{equation}\label{160606_7}
	\|\nabla\Fq\|_{L_q(\BR_-^3)} \leq
		C_q\left(\|\Bf\|_{L_q(\BR\-^3)}+\|g\|_{W_q^{1-1/q}(\BR^2)}\right).
\end{equation}

Now,
for $\Bu\in W_q^2(\BR_-^3)^3$ and $h\in W_q^{3-1/q}(\BR^2)$,
we define $K_1(\Bu)$, $K_2(h)\in \wh{W}_{q,0}^1(\BR_-^3)+W_q^1(\BR_-^3)$
as solutions to  
\begin{align*}
	&\left\{\begin{aligned}
		&(\nabla K_1(\Bu),\nabla\ph)_{\BR_-^3}
			=(\Di(\mu\BD(\Bu))-\nabla\di\Bu,\nabla\ph)_{\BR_-^3}
			\quad \text{for all $\ph\in\wh{W}_{q',0}^1(\BR_-^3)$,} \\
		&\,K_1(\Bu) = \Be_3\cdot(\mu\BD(\Bu)\Be_3)-\di\Bu
			\quad\text{on $\BR_0^3$,}
	\end{aligned}\right. \\
	&\left\{\begin{aligned}
		&(\nabla K_2(h),\nabla \ph)_{\BR_-^3} =0
			\quad\text{for any $\ph\in\wh{W}_{q',0}^1(\BR_-^3)$,} \quad 
		\,K_2(h) = (c_g-c_\si\De')h \quad \text{on $\BR_0^3$.}
	\end{aligned}\right.
\end{align*}
Then $K_1(\Bu)$ and $K_2(h)$ satisfy by \eqref{160606_7} the estimates:
\begin{equation}\label{160606_8}
	\|\nabla K_1(\Bu)\|_{L_q(\BR_-^3)} \leq C_q\|\nabla\Bu\|_{W_q^1(\BR_-^3)}, \quad 
	\|\nabla K_2(h)\|_{L_q(\BR_-^3)} \leq C_q \|(h,\De' h)\|_{W_q^{1-1/q}(\BR^2)}. 
\end{equation} 
We here set an operator $\CA_q$ and its domain $\CD(\CA_q)$ as
\begin{align*}
	&\CA_q(\Bu,h) =
		(\Di\BT(\Bu,K_1(\Bu)+K_2(h)),u_3) \quad
 		\text{for $(\Bu,h)\in \CD(\CA_q)$,} \quad 
		\CD(\CA_q)=\wt{W}_q^2(\BR_-^3)\times W_q^{3-1/q}(\BR^2), \\
	&\wt{W}_q^2(\BR_-^3) =
		\{\Bu\in W_q^2(\BR_-^3)^3\cap J_q(\BR_-^3)
		\mid \text{$\BD(\Bu)\Be_3 -(\Be_3\cdot\BD(\Bu)\Be_3)\Be_3=0$ on $\BR_0^3$}\}.
\end{align*}
System \eqref{resolvent} is then equivalent to
the reduced system: $(\la-\CA_q)(\BU,H) = (\BF, K)$
for $(\BF,K)\in X_q := J_q(\BR_-^3)\times W_q^{2-1/q}(\BR^2)$
(cf. \cite[Section 3]{Shibata13} in more detail).
Let $(\la-\CA_q)^{-1}$ be the resolvent operator of $\CA_q$.
By the equivalence mentioned above and \eqref{160609_1},
we have
\begin{equation*}
	\|(\la-\CA_q)^{-1}\|_{\CL{(X_q)}} \leq \frac{C_{q,\ga_0,\ep}}{|\la|}
		\quad (\la\in\Si_{\ep,\ga_0}).
\end{equation*}
This estimate furnishes that $\CA_q$ generates
an analytic semigroup on $X_q$,
which is stated as follows:

\begin{lemm}\label{lemm:sg_1}
Let $1<q<\infty$. Then $\CA_q$ generates
an analytic $C_0$-semigroup $\{e^{\CA_q t}\}_{t\geq0}$ on $X_q$.
In addition, there are positive constants $\ga_1$, $C_{q,\ga_1}$ such that
for any $t>0$
\begin{alignat}{2}\label{160606_9}
	\|e^{\CA_q t}\BU_0\|_{X_q} 
		&\leq C_{q,\ga_1}e^{\ga_1 t}\|\BU_0\|_{X_q}
		&\quad
		&(\BU_0\in X_q),  \\
	\|\pa_t e^{\CA_q t}\BU_0\|_{X_q}
		&\leq C_{q,\ga_1}t^{-1}e^{\ga_1 t}\|\BU_0\|_{X_q}
		&\quad
		&(\BU_0\in X_q), \notag \\
	\|\pa_t e^{\CA_q t}\BU_0\|_{X_q}
		&\leq C_{q,\ga_1}e^{\ga_1 t}\|\BU_0\|_{\CD_q(\CA)}
		&\quad
		&(\BU_0\in \CD(\CA_q)). \notag
\end{alignat}
\end{lemm}

For $1<p,q<\infty$, we set function spaces
$\CD_{q,p}$, $\wt{B}_{q,p}^{2(1-1/p)}(\BR_-^3)$ as
\begin{align*}
	&\wt{B}_{q,p}^{2(1-1/p)}(\BR_-^3)
		=(J_q(\BR_-^3),\wt{W}_q^2(\BR_-^3))_{1-1/p,p}, \\
	&\CD_{q,p}=(X_q,\CD(\CA_q))_{1-1/p.p}
		=\wt{B}_{q,p}^{2(1-1/p)}(\BR_-^3)
		\times B_{q,p}^{3-1/p-1/q}(\BR^2).	
\end{align*}
On the other hand,
let $P_1$, $P_2$ be projections defined by
$P_1: X_q\to J_q(\BR_-^3)$, $P_2:X_q\to W_q^{2-1/q}(\BR^2)$.
We then have 

\begin{lemm}\label{lemm:sg_2}
Let $1<p,q<\infty$,
and let $\Bu=P_1(e^{\CA_q t}\BU_0)$, $h=P_2(e^{\CA_q t}\BU_0)$,
and $\Fp=K_1(P_1(e^{\CA_q t}\BU_0))+K_2(P_2(e^{\CA_q t}\BU_0))$
with $\BU_0=(\Bu_0,h_0)\in\CD_{q,p}$.
Then the following assertions hold. 
\begin{enumerate}[$(1)$]
	\item
		It holds that
		\begin{align}\label{160606_5}
			&\Bu \in C^1(\BR_+,J_q(\BR_-^3))
				\cap C(\BR_+,\wt{W}_q^2(\BR_-^3)), \\
			&h\in C^1(\BR_+,W_q^{2-1/q}(\BR^2))
				\cap C(\BR_+,W_q^{3-1/q}(\BR^2)), \notag \\
			&\Fp \in C(\BR_+,\wh{W}_q^1(\BR_-^3)), \notag
		\end{align}
		and furthermore, $(\Bu,\Fp,h)$ solves uniquely System \eqref{linear3} for all $t>0$ with
		\begin{alignat*}{2}
			&\Bu \in C(\BR_+,\wt B_{q,p}^{2(1-1/p)}(\BR_-^3)), \quad
			& &\lim_{t\to 0+}\|\Bu(t)-\Bu_0\|_{\wt{B}_{q,p}^{2(1-1/p)}(\BR_-^3)}=0, \\
			&h \in C(\BR_+,B_{q,p}^{3-1/p-1/q}(\BR^2)), \quad
			& &\lim_{t\to 0+}\|h(t)-h_0\|_{B_{q,p}^{3-1/p-1/q}(\BR^2)}=0. \notag
		\end{alignat*}
	\item
		For any $T>0$, there is a positive constant $C_{p,q,T}$ such that
		\begin{align}\label{161001_10}
			&\|\Bu\|_{W_{q,p}^{2,1}(\BR_-^3\times(0,T))}
				+\|\nabla\Fp\|_{L_p((0,T),L_q(\lhs))} 
				+\|\pa_t h\|_{L_p((0,T),W_q^{2-1/q}(\BR^2))} \\ 
				&+\|h\|_{L_p((0,T),W_q^{3-1/q}(\BR^2))} 
				 + \|\pa_t\CE(h)\|_{L_p((0,T),\wh{ W}_q^2(\lhs))}
				+\|\CE(h)\|_{L_p((0,T), \wh {W}_q^3(\lhs))} \notag\\
			&\leq
				C_{p,q,T}\|\BU_0\|_{\CD_{q,p}}. \notag
		\end{align}
\end{enumerate}
\end{lemm}

\begin{proof}
We here prove (2) only.
By using \eqref{160606_8} and \eqref{160606_9},
we can prove that
\begin{align*}
	&\|\Bu\|_{W_{q,p}^{2,1}(\BR_-^3\times(0,T))}
		+\|\nabla\Fp\|_{L_p((0,T),L_q(\lhs))} 
		 +\|\pa_t h\|_{L_p((0,T),W_q^{2-1/q}(\BR^2))} 
		+\|h\|_{L_p((0,T),W_q^{3-1/q}(\BR^2))} \\
		&\leq
		C_{p,q,T}\|\BU_0\|_{\CD_{q,p}}
\end{align*}
in the same manner as in \cite[Theorem 3.9]{SS08}.
Since $\CE$ is a bounded linear operator from $W_q^{m-1/q}(\BR^2)$ to $\wh W_q^m(\lhs)$ for $m=2,3$
and since $\pa_t\CE(h)=\CE(\pa_t h)$, 
the last inequality implies \eqref{161001_10}.
This completes the proof of the lemma.
\end{proof}

Next, concerning System \eqref{linear4},
we recall \cite[Theorem 1.4]{SS12}
(cf. Subsection \ref{subsec3_1} for function spaces with exponential weight).


\begin{lemm}\label{lemm:SS12}
Let $1<p,q<\infty$, and let $\gamma_0$ be the same constant as in Lemma $\ref{lemm:SS09}$.
Then for any
\begin{equation*}
	\Bf \in {}_0L_{p,\ga_0}(\BR,L_q(\BR_-^3))^3,\quad
	k \in {}_0L_{p,\ga_0}(\BR,W_q^{2-1/q}(\BR^2)),
\end{equation*}
System \eqref{linear4}
admit a unique solution $(\Bu,\Fp,h)$ with
\begin{align*}
	&\Bu \in {}_0W_{q,p,\ga_0}^{2,1}(\BR_-^3\times\BR)^3, \quad
	\Fp \in {}_0L_{p,\ga_0}(\BR,\wh{W}_q^1(\BR_-^3)), \\
	&h \in {}_0W_{p,\ga_0}^1(\BR,W_q^{2-1/q}(\BR^2))
		\cap {}_0L_{p,\ga_0}(\BR,W_q^{3-1/q}(\BR^2)), \notag \\
	&\CE(h) \in {}_0W_{p,\ga_0}^1(\BR,\wh{W}_q^2(\BR_-^3))
		\cap {}_0L_{p,\ga_0}(\BR,\wh{W}_q^3(\BR_-^3)). \notag
\end{align*}
In addition, the solution $(\Bu,\Fp,h)$ satisfies the estimate:
\begin{align}\label{141121_5}
	&\|e^{-\ga_0 t}(\pa_t\Bu,\Bu,\nabla\Bu,\nabla^2\Bu)\|_{L_p(\BR_+,L_q(\BR_-^3))} 
		+\|e^{-\ga_0 t}\nabla^2\CE(u_3|_{\BR_0^3})\|_{L_p(\BR_+,L_q(\BR_-^3))} \\
		&+\|e^{-\ga_0 t}\nabla \Fp\|_{L_p(\BR_+,L_q(\BR_-^3))}
		+\|e^{-\ga_0 t}\pa_t h\|_{L_p(\BR_+,W_q^{2-1/q}(\BR^2))}
		+\|e^{-\ga_0 t}h\|_{L_p(\BR_+,W_q^{3-1/q}(\BR^2))} \notag \\
		&+\|e^{-\ga_0 t}\pa_t \CE(h)\|_{L_p(\BR_+,\wh{W}_q^2(\BR_-^3))}
		+\|e^{-\ga_0 t}\CE(h)\|_{L_p(\BR_+,\wh{W}_q^3(\BR_-^3))} \notag \\
	&\leq
		C_{p,q,\ga_0}\left(\|e^{-\ga_0 t}\Bf\|_{L_p(\BR_+,L_q(\BR_-^3))}
		+\|e^{-\ga_0 t}k\|_{L_p(\BR_+,W_q^{2-1/q}(\BR^2))}\right) \notag
	\end{align}
for a positive constant $C_{p,q,\ga_0}$.
\end{lemm}

The aim of this section is to prove the following two theorems:

\begin{theo}\label{theo:sec4_1}
Let $p,q$ be exponents satisfying 
\begin{equation}\label{pq:lin3}
	2<p<\infty, \quad 3<q<4, \quad p\left(\frac{2}{q}-\frac{1}{2}\right)>1
\end{equation}
and set $\bar{q}=2/q$.
Let $(\Bu_0,h_0)\in \wt\BBI_1\times\BBI_2$ with
\begin{equation*}
	\wt\BBI_1 =\wt B_{q,p}^{2-2/p}(\BR_-^3)^3 \cap \wt B_{2,p}^{2-2/p}(\BR_-^3)^3 \cap L_{\bar{q}}(\BR_-^3)^3.
\end{equation*}
Suppose that $(\Bu,\Fp,h)$ is the solution of System \eqref{linear3}
obtained in Lemma $\ref{lemm:sg_2}$, and set $\Bz = (\Bu,\Fp,h,\CE(h))$.
Let $\Fa_1,\Fa_2$ be positive numbers satisfying
\begin{equation}\label{assump:sg}
	p\left(\Fm\left(\bar{q},q\right)+1/4-\Fa_1\right)>1, \quad 
	p\left(\Fm\left(\bar{q},q\right)+1-\Fa_2\right)>1.
\end{equation}
It then holds that
\begin{equation*}
	\CN_{q,p}(\Bz;\Fa_1,\Fa_2)
		+\sum_{r\in\{q,2\}}(\CM_{r,p}(\Bz)+\CN_{r}(\Bz))+\|\pa_t\CE(h)\|_{L_\infty(\BR_+,L_2(\lhs))}  
	\leq C_{p,q,\bar{q}}\|(\Bu_0,h_0)\|_{\wt\BBI_1\times \BBI_2} 
\end{equation*}
for some positive constant $C_{p,q,\bar{q}}$.
\end{theo}

\begin{theo}\label{theo:sec4_2}
Let $p,q$ be exponents satisfying \eqref{pq:lin},
and let $\Fa_1$, $\Fa_2$, $\Fb_1$, $\Fb_2$, $\Fb_3$, and $\Fb_4$
be positive numbers satisfying \eqref{ab}.
Suppose that $\Fa_0=\max(\Fa_1,\Fa_2)$
and the right members $\Bf=\Bf_1+\Bf_2+\Bf_3$, $k$ of System \eqref{linear4}
satisfy the following  conditions:
\begin{enumerate}[$(1)$]
	\item
		$\Bf_1\in\BBF_1 \cap \wt\BBF_1(\Fa_0,\Fb_1)$,
		$\Bf_2\in\BBF_2 \cap \wt\BBF_2(\Fa_0,\Fb_2)$,
		$\Bf_3\in\BBF_3 \cap \wt\BBF_3(\Fa_0,\Fb_3);$
	\item
		$k\in \BBK \cap \wt\BBK(\Fa_0,\Fb_4)
		\cap\BBA_4\cap \BBB_1$.
\end{enumerate}	
Let $(\Bu,\Fp,h)$ be the solution of System \eqref{linear4}
obtained in Lemma $\ref{lemm:SS12}$,
and let $\Bz = (\Bu,\Fp,h,\CE(h))$.
It then holds that
\begin{equation}\label{160510_1}
	\CN_{q,p}(\Bz;\Fa_1,\Fa_2)
		+ \sum_{r\in\{q,2\}}
		\left(\CM_{r,p}(\Bz) + \CN_r(\Bz)\right) 
	\leq
		C_{p,q}\left(\sum_{i=1}^3\|\Bf_i\|_{\BBF_i\cap\wt\BBF_i(\Fa_0,\Fb_i)}
		+\|k\|_{\BBK\cap\wt\BBK(\Fa_0,\Fb_3)\cap\BBA_4\cap \BBB_1}\right) 
\end{equation}
for some positive constant $C_{p,q}$.
\end{theo}

\subsection{Preliminaries}
In this subsection, we introduce
representation formulas of solutions of \eqref{linear3}, \eqref{linear4}
and some lemmas
in order to prove Theorems \ref{theo:sec4_1} and \ref{theo:sec4_2}.
Let  $\BU=\BU(x,\la)$, $P=P(x,\la)$, and $H=H(x,\la)$
be the solutions, obtained in Lemma \ref{lemm:SS09},
to the resolvent problem \eqref{resolvent} in what follows.

First, we decompose $(\BU,P,H)$ in the same way as in
\cite[Sections 3.2, 3.4]{Saito15}, \cite[Sections 2, 3]{SaS1}.
The solutions $\BU$, $P$ can be written as 
$\BU=\BU^0+\BV$ and $P=P^0+Q$, where
$(\BU^0,P^0)$ and $(\BV,Q)$ are solutions to 
\begin{align*}
	&\left\{\begin{aligned}
		\la\BU^0 -\Di\BT(\BU^0,P^0) &= \BF
			&& \text{in $\BR_-^3$,} \\
		\di\BU^0 &= 0 
			&& \text{in $\BR_-^3$,} \\
		\BT(\BU^0,P^0)\Be_3 &= 0 
			&& \text{on $\BR_0^3$,}
	\end{aligned}\right. \\
	&\left\{\begin{aligned}
		\la\BV-\Di\BT(\BV,Q) =0& 
			&& \text{in $\BR_-^3$,} \\
		\di\BV = 0& 
			&& \text{in $\BR_-^3$,} \\
		\BT(\BV,Q)\Be_3+(c_g-c_\si\De')H\Be_3 = 0&
			&& \text{on $\BR_0^3$,} \\
		\la H -V_3 = K + U_3^0&
			&& \text{on $\BR_0^3$.}
	\end{aligned}\right.
\end{align*}
Note that $U_3^0$ is the third component of $\BU^0$ and
$V_3$ the third component of $\BV$.
For functions $f=f(x)$ defined on $\BR_-^3$,
we introduce the even extension $f^e=f^e(x)$ and 
the odd extension $f^o=f^o(x)$ as follows:
\begin{equation}\label{oddeven}
	f^e(x) =
		\left\{\begin{aligned}
			&f(x',-x_N), && x_N>0, \\
			&f(x',x_N), && x_N<0, 
		\end{aligned}\right.\quad
	f^o(x) =
		\left\{\begin{aligned}
			&-f(x',-x_N), && x_N>0, \\
			&f(x',x_N), && x_N<0.
		\end{aligned}\right.
\end{equation}
Setting $\iota\BF = \tp(F_1^o,F_2^o,F_3^e)$, 
we further decompose $\BU^0$, $P^0$ in the following manner:
$\BU^0=\BV^1+\BV^2$ and $P^0=Q^1+Q^2$,
where $(\BV^1,Q^1)$, $(\BV^2,Q^2)$ are solutions to
\begin{align*}
	&\left\{\begin{aligned}
		\la\BV^1 -\Di\BT(\BV^1,Q^1) &=\iota\BF && \text{in $\BR^3$,} \\
		\di\BV^1 &=0 && \text{in $\BR^3$,} 
	\end{aligned}\right. \\
	&\left\{\begin{aligned}
		\la\BV^2 -\Di\BT(\BV^2,Q^2) &= 0 && \text{in $\BR_-^3$,}\\
		\di\BV^2 &=0 && \text{in $\BR_-^3$,} \\
		\mu(\pa_1 V_3^2+\pa_3 V_1^2) &= -\mu(\pa_1 V_3^1+\pa_3 V_1^1)  && \text{on $\BR_0^3$, } \\
		\mu(\pa_2 V_3^2+\pa_3 V_2^2) &= -\mu(\pa_2 V_3^1+\pa_3 V_2^1)  && \text{on $\BR_0^3$, } \\
		2\mu\pa_3 V_3^2 - Q^2 &= 0 && \text{on $\BR_0^3$.}
	\end{aligned}\right.
\end{align*}
Here $\BV^l = \tp(V_1^l,V_2^l,V_3^l)$ for $l=1,2$
and we have used the fact that
$\pa_3 V_3^1=0$ on $\BR_0^3$ and $Q^1=0$ on $\BR_0^3$
by the definition of extension $\iota\BF$ (cf. \cite[Section 4]{SS03})
to obtain the last equation.

Secondly, we introduce representation formulas of
$\BV^1$, $Q^1$, $\BV^2$, $Q^2$, $\BV$, $Q$, and $H$ as above.
Such formulas are given by
\begin{align}\label{160609_10}
	\BV^1(x,\la) &=
		\CF_\xi^{-1}\left[\frac{1}{\la+\mu|\xi|^2}\left(\BI-\frac{\xi\otimes\xi}{|\xi|^2}\right)\CF_x[\iota\BF](\xi)\right](x), \quad
	Q^1(x,\la) =
		-\CF_\xi^{-1}\left[\frac{i\xi\cdot\CF_x[\iota\BF](\xi)}{|\xi|^2}\right](x),
		 \\
	\BV^2(x,\la) &=
		\int_{-\infty}^0 \CF_{\xi'}^{-1}
		\left[\CU^0(\xi',x_3,y_3,\la)\,\wh{\BF}(\xi',y_3)\right](x')\intd y_3,
		\notag \\
	Q^2(x,\la) &=
		\int_{-\infty}^0 \CF_{\xi'}^{-1}
		\left[\CP^0(\xi',x_3,y_3,\la)\,\wh{\BF}(\xi',y_3)\right](x')\intd y_3,
		\notag \\
	\BV (x,\la) &=
		\int_{-\infty}^0 \CF_{\xi'}^{-1}
		\left[\CU(\xi',x_3,y_3,\la)\,\wh{\BF}(\xi',y_3)\right](x')\intd y_3 
		+\CF_{\xi'}^{-1}\left[\CU'(\xi',x_3,\la)\,\wh{K}(\xi')\right](x'),
		\notag \\ 
	Q(x,\la) &=
		\int_{-\infty}^0\CF_{\xi'}^{-1}
		\left[\CP(\xi',x_3,y_3,\la)\wh\BF(\xi',y_3)\right](x')\intd y_3 +\CF_{\xi'}^{-1}
		\left[\CP'(\xi',x_3,\la)\wh K(\xi')\right](x'),
		\notag\\
	H(x',\la) &=
		\int_0^\infty \CF_{\xi'}^{-1}
		\left[\CH(\xi',y_3,\la)\,\wh{\BF}(\xi',y_3)\right](x')\intd y_3 
		+\CF_{\xi'}^{-1}\left[\CH'(\xi',\la)\,\wh{K}(\xi')\right](x')
		\notag
\end{align}
for some symbols $\CU^0$, $\CP^0$, $\CU$, $\CU'$, $\CP$, $\CP'$, $\CH$, and $\CH'$
that are explicitly calculated in \cite[Section 3]{SaS1},
where the formulas of $\BV^1$, $Q^1$ can be found e.g. in \cite[Section 3]{SS12}.


Let $\zeta\in C_0^\infty(\BR^2)$ with $0\leq\zeta\leq1$
and satisfy
\begin{equation*}
	\zeta = \zeta(\xi') =
		\left\{\begin{aligned}
			&1 && (|\xi'|\leq 1), \\
			&0 && (|\xi'|\geq 2).
		\end{aligned}\right. 
\end{equation*}
In addition, we set
\begin{equation*}
	\zeta_\de^1(\xi')
		= 1-\zeta\left(\frac{\xi'}{\de}\right), \quad
	\zeta_\de^2(\xi')
		 = \zeta\left(\frac{\xi'}{\de}\right) \quad \text{for $\de>0$}.
\end{equation*}
The cut-off functions $\zeta_\de^1$, $\zeta_\de^2$
yield further decompositions of $\BV$, $Q$, and $H$ in the following manner:
$\BV=\BU_\de^1+\BU_\de^2$,
$Q= P_\de^1+P_\de^2$,
and $H=H_\de^1+H_\de^2$ with
\begin{align}\label{160609_11}
	\BU_\de^l(x,\la) &=
		 \int_{-\infty}^0\CF_{\xi'}^{-1}
		\left[\zeta_\de^l(\xi')\,\CU(\xi',x_3,y_3,\la)
		\,\wh{\BF}(\xi',y_3)\right](x')\intd y_3  
		+\CF_{\xi'}^{-1}\left[\zeta_\de^l(\xi')\,\CU'(\xi',x_3,\la)
		\,\wh{K}(\xi')\right](x'),  \\
	P_\de^l(x,\la) &=
		 \int_{-\infty}^0\CF_{\xi'}^{-1}
		\left[\zeta_\de^l(\xi')\,\CP(\xi',x_3,y_3,\la)
		\,\wh{\BF}(\xi',y_3)\right](x')\intd y_3  
		+\CF_{\xi'}^{-1}\left[\zeta_\de^l(\xi')\,\CP'(\xi',x_3,\la)
		\,\wh{K}(\xi')\right](x'),
		\notag \\
	H_\de^l(x',\la)
		&= \int_{-\infty}^0\CF_{\xi'}^{-1}
			\left[\zeta_\de^l(\xi')\,\CH(\xi',y_3,\la)
			\,\wh{\BF}(\xi',y_3)\right](x')\intd y_3 
			+\CF_{\xi'}^{-1}\left[\zeta_\de^l(\xi')\,\CH'(\xi',\la)
			\,\wh{K}(\xi')\right](x')
			\notag
\end{align}
for $l=1,2$.
Summing up the above argumentation, we see that,
together with \eqref{160609_10} and \eqref{160609_11},
the solution $(\BU,P,H)$ of System \eqref{resolvent}
is represented as
\begin{align}\label{160610_1}
	\BU &= \BU^0+\BU_\de^1+\BU_\de^2 = \BV^1+\BV^2+\BU_\de^1+\BU_\de^2, \quad 
	P = P^0 +P_\de^1+P_\de^2 = Q^1+Q^2 + P_\de^1+P_\de^2,  \\
	H&= H_\de^1+H_\de^2. \notag 
\end{align}

Thirdly, we consider System \eqref{linear3}.
The solution $(\Bu,\Fp,h)$ of the system \eqref{linear3},
obtained in Lemma \ref{lemm:sg_2}, is represented as
\begin{align}\label{160611_1}
	\Bu&=\frac{1}{2\pi i}\int_\Ga e^{\la t}\BU(x,\la)\intd\la, \quad
	 \Fp=\frac{1}{2\pi i}\int_\Ga e^{\la t} P(x,\la)\intd \la, \\
	h&= \frac{1}{2\pi i}\int_\Ga e^{\la t} H(x',\la)\intd\la
	\quad \text{with $(\BF,K)=(\Bu_0,h_0)$;} \notag \\
	\Ga&= \{\la \in \BC \mid \la=2\ga_0/\sin\ep+s e^{i(\pi-\ep)},s:-\infty\to\infty \} 
	\quad (0<\ep<\pi/2), \notag
\end{align}
where $\ga_0$ is the same constant as in Lemma \ref{lemm:SS09}.
Furthermore, by \eqref{160610_1}, these formulas can be written as
\begin{align*}
	\Bu
		&= R^0(t)\Bu_0+R_\de^1(t)(\Bu_0,h_0)+R_\de^2(t)(\Bu_0,h_0), \quad 
	\Fp
		=S^0(t)\Bu_0+S_\de^1(t)(\Bu_0,h_0)+S_\de^2(t)(\Bu_0,h_0), \notag \\
	h
		&=T_\de^1(t)(\Bu_0,h_0)+T_\de^2(t)(\Bu_0,h_0)
\end{align*}
for $\de>0$, where, for $l=1,2$,
\begin{align}\label{160611_2}
	&R^0(t)\Bu_0
		= \frac{1}{2\pi i}\int_\Ga e^{\la t}\BU^0(x,\la)\intd \la, \quad
	S^0(t)\Bu_0
		=\frac{1}{2\pi i}\int_\Ga e^{\la t}P^0(x,\la)\intd \la, \\
	&R_\de^l(t)(\Bu_0,h_0)
		=\frac{1}{2\pi i}\int_\Ga e^{\la t}\BU_\de^l(x,\la)\intd\la, \quad 
	S_\de^l(t)(\Bu_0,h_0)
		=\frac{1}{2\pi i}\int_\Ga e^{\la t} P_\de^l(x,\la)\intd\la, \notag \\
	&T_\de^l(t)(\Bu_0,h_0)
		=\frac{1}{2\pi i}\int_\Ga e^{\la t} H_\de^l(x,\la)\intd\la. \notag
\end{align}


Fourthly, we consider System \eqref{linear4}.
The solution $(\Bu,\Fp,h)$ of  \eqref{linear4},
obtained in Lemma \ref{lemm:SS12}, is given by
\begin{align*}
	&\Bu=\CL_\la^{-1}[\BU(x,\la)](t), \quad
	\Fp = \CL_\la^{-1}[P(x,\la)](t), \quad
	h=\CL_\la^{-1}[H(x,\la)](t), \\
	&\text{with  $(\BF,K)=(\CL[\Bf](\la),\CL[k](\la))$ and
	$\la=\ga+i\tau$ $(\ga>2\ga_0/\sin\ep)$,} \notag
\end{align*}
where $\CL$, $\CL_\la^{-1}$ are the Laplace transform and its inverse given by \eqref{la:trans}.
We insert $(\BU,P,H)$ of the form \eqref{160610_1} into the last formulas 
in order to obtain 
\begin{equation*}
	\Bu = \CR^0(x,t)+\CR_\de^1(x,t)+\CR_\de^2(x,t), \quad
	\Fp = \CS^0(x,t)+\CS_\de^1(x,t)+\CS_\de^2(x,t), \quad
	h=\CT_\de^1(x,t)+\CT_\de^2(x,t)  
\end{equation*}
for $\de>0$, where, for $l=1,2$, 
\begin{align}\label{160611_4}
	\CR^0(x,t) &=\CR^0(t)= \CL_\la^{-1}\left[\BU^0(x,\la)\right](t), \quad  
	\CS^0(x,t) = \CS^0(t)=\CL_\la^{-1}\left[P^0(x,\la)\right](t),  \\
	\CR_\de^l(x,t) &=\CR_\de^l(t) = \CL_{\la}^{-1}\left[\BU_\de^l(x,\la)\right](t) , \quad
	\CS_\de^l(x,t) =\CS_\de^l(t) = \CL_\la^{-1}\left[P_\de^l(x,\la)\right](t) , \notag \\	
	\CT_\de^l(x,t) &= \CT_\de^l(t)=\CL_\la^{-1}\left[H_\de^l(x',\la)\right](t). \notag
\end{align}
In addition, using \eqref{160611_2},
we have by Cauchy's integral theorem
 another representation of \eqref{160611_4} as follows:
\begin{alignat}{2}\label{conv}
	\CR^0(t) &= \int_0^t R^0(t-s)\Bf(s)\intd s, \quad
	&\CS^0(t) &= \int_0^t S^0(t-s)\Bf(s)\intd s, \\
	\CR_\de^l(t) &= \int_0^t R_\de^l(t-s)(\Bf(s),k(s))\intd s, \quad
	&\CS_\de^l(t) &= \int_0^t S_\de^l(t-s)(\Bf(s),k(s))\intd s, \notag \\
	\CT_\de^l(t) &= \int_0^t T_\de^l(t-s)(\Bf(s),k(s))\intd s & (l=&1,2). \notag
\end{alignat}

Concerning $R^0(t)$, $R_\de^l(t)$, $S^0(t)$, $S_\de^l(t)$, and $T_\de^l(t)$, 
we have 

\begin{lemm}\label{lemm:SaS}
Let $3<q<4$ and $2\leq r\leq q$, and let $Y_r = L_r(\lhs)^3\times W_r^{2-1/r}(\BR^2)$.
Suppose that $\bar{q}=q/2$ and $\BU_0=(\Bu_0,h_0)$ with 
\begin{equation*}
	\Bu_0\in L_r(\BR_-^3)^3 \cap L_{\bar{q}}(\BR_-^3)^3,\quad
	h_0\in W_{r}^{2-1/r}(\BR^2) \cap L_{\bar{q}}(\BR^2).
\end{equation*}
Then there is a constant $\de\in(0,1)$, independent of $\Bu_0$, $h_0$,
such that the following assertions hold. 
\begin{enumerate}[$(1)$]
	\item\label{lemm:SaS1}
		There exists a positive constant $C_r$ such that, for any $t>0$ and $l=0,1,2$,
		\begin{align*}
			\|(\pa_t R^0(t)\Bu_0,\nabla S^0(t)\Bu_0)\|_{L_r(\lhs)}
				&\leq C_r t^{-1}\|\Bu_0\|_{L_r(\lhs)}, \\
			\|\nabla^l R^0(t)\Bu_0\|_{L_r(\lhs)}
				&\leq C_r t^{-l/2}\|\Bu_0\|_{L_r(\lhs)}.
		\end{align*}
	\item\label{lemm:SaS2}
		There exists a positive constant $C_{\bar{q},r}$ such that, for any $t>0$ and $l=0,1$,
		\begin{equation*}
			\|\nabla ^l R^0(t)\Bu_0\|_{L_r(\BR_-^3)} 
				\leq C_{\bar{q},r}t^{-\Fn(\bar{q},r)-l/2} 
				\|\Bu_0\|_{L_{\bar{q}}(\BR_-^3)}.
		\end{equation*}
	\item\label{lemm:SaS3}
		There exist positive constants $\ga_2$, $C_{r,\ga_2}$ such that, for any $t\geq 1$,
		\begin{align*}
			&\|\pa_t R_\de^1(t)\BU_0\|_{L_r(\lhs)} + \|R_\de^1(t)\BU_0\|_{W_r^2(\lhs)} + \|\nabla S_\de^1(t)\BU_0\|_{L_r(\lhs)}\\
				&+\|\pa_t\CE(T_\de^1(t)\BU_0))\|_{W_r^2(\lhs)}+	\|\CE(T_\de^1(t)\BU_0))\|_{W_r^3(\lhs)}
				+\|T_\de^1(t)\BU_0\|_{L_r(\BR^2)}  \\
				&\leq C_{r,\ga_2} e^{-\ga_2 t}
				 \|\BU_0\|_{Y_r}.
		\end{align*}
	\item\label{lemm:SaS6}
		Let $\al>0$. Then, there exists positive constants $\ga_3$, $C_{r,\ga_3}$, and $C_{r,\al,\ga_3}$ such that, for any $t>0$, 
		\begin{align*}
			\|R_\de^1(t)\BU_0\|_{W_r^1(\lhs)} + \|\CE(T_\de^1(t)\BU_0)\|_{W_r^2(\lhs)}
				&\leq C_{r,\al, \ga_3}t^{-\al}e^{-\ga_3 t}\|\BU_0\|_{Y_r}, \\
			\|(\nabla^2 R_\de^1(t)\BU_0,\nabla^3\CE(T_\de^1(t)))\|_{L_r(\lhs)} 
				&\leq C_{r,\ga_3}t^{-1}e^{-\ga_3 t}\|\BU_0\|_{Y_r}.
		\end{align*}
	\item\label{lemm:SaS4}
		If we additionally assume that $h_0\in W_{\bar{q}}^{2-1/\bar{q}}(\BR^2)$,
		then we have, for any $t>0$,	
		\begin{equation*}
			\|R_\de^1(t)\BU_0\|_{W_r^1(\BR_-^3)}+\|\CE(T_\de^1(t)\BU_0)\|_{W_r^2(\BR_-^3)} 
				+\|T_\de^1(t)\BU_0\|_{L_r(\BR^2)} 
			\leq
				C_{\bar{q},r,\ga_4}
				t^{-\Fn(\bar{q},r)-1/2}
				e^{-\ga_4 t}\|\BU_0\|_{Y_{\bar q}},
		\end{equation*}
		with positive constants $\ga_4$, $C_{\bar{q},r,\ga_4}$.
	\item\label{lemm:SaS5}
		Let $k=1,2$ and $l=0,1,2$, and let $Z_{\bar q}=L_{\bar q}(\lhs)^3\times L_{\bar{q}}(\BR^2)$
		Then there exists a positive constant $C_{\bar{q}, r}$ such that, for any $t\geq 1$,
				\begin{align*}
			\|\pa_t R_\de^2(t)\BU_0\|_{L_r(\BR_-^3)}
				&\leq C_{\bar{q},r}(t+2)^{-\Fm\left(\bar{q},r\right)-1/4}
				\|\BU_0\|_{Z_{\bar{q}}}, \\
			\|R_\de^2(t)\BU_0\|_{L_r(\BR_-^3)}
				&\leq C_{\bar{q},r}(t+2)^{-\Fm\left(\bar{q},r\right)}
				\|\BU_0\|_{Z_{\bar q}}, \\
			\|\nabla^k R_\de^2(t)\BU_0\|_{L_r(\BR_-^3)}
				&\leq C_{\bar{q},r}(t+2)^{-\Fn\left(\bar{q},r\right)-k/8}
				\|\BU_0\|_{Z_{\bar q}},
				\\
			\|\nabla S_\de^2(t)\BU_0\|_{L_r(\lhs)}
				&\leq C_{\bar{q},r}(t+2)^{-\Fm\left(\bar{q},r\right)-1/4}
				\|\BU_0\|_{Z_{\bar q}},  \\
			\|\nabla^k\pa_t \CE(T_\de^2(t)\BU_0)\|_{L_r(\BR_-^3)}
				&\leq C_{\bar{q},r}(t+2)^{-\Fm\left(\bar{q},r\right)-k/2}
				\|\BU_0\|_{Z_{\bar q}}, \\
			\|\nabla^{1+l} \CE(T_\de^2(t)\BU_0)\|_{L_r(\BR_-^3)}
				&\leq C_{\bar{q},r}(t+2)^{-\Fm\left(\bar{q},r\right)-1/4-l/2}
				\|\BU_0\|_{Z_{\bar q}}, \\
			\|T_\de^2(t)\BU_0\|_{L_r(\BR^2)}
				&\leq C_{\bar{q},r}(t+2)^{-\left(1/\bar{q}-1/r\right)}
				\|\BU_0\|_{Z_{\bar q}}.
		\end{align*}
		On the other hand, for any $0<t\leq 1$ and any $\al>0$, 
		\begin{equation*}
			\|(R_\de^2(t)\BU_0,\nabla \CE(T_\de^2(t)\BU_0))\|_{W_r^2(\BR_-^3)}
			+\|\nabla S_\de^2(t)\BU_0\|_{L_r(\lhs)} 
			+\|T_\de^2(t)\BU_0\|_{L_r(\BR^2)}
			\leq C_{\bar{q},r,\al}t^{-\al}\|\BU_0\|_{Z_{\bar q}},
		\end{equation*}
		with some positive constant $C_{\bar{q},r,\al}$.
	\item\label{lemm:SaS7}
		There exists a positive constant $C_{q}$ such that, for any $t\geq 1$,
		\begin{align*}
			&\|(\nabla^2 R_\de^2(t)\BU_0,\nabla S_\de^2(t)\BU_0)\|_{L_q(\lhs)} 
				\leq C_q (t+2)^{-\Fm(2,q)-1/4}\|\BU_0\|_{Z_2}, \\
			&\|\nabla^3\CE(T_\de^2(t)\BU_0)\|_{L_q(\lhs)}
				\leq C_q (t+2)^{-\Fm(2,q)-5/4}\|\BU_0\|_{Z_2}.
		\end{align*}
		On the other hand, for any $0<t\leq 1$ and any $\al>0$,
		\begin{equation*}
			\|(\nabla^2 R_\de^2(t)\BU_0,\nabla S_\de^2(t)\BU_0,\nabla^3\CE(T_\de^2(t)\BU_0))\|_{L_q(\lhs)} 
			\leq C_{q,\al}t^{-\al}\|\BU_0\|_{Z_{2}},
		\end{equation*}		
		with some positive constant $C_{q,\al}$.
\end{enumerate}
\end{lemm}

\begin{proof}
\eqref{lemm:SaS1}, \eqref{lemm:SaS3}, \eqref{lemm:SaS6}, \eqref{lemm:SaS5}, \eqref{lemm:SaS7}.
These estimates were proved in \cite{SaS1} (cf. also \cite[Theorem 3.1.3]{Saito15}).

\eqref{lemm:SaS2}.
The desired estimates follow from (1) 
and Sobolev's embedding theorem: 
\begin{equation}\label{sobolev:2}
	\|f\|_{L_{p_2}(\lhs)} \leq C_{p_1,p_2}
		\|\nabla f\|_{L_{p_1}(\lhs)}^{3\left(\frac{1}{p_1}-\frac{1}{p_2}\right)}
		\|f\|_{L_{p_1}(\lhs)}^{1-3\left(\frac{1}{p_1}-\frac{1}{p_2}\right)}
		, \quad 3\left(\frac{1}{p_1}-\frac{1}{p_2}\right)<1,
\end{equation}
with $p_1=\bar{q}$ and $p_2 = r$.
Here and subsequently, we note that, by $3<q<4$ and $2\leq r \leq q$,
\begin{equation*}
	3\left(\frac{1}{\bar{q}}-\frac{1}{r}\right) \leq 3\left(\frac{1}{\bar{q}}-\frac{1}{q}\right) = \frac{3}{q}<1.
\end{equation*} 


\eqref{lemm:SaS4}.
By \eqref{sobolev:2} and \eqref{lemm:SaS6}, we have, for any $\al,\beta>0$,
\begin{align*}
	\|R_\de^1(t)\BU_0\|_{L_r(\lhs)} 
		&\leq C_{\bar{q},r}\|\nabla R_\de^1(t)\BU_0\|_{L_{\bar{q}}(\lhs)}^{3\left(\frac{1}{\bar{q}}-\frac{1}{r}\right)}
		\|R_\de^1(t)\BU_0\|_{L_{\bar{q}}(\lhs)}^{1-3\left(\frac{1}{\bar{q}}-\frac{1}{r}\right)} 
		\leq C_{\bar{q},r,\al,\ga_3}t^{-\al}e^{-\ga_3 t}\|\BU_0\|_{Y_{\bar{q}}}, \\
	\|\nabla R_\de^1(t)\BU_0\|_{L_r(\lhs)} 
		&\leq C_{\bar{q},r}\|\nabla^2 R_\de^1(t)\BU_0\|_{L_{\bar{q}}(\lhs)}^{3\left(\frac{1}{\bar{q}}-\frac{1}{r}\right)}
		\|\nabla R_\de^1(t)\BU_0\|_{L_{\bar{q}}(\lhs)}^{1-3\left(\frac{1}{\bar{q}}-\frac{1}{r}\right)} \\
		&\leq C_{\bar{q},r,\beta,\ga_3}t^{-3\left(\frac{1}{\bar{q}}-\frac{1}{r}\right)-\beta\left\{1-3\left(\frac{1}{\bar{q}}-\frac{1}{r}\right)\right\}}
		e^{-\ga_3 t}\|\BU_0\|_{Y_{\bar{q}}}.
\end{align*}
Thus, setting $\al=\Fn(\bar{q},r)+1/2$ and $\beta=1/2$ in the first inequality and the second inequality, respectively,
implies that 
\begin{equation*}
	\|R_\de^1(t)\BU_0\|_{W_r^1(\lhs)} 
		\leq C_{\bar{q},r,\ga_3}t^{-\Fn(\bar{q},r)-1/2}e^{-\ga_3 t}\|\BU_0\|_{Y_{\bar{q}}}.
\end{equation*}
Analogously, we observe that
\begin{equation*}
	\|\CE(T_\de^1(t)\BU_0)\|_{W_r^2(\lhs)}
		\leq C_{\bar{q},r,\ga_3}t^{-\Fn(\bar{q},r)-1/2}e^{-\ga_3 t}\|\BU_0\|_{Y_{\bar{q}}},
\end{equation*}
which, combined with the trace theorem, furnishes that
\begin{equation*}
	\|T_\de^1(t)\BU_0\|_{L_r(\BR^2)} \leq C_r\|\CE(T_\de^1(t)\BU_0)\|_{W_r^1(\lhs)}
		\leq  C_{\bar{q},r,\ga_3} t^{-\Fn(\bar{q},r)-1/2}e^{-\ga_3 t}\|\BU_0\|_{Y_{\bar{q}}}.
\end{equation*}
This completes the proof of the lemma.
\end{proof}

Finally, we introduce some embedding properties (cf. e.g. \cite{MS12} and \cite{Shibata18}).

\begin{lemm}\label{lemm:embed}
Let $T\in(0,\infty)$ or $T=\infty$, and set $J=(0,T)$.
Then the following properties hold.
\begin{enumerate}[$(1)$]
\item\label{embed:0}
	If $1<p,q<\infty$, then
		$W_{q,p}^{2,1}(\lhs\times J) \hookrightarrow H_p^{1/2}(J,W_q^1(\lhs))$.
\item\label{embed:1}
	If $2<p<\infty$ and $1<q<\infty$, then
		$W_{q,p}^{2,1}(\lhs\times J)\hookrightarrow BUC(\overline{J},W_q^1(\BR_-^3))$.
\item\label{embed:2}
	If $2<p<\infty$, $3<q<\infty$, and $2/p+3/q<1$, then
		$W_{q,p}^{2,1}(\lhs\times J)\hookrightarrow BUC(\overline{J},BUC^1(\BR_-^3))$.
\end{enumerate}
\end{lemm}

%
%
%
%
%
%
%
%
%
%
%
%

\subsection{Proof of Theorem \ref{theo:sec4_1}}
In this subsection, we prove Theorem \ref{theo:sec4_1}.
In this proof, we omit the subscript $\de$ of
$R_\de^l$, $S_\de^l$, and $T_\de^l$ $(l=1,2)$ for simplicity
and suppose that $r\in\{q,2\}$. 
Let $(\Bu,\Fp,h)$ be the solution obtained in Lemma \ref{lemm:sg_2} in what follows.

{\bf Step 1.}
In this step, we prove the following estimates:
\begin{align}
	\|\nabla\pa_t\CE(h)\|_{L_\infty^{\Fm(\bar{q},r)+1/2}(\BR_+,L_r(\lhs))}
		&\leq C_{p,\bar{q},r}\|(\Bu_0,h_0)\|_{\wt\BBI_1\times\BBI_2}, \label{161001_12} \\
	\|\pa_t\CE(h)\|_{L_\infty(\BR_+,L_2(\lhs))} 
		&\leq C_{p,\bar{q}}\|(\Bu_0,h_0)\|_{\wt\BBI_1\times\BBI_2} \label{161001_11} 
\end{align}
with positive constants $C_{p,\bar{q}}$, $C_{p,\bar{q},r}$.
To this end, we introduce

\begin{lemm}\label{lemm:ap}
Let $3<q<4$ and $2\leq r\leq q$, and let $\bar{q}=q/2$.
Suppose that $Y_r$, $Z_{\bar{q}}$ are given in Lemma $\ref{lemm:SaS}$
and that $\BU_0=(\Bu_0,h_0)\in Y_r\cap Z_{\bar{q}}$. 
Let 
\begin{equation*}
	R(t)\BU_0 = R^0(t)\Bu_0+R^1(t)(\Bu_0,h_0)+R^2(t)(\Bu_0,h_0),
\end{equation*}
and let $(R(t)\BU_0)_i$ be the $i$th component of $R(t)\BU_0$ for $i=1,2,3$. 
Then there exist positive constants $\ga_5$, $C_{\ga_5,r}$, and $C_{\bar{q},r}$ such that, for $k=0,1$, 
\begin{align}
	\|\CE((R(t)\BU_0)_3|_{\BR_0^3})\|_{L_r(\lhs)}
		&\leq C_{r,\ga_5}e^{\ga_5 t}\|\BU_0\|_{Y_r} \quad (t>0), \label{161028_1} \\
	\|\nabla^k\CE((R(t)\BU_0)_3|_{\BR_0^3})\|_{L_r(\lhs)}
		&\leq C_{\bar{q},r}(t+2)^{-\Fm(\bar{q},r)-k/2}\|\BU_0\|_{Y_r\cap Z_{\bar{q}}} \quad (t\geq 1), \label{161028_2} \\
	\|\nabla^k\CE((R(t)\BU_0)_3|_{\BR_0^3})\|_{L_r(\lhs)}
		&\leq C_{\bar{q},r}t^{-\Fn(\bar{q},r)-1/2}\|\BU_0\|_{Y_r\cap Z_{\bar{q}}} \quad (0<t\leq 1), \label{161028_3} \\
	\|\nabla^2\CE((R^2(t)\BU_0)_3|_{\BR_0^3})\|_{L_r(\lhs)}
		&\leq C_{\bar{q},r}(t+2)^{-\Fm(\bar{q},r)-1}\|\BU_0\|_{Z_{\bar{q}}} \quad (t\geq 1). \label{161219_21} \\
	\|\nabla^2\CE((R^2(t)\BU_0)_3|_{\BR_0^3})\|_{L_r(\lhs)}
		&\leq C_{\bar{q},r,\al}t^{-\al}\|\BU_0\|_{Z_{\bar{q}}} \quad (0< t\leq 1), \label{161220_2}		
\end{align}
where $\al>0$ and $C_{\bar{q},r,\al}$ is a positive constant depending on $\bar{q}$, $r$, and $\al$.
\end{lemm}

\begin{proof}
We can prove the first inequality similarly to Lemma \ref{lemm:sg_1} and
the other inequalities similarly to \cite{SaS1} (cf. also \cite[Chapter 4]{Saito15}) and Lemma \ref{lemm:SaS},
so that we may omit the detailed proof.
\end{proof}

Since $\pa_ t h - u_3=0$ on $\BR_0^3$, we have
\begin{equation}\label{161006_10}
	\pa_t\CE(h)=\CE(u_3|_{\BR_0^3})=\CE((R(t)\BU_0)_3|_{\BR_0^3}) \quad \text{in $\lhs$.}
\end{equation}
Combining this identity with \eqref{161028_2} furnishes that
\begin{equation}\label{160927_1}
	\|\nabla\pa_t\CE(h)\|_{L_\infty^{\Fm(\bar{q},r)+1/2}((1,\infty),L_r(\lhs))}
		\leq C_{p,\bar{q},r}\|(\Bu_0,h_0)\|_{\wt\BBI_1\times\BBI_2},
\end{equation}
where we have used the fact that
\begin{equation}\label{embed:5}
	\|\Bu_0\|_{L_r(\lhs)}+\|\Bu_0\|_{L_{\bar{q}}(\lhs)}
		\leq C_{p,r}\|\Bu_0\|_{\wt\BBI_1}, \quad
	\|h_0\|_{W_r^{2-1/r}(\BR^2)}+\|h_0\|_{L_{\bar{q}}(\BR^2)} 
		\leq C_{p,r}\|h_0\|_{\BBI_2}. 
\end{equation}
It holds that, for $1<q<\infty$ and $k=1,2,3$,
\begin{equation}\label{161007_1}
	\|\nabla^k\CE(f|_{\BR_0^3})\|_{L_q(\lhs)} \leq C_q\|\nabla^k f\|_{L_q(\lhs)}, 
\end{equation}
and thus we observe that, by \eqref{161001_10}, \eqref{161006_10}, Lemma \ref{lemm:embed} \eqref{embed:0}, and Sobolev's embedding theorem,
\begin{align*}
	&\|\nabla\pa_t\CE(h)\|_{BUC((0,2),L_r(\lhs))}  
		\leq C_r \|\nabla u_3\|_{BUC((0,2),L_r(\lhs))}\\
		&\leq C_{p,r}\|\nabla u_3\|_{H_p^{1/2}((0,2),L_r(\lhs))}
			\leq C_{p,r}\|u_3\|_{H_p^{1/2}((0,2),W_r^1(\lhs))} \\
		&\leq	C_{p,r}\|u_3\|_{W_{r,p}^{2,1}(\lhs\times (0,2))}
		\leq  C_{p,\bar{q},r}\|(\Bu_0,h_0)\|_{\wt\BBI_1\times\BBI_2}, 
\end{align*}
which, combined with \eqref{160927_1}, furnishes \eqref{161001_12}.
The estimate \eqref{161001_11} follows directly from \eqref{161028_1}, \eqref{161028_2}, \eqref{161006_10}, and \eqref{embed:5}.
This completes Step 1.

%

{\bf Step 2.}
In this step, we prove, for $\Bz = (\Bu,\Fp,h,\CE(h))$, 
\begin{equation}\label{160923_1}
	\CN_{q,p}(\Bz;\Fa_1,\Fa_2) + \sum_{r\in\{q,2\}}\CM_{r,p}(\Bz)
		\leq C_{p,q,\bar{q}}\|(\Bu_0,h_0)\|_{\wt\BBI_1\times\BBI_2}
\end{equation}
with some positive constant $C_{p,q,\bar{q}}$.
By Lemma \ref{lemm:SaS} \eqref{lemm:SaS1}, we have
\begin{equation}\label{160923_11}
	\|(\pa_t R^0(t)\Bu_0,\nabla^2 R^0(t)\Bu_0)\|_{L_p^{\Fa_1}((1,\infty),L_q(\lhs))}
		\leq C_{p,q}\|\Bu_0\|_{L_q(\lhs)},
\end{equation}
because $p(1-\Fa_1)> p(\Fm(\bar{q},q)+1/4-\Fa_1)>1$.
In addition, it follows from Lemma \ref{lemm:SaS} \eqref{lemm:SaS3} that
\begin{align}\label{160923_12}
	&\|(\pa_t R^1(t)\BU_0,\nabla^2 R^1(t)\BU_0)\|_{L_p^{\Fa_1}((1,\infty),L_q(\lhs))} 
		+\|(\nabla^2\pa_t\CE(T^1(t)\BU_0),\nabla^3\CE(T^1(t)\BU_0))\|_{L_p^{\Fa_2}((1,\infty),L_q(\lhs))}  \\
		&\leq C_{p,q}\left(\|\Bu_0\|_{L_q(\lhs)}+\|h_0\|_{W_q^{2-1/q}(\BR^2)}\right). \notag
\end{align}
Noting $\Fn(\bar{q},q)>\Fm(\bar{q},q)$ and \eqref{assump:sg},
we observe that, by Lemma \ref{lemm:SaS} \eqref{lemm:SaS5}, 
\begin{align*}
	&\|(\pa_t R^2(t)\BU_0,\nabla^2 R^2(t)\BU_0)\|_{L_p^{\Fa_1}((1,\infty),L_q(\lhs))} 
		 +\|\nabla^2\pa_t\CE(T^2(t)\BU_0),\nabla^3\CE(T^2(t)\BU_0))\|_{L_p^{\Fa_2}((1,\infty),L_q(\lhs))} \\
		&\leq C_{p,q,\bar{q}}\left(\|\Bu_0\|_{L_{\bar{q}}(\lhs)} + \|h_0\|_{L_{\bar{q}}(\BR^2)}\right),
\end{align*}
which, combined with \eqref{160923_11} and \eqref{160923_12}, furnishes that
\begin{align}\label{160907_1}
	&\|(\pa_t\Bu,\nabla^2\Bu)\|_{L_p^{\Fa_1}((1,\infty),L_q(\BR_-^3))}
	+\|(\nabla^2\pa_t \CE(h),\nabla^3\CE(h))\|_{L_p^{\Fa_2}((1,\infty),L_q(\lhs))} \\
	&\leq C_{p,q,\bar{q}}\left(\|\Bu_0\|_{L_q(\lhs)\cap L_{\bar{q}}(\lhs)}+\|h_0\|_{W_q^{2-1/q}(\BR^2)\cap L_{\bar{q}}(\BR^2)}\right). \notag 
\end{align}
By using Lemma \ref{lemm:SaS} \eqref{lemm:SaS1}, \eqref{lemm:SaS2}, \eqref{lemm:SaS3}, and \eqref{lemm:SaS5},
we have
\begin{align*}
	&\|(\pa_t R^0(t)\Bu_0,R^0(t)\Bu_0,\nabla R^0(t)\Bu_0,\nabla^2 R^0(t)\Bu_0,\nabla S^0(t)\Bu_0)\|_{L_p((1,\infty),L_r(\lhs))} \\
		&\leq C_{p,\bar{q},r}\left(\|\Bu_0\|_{L_r(\lhs)}+\|\Bu_0\|_{L_{\bar{q}}(\lhs)}\right), \\
	&\|(\pa_t R^1(t)\BU_0,R^1(t)\BU_0,\nabla R^1(t)\BU_0,\nabla^2 R^1(t)\BU_0,\nabla S^1(t)\BU_0)\|_{L_p((1,\infty),L_r(\lhs))} \\
		& +\|T^1(t)\BU_0\|_{L_p((1,\infty),L_r(\BR^2))} +\|\pa_t\CE(T^1(t)\BU^0)\|_{L_p((1,\infty),\wh W_r^2(\lhs))} 
	+ \|\CE(T^1(t)\BU_0)\|_{L_p((1,\infty),\wh W_r^3(\lhs))} \\
	&\leq C_{p,r}\left(\|\Bu_0\|_{L_r(\lhs)} + \|h_0\|_{W_r^{2-1/r}(\BR^2)}\right), \\
	&\|(\pa_t R^2(t)\BU_0,R^2(t)\BU_0,\nabla R^2(t)\BU_0,\nabla^2 R^2(t)\BU_0,\nabla S^2(t)\BU_0)\|_{L_p((1,\infty),L_r(\lhs))} \\
	& + \|T^2(t)\BU_0\|_{L_p((1,\infty),L_r(\BR^2))} +\|\pa_t\CE(T^2(t)\BU^0)\|_{L_p((1,\infty),\wh W_r^2(\lhs))} 
	 +\|\CE(T^2(t)\BU_0)\|_{L_p((1,\infty),\wh W_r^3(\lhs))} \\
	&\leq C_{p,\bar{q},r}\left(\|\Bu_0\|_{L_{\bar{q}}(\lhs)}+\|h_0\|_{L_{\bar{q}}(\BR^2)}\right),
\end{align*}
because it holds that, by \eqref{pq:lin3},
\begin{equation*}
	p\cdot\Fn(\bar{q},r)>p\cdot\Fm(\bar{q},r) \geq p\cdot \Fm(\bar{q},2) = p\left(\frac{2}{q}-\frac{1}{2}\right)>1.
\end{equation*}
These inequalities imply that
\begin{align*}
	&\|(\pa_t\Bu,\Bu,\nabla\Bu,\nabla^2\Bu,\nabla\Fp)\|_{L_p((1,\infty),L_r(\lhs))} +\|h\|_{L_p((1,\infty),L_r(\BR^2))}\\
	&+\|\pa_t\CE(h)\|_{L_p((1,\infty),\wh W_r^2(\lhs))} + \|\CE(h)\|_{L_p((1,\infty),\wh W_r^3(\lhs))} \notag \\
	&\leq C_{p,\bar{q},r}\left(\|\Bu_0\|_{L_r(\lhs)\cap L_{\bar{q}}(\lhs)} + \|h_0\|_{W_r^{2-1/r}(\BR^2)\cap L_{\bar{q}}(\BR^2)}\right), \notag
\end{align*}
which, combined with \eqref{lemm:embed} for $T=2$, \eqref{embed:5}, and \eqref{160907_1},
furnishes that 
\begin{align}\label{160923_14}
	&\|(\pa_t\Bu,\nabla^2\Bu)\|_{L_p^{\Fa_1}(\BR_+,L_q(\lhs))}
	 +\|\nabla^2\pa_t\CE(h),\nabla^3\CE(h)\|_{L_p^{\Fa_2}(\BR_+,L_q(\lhs))} \\
	&+\|\Bu\|_{W_{r,p}^{2,1}(\lhs\times\BR_+)}+\|\nabla\Fp\|_{L_p(\BR_+,L_r(\lhs))}
	+\|h\|_{L_p(\BR_+,L_r(\BR^2))} \notag \\
	&+\|\pa_t\CE(h)\|_{L_p(\BR_+,\wh W_r^2(\lhs))} + \|\CE(h)\|_{L_p(\BR_+,\wh W_r^3(\lhs))} \notag \\
	&\leq C_{p,\bar{q},r}\|(\Bu_0,h_0)\|_{\wt\BBI_1\times \BBI_2}. \notag
\end{align}
Since it holds that, by the equation $\pa_t h-u_3=0$ on $\BR_0^3$ and by the trace theorem,
\begin{align*}
	\|\pa_t h(t)\|_{W_r^{2-1/r}(\BR^2)} &=\|u_3(t)\|_{W_r^{2-1/r}(\BR_0^3)} \leq C_r\|u_3(t)\|_{W_r^2(\lhs)}, \\
	\|\nabla' h(t)\|_{W_r^{2-1/r}(\BR^2)} &\leq C_r \|\nabla' \CE( h(t))\|_{W_r^2(\lhs)} 
\end{align*}
with a positive constant $C_r$, we have, by \eqref{160923_14},
\begin{equation*}
	\|\pa_t h\|_{L_p(\BR_+,W_r^{2-1/r}(\lhs))} + \|h\|_{L_p(\BR_+,W_r^{3-1/r}(\BR^2))}
		\leq C_{p,\bar{q},r}\|(\Bu_0,h_0)\|_{\wt\BBI_1\times \BBI_2}.
\end{equation*}
Combining this inequality with \eqref{160923_14} implies \eqref{160923_1}.

{\bf Step 3.}
In this step,  we prove, for $\Bz = (\Bu,\Fp,h,\CE(h))$,
\begin{equation}\label{160923_15}
	\sum_{r\in\{q,2\}}\CN_r(\Bz) \leq C_{p,q,\bar{q}}\|(\Bu_0,h_0)\|_{\wt\BBI_1\times \BBI_2}
\end{equation} 
with some positive constant $C_{p,q,\bar{q}}$.
Noting that $\Fn(\bar{q},r)>\Fm(\bar{q},r)$,
we see that, by Lemma \ref{lemm:SaS} \eqref{lemm:SaS2}, \eqref{lemm:SaS3}, \eqref{lemm:SaS5}, and \eqref{embed:5},
\begin{align}\label{160923_2}
	&\|\Bu\|_{L_\infty^{\Fm(\bar{q},r)}((1,\infty),L_r(\BR_-^3))}
	+\|\nabla \Bu(t)\|_{L_\infty^{\Fn(\bar{q},r)+1/8}((1,\infty),L_r(\BR_-^3))} \\
	&\quad  +\|h\|_{L_\infty^{1/\bar{q}-1/r}((1,\infty),L_r(\BR^2))}
	+\|\nabla\CE(h)\|_{L_\infty^{\Fm(\bar{q},r)+1/4}((1,\infty),W_r^1(\BR\-^3))} 
	\leq C_{p,\bar{q},r}\|(\Bu_0,h_0)\|_{\wt\BBI_1\times\BBI_2}. \notag
\end{align}
On the other hand, we have, by \eqref{161001_10} with $T=2$ and Lemma \ref{lemm:embed} \eqref{embed:1},
\begin{equation*}
	\|\Bu\|_{L_\infty((0,2),W_r^1(\lhs))}
	+\|\nabla\CE(h)\|_{L_\infty((0,2),W_r^1(\lhs))} \leq C_{p,r}\|(\Bu_0,h_0)\|_{\wt\BBI_1\times\BBI_2},
\end{equation*}
and also we observe that, by Sobolev's embedding theorem, 
\begin{align*}
	\|h\|_{L_\infty((0,2),L_r(\BR^2))} \leq C_p\|h\|_{W_p^1((0,2),L_r(\BR^2))}
	\leq C_{p,r}\|(\Bu_0,h_0)\|_{\wt\BBI_1\times\BBI_2}.
\end{align*}
Combining the last two estimates with \eqref{160923_2} furnishes that
\begin{align}\label{160923_4}
	&\|\Bu\|_{L_\infty^{\Fm(\bar{q},r)}(\BR_+,L_r(\BR_-^3))}
	+\|\nabla \Bu(t)\|_{L_\infty^{\Fn(\bar{q},r)+1/8}(\BR_+,L_r(\BR_-^3))} \\
	&\quad  +\|h\|_{L_\infty^{1/\bar{q}-1/r}(\BR_+,L_r(\BR^2))}
	+\|\nabla\CE(h)\|_{L_\infty^{\Fm(\bar{q},r)+1/4}(\BR_+,W_r^1(\BR\-^3))} 
	\leq C_{p,\bar{q},r}\|(\Bu_0,h_0)\|_{\wt\BBI_1\times\BBI_2}. \notag
\end{align}

To estimate $\pa_t h$, we use the equation: $\pa_th -u_3=0$ on $\BR_0^3$.
It is clear that, by the trace theorem,
\begin{equation*}
	\|\pa_t h(t)\|_{L_r(\BR^2)} =\|u_3(t)\|_{L_r(\BR_0^3)} \leq C_r \|u_3(t)\|_{W_r^1(\lhs)},
\end{equation*}
which, combined with \eqref{160923_4}, furnishes that
\begin{equation}\label{160923_16}
	\|\pa_t h\|_{L_\infty^{\Fm(\bar{q},r)}(\BR_+,L_r(\BR^2))} 
		\leq C_{p,\bar{q},r}\|(\Bu_0,h_0)\|_{\wt\BBI_1\times\BBI_2}. 
\end{equation}
By \eqref{161001_12}, \eqref{160923_4}, and \eqref{160923_16},
we have \eqref{160923_15}.
Combining \eqref{161001_11}, \eqref{160923_1}, and \eqref{160923_15}
completes the proof of Theorem \ref{theo:sec4_1}.

\subsection{Proof of Theorem \ref{theo:sec4_2}}
In this subsection, we prove Theorem \ref{theo:sec4_2}.
In this proof, we omit the subscript $\de$ of
$R_\de^l$, $S_\de^l$, $T_\de^l$, $\CR_\de^l$ $\CS_\de^l$, and $\CT_\de^l$ $(l=1,2)$ for simplicity.
Suppose that $r\in\{q,2\}$, $\bar{q}=q/2$, and 
$\Fb_0=\min(\Fb_1,\Fb_2,\Fb_3,\Fb_4)$ in what follows.

{\bf Step 1.}
The aim of this step is to show the following estimates:
\begin{align}
	&\|(\nabla^2\CR^0,\nabla \CS^0,\nabla^2\CE(\SSR^0|_{\BR_0^3}))\|_{L_p(\BR_+,L_r(\BR_-^3))}
		\leq C_{p,r}\|\Bf\|_{L_p(\BR_+,L_r(\BR_-^3))},
		 \label{160607_1} \\
	&\|\CR^0\|_{L_\infty^{\Fm(\bar{q},r)}(\BR_+,L_r(\BR_-^3))}
		+\|\nabla\CR^0\|_{L_\infty^{\Fn(\bar{q},r)+1/8}(\BR_+,L_r(\BR_-^3))} \label{160607_2} 
		\leq C_{p,\bar{q},r}
		\sum_{i=1}^3\|\Bf_i\|_{\BBF_i\cap\wt\BBF_i(\Fa_0,\Fb_i)}  
\end{align}
for positive constants $C_{p,r}$, $C_{p,\bar{q},r}$,
where $\SSR^0$ denotes the third component of $\CR^0$.

We know that \cite[Sections 2 and 3]{SS12} essentially proved the estimate \eqref{160607_1} and
\begin{equation}\label{141015_5}
	\|e^{-\ga_6 t}(\pa_t\CR^0,\CR^0,\nabla\CR^0,
	\nabla^2\CR^0)\|_{L_p(\BR_+,L_r(\BR_-^3))}  
		\leq
			C_{p,r,\ga_6}
			\|e^{-\ga_6 t}\Bf\|_{L_p(\BR_+,L_r(\BR_-^3))} 
		\leq
			C_{p,r,\ga_6}
			\sum_{i=1}^3\|\Bf_i\|_{\BBF_i} 
\end{equation}
for positive constants $\ga_6$, $C_{p,r,\ga_6}$. We here set
\begin{equation*}
	\CR^0(t) =
		\left(\int_0^{t/2}+\int_{t/2}^t\right)R^0(t-s)\Bf(s)\intd s
				=:\CR_1^0(t)+\CR_2^0(t)\quad(t>0).
\end{equation*}
Then, by Lemma \ref{lemm:SaS} \eqref{lemm:SaS2}, it holds that, for $l=0,1$ and $t>0$,
\begin{align*}
	\|\nabla^l\CR_1^0(t)\|_{L_r(\BR_-^3)} 
		&\leq
			C_{\bar{q},r} \int_{0}^{t/2}(t-s)^{-\frac{3}{2}
			\left(\frac{1}{\bar{q}}-\frac{1}{r}\right)-\frac{l}{2}}
			\|\Bf(s)\|_{L_{\bar{q}}(\BR_-^3)}\intd s \\
		&\leq
			C_{\bar{q},r}\,t^{-\frac{3}{2}\left(\frac{1}{\bar{q}}-\frac{1}{r}\right)-\frac{l}{2}}
			\Bigg\{\sum_{i=1}^2\bigg(\int_0^{t/2}(s+2)^{-\Fb_i}\intd s\bigg)
			\,\|\Bf_i\|_{L_\infty^{\Fb_i}(\BR_+,L_{\bar{q}}(\BR_-^3))}  \\
			&+\bigg(\int_0^{t/2}(s+2)^{-p'\Fb_3 }\intd s\bigg)^{1/p'}
			\|\Bf_3\|_{L_p^{\Fb_3}(\BR_+,L_{\bar{q}}(\BR_-^3))}\Bigg\} \\
		&\leq
			C_{p,\bar{q},r}\,t^{-\frac{3}{2}\left(\frac{1}{\bar{q}}-\frac{1}{r}\right)-\frac{l}{2}}
			\sum_{i=1}^3\|\Bf_i\|_{\wt\BBF_i(\Fa_0,\Fb_i)}, \\
	\|\nabla^l \CR_2^0(t)\|_{L_r(\BR_-^3)} 
		&\leq
			C_{\bar{q},r}
			\Bigg\{\sum_{i=1}^2(t+2)^{-\Fb_i}
			\bigg(\int_{t/2}^t
			(t-s)^{-\frac{3}{2}\left(\frac{1}{\bar{q}}-\frac{1}{r}\right)-\frac{l}{2}}
			\intd s\bigg) \|\Bf_i\|_{L_\infty^{\Fb_i}(\BR_+,L_{\bar{q}}(\BR_-^3))} \\
			&+(t+2)^{-\Fb_3}\bigg(\int_{t/2}^t(t-s)^{-p' \left\{\frac{3}{2}
			\left(\frac{1}{\bar{q}}-\frac{1}{r}\right)+\frac{l}{2}\right\}}\intd s\bigg)^{1/p'}
			\|\Bf_3\|_{L_p^{\Fb_3}(\BR_+,L_{\bar{q}}(\BR_-^3))} \Bigg\}\\
		&\leq
			C_{p,\bar{q},r}
			(t+2)^{-\Fb_0-\frac{3}{2}\left(\frac{1}{\bar{q}}-\frac{1}{r}\right)-\frac{l}{2}}
			\left((t+2)+(t+2)^{1/p'}\right)
			\sum_{i=1}^3\|\Bf_i\|_{\wt\BBF_i(\Fa_0,\Fb_i)}
			\\
		&\leq
			C_{p,\bar{q},r}
			(t+2)^{-\frac{3}{2}\left(\frac{1}{\bar{q}}-\frac{1}{r}\right)-\frac{l}{2}}
			\sum_{i=1}^3\|\Bf_i\|_{\wt\BBF_i(\Fa_0,\Fb_i)},
\end{align*}
where $p'=p/(p-1)$ and 
we have noted that, by $2/p+3/q<1$,
\begin{equation}\label{singular:1}
	\frac{3}{2}\left(\frac{1}{\bar{q}}-\frac{1}{r}\right)+\frac{l}{2}
		<p'\left\{\frac{3}{2}\left(\frac{1}{\bar{q}}-\frac{1}{r}\right)+\frac{l}{2}\right\}
		<p'\left(\frac{3}{2q}+\frac{1}{2}\right)<1.
\end{equation}
Hence, we have
\begin{align}
	&\|\CR^0(t)\|_{L_r(\BR_-^3)} 
		\leq
			C_{p,\bar{q},r}
			(t+2)^{-\Fm(\bar{q},r)}
			\sum_{i=1}^3\|\Bf_i\|_{\wt\BBF_i(\Fa_0,\Fb_i)} \quad (t\geq1), \label{141122_4} \\
	&\|\nabla \CR^0(t)\|_{L_r(\BR_-^3)} 
		\leq
			C_{p,\bar{q},r}
			(t+2)^{-\Fn(\bar{q},r)-1/8}
			\sum_{i=1}^3\|\Bf_i\|_{\wt\BBF_i(\Fa_0,\Fb_i)}\quad (t\geq1), \notag
\end{align}
because $\Fm(\bar{q},r)\leq \Fn(\bar{q},r)$ and $\Fn(\bar{q},r)+1/2\leq \Fn(\bar{q},r)+1/8$.
On the other hand,
by Lemma \ref{lemm:embed} \eqref{embed:1} and \eqref{141015_5},
\begin{equation}\label{160607_5}
	\sup_{0<t<2}\|\CR^0(t)\|_{W_r^1(\BR_-^3)}
		\leq C_{p,r}\|\CR^0\|_{W_{r,p}^{2,1}(\BR_-^3\times(0,2))}
		\leq  C_{p,r}\sum_{i=1}^3\|\Bf_i\|_{\BBF_i},
\end{equation}
which, combined with \eqref{141122_4}, furnishes the estimate \eqref{160607_2}.

{\bf Step 2.}
The aim of this step is to show the following estimates:
\begin{align}
	&\|(\nabla^2\CR^1,\nabla\CS^1,\nabla^3\CE(\CT^1),\nabla^2\CE(\SSR^1|_{\BR_0^3}))\|_{L_p(\BR_+,L_r(\BR_-^3))}  \label{160607_7} \\
		&\quad \leq
			C_{p,r}\left(\|\Bf\|_{L_p(\BR_+,L_r(\BR_-^3))}
			+\|k\|_{L_p(\BR_+,W_r^{2-1/r}(\BR_0^3))}\right),  
			\notag \\
	&\|\CR^1\|_{L_\infty^{\Fm(\bar{q},r)}(\BR_+,L_r(\BR_-^3))}
	+\|\nabla\CR^1\|_{L_\infty^{\Fn(\bar{q},r)+1/8}(\BR_+,L_r(\BR_-^3))} \label{160607_8} 
	+\|\nabla\CE(\CT^1)\|_{L_\infty^{\Fm(\bar{q},r)+1/4}(\BR_+,W_r^1(\BR_-^3))} \\
	&+\|\CT^1\|_{L_\infty^{1/\bar{q}-1/r}(\BR_+,L_r(\BR^2))} 
		\leq
			C_{p,\bar{q},r}
			\left(\sum_{i=1}^3\|\Bf_i\|_{\wt\BBF_i(\Fa_0,\Fb_i)}
			+\|k\|_{\wt\BBK(\Fa_0,\Fb_4)}\right),	\notag 
\end{align}
where $\SSR^1$ denotes the third component of $\CR^1$.

The estimate \eqref{160607_7} follows directly from \cite[Appendix B]{Saito15},
so that we here prove \eqref{160607_8} only. Set
\begin{align*}
	\CR^1(t)
		&=\left(\int_{0}^{t/2}+\int_{t/2}^t\right)R^1(t-s)(\Bf(s),k(s))\intd s
		=: \CR_1^1(t)+\CR_2^1(t)\quad(t>0),  \\
	\CT^1(t)
		&=\left(\int_{0}^{t/2}+\int_{t/2}^{t}\right)T^1(t-s)(\Bf(s),k(s))\intd s
		=:\CT_1^1(t)+\CT_2^1(t)\quad(t>0). 
\end{align*}
Noting that $\Fn(\bar{q},r)+1/2<p'(\Fn(\bar{q},r)+1/2)<1$ by \eqref{singular:1},
we have, by Lemma \ref{lemm:SaS} \eqref{lemm:SaS4} together with the trace theorem
and for any $t>0$, 
\begin{align*}
	&\|(\CR_1^1(t),\nabla\CE(\CT_1^1(t)))\|_{W_r^1(\BR_-^3)}
	+\|\CT_1^1(t)\|_{L_r(\BR^2)}  \\
		&\quad \leq	
			C_{\bar{q},r}e^{-(\ga_4/2)t}
			\Bigg\{\sum_{i=1}^2
			\bigg(\int_0^{t/2}(t-s)^{-\Fn\left(\bar{q},r\right)-\frac{1}{2}}(s+2)^{-\Fb_i}\intd s\bigg)
			\|\Bf_i\|_{L_\infty^{\Fb_i}(\BR_+,L_{\bar{q}}(\BR_-^3))}  \\
			& \quad + \bigg(\int_0^{t/2}(t-s)^{-p'\left(\Fn(\bar{q},r)
			+\frac{1}{2}\right)}(s+2)^{-p'\Fb_3}\intd s\bigg)^{1/p'}
			\|\Bf_3\|_{L_p^{\Fb_3}(\BR_+,L_{\bar{q}}(\BR_-^3))} \\
			& \quad + \bigg(\int_0^{t/2}(t-s)^{-p'\left(\Fn(\bar{q},r)
			+\frac{1}{2}\right)}(s+2)^{-p'\Fb_4}\intd s \bigg)^{1/p'}
			\|k\|_{L_p^{\Fb_4}(\BR_+,W_{\bar{q}}^{2-1/\bar{q}}(\BR_0^3))} \Bigg\} \\
		&\quad\leq
			C_{p,\bar{q},r}
			e^{-(\ga_4/2)t}\left(t^{1-\Fn(\bar{q},r)-\frac{1}{2}}+t^{\frac{1}{p'}-\Fn(\bar{q},r)-\frac{1}{2}}\right)
			\left(\sum_{i=1}^3\|\Bf_i\|_{\wt{\BBF}_i(\Fa_0,\Fb_i)}
			+\|k\|_{\wt\BBK(\Fa_0,\Fb_4)}\right)
			\\
		&\quad \leq
			C_{p,\bar{q},r}
			e^{-(\ga_4/4)t}
			\left(\sum_{i=1}^3\|\Bf_i\|_{\wt{\BBF}_i(\Fa_0,\Fb_i)}
			+\|k\|_{\BBK(\Fa_0,\Fb_4)}\right),
			\\
	&\|(\CS_2^1(t),\nabla\CE(\CT_2^1(t)))\|_{W_r^1(\BR_-^3)}
	+\|\CT_2^1(t)\|_{L_r(\BR^2)} \\
		&\quad \leq 
			C_{\bar{q},r}(t+2)^{-\Fb_0}
			\Bigg\{\sum_{i=1}^2\bigg(\int_{t/2}^{t}(t-s)^{-\Fn(\bar{q},r)-\frac{1}{2}}e^{-\ga_4(t-s)}\intd s\bigg)
			\|\Bf_i\|_{L_\infty^{\Fb_i}(\BR_+,L_{\bar{q}}(\BR_-^3))} \\
			&\quad +\bigg(\int_{t/2}^{t}(t-s)^{-p'\left(\Fn(\bar{q},r)+\frac{1}{2}\right)}
			e^{-p'\ga_4(t-s)}\intd s\bigg)^{1/p'}
			\|\Bf_3\|_{L_p^{\Fb_3}(\BR_+,L_{\bar{q}}(\BR_-^3))} \\
			&\quad +\bigg(\int_{t/2}^{t}(t-s)^{-p'\left(\Fn(\bar{q},r)+\frac{1}{2}\right)}
			e^{-p'\ga_4(t-s)}\intd s\bigg)^{1/p'}
			\|k\|_{L_p^{\Fb_4}(\BR_+,W_{\bar{q}}^{2-1/\bar{q}}(\BR_0^3))}\Bigg\} \\
		&\quad \leq C_{p,\bar{q},r}(t+2)^{-\Fb_0}
			\left(\sum_{i=1}^3\|\Bf_i\|_{\wt{\BBF}_i(\Fa_0,\Fb_i)}
			+\|k\|_{\wt{\BBK}(\Fa_0,\Fb_4)}\right),
\end{align*}
which furnishes that
\begin{equation*}
	\|(\CR^1,\nabla\CE(\CT^1))\|_{L_\infty^{\Fb_0}(\BR_+,W_r^1(\BR_-^3))}
	+\|\CT^1\|_{L_\infty^{\Fb_0}(\BR_+,L_r(\BR^2))}  
		\leq
			C_{p,\bar{q},r}
			\left(\sum_{i=1}^3\|\Bf_i\|_{\wt{\BBF}_i(\Fa_0,\Fb_i)}
			+\|k\|_{\wt{\BBK}(\Fa_0,\Fb_4)} \right).
\end{equation*}
This inequality implies the estimate \eqref{160607_8}, because
\begin{equation*}
	0<\frac{1}{\bar{q}}-\frac{1}{r}\leq  \Fm(\bar{q},r)
	< \Fn(\bar{q},r)+\frac{1}{8} < \Fm(\bar{q},r)+\frac{1}{4} <1 \leq \Fb_0.
\end{equation*}

{\bf Step 3.}
The aim of this step is to show estimates as follows:
\begin{align}
	&\|(\nabla^2\CR^2,\nabla\CS^2)\|_{L_p(\BR_+,L_q(\BR_-^3))}
	\label{160608_1}  \\
		&\quad \leq C_{p,q}
			\bigg(\|\Bf\|_{L_\infty^{\Fe_1}(\BR_+,L_2(\BR_-^3))\cap L_p(\BR_+,L_q(\BR_-^3))}  
			+\|k\|_{L_\infty^{\Fe_1}(\BR_+,L_2(\BR_0^3))\cap L_p(\BR_+,W_q^{2-1/q}(\BR_0^3))}\bigg),  \notag \\
	&\|\nabla^3\CE(\CT^2)\|_{L_p(\BR_+,L_q(\BR_-^3))}
	\label{161219_1} \\
		&\quad \leq C_{p,q}
			\bigg(\|\Bf\|_{L_\infty^{\Fe_2}(\BR_+,L_2(\BR_-^3))\cap L_p(\BR_+,L_q(\BR_-^3))}  
			 +\|k\|_{L_\infty^{\Fe_2}(\BR_+,L_2(\BR_0^3))\cap L_p(\BR_+,W_q^{2-1/q}(\BR_0^3))}\bigg), \notag \\
	&\|\nabla^2\CE(\SSR^2|_{\BR_0^3})\|_{L_p(\BR_+,L_q(\BR_-^3))}
	\label{161220_1}  \\
		&\quad \leq C_{p,q}
			\bigg(\|\Bf\|_{L_\infty^{\Fe_3}(\BR_+,L_2(\BR_-^3))\cap L_p(\BR_+,L_q(\BR_-^3))} 
			 +\|k\|_{L_\infty^{\Fe_3}(\BR_+,L_2(\BR_0^3))\cap L_p(\BR_+,W_q^{2-1/q}(\BR_0^3))}\bigg), \notag 
\end{align}
where $\SSR^2$ denotes the third component of $\CR^2$ and
$\Fe_1$, $\Fe_2$, $\Fe_3$ are positive constants satisfying the following conditions:
\begin{alignat}{2}\label{160613_1}
	&p\left(\Fe_1 +\Fm(2,q)-\frac{3}{4}\right)>1,  \quad && p\Fe_1>1,  \\
	&p \left(\Fe_2+\Fm(2,q)+\frac{1}{4}\right) > 1, \quad && p\Fe_2>1, \notag \\
	&p \left(\Fe_3+\Fm(2,q)\right) > 1, \quad && p\Fe_3>1, \notag 	
\end{alignat}
and furthermore,
\begin{align}
	&\|(\nabla^2\CR^2,\nabla\CS^2)\|_{L_p(\BR_+,L_q(\BR_-^3))}
	\label{160608_2}  
		\leq C_{p,q,\bar{q}}
			\bigg(\sum_{i=1}^2\|\Bf_i\|_{L_\infty^{\Fb_i-\Fa_1}(\BR_+,L_{\bar{q}}(\BR_-^3))\cap L_p(\BR_+,L_q(\BR_-^3))}  \\
			& \quad +\|\Bf_3\|_{L_p^{\Fb_3-\Fa_1}(\BR_+,L_{\bar{q}}(\BR_-^3))\cap L_p(\BR_+,L_q(\BR_-^3))}
			+\|k\|_{L_p^{\Fb_4-\Fa_1}(\BR_+,W_{\bar{q}}^2(\BR_-^3))\cap L_p(\BR_+,W_q^2(\BR_-^3))}\bigg),
			\notag\\
		&\|\nabla^3\CE(\CT^2)\|_{L_p(\BR_+,L_q(\BR_-^3))}+\|\nabla^2\CE(\SSR^2|_{\BR_0^3})\|_{L_p(\BR_+,L_q(\BR_-^3))}
	\label{161218_1}  \\
		&\quad \leq C_{p,q,\bar{q}}
			\bigg(\sum_{i=1}^2\|\Bf_i\|_{L_\infty^{\Fb_i-\Fa_2}(\BR_+,L_{\bar{q}}(\BR_-^3))\cap L_p(\BR_+,L_q(\BR_-^3))} \notag\\
			&\quad +\|\Bf_3\|_{L_p^{\Fb_3-\Fa_2}(\BR_+,L_{\bar{q}}(\BR_-^3))\cap L_p(\BR_+,L_q(\BR_-^3))}
			+\|k\|_{L_p^{\Fb_4-\Fa_2}(\BR_+,W_{\bar{q}}^2(\BR_-^3))\cap L_p(\BR_+,W_q^2(\BR_-^3))}\bigg),
			\notag\\
	& \|(\nabla^2\CR^2,\nabla\CS^2,\nabla^3\CE(\CT^2))\|_{L_p(\BR_+,L_r(\BR_-^3))} 
			\label{160608_3} 
			+\|\CR^2\|_{L_\infty^{\Fm(\bar{q},r)}(\BR_+,L_r(\BR_-^3))}
			 \\
			&\quad +\|\nabla\CR^2\|_{L_\infty^{\Fn(\bar{q},r)+1/8}(\BR_+,L_r(\BR_-^3))} 
			+\|\nabla\CE(\CT^2)\|_{L_\infty^{\Fm(\bar{q},r)+1/4}(\BR_+,W_r^1(\BR_-^3))} 
			+\|\CT^2\|_{L_\infty^{1/\bar{q}-1/r}(\BR_+,L_r(\BR^2))}
			\notag\\
		&\quad \leq C_{p,\bar{q},r}
			\bigg(\sum_{i=1}^3\|\Bf_i\|_{\BBF_i\cap\wt\BBF_i(\Fa_0,\Fb_i)}
			+\|k\|_{\BBK\cap\wt\BBK(\Fa_0,\Fb_4)}\bigg), \notag 
\end{align}
with some positive constants $C_{p,q,\bar{q}}$, $C_{p,\bar{q},r}$. 

First, we show the estimates \eqref{160608_1}-\eqref{161220_1}.
We here set, for $t\geq 2$,
\begin{align*}
	\CR^2(t)
		&=\left(\int_0^{t/2}+\int_{t/2}^{t-1}+\int_{t-1}^t\right)
		 R^2(t-s)(\Bf(s),k(s))\intd s 
		 =:\CR_1^2(t)+\CR_2^2(t)+\CR_3^2(t), \notag \\
	\CS^2(t)
		&=\left(\int_0^{t/2}+\int_{t/2}^{t-1}+\int_{t-1}^t\right)
		 S^2(t-s)(\Bf(s),k(s))\intd s 
		 =:\CS_1^2(t)+\CS_2^2(t)+\CS_3^2(t), \notag \\
	\CT^2(t)
		&=\left(\int_0^{t/2}+\int_{t/2}^{t-1}+\int_{t-1}^t\right)
		 T^2(t-s)(\Bf(s),k(s))\intd s 
		=:\CT_1^2(t)+\CT_2^2(t)+\CT_3^2(t). \notag
\end{align*}
By Lemma \ref{lemm:SaS} \eqref{lemm:SaS7}, we see that, for $t\geq 2$,
%
%
\begin{align*}
	&\|(\nabla^2\CR_1^2(t),\nabla^2\CS_1^2(t))\|_{L_q(\BR_-^3)} \\
		&\quad \leq C_q(t+2)^{-\Fm(2,q)-1/4}\int_0^{t/2}(s+2)^{-\Fe_1}\intd s  
			\left(\|\Bf\|_{L_\infty^{\Fe_1}(\BR_+,L_2(\BR_-^3))} +
			\|k\|_{L_\infty^{\Fe_1}(\BR_+,L_2(\BR_0^3))}\right) \\
		&\quad \leq
			C_q(t+2)^{-\Fm(2,q)-1/4}\left(1+\log(t+2)+(t+2)^{1-\Fe_1}\right) 
			 \left(\|\Bf\|_{L_\infty^{\Fe_1}(\BR_+,L_2(\BR_-^3))}
			+\|k\|_{L_\infty^{\Fe_1}(\BR_+,L_2(\BR_0^3))}\right), 
			\\
	&\|(\nabla^2\CR_1^2(t),\nabla^2\CS_2^2(t))\|_{L_q(\BR_-^3)} \\
		&\quad \leq C_q (t+2)^{-\Fe_1}\int_{t/2}^{t-1}(t+2-s)^{-\Fm(2,q)-1/4}\intd s 
			\left(\|\Bf\|_{L_\infty^{\Fe_1}(\BR_+,L_2(\BR_-^3))}
			+\|k\|_{L_\infty^{\Fe_1}(\BR_+,L_2(\BR_0^3))}\right) \\
		&\quad \leq C_q
			(t+2)^{1-\Fe_1-\Fm(2,q)-1/4}
			\left(\|\Bf\|_{L_\infty^{\Fe_1}(\BR_+,L_2(\BR_-^3))}
			+\|k\|_{L_\infty^{\Fe_1}(\BR_+,L_2(\BR_0^3))}\right), \\
	&\|(\nabla^2\CR_3^2(t),\nabla^2\CS_3^2(t))\|_{L_q(\BR_-^3)} \\
		&\quad \leq C_q 
			(t+2)^{-\Fe_1}\int_{t-1}^t (t-s)^{-\al}\intd s
			\left(\|\Bf\|_{L_\infty^{\Fe_1}(\BR_+,L_2(\BR_-^3))}
			+\|k\|_{L_\infty^{\Fe_1}(\BR_+,L_2(\BR_0^3))}\right) \\
			&\quad \leq C_q 
			(t+2)^{-\Fe_1}\left(\|\Bf\|_{L_\infty^{\Fe_1}(\BR_+,L_2(\BR_-^3))}
			+\|k\|_{L_\infty^{\Fe_1}(\BR_+,L_2(\BR_0^3))}\right)	 \quad  \text{for $0<\al<1$,} \\ 
	&\|\nabla^3\CE(\CT_1^2(t))\|_{L_q(\lhs)} \\
		&\quad \leq C_q(t+2)^{-\Fm(2,q)-5/4}\int_0^{t/2}(s+2)^{-\Fe_2}\intd s 
		\left(\|\Bf\|_{L_\infty^{\Fe_2}(\BR_+,L_2(\BR_-^3))} +
			\|k\|_{L_\infty^{\Fe_2}(\BR_+,L_2(\BR_0^3))}\right) \\
		&\quad \leq C_q(t+2)^{-\Fm(2,q)-5/4}\left(1+\log(t+2)+(t+2)^{1-\Fe_2}\right) 
		\left(\|\Bf\|_{L_\infty^{\Fe_2}(\BR_+,L_2(\BR_-^3))} +
			\|k\|_{L_\infty^{\Fe_2}(\BR_+,L_2(\BR_0^3))}\right), \\
	&\|\nabla^3\CE(\CT_2^2(t))\|_{L_q(\lhs)} \\
		&\quad \leq C_q(t+2)^{-\Fe_2}\int_{t/2}^{t-1}(t+2-s)^{-\Fm(2,q)-5/4}\intd s 
		\left(\|\Bf\|_{L_\infty^{\Fe_2}(\BR_+,L_2(\BR_-^3))} +
			\|k\|_{L_\infty^{\Fe_2}(\BR_+,L_2(\BR_0^3))}\right) \\
		&\quad \leq C_q(t+2)^{-\Fe_2}\left(\|\Bf\|_{L_\infty^{\Fe_2}(\BR_+,L_2(\BR_-^3))} +
			\|k\|_{L_\infty^{\Fe_2}(\BR_+,L_2(\BR_0^3))}\right), \\
	&\|\nabla^3\CE(\CT_3^2(t))\|_{L_q(\lhs)} \\
		&\quad \leq C_q(t+2)^{-\Fe_2}\int_{t-1}^t (t-s)^{-\beta}\intd s 
		\left(\|\Bf\|_{L_\infty^{\Fe_2}(\BR_+,L_2(\BR_-^3))} +
			\|k\|_{L_\infty^{\Fe_2}(\BR_+,L_2(\BR_0^3))}\right) \\
		&\quad \leq C_q(t+2)^{-\Fe_2}\left(\|\Bf\|_{L_\infty^{\Fe_2}(\BR_+,L_2(\BR_-^3))} +
			\|k\|_{L_\infty^{\Fe_2}(\BR_+,L_2(\BR_0^3))}\right) \quad \text{for $0<\beta<1$}.
\end{align*}
Noting that, by \eqref{pq:lin},
\begin{equation*}
	\|(t+2)^{-\Fm(2,q)-1/4}\log(t+2)\|_{L_p((2,\infty))}<\infty,
\end{equation*}
we have by \eqref{160613_1} 
\begin{align}\label{160613_2}
	&\|(\nabla^2\CR^2,\nabla\CS^2)\|_{L_p((2,\infty),L_q(\BR_-^3))} 
		\leq C_{p,q}
			\left(\|\Bf\|_{L_\infty^{\Fe_1}(\BR_+,L_2(\BR_-^3))}
			+\|k\|_{L_\infty^{\Fe_1}(\BR_+,L_2(\BR_0^3))}\right),  \\
	&\|\nabla^3\CE(\CT^2)\|_{L_p((2,\infty),L_q(\BR_-^3))} 
		 \leq C_{p,q}
			\left(\|\Bf\|_{L_\infty^{\Fe_2}(\BR_+,L_2(\BR_-^3))}
			+\|k\|_{L_\infty^{\Fe_2}(\BR_+,L_2(\BR_0^3))}\right). \notag
\end{align}
Analogously, it follows from \eqref{161219_21} and \eqref{161220_2} that
\begin{equation}\label{161220_3}
	\|\nabla^2\CE(\SSR^2|_{\BR_0^3})\|_{L_p((2,\infty),L_q(\lhs))} 
			\leq C_{p,q}
			\left(\|\Bf\|_{L_\infty^{\Fe_3}(\BR_+,L_2(\BR_-^3))}
			+\|k\|_{L_\infty^{\Fe_3}(\BR_+,L_2(\BR_0^3))}\right).  
\end{equation}
Since it holds that 
\begin{equation}\label{160613_5}
	\CR^2 = \Bu - \CR^0-\CR^1, \quad
	\CS^2 = \Fp - \CS^0-\CS^1, \quad	
	\CT^2 = h-\CT^1, 
\end{equation}
we have, by \eqref{141121_5}, \eqref{160607_1}, and \eqref{160607_7},
\begin{equation}\label{160613_6}
	\|(\nabla^2\CR^2,\nabla\CS^2,\nabla^3\CE(\CT^2),\nabla^2\CE(\SSR^2|_{\BR_0^3}))\|_{L_p((0,2),L_q(\BR_-^3))} 
		\leq
			C_{p,q}\left(\|\Bf\|_{L_p(\BR_+,L_q(\BR_-^3))}+\|k\|_{L_p(\BR_+,W_q^{2-1/q}(\BR_0^3))}\right), 
\end{equation}
which, combined with \eqref{160613_2} and \eqref{161220_3},
yields the estimates \eqref{160608_1}-\eqref{161220_1}.

Secondly, we prove \eqref{160608_2} and \eqref{161218_1}. Let $\SSF_j$ be given by
\begin{equation*}
	\SSF_j = \sum_{i=1}^2\|\Bf_i\|_{L_\infty^{\Fb_i-\Fa_j}(\BR_+,L_{\bar{q}}(\BR_-^3))}
		+\|\Bf_3\|_{L_p^{\Fb_3-\Fa_j}(\BR_+,L_{\bar{q}}(\BR_-^3))} 
		+\|k\|_{L_p^{\Fb_4-\Fa_j}(\BR_+,W_{\bar{q}}^2(\BR_-^3))} \quad (j=1,2),
\end{equation*}
and note that $\Fm(\bar{q},q)+1/4= 1/2+1/(2q)<1$ by \eqref{pq:lin}.
It then holds that, by Lemma \ref{lemm:SaS} \eqref{lemm:SaS5} with the trace theorem and for any $t\geq 2$,
\begin{align*}
	&\|(\nabla^2\CR_1^2(t),\nabla\CS_1^2(t))\|_{L_q(\BR_-^3)} \\
		&\quad \leq C_{q,\bar{q}}(t+2)^{-\Fm(\bar{q},q)-1/4} 
			\Bigg\{\sum_{i=1}^2
			\left(\int_0^{t/2}(s+2)^{-(\Fb_i-\Fa_1)}\intd s\right)
			\|\Bf_i\|_{L_\infty^{\Fb_i-\Fa_1}(\BR_+,L_{\bar{q}}(\BR_-^3))} \\
			&\quad +\left(\int_0^{t/2}(s+2)^{-p'\left(\Fb_3-\Fa_1\right)}\intd s\right)^{1/p'}
			\|\Bf_3\|_{L_p^{\Fb_3-\Fa_1}(\BR_+,L_{\bar{q}}(\BR_-^3))} \\
			&\quad+\left(\int_0^{t/2}(s+2)^{-p'\left(\Fb_4-\Fa_1\right)}\intd s\right)^{1/p'}
			\|k\|_{L_p^{\Fb_4-\Fa_1}(\BR_+,W_{\bar{q}}^2(\BR_-^3))}
			\Bigg\} \\
		&\quad \leq C_{q,\bar{q}}(t+2)^{-\Fm(\bar{q},q)-1/4} 
			\Big(1+(t+2)^{1-(\Fb_0-\Fa_1)}+\log(t+2)\Big)\SSF_1
		\leq 
			C_{q,\bar{q}}(t+2)^{\Fa_1-\Fm(\bar{q},q)-1/4}\SSF_1, \\
	&\|(\nabla^2\CR_2^2(t),\nabla\CS_2^2(t))\|_{L_q(\BR_-^3)} \\
		&\quad \leq C_{q,\bar{q}}(t+2)^{-(\Fb_0-\Fa_1)} 
				\Bigg\{
				\sum_{i=1}^2\left(\int_{t/2}^{t-1}(t+2-s)^{-\Fm(\bar{q},q)-1/4}\intd s\right)
				\|\Bf_i\|_{L_\infty^{\Fb_i-\Fa_1}(\BR_+,L_{\bar{q}}(\BR_-^3))} \\
				&\quad+\left(\int_{t/2}^{t-1}(t+2-s)^{-p'\left(\Fm(\bar{q},q)+1/4\right)}\intd s\right)^{1/p'}
				\|\Bf_3\|_{L_p^{\Fb_3-\Fa_1}(\BR_+,L_{\bar{q}}(\BR_-^3))} \\
				&\quad+\left(\int_{t/2}^{t-1}(t+2-s)^{-p'\left(\Fm(\bar{q},q)+1/4\right)}\intd s\right)^{1/p'}
				\|k\|_{L_p^{\Fb_4-\Fa_1}(\BR_+,W_{\bar{q}}^2(\BR_-^3))} 
				\Bigg\} \\
		&\quad \leq C_{q,\bar{q}}(t+2)^{-(\Fb_0-\Fa_1)+1-\Fm(\bar{q},q)-1/4}\SSF_1
			 \leq C_{q,\bar{q}}(t+2)^{\Fa_1-\Fm(\bar{q},q)-1/4}\SSF_1, \\ 
	&\|(\nabla^2\CR_3^2(t),\nabla\CS_3^2(t))\|_{L_q(\BR_-^3)} \\
			&\quad \leq C_{q,\bar{q}}(t+2)^{-(\Fb_0-\Fa_1)} 
			\Bigg\{\sum_{i=1}^2\left(\int_{t-1}^t(t-s)^{-\al}\intd s\right)\|\Bf_i\|_{L_\infty^{\Fb_i-\Fa_1}(\BR_+,L_{\bar{q}}(\BR_-^3))} \\
			&\quad +\left(\int_{t-1}^t(t-s)^{-\al p'}\intd s\right)^{1/p'}\|\Bf_3\|_{L_p^{\Fb_3-\Fa_1}(\BR_+,L_{\bar{q}}(\BR_-^3))} \\
			&\quad +\left(\int_{t-1}^t(t-s)^{-\al p'}\intd s\right)^{1/p'}\|k\|_{L_p^{\Fb_4-\Fa_1}(\BR_+,W_{\bar{q}}^2(\BR_-^3))}
			\Bigg\} \\
			&\quad \leq C_{q,\bar{q}}(t+2)^{-(\Fb_0-\Fa_1)}\SSF_1 	 \quad \text{for $0<\al<1/p'$.}
\end{align*}
The above inequalities impliy that, by \eqref{ab},
\begin{equation*}
	\|(\nabla^2\CR^2,\nabla\CS^2)\|_{L_p((2,\infty),L_q(\BR_-^3))} \\
		\leq C_{p,q,\bar{q}}\SSF_1.
\end{equation*}
Analogously, noting \eqref{ab}, we can prove that, by Lemma \ref{lemm:SaS} \eqref{lemm:SaS5}, \eqref{161219_21}, and \eqref{161220_2},
\begin{equation*}
	\|(\nabla^3\CT^2,\nabla^2\CE(\SSR^2|_{\BR_0^3}))\|_{L_p((2,\infty),L_q(\lhs))}
		\leq C_{p,q,\bar{q}}\SSF_2.
\end{equation*}
Combining the last two inequalities with \eqref{160613_6} furnishes \eqref{160608_2} and \eqref{161218_1}.

Finally, we prove the estimate \eqref{160608_3}.
To this end, we show 
\begin{lemm}\label{lemm:low}
Let $p,q$ be exponents satisfying \eqref{pq:lin} and $\bar{q}=q/2$,
and let $r$ be another exponent satisfying $2\leq r\leq q$.
Suppose that $\Fb_1,\Fb_2>1$, $\Fb_3,\Fb_4\geq1$, and 
\begin{equation*}
	\Bf_i \in L_\infty^{\Fb_i}(\BR_+,L_{\bar{q}}(\lhs))^3 \quad (i=1,2), 
	\Bf_3 \in L_p^{\Fb_3}(\BR_+,L_{\bar{q}}(\BR_-^3))^3, \quad
	k \in L_p^{\Fb_4}(\BR_+,W_{\bar{q}}^2(\BR_-^3)).
\end{equation*}
In addition,
let $Y_r$ be defined as in Lemma $\ref{lemm:SaS}$, and let
$\BS(t),\BT(t)$ be operators with
\begin{equation*}
	\BS(t)\in \CL(Y_{\bar{q}},L_r(\BR_-^3)^3),\quad
	\BT(t)\in \CL(Y_{\bar{q}},L_r(\BR^2))\quad(t>0)
\end{equation*}
that satisfy
\begin{equation*}
	\|\BS(t)\BG\|_{L_r(\BR_-^3)}\leq C_{\bar{q},r}(t+2)^{-a}\|\BG\|_{Y_{\bar{q}}}  \quad (t\geq 1), \quad
	\|\BT(t)\BG\|_{L_r(\BR^2)}\leq C_{\bar{q},r}(t+2)^{-b}\|\BG\|_{Y_{\bar{q}}}  \quad	(t\geq1)
\end{equation*}
for $\BG\in Y_{\bar q}$ and for real numbers $0<a,b<1$
with a positive constant $C_{\bar{q},r}$ independent of time $t$ and $\BG$, while	
\begin{equation*}
	\|\BS(t)\BG\|_{L_r(\BR_-^3)}\leq Ct^{-\al}\|\BG\|_{Y_{\bar q}}  \quad (0<t\leq 1), \\quad
	\|\BT(t)\BG\|_{L_r(\BR^2)} \leq Ct^{-\beta}\|\BG\|_{Y_{\bar q}} \quad (0<t\leq 1)
\end{equation*}
for $\BG\in Y_{\bar q}$ and for real numbers  $\al,\beta>0$ satisfying $p'\al<1$ and $p'\beta<1$ with $p'=p/(p-1)$.
Then, setting
\begin{equation*}
	\SSG =
		\sum_{i=1}^2\|\Bf_i\|_{L_\infty^{\Fb_i}(\BR_+,L_{\bar{q}}(\BR_-^3))}
		+\|\Bf_3\|_{L_p^{\Fb_3}(\BR_+,L_{\bar{q}}(\BR_-^3))}
		+\|k\|_{L_p^{\Fb_4}(\BR_+,W_{\bar{q}}^2(\BR_-^3))},
\end{equation*}
we see that the following assertions hold.
\begin{enumerate}[$(1)$]
	\item
		Let $\BF=(\Bf,k|_{\BR_0^3})$ with $\Bf=\Bf_1+\Bf_2+\Bf_3$.
		Then, there is a positive constant $C_{p,\bar{q},r}$,
		independent of $\Bf_1$, $\Bf_2$, $\Bf_3$, and $k$,
		such that 
		\begin{alignat*}{2}
			\left\|\int_0^t\BS(t-s)\BF(s)\intd s\right\|_{L_r(\BR_-^3)}
				&\leq C_{p,\bar{q},r} (t+2)^{-a}\SSG && \quad (t\geq 2), \\
			\left\|\int_0^t\BT(t-s)\BF(s)\intd s\right\|_{L_r(\BR^2)}
				&\leq C_{p,\bar{q},r} (t+2)^{-b}\SSG && \quad (t\geq 2).
		\end{alignat*}
	\item
		Let $\BF=(\Bf,k|_{\BR_0^3})$ with $\Bf=\Bf_1+\Bf_2+\Bf_3$.
		If $pa>1$ and $p b>1$, then there is a positive constant $C_{p,\bar{q},r}$ such that
		\begin{equation*}
			\left\|\int_0^t\BS(t-s)\BF(s)\intd s\right\|_{L_p((2,\infty),L_r(\BR_-^3))} 
				\leq C_{p,\bar{q},r}\SSG, \quad
			\left\|\int_0^t\BT(t-s)\BF(s)\intd s\right\|_{L_p((2,\infty),L_r(\BR^2))}
				\leq C_{p,\bar{q},r}\SSG.
		\end{equation*}
\end{enumerate}
\end{lemm}

\begin{proof}
We here prove the case of $\BS(t)$ only for both (1) and (2).
Let $t\geq 2$, and set
\begin{equation*}
	\int_0^t\BS(t-s)\BF(s)\intd s
		=\left(\int_0^{t/2}+\int_{t/2}^{t-1}+\int_{t-1}^t\right)\BS(t-s)\BF(s)\intd s 
		=:\BS_1(t)+\BS_2(t)+\BS_3(t). 
\end{equation*}	
Setting $\Fb_0=\min(\Fb_1,\Fb_2,\Fb_3,\Fb_4)$, we have $\Fb_0\geq 1$.
It then holds that, for $t\geq 2$,
\begin{align*}
	&\|\BS_1(t)\|_{L_r(\BR_-^3)}
		\leq
			C_{\bar{q},r}(t+2)^{-a}\Bigg\{\sum_{i=1}^2\left(\int_0^{t/2}(s+2)^{-\Fb_i}\intd s\right)
			\|\Bf_i\|_{L_\infty^{\Fb_i}(\BR_+,L_{\bar{q}}(\BR_-^3))} \\
		&\quad+ \left(\int_0^{t/2}(s+2)^{-p'\Fb_3}\intd s\right)^{1/p'}
			\|\Bf_3\|_{L_p^{\Fb_3}(\BR_+,L_{\bar{q}}(\BR_-^3))}\\
		&\quad+ \left(\int_0^{t/2}(s+2)^{-p'\Fb_3}\intd s\right)^{1/p'}
			\|k\|_{L_p^{\Fb_4}(\BR_+,W_{\bar{q}}^2(\BR_-^3))}\Bigg\} \\
		&\leq
			C_{p,\bar{q},r}(t+2)^{-a}\SSG, \\
	&\|\BS_2(t)\|_{L_r(\BR_-^3)}
		\leq
			C_{\bar{q},r}(t+2)^{-\Fb_0} 
			\Bigg\{\sum_{i=1}^2\left(\int_{t/2}^{t-1}(t+2-s)^{-a}\intd s\right)
			\|\Bf_i\|_{L_\infty^{\Fb_i}(\BR_+,L_{\bar{q}}(\BR_-^3))} \\
		&\quad+\left(\int_{t/2}^{t-1}(t+2-s)^{-p' a}\intd s \right)^{1/p'} 
			\|\Bf_3\|_{L_p^{\Fb_3}(\BR_+,L_{\bar q}(\BR_-^3))} \\ 
		&\quad+\left(\int_{t/2}^{t-1}(t+2-s)^{-p' a}\intd s \right)^{1/p'}
			\|k\|_{L_p^{\Fb_4}(\BR_+,W_{\bar q}^2(\BR_-^3))}\Bigg\} \\
		&\leq
			C_{p,\bar{q},r}(t+2)^{-\Fb_0}\left((t+2)^{1-a}+(t+2)^{1/p'-a}\right)\SSG
		\leq
			C_{p,\bar{q},r}(t+2)^{-a}\SSG, \\ 
	&\|\BS_3(t)\|_{L_r(\BR_-^3)}
		\leq
			C_{\bar{q},r}(t+2)^{-\Fb_0} 
			\Bigg\{\sum_{i=1}^2\left(\int_{t-1}^t(t-s)^{-\al}\intd s\right)
			\|\Bf_i\|_{L_\infty^{\Fb_i}(\BR_+,L_{\bar{q}}(\BR_-^3))} \\
			&\quad +\left(\int_{t-1}^{t}(t-s)^{-p'\al}\intd s \right)^{1/p'}
				 \|\Bf_3\|_{L_p^{\Fb_3}(\BR_+,L_{\bar{q}}(\BR_-^3))} \\
			&\quad+\left(\int_{t-1}^{t}(t-s)^{-p'\al}\intd s \right)^{1/p'}
				\|k\|_{L_p^{\Fb_4}(\BR_+,W_{\bar{q}}^2(\BR_-^3))}\Bigg\} \\
		&\leq
			C_{p,\bar{q},r}(t+2)^{-\Fb_0}\SSG \leq C_{p,\bar{q},r}(t+2)^{-a}\SSG.
\end{align*}
These inequalities complete the required estimate of $\BS(t)$ in (1). 
Then, taking $L_p$-norm, with respect to $t\in(2,\infty)$,
of the inequality obtained in (1)
yields the required estimate of $\BS(t)$ in (2).
\end{proof}

By Lemma \ref{lemm:low} and Lemma \ref{lemm:SaS} \eqref{lemm:SaS5}, we have
\begin{align}\label{141108_20}
	&
	\|\CR^2\|_{L_\infty^{\Fm(\bar{q},r)}((2,\infty),L_r(\lhs))}
	+\|\nabla\CR^2(t)\|_{L_\infty^{\Fn(\bar{q},r)+1/8}((2,\infty),L_r(\lhs))} \\
	&\quad+\|\nabla \CE(\CT^2(t))\|_{L_\infty^{{\Fm(\bar{q},r)+1/4}}((2,\infty),W_r^1(\lhs))}
	+\|\CT^2(t)\|_{L_\infty^{1/\bar{q}-1/r}((2,\infty),L_r(\BR^2))} \notag \\
		&\leq
			 C_{p,\bar{q},r}\left(\sum_{i=1}^3\|\Bf_i\|_{\wt{\BBF}_i(\Fa_0,\Fb_i)}
			+\|k\|_{\wt{\BBK}(\Fa_0,\Fb_4)}\right) \notag
\end{align}
with some positive constant $C_{p,\bar{q},r}$, and furthermore, we have, by \eqref{pq:lin},
\begin{align}\label{141108_21}
	&\|(\CR^2,\nabla\CR^2,\nabla^2\CR^2,\nabla\CS^2)\|_{L_p((2,\infty),L_r(\BR_-^3))} 
	+ \|\nabla\CE(\CT^2)\|_{L_p((2,\infty),W_r^2(\lhs))} \\
	&\quad +\|\CT^2\|_{L_p((2,\infty),L_r(\BR^2))}  
		\leq C_{p,\bar{q},r}\left(\sum_{i=1}^3\|\Bf_i\|_{\wt{\BBF}_i(\Fa_0,\Fb_i)}
			+\|k\|_{\wt{\BBK}(\Fa_0,\Fb_3)}\right). \notag		
\end{align}
On the other hand,
since it holds by Lemma \ref{lemm:embed} \eqref{embed:1} and \eqref{141121_5} that
\begin{equation*}
	\|\Bu\|_{L_\infty((0,3),W_r^1(\BR_-^3))} + \|\nabla\CE(h)\|_{L_\infty((0,3),W_r^1(\lhs))} 
		\leq  C_{p,r}\left(\sum_{i=1}^3\|\Bf_i\|_{\BBF_i}+\|k\|_{\BBK}\right),
\end{equation*}
and since it holds by Sobolev's embedding theorem and \eqref{141121_5} that
\begin{equation*}
	\|h\|_{L_\infty((0,3),L_r(\BR^2))}
		\leq C_{p,r}\left(\sum_{i=1}^3\|\Bf_i\|_{\BBF_i}+\|k\|_{\BBK}\right),
\end{equation*}
%
we have, by \eqref{141121_5}, \eqref{160607_1}, \eqref{160607_2}, \eqref{160607_7}, \eqref{160607_8},
and \eqref{160613_5}, 
\begin{align*}
&\|(\CR^2,\nabla \CE(\CT^2))\|_{L_\infty((0,3),W_r^1(\lhs))}
	+\|\CT^2\|_{L_\infty((0,3)L_r(\BR^2))}  \\
	&\quad +\|(\CR^2,\nabla\CR^2,\nabla^2\CR^2,\nabla\CS^2)\|_{L_p((0,3),L_r(\BR_-^3))} 
 + \|\nabla\CE(\CT^2)\|_{L_p((0,3),W_r^2(\lhs))}
	+\|\CT^2\|_{L_p((0,3),L_r(\BR^2))}  \notag \\
			&\leq C_{p,r}\left(\sum_{i=1}^3\|\Bf_i\|_{\BBF_i\cap\wt\BBF_i(\Fa_0,\Fb_i)}
			+\|k\|_{\BBK\cap \wt\BBK(\Fa_0,\Fb_4)}\right) \notag
\end{align*}
with some positive constant $C_{p,r}$.
Combining this inequality with \eqref{141108_20} and \eqref{141108_21} furnishes \eqref{160608_3}. 

Summing up Step 1-Step 3 (i.e., \eqref{160607_1}, \eqref{160607_2}, \eqref{160607_7}, \eqref{160607_8},
\eqref{160608_1}, \eqref{161219_1}, \eqref{161220_1},
\eqref{160608_2}, \eqref{161218_1}, and \eqref{160608_3}) and recalling
\begin{equation*}
	\Bu = \CR^0+\CR^1+\CR^2, \quad
	\Fp = \CS^0+\CS^1+\CS^2, \quad
	h=\CT^1+\CT^2,
\end{equation*}
we have obtained the following estimates:
\begin{align}
	&\|(\nabla^2\Bu,\nabla\Fp)\|_{L_p(\BR_+,L_q(\BR_-^3))} \label{160608_10} \\
		&\quad \leq C_{p,q,\bar{q}}
			\bigg(\|\Bf\|_{L_\infty^{\Fe_1}(\BR_+,L_2(\BR_-^3)\cap L_p(\BR_+,L_q(\BR_-^3))}
			+\|k\|_{L_\infty^{\Fe_1}(\BR_+,L_2(\BR_0^3)) \cap L_p(\BR_+,W_q^{2-1/q}(\BR_0^3))}\bigg),
			\notag \\
	&\|\nabla^3\CE(h)\|_{L_p(\BR_+,L_q(\BR_-^3))} \label{161219_7} \\
		&\quad \leq C_{p,q,\bar{q}}
			\bigg(\|\Bf\|_{L_\infty^{\Fe_2}(\BR_+,L_2(\BR_-^3)\cap L_p(\BR_+,L_q(\BR_-^3))} 
			+\|k\|_{L_\infty^{\Fe_2}(\BR_+,L_2(\BR_0^3)) \cap L_p(\BR_+,W_q^{2-1/q}(\BR_0^3))}\bigg),
			\notag \\
	&\|\nabla^2\CE(u_3|_{\BR_0^3})\|_{L_p(\BR_+,L_q(\BR_-^3))} \label{161220_5} \\
		&\quad \leq C_{p,q,\bar{q}}
			\bigg(\|\Bf\|_{L_\infty^{\Fe_3}(\BR_+,L_2(\BR_-^3)\cap L_p(\BR_+,L_q(\BR_-^3))} 
			+\|k\|_{L_\infty^{\Fe_3}(\BR_+,L_2(\BR_0^3)) \cap L_p(\BR_+,W_q^{2-1/q}(\BR_0^3))}\bigg),
			\notag \\
	&\|(\nabla^2\Bu,\nabla\Fp)\|_{L_p(\BR_+,L_q(\BR_-^3))} \label{160608_11} \\
		&\quad \leq C_{p,q,\bar{q}}
		\Bigg(\sum_{i=1}^2\|\Bf_i\|_{L_\infty^{\Fb_i-\Fa_1}(\BR_+,L_{\bar{q}}(\BR_-^3))\cap L_p(\BR_+,L_q(\BR_-^3))} 
		+\|\Bf_3\|_{L_p^{\Fb_3-\Fa_1}(\BR_+,L_{\bar{q}}(\BR_-^3))\cap L_p(\BR_+,L_q(\BR_-^3))}
		\notag \\
		&\quad +\|k\|_{L_p^{\Fb_4-\Fa_1}(\BR_+,W_{\bar{q}}^2(\BR_-^3))\cap L_p(\BR_+,W_q^2(\BR_-^3))}\Bigg)
		\notag\\
	&\|(\nabla^3\CE(h),\nabla^2\CE(u_3|_{\BR_0^3}))\|_{L_p(\BR_+,L_q(\BR_-^3))} \label{160608_15} \\
		&\quad \leq C_{p,q,\bar{q}}
		\Bigg(\sum_{i=1}^2\|\Bf_i\|_{L_\infty^{\Fb_i-\Fa_2}(\BR_+,L_{\bar{q}}(\BR_-^3))\cap L_p(\BR_+,L_q(\BR_-^3))} 
		+\|\Bf_3\|_{L_p^{\Fb_3-\Fa_2}(\BR_+,L_{\bar{q}}(\BR_-^3))\cap L_p(\BR_+,L_q(\BR_-^3))}
		\notag \\
		&\quad +\|k\|_{L_p^{\Fb_4-\Fa_2}(\BR_+,W_{\bar{q}}^2(\BR_-^3))\cap L_p(\BR_+,W_q^2(\BR_-^3))}\Bigg)
		\notag\\
	&\|(\nabla^2\Bu,\nabla\Fp,\nabla^3\CE(h))\|_{L_p(\BR_+,L_r(\BR_-^3))}
	+\|\Bu\|_{L_\infty^{\Fm(\bar{q},r)}(\BR_+,L_r(\BR_-^3))}
	 \label{160608_12} \\
	&\quad +\|\nabla\Bu\|_{L_\infty^{\Fn(\bar{q},r)+1/8}(\BR_+,L_r(\BR_-^3))}
	+\|\nabla\CE(h)\|_{L_\infty^{\Fm(\bar{q},r)+1/4}(\BR_+,W_r^1(\BR_-^3))}
	+\|h\|_{L_\infty^{1/\bar{q}-1/r}(\BR_+,L_r(\BR^2))} \notag  \\
	&\leq C_{p,\bar{q},r}
			\left(\sum_{i=1}^3\|\Bf_i\|_{\BBF_i\cap\wt\BBF_i(\Fa_0,\Fb_i)}
		+\|k\|_{\BBK\cap\wt\BBK(\Fa_0,\Fb_4)}\right) \notag 
\end{align}
with positive numbers $\Fe_1,\Fe_2,\Fe_3$ defined as in \eqref{160613_1}.

{\bf Step 4.}
This step is concerned with estimates of the time derivatives: $\pa_t\Bu$, $\pa_t h$, and $\pa_t\CE(h)$.
By using the first and the fourth equations of the system \eqref{linear4}, 
\eqref{160608_12}, and the trace theorem, we have
\begin{align}
	&\|\pa_t\Bu\|_{L_p(\BR_+,L_r(\BR_-^3))} 
		\leq C_{p,q,\bar{q},r}
		\left(\sum_{i=1}^3\|\Bf_i\|_{\BBF_i\cap\wt\BBF_i(\Fa_0,\Fb_i)}
		+\|k\|_{\BBK\cap\wt\BBK(\Fa_0,\Fb_4)}\right),
		\label{160613_12} \\
	&\|\pa_t h\|_{L_\infty^{\Fm(\bar{q},r)}(\BR_+,L_r(\BR^2))}+\|\pa_t h\|_{L_p(\BR_+,L_r(\BR^2))} \label{160613_15} \\
		&\leq C_{p,q,\bar{q},r}
		\left(\sum_{i=1}^3\|\Bf_i\|_{\BBF_i\cap\wt\BBF_i(\Fa_0,\Fb_i)}
		+\|k\|_{\BBK\cap\wt\BBK(\Fa_0,\Fb_4)}+\|k\|_{\BBB_1}\right). \notag 
\end{align}

Next, we prove an estimate of $\nabla\pa_t\CE(h)$ as follows:
\begin{equation}\label{161007_14}
	\|\nabla\pa_t\CE(h)\|_{L_\infty^{\Fm(\bar{q},r)+1/2}(\BR_+,L_r(\lhs))} 
		\leq C_{p,q}\left(\sum_{i=1}^3\|\Bf_i\|_{\BBF_i\cap \wt\BBF_i(\Fa_0,\Fb_i)}+\|k\|_{\BBK\cap\wt\BBK(\Fa_0,\Fb_4)}+\|k\|_{\BBA_4}\right).
\end{equation}
Note that, by the fourth equation of \eqref{linear4},
\begin{equation*}
	\pa_t\CE(h) = \CE(k|_{\BR_0^3})+\int_0^t\CE((R(t-s)(\Bf(s),k(s)))_3|_{\BR_0^3})\intd s.
\end{equation*}
Then, for $t\geq 2$, we have, by Lemmas \ref{lemm:ap} and \ref{lemm:low},
\begin{equation*}
	\left\|\nabla\int_0^t\CE((R(t-s)(\Bf(s),k(s)))_3|_{\BR_0^3})\intd s\right\|_{L_r(\lhs))}  
	\leq C_{p,q}(t+2)^{-\Fm(\bar{q},r)-1/2}\SSG \quad (t\geq 2).
\end{equation*}
Combining this identity with \eqref{161007_1}
furnishes that
\begin{equation}\label{161010_11}
	\|\nabla\pa_t\CE(h)\|_{L_\infty^{\Fm(\bar{q},r)+1/2}((2,\infty),L_r(\lhs))} 
		\leq C_{p,q}\left(\sum_{i=1}^3\|\Bf_i\|_{\BBF_i\cap \wt\BBF_i(\Fa_0,\Fb_i)}+\|k\|_{\BBK\cap\wt\BBK(\Fa_0,\Fb_4)}+\|k\|_{\BBA_4}\right).
\end{equation}
On the other hand,
by Lemma \ref{lemm:embed} \eqref{embed:0}, Sobolev's embedding theorem,
\eqref{141121_5}, and \eqref{161007_1}, 
we observe that
\begin{align*}
	&\|\nabla\pa_t\CE(h)\|_{BUC((0,2),L_r(\lhs))} \leq C_r\|(\nabla u_3,\nabla k)\|_{BUC((0,2)L_r(\lhs))} \\
	&\leq C_{p,r}\left(\|\nabla u_3\|_{W_p^{1/2}((0,2),L_r(\lhs))}+\|k\|_{\BBA_4}\right) \\
	&\leq C_{p,r}\left(\|u_3\|_{W_{r,p}^{2,1}(\lhs\times(0,2))}+\|k\|_{\BBA_4}\right) \\
	& \leq
		C_{p,r}\left(\sum_{i=1}^3\|\Bf_i\|_{\BBF_i\cap \wt\BBF_i(\Fa_0,\Fb_i)}+\|k\|_{\BBK\cap\wt\BBK(\Fa_0,\Fb_4)}+\|k\|_{\BBA_4}\right),
\end{align*}
which, combined with \eqref{161010_11}, furnishes \eqref{161007_14}.

Summing up \eqref{160608_12}, \eqref{160613_12}, \eqref{160613_15}, 
and \eqref{161007_14} and noting \eqref{pq:lin}, we have
\begin{equation}\label{161007_17}
	\sum_{r\in\{q,2\}}\left(\CM_{r,p}(\Bz)+\CN_r(\Bz)\right) 
	\leq C_{p,q}\left(\sum_{i=1}^3\|\Bf_i\|_{\BBF_i\cap \wt\BBF_i(\Fa_0,\Fb_i)}
		+\|k\|_{\BBK\cap\wt\BBK(\Fa_0,\Fb_4)\cap \BBA_4\cap \BBB_1}\right). 
\end{equation}
%
%
%
%
%
%
%

{\bf Step 5.}
The aim of this step is to prove the following estimate:
\begin{align}\label{160613_14}
	&\|(\pa_t\Bu,\nabla^2\Bu)\|_{L_p^{\Fa_1}(\BR_+,L_q(\BR_-^3))}
	+\|(\nabla^2\pa_t\CE(h),\nabla^3\CE(h))\|_{L_p^{\Fa_2}(\BR_+,L_q(\BR_-^3))} \\
	&\leq C_{p,q,\bar{q}}\left(\sum_{i=1}^3\|\Bf_i\|_{\BBF_i\cap\wt\BBF_i(\Fa_0,\Fb_i)}
		+\|k\|_{\BBK\cap\wt\BBK(\Fa_0,\Fb_4)\cap \BBA_4\cap \BBB_1}\right).
	\notag
\end{align}

We here set
	$\Bu^{\Fa_j}=(t+2)^{\Fa_j} \Bu$,
	$\Fp^{\Fa_j}=(t+2)^{\Fa_j} \Fp$, and
	$h^{\Fa_j}=(t+2)^{\Fa_j} h$ for $j=1,2$.
Then, $(\Bu^{\Fa_j},\Fp^{\Fa_j},h^{\Fa_j})$ satisfies
\begin{equation}\label{eq:weight}
	\left\{\begin{aligned}
		\pa_t\Bu^{\Fa_j}-\Di\BT(\Bu^{\Fa_j},\Fp^{\Fa_j})
			&=\Bf^{\Fa_j} && \text{in $\lhs$,} \\
		\di\Bu^{\Fa_j}
			&=0 && \text{in $\lhs$,} \\
		\BT(\Bu^{\Fa_j},\Fp^{\Fa_j})\Be_3+(c_g-c_\si\De')h^{\Fa_j}\Be_3
			&=0 && \text{on $\BR_0^3$,} \\
		\pa_t h^{\Fa_j}-u_3^{\Fa_j}
			&=k^{\Fa_j} && \text{on $\BR_0^3$,} \\
		\Bu^{\Fa_j}|_{t=0} &=0 && \text{in $\lhs$,} \\
		h^{\Fa_j}|_{t=0}&=0 && \text{on $\BR^2$,}
	\end{aligned}\right.
\end{equation}
where we have set $\Bu^{\Fa_j}=\tp(u_1^{\Fa_j},u_2^{\Fa_j},u_3^{\Fa_j})$ and
\begin{equation*}
	\Bf^{\Fa_j}=(t+2)^{\Fa_j}\Bf-\Fa_j (t+2)^{-1+{\Fa_j}}\Bu, \quad
	k^{\Fa_j}=(t+2)^{\Fa_j} k-\Fa_j(t+2)^{-1+{\Fa_j}}h.
\end{equation*}

To estimate $(\Bu^{\Fa_j},\Fp^{\Fa_j},h^{\Fa_j})$, we rewrite them as follows:
	$\Bu^{\Fa_j}=\Bv^{\Fa_j}+\Bw^{\Fa_j}$, 
	$\Fp^{\Fa_j} = \Fq^{\Fa_j}+\Fr^{\Fa_j}$, and 
	$h^{\Fa_j} = \ell ^{\Fa_j} + m^{\Fa_j}$
that are solutions to
\begin{align}
	&\left\{\begin{aligned}\label{eq:weight1}
		\pa_t\Bv^{\Fa_j}-\Di\BT(\Bv^{\Fa_j},\Fq^{\Fa_j})
			&=(t+2)^{\Fa_j}\Bf && \text{in $\lhs$,} \\
		\di\Bv^{\Fa_j}
			&=0 && \text{in $\lhs$,} \\
		\BT(\Bv^{\Fa_j},\Fq^{\Fa_j})\Be_3+(c_g-c_\si\De')\ell^{\Fa_j}\Be_3
			&=0 && \text{on $\BR_0^3$,} \\
		\pa_t \ell^{\Fa_j}-v_3^{\Fa_j}
			&=(t+2)^{\Fa_j} k && \text{on $\BR_0^3$,} \\
		\Bv^{\Fa_j}|_{t=0} &=0 && \text{in $\lhs$,} \\
		\ell^{\Fa_j}|_{t=0}&=0 && \text{on $\BR^2$,}
	\end{aligned}\right. \\
	&\left\{\begin{aligned}\label{eq:weight2}
		\pa_t\Bw^{\Fa_j}-\Di\BT(\Bw^{\Fa_j},\Fr^{\Fa_j})
			&=-{\Fa_j}(t+2)^{-1+{\Fa_j}}\Bu && \text{in $\lhs$,} \\
		\di\Bw^{\Fa_j}
			&=0 && \text{in $\lhs$,} \\
		\BT(\Bw^{\Fa_j},\Fr^{\Fa_j})\Be_3+(c_g-c_\si\De')m^{\Fa_j}\Be_3
			&=0 && \text{on $\BR_0^3$,} \\
		\pa_t m^{\Fa_j}-w_3^{\Fa_j}
			&=-{\Fa_j}(t+2)^{-1+{\Fa_j}}h && \text{on $\BR_0^3$,} \\
		\Bw^{\Fa_j}|_{t=0} &=0 && \text{in $\lhs$,} \\
		m^{\Fa_j}|_{t=0}&=0 && \text{on $\BR^2$,}
	\end{aligned}\right.
\end{align}
where $\Bv^{\Fa_j}=\tp(v_1^{\Fa_j},v_2^{\Fa_j}, v_3^{\Fa_j})$
and $\Bw^{\Fa_j}=\tp(w_1^{\Fa_j},w_2^{\Fa_j},w_3^{\Fa_j})$
for $j=1,2$.
It is clear that, for $i,j=1,2$,
\begin{align}\label{161218_5}
	&\|(t+2)^{\Fa_j}\Bf_i\|_{L_p(\BR_+,L_q(\BR_-^3)\cap L_\infty^{\Fb_i-\Fa_j}(\BR_+,L_{\bar{q}}(\BR_-^3))}
		\leq \|\Bf_i\|_{\wt{\BBF}(\Fa_0,\Fb_i)}, \\
	&\|(t+2)^{\Fa_j}\Bf_3\|_{L_p(\BR_+,L_q(\BR_-^3) \cap L_p^{\Fb_3-\Fa_j}(\BR_+,L_{\bar{q}}(\BR_-^3))}
		\leq \|\Bf_3\|_{\wt\BBF(\Fa_0,\Fb_3)}, \notag \\
	&\|(t+2)^{\Fa_j}k\|_{L_p(\BR_+,W_q^2(\BR_-^3)\cap L_p^{\Fb_4-\Fa_j}(\BR_+,W_{\bar{q}}^2(\BR_-^3))}
		\leq  \|k\|_{\wt\BBK(\Fa_0,\Fb_4)}. \notag
\end{align}
Since $(\Bv^{\Fa_j},\Fq^{\Fa_j},\ell^{\Fa_j})$ satisfies \eqref{eq:weight1},
we apply the estimates \eqref{160608_11}, \eqref{160608_15}, and \eqref{161218_5} to $(\Bv^{\Fa_j},\Fq^{\Fa_j},\ell^{\Fa_j})$
in order to obtain
\begin{equation}\label{160927_20}
	\|(\nabla^2\Bv^{\Fa_1},\nabla \Fq^{\Fa_1},\nabla^3\CE(\ell^{\Fa_2}),\nabla^2\CE(v_3^{\Fa_2}|_{\BR_0^3}))\|_{L_p(\BR_+,L_q(\lhs))} 
		\leq C_{p,q}\left(\sum_{i=1}^3\|\Bf_i\|_{\BBF_i\cap\wt\BBF_i(\Fa_0,\Fb_i)}
		+\|k\|_{\BBK\cap\wt\BBK(\Fa_0,\Fb_4)}\right). 
\end{equation}

Next, we estimate $(\Bw^{\Fa_j},\Fr^{\Fa_j},m^{\Fa_j})$. 
Since it holds by \eqref{ab} that
\begin{align*}
	&p(1+\Fm(\bar{q},q)-\Fa_j) \geq p \min\{1/4+\Fm(\bar{q},q)-\Fa_1,1/2+2/q-\Fa_2\}>1, \\
	&p(1+1/q-\Fa_j)\geq p\min\{1/4+\Fm(\bar{q},q)-\Fa_1,1/2+2/q-\Fa_2\}>1,
\end{align*}
we observe, for $j=1,2$, that 
\begin{align*}
	\|(t+2)^{-1+\Fa_j}\Bu\|_{L_p(\BR_+,L_q(\lhs))}  
		&\leq
			\|(t+2)^{-1+\Fa_j-\Fm(\bar{q},q)}\|_{L_p(\BR_+)}
			\|\Bu\|_{L_\infty^{\Fm(\bar{q},q)}(\BR_+,L_q(\lhs))} \notag  \\
		&\leq
			C_{p,q,\bar{q}}\|\Bu\|_{L_\infty^{\Fm(\bar{q},q)}(\BR_+,L_q(\lhs))}, \notag  
\end{align*}
and also that
\begin{align*}
	\|(t+2)^{-1+\Fa_j}h\|_{L_p(\BR_+,W_q^{2-1/q}(\BR^2))} 
		&\leq 
			\|(t+2)^{-1+\Fa_j}h\|_{L_p(\BR_+,L_q(\BR^2))}
			+\|(t+2)^{-1+\Fa_j}\nabla\CE(h) \|_{L_p(\BR_+,W_q^{1}(\BR_-^3))} \notag \\
		&\leq
			C_{p,q}\Big(\|(t+2)^{-1+\Fa_j-1/q}\|_{L_p(\BR_+)}
			\|h\|_{L_\infty^{1/q}(\BR_+,L_q(\BR^2))} \notag \\
			&+\|(t+2)^{-1+\Fa_j-\Fm(\bar{q},q)-1/4}\|_{L_p(\BR_+)}
			\|\nabla\CE(h)\|_{L_\infty^{\Fm(\bar{q},q)+1/4}(\BR_+,W_q^1(\lhs))}\Big) \notag \\
		&\leq
			C_{p,q,\bar{q}}\left(\|h\|_{L_\infty^{1/q}(\BR_+,L_q(\BR^2))}
			+\|\nabla\CE(h)\|_{L_\infty^{\Fm(\bar{q},q)+1/4}(\BR_+,W_q^1(\lhs))}\right). \notag 
\end{align*}
Furthermore, setting
\begin{equation*}
	\Fe_1 = 1-\Fa_1+\frac{1}{\bar{q}}-\frac{1}{2} = \frac{1}{2}+\frac{2}{q} -\Fa_1, \quad
	\Fe_2 = \Fe_3= 1-\Fa_2+\frac{1}{\bar{q}}-\frac{1}{2} = \frac{1}{2}+\frac{2}{q} -\Fa_2, 
\end{equation*}
we see that, by \eqref{pq:lin} and \eqref{ab},
\begin{align*}
	&p\left(\Fe_1+\Fm(2,q)-\frac{3}{4}\right) =p\left(\Fm(\bar{q},q)+\frac{1}{4}-\Fa_1\right)>1, \\
	&p\Fe_1=p\left(\frac{1}{2}+\frac{2}{q}-\Fa_1\right) > p\left(\Fm(\bar{q},q)+\frac{1}{4}-\Fa_1\right)>1, \\
	&p\left(\Fe_2+\Fm(2,q)+\frac{1}{4}\right)=p\left(\frac{3}{2}+\frac{1}{2q}-\Fa_2\right)
		>p\left(\frac{1}{2}+\frac{2}{q}-\Fa_2\right)>1, \\
	&p\Fe_2 = p\left(\frac{1}{2}+\frac{2}{q}-\Fa_2\right)>1, \quad
	p\left(\Fe_3 + \Fm(2,q)\right) >1, \quad p\Fe_3>1,
\end{align*}
which implies that the condition \eqref{160613_1} holds 
and that
\begin{align*}
	\|(t+2)^{-1+\Fa_j}\Bu\|_{L_\infty^{\Fe_j}(\BR_+,L_2(\BR_-^3))} 
		&\leq \|\Bu\|_{L_\infty^{\Fm(\bar{q},2)}(\BR_+,L_2(\lhs))}, \\
	\|(t+2)^{-1+\Fa_j}h\|_{L_\infty^{\Fe_j}(\BR_+,L_2(\BR^2))} 
		&\leq \|h\|_{L_\infty^{1/\bar{q}-1/2}(\BR_+,L_2(\lhs))}.
\end{align*}
Summing up the above estimates, we obtain, by \eqref{161007_17},
\begin{align}\label{161219_15}
	&\|(t+2)^{-1+\Fa_j}\Bu\|_{L_p(\BR_+,L_q(\lhs))\cap L_\infty^{\Fe_j}(\BR_+,L_2(\BR_-^3))}  
	+ \|(t+2)^{-1+\Fa_j}h\|_{L_p(\BR_+,W_q^{2-1/q}(\BR^2))\cap L_\infty^{\Fe_j}(\BR_+,L_2(\BR^2))} \\
		&\leq C_{p,q}\left(\sum_{i=1}^3\|\Bf_i\|_{\BBF_i\cap\wt\BBF_i(\Fa_0,\Fb_i)}
		+\|k\|_{\BBK\cap\wt\BBK(\Fa_0,\Fb_4)\cap\BBA_4\cap\BBB_1}\right) \notag
\end{align}
with some positive constant $C_{p,q}$ for $j=1,2$.
Combining this inequality with \eqref{160608_10}, \eqref{161219_7}, and \eqref{161220_5}
furnishes that
\begin{align}\label{161219_18}
	&\|(\nabla^2\Bw^{\Fa_1},\nabla\Fr^{\Fa_1},\nabla^3\CE(m^{\Fa_2}),\nabla^2\CE(w_3^{\Fa_2}|_{\BR_0^3}))\|_{L_p(\BR_+,L_q(\BR_-^3))} \\
		&\leq C_{p,q}\left(\sum_{i=1}^3\|\Bf_i\|_{\BBF_i\cap\wt\BBF_i(\Fa_0,\Fb_i)}
		+\|k\|_{\BBK\cap\wt\BBK(\Fa_0,\Fb_4)\cap\BBA_4\cap\BBB_1}\right).  \notag
\end{align}
By \eqref{160927_20}, \eqref{161219_18} and by using the first equation of \eqref{eq:weight} with \eqref{161219_15},
we have
\begin{align}\label{161219_19}
	&\|(\pa_t\Bu^{\Fa_1},\nabla^2\Bu^{\Fa_1},\nabla^3\CE(h^{\Fa_2}),\nabla^2\CE(u_3^{\Fa_2}|_{\BR_0^3}))\|_{L_p(\BR_+,L_q(\lhs))} \\
	&\leq  C_{p,q}\left(\sum_{i=1}^3\|\Bf_i\|_{\BBF_i\cap\wt\BBF_i(\Fa_0,\Fb_i)}
		+\|k\|_{\BBK\cap\wt\BBK(\Fa_0,\Fb_4)\cap\BBA_4\cap\BBB_1}\right), \notag
\end{align}
which, combined with the identity:
\begin{equation*}
	(t+2)^{\Fa_1}\pa_t \Bu
		=\pa_t \Bu^{\Fa_1}-\Fa_1 (t+2)^{-1+{\Fa_1}}\Bu
\end{equation*}
and with \eqref{161219_15}, furnishes that
\begin{equation*}
	\|\pa_t\Bu\|_{L_p^{\Fa_1}(\BR_+,L_q(\lhs))} 
		\leq 	C_{p,q}\left(\sum_{i=1}^3\|\Bf_i\|_{\BBF_i\cap\wt\BBF_i(\Fa_0,\Fb_i)}
		+\|k\|_{\BBK\cap\wt\BBK(\Fa_0,\Fb_4)\cap\BBA_4\cap\BBB_1}\right).
\end{equation*}

Finally, we consider $\nabla^2\pa_t\CE(h)$.
Since 
\begin{equation*}
	\nabla^2\pa_t\CE(h)=\nabla^2\CE(k|_{\BR_0^3})+\nabla^2\CE(u_3|_{\BR_0^3}),
\end{equation*}
we observe that, by \eqref{161007_1} and \eqref{161219_19},
\begin{align*}
	\|\nabla^2\pa_t\CE(h)\|_{L_p^{\Fa_2}(\BR_+,L_q(\lhs))} 
		&\leq 
			\|(\nabla^2\CE(k^{\Fa_2}|_{\BR_0^3}),\nabla^2\CE(u_3^{\Fa_2}|_{\BR_0^3}))\|_{L_p(\BR_+,L_q(\lhs))} \\
	&\leq C_{p,q}\left(\sum_{i=1}^3\|\Bf_i\|_{\BBF_i\cap\wt\BBF_i(\Fa_0,\Fb_i)}
		+\|k\|_{\BBK\cap\wt\BBK(\Fa_0,\Fb_4)\cap\BBA_4\cap\BBB_1}\right).
\end{align*}
Combining this inequality with \eqref{161219_19} implies \eqref{160613_14},
which completes the proof of Theorem \ref{theo:sec4_2}.

\section{Proof of Theorem \ref{theo:main1}}\label{sec5}
In this section, we prove Theorem \ref{theo:main1} by using Theorem \ref{theo:sec3_2}
and Theorem \ref{theo:sec4_2}.
Suppose that $r\in\{q,2\}$ in what follows.

%
%
%
%
%
%
%
%
%
%


\begin{rema}
We can assume that $\Bh|_{t=0}=0$ on $\BR_0^3$ without loss of generality.
In fact, it suffices to replace $\Fp$ by $\Fp-\Bh\cdot\Be_3$ in \eqref{linear1}.
\end{rema}

{\bf Step 1.}
For $\Bg=(g_1,g_2,g_3)$, we set $\wt \Bg=(g_1^o,g_2^o,g_3^e)$ with \eqref{oddeven}.
Then $\di \wt \Bg= (\di \Bg)^o=g^o$.
Let 
$\Bd=\tp(d_1,d_2,d_3)=\tp (d_1(x,t),d_2(x,t),d_3(x,t))$ with
\begin{equation*}
	d_j (x,t)= -\CF_{\xi}^{-1}\left[\frac{i\xi_j}{|\xi|^2}\CF[g^o(\cdot, t)](\xi)\right](x) \quad (j=1,2,3).
\end{equation*}
Note that 
\begin{equation*}
	\|g(t)\|_{\wh W_q^{-1}(\lhs)}\leq \|\Bg(t)\|_{L_q(\lhs)}, \quad
	\|\pa_t g(t)\|_{\wh W_q^{-1}(\lhs)}\leq \|\pa_t \Bg(t)\|_{L_q(\lhs)}
\end{equation*}
for $t>0$.
Thus, as was seen in \cite[Lemma 4.1 and its proof]{SS12},
$\Bd$ solves the divergence equation:
\begin{equation}\label{eq:div}
	\di \Bd =g=\di\Bg\quad
	\text{in $\BR_-^3$, $t>0$},
\end{equation}
and satisfies 
$\Bd|_{t=0}=0\text{ in $\BR_-^3$}$ and the following estimates:
\begin{alignat*}{2}
	\|\pa_t \Bd(t)\|_{L_r(\lhs)} 
		&\leq C_r\|\pa_t \Bg(t)\|_{L_r(\lhs)}, \quad 
	&\|\Bd(t)\|_{L_r(\lhs)}
		&\leq C_r \|\Bg(t)\|_{L_r(\lhs)}, \\
	\|\nabla \Bd(t)\|_{L_r(\lhs)}
		&\leq C_r\|g(t)\|_{L_s(\lhs)}, 
	&\|\nabla^2\Bd(t)\|_{L_r(\lhs)}
		&\leq C_r \|\nabla g(t)\|_{L_r(\lhs)}, \quad 
\end{alignat*}
with some positive constant $C_r$ independent of $\Bd$, $\Bg$, $g$, and $t$.
These inequalities tell us that the following estimates hold:
First, since it holds that
\begin{equation*}
	\|\Bd\|_{L_\infty^{\Fm(\bar{q},r)}(\BR_+,L_r(\lhs))}\leq \|\Bg\|_{\BBA_1}, \quad
	\|\nabla\Bd\|_{L_\infty^{\Fn(\bar{q},r)+1/8}(\BR_+,L_r(\lhs))} \leq \|g\|_{\BBA_2},
\end{equation*}
we have, for $\Bx = (\Bd,0,0,0)$, 
\begin{align}\label{141109_1}
	&\CN_{q,p}(\Bx;\Fa_1,\Fa_2)
	+\sum_{r\in\{q,2\}}
		\left(\CM_{r,p}(\Bx)+\CN_q(\Bx)\right) \\
		&\leq
			C_{p,q}\left(\|\Bg\|_{\CG\cap \BBA_1}
			+\|g\|_{\BBG\cap\BBA_2}+\|(\pa_t\Bg,\nabla g)\|_{L_p^{\Fa_1}(\BR_+,L_q(\lhs))}\right);\notag
\end{align}
Secondly, noting that $\|\Bd\|_{\BBA_4} \leq C_q\|g\|_{\BBA_2}$ and $\|\Bd\|_{\BBB_1}\leq C_q(\|\Bg\|_{\BBA_1}+\|g\|_{\BBA_2})$, we have
\begin{align}
 	\|(\pa_t\Bd,\nabla^2\Bd)\|_{\BBF_3\cap \wt\BBF_3(\Fa_0,\Fb_3)} 
		&\leq C_{p,q}\left(\|\Bg\|_{\CG\cap\wt\CG(\Fa_0,\Fb_3)}+\|g\|_{\BBG\cap\wt\BBG(\Fa_0,\Fb_3)}\right),
		\label{est:7} \\
	\|\Bd\|_{\BBK\cap\wt\BBK(\Fa_0,\Fb_4)\cap \BBA_4\cap\BBB_1}
		&\leq C_{p,q}\left(\|\Bg\|_{\CG\cap\wt\CG(\Fa_0,\Fb_4)\cap \BBA_1}+\|g\|_{\BBG\cap\wt\BBG(\Fa_0,\Fb_4)\cap\BBA_2}\right),
		\label{est:8} \\
 	\|\nabla\Bd\|_{L_p^{\Fc_1}(\BR_+,W_q^1(\lhs))}
 		&\leq C_{p,q}\|g\|_{L_p^{\Fc_1}(\BR_+,W_q^1(\lhs))}, 
		\label{est:5} \\
 	\|\nabla\Bd\|_{L_p^{\Fd_1}(\BR_+,W_{\bar{q}}^1(\lhs))}
 		&\leq C_{p,q}\|g\|_{L_p^{\Fd_1}(\BR_+,W_{\bar{q}}^1(\lhs))}, 
		\label{est:6}  \\
	\|\nabla\Bd\|_{\BBA_3}&\leq C_{p,q} \|g\|_{\BBA_3};
		\label{est:1} 
\end{align}
Thirdly, it holds that
\begin{align}
	&\|\nabla\Bd\|_{H_{r,p}^{1,1/2}(\lhs\times\BR_+)}
		\leq C_{p,r}\left(\|\Bg\|_{\CG}+\|g\|_{\BBG}\right), \label{est:2} \\
	&\|(t+2)^{\Fa_0}\nabla\Bd\|_{H_{q,p}^{1,1/2}(\BR_-^3\times\BR_+)} \label{est:3} 
		 \leq C_{p,q}\left(\|\Bg\|_{W_p^{1,\Fa_0}(\BR_+,L_q(\lhs))}+\|g\|_{L_p^{\Fa_0}(\BR_+,W_q^1(\lhs))}\right),  \\
	&\|(t+2)^{\Fb_j}\nabla\Bd\|_{H_{\bar{q},p}^{1,1/2}(\lhs\times\BR_+)} \label{est:4} 
		 \leq C_{p,\bar{q}}\left(\|\Bg\|_{W_p^{1,\Fb_j}(\BR_+,L_{\bar{q}}(\lhs))}+\|g\|_{L_p^{\Fb_j}(\BR_+,W_{\bar{q}}^1(\lhs))}\right),
		\quad j=3,4.  
\end{align}
In fact, by Lemma \ref{lemm:embed} (1),
\begin{align*}
	&\|\nabla\Bd\|_{H_{r,p}^{1,1/2}(\lhs\times\BR_+)}
		=\|\nabla\Bd\|_{H_p^{1/2}(\BR_+,L_r(\lhs))} + \|\nabla\Bd\|_{L_p(\BR_+,W_r^1(\lhs))} \\
	&\leq \|\Bd\|_{H_p^{1/2}(\BR_+,W_r^1(\lhs))} + \|\Bd\|_{L_p(\BR_+,W_r^2(\lhs))} \\
	&\leq C_{p,r}\|\Bd\|_{W_{r,p}^{2,1}(\BR_-^3\times\BR_+)}
		\leq C_{p,r}\|(\pa_t\Bg,\Bg,g,\nabla g)\|_{L_p(\BR_+,L_r(\lhs))} ,
\end{align*}
and also \eqref{est:3} and \eqref{est:4} can be proved similarly.

%
%
%

Next, we prove that
\begin{equation}\label{161010_13}
	\|\CE(d_3|_{\bdry})\|_{L_\infty(\BR_+,L_2(\lhs))}
		\leq C\|\Bg\|_{L_\infty(\BR_+,L_2(\lhs))}.
\end{equation}
Since
\begin{equation*}
	\CF[g^o(\cdot, t)](\xi) = \int_{-\infty}^0\left(-e^{iy_N\xi_N}+e^{-iy_N\xi_N} \right)\wh g(\xi',y_N,t)\intd y_N,
\end{equation*}
we see that
\begin{equation*}
	\wh d_3 (\xi',x_3,t) 
		 =-\int_{-\infty}^0 \wh g(\xi',y_N,t)
		\left\{\frac{1}{2\pi}\int_{-\infty}^\infty\frac{i\xi_3}{|\xi|^2}\left(-e^{i(x_3+y_3)\xi_3}+e^{i(x_3-y_3)\xi_3}\right)\intd\xi_3\right\}\intd y_3.
\end{equation*}
Combining this formula with
\begin{equation*}
	\frac{1}{2\pi}\int_{-\infty}^{\infty} \frac{i\xi_3e^{i a\xi_3}}{|\xi|^2}\intd\xi_3
		= -{\rm sign}(a)\frac{e^{-|a||\xi'|}}{2} \quad (a\in\BR\setminus\{0\}),
\end{equation*}
which can be proved by the residue theorem,
yields that
\begin{equation}\label{161010_15}
	\CE(d_3|_{\bdry}) 
		= -\int_{-\infty}^0\CF_{\xi'}^{-1}\left[e^{(x_3+y_3)|\xi'|}\wh g(\xi',y_3,t)\right](x')\intd y_3 \quad (x_3 < 0).
\end{equation}
Since $g=\di\Bg$ in $\lhs$ with $\Bg=\tp (g_1(x,t),g_2(x,t),g_3(x,t))$, we have
\begin{equation*}
	\wh g(\xi',y_3,t) = \sum_{j=1}^2i\xi_j\wh g_j(\xi',y_3,t) + \pa_3 \wh g_3(\xi',y_3,t).
\end{equation*}
We insert this identity into \eqref{161010_15} and use the integration by parts
in order to obtain
\begin{align*}
	\CE(d_3|_{\bdry}) &= 
-\int_{-\infty}^0\CF_{\xi'}^{-1}\left[i\xi_j e^{(x_3+y_3)|\xi'|}\wh g_j(\xi',y_3,t)\right](x')\intd y_3 \\
	&-\int_{-\infty}^0\CF_{\xi'}^{-1}\left[|\xi'| e^{(x_3+y_3)|\xi'|}\wh g_3(\xi',y_3,t)\right](x')\intd y_3,
\end{align*}
which, combined e.g. with \cite[Lemma 5.4]{SS12}, furnishes that
\begin{equation*}
	\|\CE(d_3|_{\bdry})\|_{L_2(\lhs)} \leq C\|\Bg(t)\|_{L_2(\lhs)}.
\end{equation*}
This estimate implies \eqref{161010_13} immediately.

{\bf Step 2.}
Noting that $\Bd|_{t=0}=0$ in $\lhs$ as was discussed in Step 1,
we set $\Bu = \Bd+\wt \Bu$ in \eqref{linear1} in order to obtain
\begin{equation}\label{redueq:1}
	\left\{\begin{aligned}
		\pa_t \wt\Bu -\Di\BT(\wt\Bu,\Fp) = \Bf-\pa_t\Bd+\Di(\mu\BD(\Bd)) & && \text{in $\BR_-^3$, $t>0$,} \\
		\di\wt\Bu =0& && \text{in $\BR_-^3$, $t>0$,} \\
		\BT(\wt\Bu,\Fp)\Be_3 +(c_g -c_\si\De')h\Be_3 = \Bh-\mu\BD(\Bd)\Be_3& && \text{on $\BR_0^3$, $t>0$,} \\
		\pa_t h - \wt u_3 = k+d_3& && \text{on $\BR_0^3$, $t>0$,} \\
		\wt\Bu|_{t=0} = 0& && \text{in $\BR_-^3$,} \\
		h|_{t=0} = 0 & && \text{on $\BR^2$,}
\end{aligned}\right.
\end{equation}
where $\BT=(\wt\Bu,\Fp)=\mu\BD(\wt\Bu)-\Fp\BI$ and $\wt\Bu=\tp(\wt u_1,\wt u_2,\wt u_3)$.
To estimate the solution $(\wt\Bu,\Fp,h)$, we decompose $\wt\Bu$ and $\Fp$ as follows:
	$\wt\Bu= \Bv+ \Bw$, $\Fp=\Fq+\Fr$
that satisfy 
\begin{align}
	&\left\{\begin{aligned}\label{redueq:2}
		\pa_t \Bv +\Bv-\Di\BT(\Bv,\Fq) &= 0 && \text{in $\BR_-^3$, $t>0$,} \\
		\di\Bv &= 0 && \text{in $\BR_-^3$, $t>0$,} \\
		\BT(\Bv,\Fq)\Be_3 &= \BH && \text{on $\BR_0^3$, $t>0$,} \\
		\Bv|_{t=0} &= 0 && \text{in $\BR_-^3$,} \\
	\end{aligned}\right. \\
	&\left\{\begin{aligned}\label{redueq:3}
		\pa_t \Bw -\Di\BT(\Bw,\Fr) &= \BF && \text{in $\BR_-^3$, $t>0$,} \\
		\di\Bw &=0  && \text{in $\BR_-^3$, $t>0$,} \\
		\BT(\Bw,\Fr)\Be_3 +(c_g -c_\si\De')h\Be_3 &= 0 && \text{on $\BR_0^3$, $t>0$,} \\
		\pa_t h -  w_3 &= K && \text{on $\BR_0^3$, $t>0$,} \\
		\Bw|_{t=0} &= 0 && \text{in $\BR_-^3$,} \\
		h|_{t=0} &= 0 && \text{on $\BR^2$,}
\end{aligned}\right.
\end{align}
where we have set
\begin{align*}
	&\BF = \BF_1 + \BF_2 + \BF_3, \quad 
	\BF_i = \Bf_i \quad (i=1,2), \quad 
	\BF_3 = \Bf_3-\pa_t\Bd+\Di(\mu\BD(\Bd))+\Bv, \\
	&\BH = \Bh-\mu\BD(\Bd)\Be_3, \quad 
	K = k+d_3+v_3.
\end{align*}

In the following steps, we estimate $(\Bv,\Fq)$ and $(\Bw,\Fr,h)$
by using Theorem \ref{theo:sec3_2} and Theorem \ref{theo:sec4_2}, respectively.

{\bf Step 3.}
By Theorem \ref{theo:sec3_1} and Theorem \ref{theo:sec3_2},
we observe that
\begin{align*}
	&\|(\pa_t \Bv,\Bv,\nabla\Bv,\nabla^2\Bv,\nabla\Fq)\|_{L_p(\BR_+,L_r(\lhs))}
		 \leq C_{p,r}\|\BH\|_{H_{r,p}^{1,1/2}(\BR_-^3\times\BR_+)};  \\ 
	&\|(\pa_t \Bv,\Bv,\nabla\Bv,\nabla^2\Bv,\nabla\Fq)\|_{L_p^{\Fa_0}(\BR_+,L_q(\lhs))} \\ 
		&\quad \leq C_{p,q}\left(\|\BH\|_{L_p^{\Fc_1}(\BR_+,W_q^1(\lhs))}
		+\|(t+2)^{\Fa_0}\BH\|_{H_{q,p}^{1,1/2}(\BR_-^3\times\BR_+)}\right); \\  
	&\|(\pa_t \Bv,\Bv,\nabla\Bv,\nabla^2\Bv,\nabla\Fq)\|_{L_p^{\Fb_j}(\BR_+,L_{\bar{q}}(\lhs))} \\ 
		&\quad  \leq C_{p,q}\left(\|\BH\|_{L_p^{\Fd_1}(\BR_+,W_{\bar{q}}^1(\lhs))}
		+\|(t+2)^{\Fb_j}\BH\|_{H_{\bar{q},p}^{1,1/2}(\BR_-^3\times\BR_+)}\right),  \quad j=3,4; \\ 
	&\|\Bv\|_{L_\infty^{\Fn(\bar{q},r)+1/8}(\BR_+,W_r^1(\lhs))}+\|\Bv\|_{\BBA_4\cap\BBB_1}
		\leq C_{p,q}\|\BH\|_{\BBA_3}.
\end{align*}
Combining these inequalities with \eqref{est:5}-\eqref{est:4} furnishes that,
for $\By = (\Bv,\Fq,0,0)$,
\begin{align}
	&\CN_{q,p}(\By;\Fa_1,\Fa_2) + \sum_{r\in\{q,2\}}\left(\CM_{r,p}(\By)+\CN_r(\By)\right) 
		\leq C_{p,q}\SSN; \label{161220_13} \\
	&\|\Bv\|_{\BBF_3\cap \wt\BBF_3(\Fa_0,\Fb_3)\cap \BBK\cap\wt\BBK(\Fa_0,\Fb_4)\cap\BBA_4\cap\BBB_1}
		\leq C_{p,q}\SSN. \label{161220_15}
\end{align}
%
%
%
%
Analogously, it holds by Theorem \ref{theo:sec3_1} \eqref{theo:sec3_1_2}  that
\begin{equation*}
	\|\CE(v_3|_{\BR_0^3})\|_{W_{2,p}^{2,1}(\BR_-^3\times\BR_+)}
		\leq C_p\|\BH\|_{H_{2,p}^{1,1/2}(\BR_-^3\times\BR_+)}
		\leq C_{p,q}\SSN,
\end{equation*}
which, combined with Lemma \ref{lemm:embed} \eqref{embed:1}, furnishes that
\begin{equation}\label{161027_2}
	\|\CE(v_3|_{\BR_0^3})\|_{L_\infty(\BR_+,L_2(\lhs))}
		\leq C_{p,q}\SSN.
\end{equation}

{\bf Step 4.}
The aim of this step is to prove the following estimates:
For $\Bz=(\Bw,\Fr,h,\CE(h))$,
\begin{equation}\label{161010_7}
	\CN_{q,p}(\Bz;\Fa_1,\Fa_2)+\sum_{r\in\{q,2\}}\left(\CM_{r,p}(\Bz)+\CN_r(\Bz)\right) 
	\leq C_{p,q}\SSN.
\end{equation}

By Theorem \ref{theo:sec4_2}, we have
\begin{align*}
	&\CN_{q,p}(\Bz;\Fa_1,\Fa_2)+\sum_{r\in\{q,2\}}\left(\CM_{r,p}(\Bz)+\CN_r(\Bz)\right) \\
	&\leq C_{p,q,\bar{q}}\bigg(\sum_{i=1}^3\|\BF_i\|_{\BBF_i\cap \BBF_i(\Fa_0,\Fb_i)}
		+\|\BK\|_{\BBK\cap\BBK(\Fa_0,\Fb_4)\cap \BBA_4\cap \BBB_1}\bigg) \notag \\
	&\leq C_{p,q,\bar{q}}\bigg(\sum_{i=1}^3\|\Bf_i\|_{\BBF_i\cap\wt\BBF_i(\Fa_0,\Fb_i)}
	+\|k\|_{\BBK\cap\wt\BBK(\Fa_0,\Fb_4)\cap\BBA_4\cap\BBB_1} \notag \\
	&\quad
	+\|(\pa_t\Bd,\nabla^2\Bd,\Bv)\|_{\BBF_3\cap\wt\BBF_3(\Fa_0,\Fb_3)}
	+\|(d_3,v_3)\|_{\BBK\cap\wt\BBK(\Fa_0,\Fb_4)\cap\BBA_4\cap\BBB_1}\bigg).
	\notag
\end{align*}
Combining this inequality with \eqref{est:7}, \eqref{est:8}, and \eqref{161220_15}
implies \eqref{161010_7}.


{\bf Step 5.}
In this step, we prove that
\begin{equation}\label{161016_11}
	\|\pa_t\CE(h)\|_{L_\infty(\BR_+,L_2(\lhs))} \leq C_{p,q,\te}\SSN.
\end{equation}

By the equation $\pa_t h -w_3 = k +d_3+v_3$ on $\bdry$, we see that
\begin{equation}\label{161016_13}
	\pa_t\CE(h) = \CE(k|_\bdry)+ \CE(d_3|_{\bdry})+\CE(v_3|_{\bdry}) + \CE(w_3|_\bdry) \quad \text{in $\lhs$}.
\end{equation}
In what follows, we estimate the four terms on the right-hand side of \eqref{161016_13} in order to obtain \eqref{161016_11}.

First, we show that
\begin{equation}\label{161013_14}
	\|\CE(k|_\bdry)\|_{L_\infty(\BR_+,L_2(\lhs))}
		\leq C_\te\|k\|_{\BBB_2(\te)}.
\end{equation}
By Perseval's theorem,
\begin{align*}
	&\|\CE(k|_{\BR_0^3})\|_{L_2(\lhs)}^2 
		=\int_{-\infty}^0\|\CF_{\xi'}^{-1}[e^{|\xi'|x_3}\wh k(\xi',0,t)]\|_{L_2(\BR^2)}^2\intd x_3 \\
		&= \int_{-\infty}^0\|e^{|\xi'|x_3}\wh k(\xi',0,t)\|_{L_2(\BR^2)}^2\intd x_3
		=\int_{\BR^2}\left(\int_{-\infty}^0e^{2|\xi'|x_3}\intd x_3\right)|\wh k(\xi',0,t)|^2\intd \xi' \\
		&=\frac{1}{2}\int_{\BR^2}\frac{|\wh k(\xi',0,t)|^2}{|\xi'|} \intd\xi'
		=\frac{1}{2}\left(\int_{|\xi'|\leq 1}+\int_{|\xi'|\geq 1}\right)\frac{|\wh k(\xi',0,t)|^2}{|\xi'|} \intd\xi' \\
		&=:I_1(t)+I_2(t).
\end{align*}
It then holds that, by H\"older's inequality,
\begin{align*}
	I_1(t) \leq \frac{1}{2}\left(\int_{|\xi'|\leq 1}\frac{d\xi'}{|\xi'|^\al}\right)^{1/\al}
		\left(\int_{|\xi'|\leq 1}|\wh k(\xi',0,t)|^{2\al'}\right)^{1/\al'},
\end{align*}
where $\al=2-\te$ and $\al'=\al/(\al-1)$.
Setting $\beta=2\al'$, we observe that
\begin{align*}
	I_1(t) \leq C_\te \|\wh k(\cdot,0,t)\|_{L_\beta(\BR^2)}^2,
\end{align*}
which, combined with the Hausdorff-Young inequality (cf. \cite[Theorem 1.2.1]{BL76}):
\begin{equation*}
	\|\wh f\|_{L_p(\BR^2)} \leq C_p\|f\|_{L_{p'}(\BR^2)} \quad \text{for $2\leq p \leq \infty$,}
\end{equation*}
furnishes that
\begin{equation}\label{161222_7}
	I_1(t)\leq C_\te\|k(t)\|_{L_{q(\te)}(\BR_0^3)}^2.
\end{equation}
On the other hand,
by Perseval's theorem, 
\begin{align*}
	I_2(t) \leq \frac{1}{2}\int_{|\xi'|\geq 1}|\wh k(\xi',0,t)|^2\intd\xi' \leq \frac{1}{2}\|\wh k(\cdot,0,t)\|_{L_2(\BR^2)}^2
	= \frac{(2\pi)^2}{2}\|k(t)\|_{L_2(\BR_0^3)}^2. 
\end{align*}
Combining this inequality with \eqref{161222_7} implies \eqref{161013_14}.
%
%
%
%

Next, we consider $\CE(w_3|_{\BR_0^3})$.
Note that $\CE(w_3|_{\BR_0^3})$ can be written as
\begin{equation*}
	\CE(w_3|_{\BR_0^3}) = \int_0^t \CE((R(t-s)(\BF(s),K(s)))_3|_{\BR_0^3})\intd s.
\end{equation*}
Let $0<t<3$, and then we have, by \eqref{161028_1} and the trace theorem,
\begin{align*}
	&\|\CE(w_3|_{\BR_0^3})\|_{L_2(\lhs)}
		\leq C_{p,\ga_5}\int_0^te^{\ga_5(t-s)}\|(\BF(s),K(s))\|_{L_2(\lhs)\times W_2^2(\lhs)}\intd s \\
		&\leq C_{p,\ga_5}\left(\int_0^t e^{p'\ga_5(t-s)}\intd s\right)^{1/p'}
			\left(\|\BF\|_{L_p(\BR_+,L_2(\lhs))}+\|K\|_{L_p(\BR_+,W_2^2(\lhs))}\right) 
\end{align*}
which, combined with \eqref{141109_1} and \eqref{161220_13}, furnishes that
\begin{equation}\label{161007_12}
	\|\CE(w_3|_{\BR_0^3})\|_{L_\infty((0,3),L_2(\lhs))} 
		\leq C_{p,\ga_5}\SSN. 
\end{equation}
On the other hand, for $t\geq 2$, 
we have, by Lemma \ref{lemm:ap} and Lemma \ref{lemm:low},
\begin{equation*}
	\|\CE(w_3|_{\BR_0^3})\|_{L_\infty((2,\infty),L_2(\lhs))}
		\leq C_{p,q}\left(\sum_{i=1}^3\|\BF_i\|_{\BBF_i(\Fa_0,\Fb_i)}+\|K\|_{\BBK(\Fa_0,\Fb_4)}\right),
\end{equation*}
which, combined with \eqref{161007_12}, furnishes that
\begin{equation*}
	\|\CE(w_3|_{\BR_0^3})\|_{L_\infty(\BR_+,L_2(\lhs))} \leq C_{p,q}\SSN.
\end{equation*}
Combining this inequality with \eqref{161010_13}, \eqref{161027_2}, \eqref{161016_13}, and \eqref{161013_14} implies \eqref{161016_11}.

Summing up \eqref{141109_1}, \eqref{161220_13}, \eqref{161010_7}, and \eqref{161016_11},
we have completed the proof of Theorem \ref{theo:main1}.

\section{Proof of Theorem \ref{theo:main2}}\label{sec6}
In this section, we prove Theorem \ref{theo:main2} 
by using Theorem \ref{theo:main1} and Theorem \ref{theo:sec4_1}.

\subsection{Initial flow}
In this subsection, we  construct an initial flow.

{\bf Step 1.}
Let $(\Bu^*,\Fq^*,h^*)$
be the solution to the following equations: 
\begin{equation}\label{eq:ini1}
	\left\{\begin{aligned}
		\pa_t \Bu^*-\Di\BT(\Bu^*,\Fq^*) &=0
			 && \text{in $\BR_-^3$, $t>0$,} \\
		\di\Bu^*&=0
			&& \text{in $\BR_-^3$, $t>0$,} \\
		\BT(\Bu^*,\Fq^*)\Be_3+(c_g-c_\si\De')h^*\Be_3 &=0
			&& \text{on $\BR_0^3$, $t>0$,} \\
		\pa_t h^* -u_3^* &=0
			&& \text{on $\BR_0^3$, $t>0$,} \\
		\Bu^*|_{t=0}&=0
			&& \text{in $\BR_-^3$,} \\
		h^*|_{t=0}&=h_0
			&& \text{on $\BR^2$.}
	\end{aligned}\right.
\end{equation}
By Theorem \ref{theo:sec4_1}, we have,
for $\Bx*=(\Bu^*,\Fq^*,h^*,\CE(h^*))$,
\begin{equation}\label{160805_1}
	\CN_{q,p}(\Bx^*;\Fa_1,\Fa_2) +
	\sum_{r\in\{q,2\}}\left(\CM_{r,p}(\Bx^*)+\CN_r(\Bx^*)\right) 
	+\|\pa_t\CE(h^*)\|_{L_\infty(\BR_+,L_2(\lhs))} 
	\leq C_{p,q}\|h_0\|_{\BBI_2} 
\end{equation}
for a positive constant $C_{p,q}$.


{\bf Step 2.}
Let $\wt\Bv_0$ be an extension of $\Bv_0$ satisfying
$\wt\Bv_0=\Bv_0$ in $\BR_-^3$ and 
\begin{equation}\label{141110_21}
	\|\wt\Bv_0\|_{B_{s,p}^{2-2/p}(\BR^3)} \leq
		C_{p,q,s}\|\Bv_0\|_{B_{s,p}^{2-2/p}(\BR_-^3)}
\end{equation}
with a positive constant $C_{p,q,s}$,
where $q(\te) \leq s \leq q$, here and hereafter.
%
%
%
%
Let $\Bw^*$ be the solution to
\begin{equation}\label{eq3_141024}
	\left\{\begin{aligned}
		\pa_t\Bw^*+\Bw^*-\mu\De\Bw^*&=0
			&& \text{in $\BR^3$, $t>0,$} \\
		\Bw^*|_{t=0}&=\wt\Bv_0 && \text{in $\BR^3$}.
	\end{aligned}\right.
\end{equation}
It then holds by \eqref{141110_21} that 
\begin{align}
	&\|e^{\ga_6 t}(\pa_t\Bw^*,\Bw^*,\nabla\Bw^*,\nabla^2\Bw^*)\|_{L_p(\BR_+,L_s(\lhs))}
		\leq C_{p,s}\|\Bv_0\|_{\BBI_1(\te)}, \label{161221_1} \\
	&\|\Bw^*(t)\|_{W_s^1(\lhs)}
		\leq C_se^{-\ga_ 7 t}\|\Bv_0\|_{\BBI_1(\te)} \quad (t\geq 1) \label{161221_2}
\end{align}
for some positive constants $\ga_6$, $\ga_7$, $C_{p,s}$, and $C_s$.
By Lemma \ref{lemm:embed} \eqref{embed:1} and \eqref{161221_1},
\begin{equation*}
	\|\Bw^*\|_{L_\infty((0,2),W_s^1(\lhs))} \leq C_{p,s}\|\Bw^*\|_{W_{s,p}^{2,1}(\lhs\times(0,2))}
		\leq C_{p,s}\|\Bv_0\|_{\BBI_1(\te)},
\end{equation*}
which, combined with \eqref{161221_2}, furnishes that
\begin{equation}\label{170123_11}
	\|e^{\ga_7 t}\Bw^*\|_{L_\infty(\BR_+,W_s^1(\lhs))} \leq C_{p,s}\|\Bv_0\|_{\BBI_1(\te)}.
\end{equation}
Thus, combining this inequality with \eqref{161221_1} implies that, for $\By^*=(\Bw^*,0,0,0)$,
\begin{equation}\label{160805_2}
	\CN_{q,p}(\By^*;\Fa_1,\Fa_2) +
	\sum_{r\in\{q,2\}}\left(\CM_{r,p}(\By^*)+\CN_r(\By^*)\right)
	\leq C_{p,q}\|\Bv_0\|_{\BBI_1(\te)}
\end{equation}
with a positive constant $C_{p,q}$.

{\bf Step 3.}
Let $\Bv^*=\Bu^*+\Bw^*$ and
$\Bz^*=\Bx^*+\By^*=(\Bv^*,\Fq^*,h^*,\CE(h^*))$.
We call $\Bz^*$ the initial flow in this paper,
and have, by \eqref{160805_1} and \eqref{160805_2},
\begin{equation}\label{160805_3}
	\CN_{q,p}(\Bz^*;\Fa_1,\Fa_2) +
	\sum_{r\in \{q,2\}}\left(\CM_{r,p}(\Bz^*)+\CN_r(\Bz^*)\right) + \|\pa_t\CE(h^*)\|_{L_\infty(\BR_+,L_2(\lhs))} 
	\leq C_{p,q}\|(\Bv_0,h_0)\|_{\BBI_1(\te)\times\BBI_2} 
\end{equation}
for a positive constant $C_{p,q}$.
In addition, setting in the equations \eqref{NS3_1}-\eqref{NS3_5}
\begin{equation*}
	\Bv=\bar{\Bv}+\Bv^*, \quad  \Fq = \bar{\Fq}+\Fq^*, \quad h=\bar{h}+h^*
\end{equation*}
and denoting $(\bar\Bv,\bar\Fq,\bar h)$ of the resultant equations by $(\Bv,\Fq,h)$ again,
we achieve
\begin{equation}\label{eq:comp}
	\left\{\begin{aligned}
		\pa_t\Bv -\Di\BT(\Bv,\Fq) =\SSF(\Bv,\CE(h))&
			&& \text{in $\BR_-^3$, $t>0$,}  \\
		\di\Bv = \SSG(\Bv,\CE(h)) = \di\Fg(\Bv,\CE(h))&
			&& \text{in $\BR_-^3$, $t>0$,} \\
		\BT(\Bv,\Fq)\Be_3-(c_g-c_\si\De')h\Be_3 =\SSH(\Bv,\CE(h))&
			&& \text{on $\BR_0^3$, $t>0$,} \\
		\pa_t h - v_3 =\SSK(\Bv,\CE(h))&
			&& \text{on $\BR_0^3$, $t>0$,} \\
		\Bv|_{t=0} =0&
 			&& \text{in $\BR_-^3$,} \\
		h|_{t=0} =0 &
			&& \text{on $\BR^2$,}
	\end{aligned}\right.
\end{equation}  
where we have set
\begin{align*}
	\SSF(\Bv,\CE(h)) &=
		\BF(\Bv+\Bv^*,\CE(h)+\CE(h^*)) +\Bw^*+\mu\nabla\di\Bw^*, \\
	\Fg(\Bv,\CE(h)) &=
		\BG(\Bv+\Bv^*,\CE(h)+\CE(h^*))-\Bw^*, \\
	\SSG(\Bv,\CE(h)) &=
		G(\Bv+\Bv^*,\CE(h)+\CE(h^*))-\di\Bw^*, \\
	\SSH(\Bv,\CE(h)) &=
		\BH(\Bv+\Bv^*,\CE(h)+\CE(h^*))-\mu\BD(\Bw^*)\Be_3, \\
	\SSK(\Bv,\CE(h)) &=
		K(\Bv+\Bv^*,\CE(h)+\CE(h^*))+ w_3^*.
\end{align*}


\subsection{Construction of solutions to \eqref{eq:comp}.}
Since it suffices to construct solutions to System \eqref{eq:comp}
in order to prove Theorem \ref{theo:main2},
we deal with System \eqref{eq:comp} in this subsection.
Let $\eta=\CE(h)$, and then $\eta$ is a unique solution to
\begin{equation}\label{eq:auxi}
	\left\{\begin{aligned}
		\De\eta&=0 && \text{in $\BR_-^3$, $t>0$,} \\
		\eta &= h && \text{on $\BR_0^3$, $t>0$.}
	\end{aligned}\right.
\end{equation}
From this viewpoint,
instead of System \eqref{eq:comp},
we solve 
\begin{equation}
	\left\{\begin{aligned}
		\pa_t\Bv -\Di\BT(\Bv,\Fq) =\SSF(\Bv,\eta)&
			&& \text{in $\BR_-^3$, $t>0$,}  \\
		\di\Bv = \SSG(\Bv,\eta) = \di\Fg(\Bv,\eta)&
			&& \text{in $\BR_-^3$, $t>0$,} \\
		\BT(\Bv,\Fq)\Be_3-(c_g-c_\si\De')h\Be_3 =\SSH(\Bv,\eta)&
			&& \text{on $\BR_0^3$, $t>0$,} \\
		\pa_t h - v_3 =\SSK(\Bv,\eta)&
			&& \text{on $\BR_0^3$, $t>0$,} \\
		\Bv|_{t=0} =0&
 			&& \text{in $\BR_-^3$,} \\
		h|_{t=0} =0 &
			&& \text{on $\BR^2$,}
	\end{aligned}\right.
\end{equation}
together with the auxiliary problem \eqref{eq:auxi} in what follows.

First, we introduce an underlying space ${}_0X_{q,p}(r;\Fa_1,\Fa_2)$
that is used for the contraction mapping principle.
Let $\|\cdot\|_{X_{q,p}(\Fa_1,\Fa_2)}$ be defined as \eqref{160920_1},
and then
\begin{equation*}
	X_{q,p}(\Fa_1,\Fa_2) := \{\Bz \mid \|\Bz\|_{X_{q,p}(\Fa_1,\Fa_2)}<\infty\}.
\end{equation*}
In addition, we set
\begin{align*}
	&{}_0X_{q,p}(\Fa_1,\Fa_2) 
		= \{(\Bv,\Fq,h,\eta) \in X_{q,p}(\Fa_1,\Fa_2) \mid
		 \text{$\Bv|_{t=0}=0$ in $\BR_-^3$, $\eta|_{t=0}=0$ in $\BR_-^3$} \}, \\
	&{}_0X_{q,p}(r;\Fa_1,\Fa_2)
		=\{\Bz =(\Bv,\Fq,h,\eta) \in {}_0X_{q,p}(\Fa_1,\Fa_2) \mid 
		\|\Bz\|_{X_{q,p}(\Fa_1,\Fa_2)}\leq r\} \quad (r>0).
\end{align*}

Next, we apply Theorem \ref{theo:main1} with
\begin{align}\label{141225_5}
	&\Fa_1=\frac{1}{2}, \quad \Fa_2=\frac{3}{4}, \quad \Fc_1=0, \quad \Fd_1=\frac{1}{4}, \\
	&\Fb_1=\Fb_2=\Fm\left(\bar{q},q\right)+\Fn\left(\bar{q},q\right)+\frac{1}{8}=\frac{2}{q}+\frac{3}{8}, \quad
	\Fb_3=\Fb_4=1. \notag
\end{align}
We then note the following relations: First, the assumption $3<q<16/5$ implies that $\Fb_1=\Fb_2>1$.
Secondly, since $5/16<1/q<1/3$ from the second condition of \eqref{pq},
it holds by the third condition of \eqref{pq} that
\begin{equation*}
\frac{2}{p}<1-\frac{3}{q}<1-3\cdot \frac{5}{16}.
\end{equation*}
This gives us $p>32$,
and we collect this inequality and several conditions used in the following argumentation as follows:
\begin{align}\label{141225_3}
	&p>32, \quad p\left(1+\Fc_1-\frac{3}{4}\right)=p\left(1+\Fd_1-1\right)=\frac{p}{4}>1, \\
	&p(\min\{\Fb_1,\Fb_2,\Fb_3,\Fb_4\}-\Fa_1) =\frac{p}{2}>1, \notag \\
	&p(\min\{\Fb_1,\Fb_2,\Fb_3,\Fb_4\}-\Fa_2) =\frac{p}{4}>1, \notag \\
	&p\left(\Fm(\bar{q},q)+\frac{1}{4}-\Fa_1\right) = \frac{p}{2q}>1, \quad
	p\left(\frac{1}{2}+\frac{2}{q}-\Fa_2\right)>\frac{p}{8}>1.\notag
\end{align}
Thirdly, by Lemma \ref{lemm:embed} and \eqref{pq}, 
\begin{align}\label{141224_1}
	\|(\Bv,\nabla\eta)\|_{L_\infty(\BR_+,W_\infty^1(\BR_-^3))} &\leq M\|\Bz\|_{X_{q,p}(1/2,3/4)}, \\
	\|(\Bv,\nabla\eta)\|_{L_\infty(\BR_+,W_q^1(\BR_-^3))} &\leq M\|\Bz\|_{X_{q,p}(1/2,3/4)}, \notag \\
	\|(\Bv,\nabla\eta)\|_{L_\infty(\BR_+,W_2^1(\BR_-^3))} &\leq M\|\Bz\|_{X_{q,p}(1/2,3/4)} \notag
\end{align}
for some positive constant $M$, depending only on $p$ and $q$, 
and for $\Bz=(\Bv,\Fq,h,\eta)\in X_{q,p}(1/2,3/4)$. 

\begin{rema}
The compatibility condition \eqref{comp:1} implies that
\begin{equation*}
	\SSG(\Bv,\eta)|_{t=0}=0 \quad \text{in $\lhs$,} \quad	[\SSH(\Bv,\eta)]_\tau|_{t=0}=0 \quad \text{on $\BR_0^3$}
\end{equation*}
for $\Bz = (\Bv,\Fq,h,\eta)\in {}_0 X_{q,p}(1/2,3/4)$.
\end{rema}

From now on, to use Theorem \ref{theo:main1} under the condition \eqref{141225_5},
we show that there exists a positive number $\ep_1\in(0,1/4)$ such that,
for any $\bar\Bz=(\bar\Bv,\bar\Fq,\bar h,\bar\eta)\in {}_0X_{q,p}(\ep_1;1/2,3/4)$,
\begin{align}\label{141224_10}
	&\BF_i(\bar\Bv+\Bv^*,\bar\eta+\CE(h^*)) \in \BBF_i\cap \widetilde{\BBF}_i(3/4,2/q+3/8) \quad (i=1,2), \\
	&\BF_3(\bar\Bv+\Bv^*,\bar\eta+\CE(h^*))+\Bw^*+\mu\nabla\di\Bw^* \in \BBF_3\cap \wt\BBF_3(3/4,1), \notag \\
	&\Fg(\bar\Bv,\bar\eta) \in \CG \cap \wt{\CG}(3/4,1)\cap\BBA_1, \quad
	\SSG(\bar\Bv,\bar\eta) \in \BBG \cap \wt{\BBG}(3/4,1) \cap \BBA_2 \cap \BBA_3, \notag \\
	&\SSH(\bar\Bv,\bar\eta) \in \BBH \cap \widetilde{\BBH}(3/4,1)\cap \BBA_3, \quad
	\SSK(\bar\Bv,\bar\eta) \in \BBK \cap \widetilde{\BBK}(3/4,1)\cap\BBA_4\cap\BBB_1\cap\BBB_2(\te),\notag \\
	&\SSG(\bar\Bv,\bar\eta)\in L_p(\BR_+,W_q^1(\BR_-^3)) \cap L_p^{1/4}(\BR_+,W_{\bar{q}}^1(\BR_-^3)), \notag \\
	&\SSH(\bar\Bv,\bar\eta)\in L_p(\BR_+,W_q^1(\BR_-^3))^3 \cap L_p^{1/4}(\BR_+,W_{\bar{q}}^1(\BR_-^3))^3, \notag 
\end{align}
with the inequality:
\begin{align}\label{141224_11}
	&\sum_{i=1}^2\|\BF_i(\bar\Bv+\Bv^*,\bar\eta+\CE(h^*))\|_{\BBF_i \cap \wt{\BBF}_i(3/4,2/q+3/8)} 
		+\|(\BF_3(\bar\Bv+\Bv^*,\bar\eta+\CE(h^*)),\Bw^*,\nabla\di\Bw^*)\|_{\BBF_3 \cap \widetilde{\BBF}_3(3/4,1)} \\
		&+ \|(\Fg(\bar\Bv,\bar\eta),\Bw^*)\|_{\CG \cap \wt\CG(3/4,1)\cap\BBA_1}
		+ \|(\SSG(\bar\Bv,\bar\eta),\di\Bw^*)\|_{\BBG \cap \wt{\BBG}(3/4,1) \cap \BBA_2 \cap \BBA_3}
		 \notag \\
		&+ \|(\SSH(\bar\Bv,\bar\eta),\BD(\Bw^*))\|_{\BBH \cap \wt{\BBH}(3/4,1)\cap \BBA_3}
 		+ \|(\SSK(\bar\Bv,\bar\eta),w_3^*)\|_{\BBK \cap \wt{\BBK}(3/4,1)\cap\BBA_4\cap\BBB_1\cap\BBB_2(\te)} \notag  \\
		&+\|(\SSG(\bar\Bv,\bar\eta),\di\Bw^*)\|_{L_p(\BR_+,W_q^1(\BR_-^3)) \cap L_p^{1/4}(\BR_+,W_{\bar{q}}^1(\BR_-^3))} \notag \\
		&  +\|(\SSH(\bar\Bv,\bar\eta),\BD(\Bw^*))\|_{L_p(\BR_+,W_q^1(\BR_-^3)) \cap L_p^{1/4}(\BR_+,W_{\bar{q}}^1(\BR_-^3))} \notag \\
	&\leq C_{p,q,\te}\left(\|(\Bv_0,h_0)\|_{\BBI_1(\te)\times\BBI_2}+\|\bar{\Bz}\|_{X_{q,p}(1/2,3/4)}^2\right) \notag
\end{align}
for a positive constant $C_{p,q,\te}$. 
Let $\Bz=(\Bv,\Fq,h,\eta)=(\bar{\Bv}+\Bv^*,\bar\Fq+\Fq^*,\bar h +h^*,\bar\eta+\eta^*)$
in the following proof.
To obtain \eqref{141224_10} and \eqref{141224_11}, it is enough to show
\begin{align}\label{160902_10}
	&\sum_{i=1}^2\|\BF_i(\Bv,\eta)\|_{\BBF_i \cap \widetilde{\BBF}_i(3/4,2/q+3/8)}+\|\BF_3(\Bv,\eta)\|_{\BBF_3 \cap \wt{\BBF}_3(3/4,1)} \\
		& \quad + \|\Fg(\Bv,\eta)\|_{\CG \cap \wt{\CG}(3/4,1)\cap\BBA_1}
		+ \|\SSG(\Bv,\eta)\|_{\BBG \cap \wt{\BBG}(3/4,1) \cap \BBA_2 \cap \BBA_3} \notag \\
		&\quad + \|\SSH(\Bv,\eta)\|_{\BBH \cap \wt{\BBH}(3/4,1)\cap \BBA_3}
		+ \|\SSK(\Bv,\eta)\|_{\BBK \cap \wt{\BBK}(3/4,1)\cap\BBA_4\cap\BBB_1\cap\BBB_2(\te)} \notag \\
		&\quad + \|\SSG(\Bv,\eta)\|_{L_p(\BR_+,W_q^1(\BR_-^3)) \cap L_p^{1/4}(\BR_+,W_{\bar{q}}^1(\BR_-^3))} \notag \\
		&\quad + \|\SSH(\Bv,\eta)\|_{L_p(\BR_+,W_q^1(\BR_-^3)^3) \cap L_p^{1/4}(\BR_+,W_{\bar{q}}^1(\BR_-^3))} \notag \\
	&\leq C_{p,q,\te}\|\Bz\|_{X_{q,p}(1/2,3/4)}^2. \notag	
\end{align} 
In fact, if we have \eqref{160902_10}, then \eqref{161221_1} and \eqref{170123_11} imply that \eqref{141224_10} and \eqref{141224_11} hold.
From this viewpoint, we prove the inequality \eqref{160902_10} in what follows.

{\bf Step 1: Estimates of $\BF_1(\Bv,\eta)$.}
By \eqref{141224_1}, it is clear that for $r\in\{q,2\}$
\begin{equation}\label{141224_16}
	\|(\Bv\cdot\nabla)\Bv\|_{L_p(\BR_+,L_r(\BR_-^3))}
		\leq \|\Bv\|_{L_\infty(\BR_+,L_\infty(\BR_-^3))}\|\nabla\Bv\|_{L_p(\BR_+,L_r(\BR_-^3))} 
		\leq C_{p,q}\|\Bz\|_{X_{q,p}(1/2,3/4)}^2. 
\end{equation}
Furthermore, by Sobolev's embedding theorem and H\"older's inequality,
\begin{align*}
	&\|(\Bv(t)\cdot\nabla)\Bv(t)\|_{L_q(\BR_-^3)}
		\leq \|\Bv(t)\|_{L_\infty(\BR_-^3)}\|\nabla\Bv(t)\|_{L_q(\BR_-^3)} \\
		&\quad\leq C_q \|\Bv(t)\|_{W_q^1(\BR_-^3)}\|\nabla\Bv(t)\|_{L_q(\BR_-^3)} 
		\leq C_q (t+2)^{-(2/q+3/8)}\|\Bz\|_{X_{q,p}(1/2,3/4)}^2, \\
	&\|(\Bv(t)\cdot\nabla)\Bv(t)\|_{L_{\bar{q}}(\BR_-^3)}
		\leq \|\Bv(t)\|_{L_q(\BR_-^3)}\|\nabla\Bv(t)\|_{L_q(\BR_-^3)} 
		\leq (t+2)^{-(2/q+3/8)}\|\Bz\|_{X_{q,p}(1/2,3/4)}^2
\end{align*}
for every $t>0$.
Then, noting that $p(2/q+3/8-3/4)>p/4>1$ by the assumption: $3<q<16/5$ and \eqref{141225_3},
we have
\begin{equation*}
	\|(\Bv\cdot\nabla)\Bv\|_{L_p^{3/4}(\BR_+,L_q(\BR_-^3))}
		\leq C_{p,q}\|\Bz\|_{X_{q,p}(1/2,3/4)}^2, \quad 
	\|(\Bv\cdot\nabla)\Bv\|_{L_\infty^{2/q+3/8}(\BR_+,L_{\bar{q}}(\BR_-^3))}
		\leq C_{p,q}\|\Bz\|_{X_{q,p}(1/2,3/4)}^2,
\end{equation*}
which, combined with \eqref{141224_16}, furnishes that
\begin{equation}\label{141225_10}
	\|(\Bv\cdot\nabla)\Bv\|_{\BBF_1\cap\widetilde{\BBF}_1(3/4,2/q+3/8)}
		\leq C_{p,q}\|\Bz\|_{X_{q,p}(1/2,3/4)}^2.
\end{equation}
Similarly, it holds that
\begin{equation*}
	\|(\Bv\cdot\nabla\eta)D_3\Bv\|_{\BBF_1\cap\widetilde{\BBF}_1(3/4,2/q+3/8)}
		\leq C_{p,q}\|\Bz\|_{X_{q,p}(1/2,3/4)}^2,
\end{equation*}
which, combined with \eqref{141224_1} and \eqref{141225_10}, furnishes that
\begin{align*}
	\|\BF_1(\Bv,\eta)\| \leq C_{p,q}(1+M)\left(1+\frac{1}{1-M\|\Bz\|_{X_{q,p}(1/2,3/4)}}\right)\|\Bz\|_{X_{q,p}(1/2,3/4)}^2.
\end{align*}
Thus, noting that $\|\Bz\|_{X_{q,p}(1/2,3/4)}\leq C_{p,q}(\de_0+\ep_1)$,
we choose $\de_0$, $\ep_1$ so small that
$C_{p,q}M(\de_0+\ep_1)\leq 1/2$ in what follows in order to obtain
\begin{equation}\label{170108_1}
	\|\BF_1(\Bv,\eta)\|_{\BBF_1\cap\widetilde{\BBF}_1(3/4,2/q+3/8)} \leq C_{p,q}\|\Bz\|_{X_{q,p}(1/2,3/4)}^2.
\end{equation}

{\bf Step 2: Estimates of $\BF_2(\Bv,\eta)$.}
We use Sobolev's inequality (cf. \cite[Theorem 4.31]{AF03}):
\begin{equation*}
	\|f\|_{L_6(\BR_-^3)}\leq C\|\nabla f\|_{L_2(\BR_-^3)} \quad \text{for any $f\in W_2^1(\lhs)$}
\end{equation*}
with a positive constant $C$ independent of $f$.
Since
	$\pa_t\eta\in L_\infty(\BR_+,W_2^1(\lhs))$,
we observe that, by Sobolev's inequality, H\"older's inequality, and Sobolev's embedding theorem,
\begin{align}\label{141224_12}
	&\|\pa_t \eta(t)D_3\Bv(t)\|_{L_q(\BR_-^3)} \\
		&\quad\leq
			\|\pa_t \eta(t)\|_{L_6(\BR_-^3)}\|\nabla \Bv(t)\|_{L_r(\BR_-^3)} \quad (1/6+1/r=1/q) \notag \\
		&\quad\leq
			C\|\nabla\pa_t \eta(t)\|_{L_2(\BR_-^3)}\|\nabla\Bv(t)\|_{W_q^1(\BR_-^3)} \notag \displaybreak[0] \\
		&\|\pa_t \eta(t)D_3\Bv(t)\|_{L_2(\BR_-^3)} 
		 \leq
			\|\pa_t \eta(t)\|_{L_6(\BR_-^3)}\|\nabla\Bv(t)\|_{L_3(\BR_-^3)} \notag \\
		&\quad \leq
			 C\|\nabla\pa_t \eta(t)\|_{L_2(\BR_-^3)}\|\nabla\Bv(t)\|_{L_2(\BR_-^3)}^{\al}\|\nabla \Bv(t)\|_{L_q(\BR_-^3)}^{1-\al},
			\notag \displaybreak[0] \\
	&\|\pa_t \eta(t)D_3\Bv(t)\|_{L_{\bar{q}}(\BR_-^3)} \notag \\
		&\quad\leq
			\|\pa_t \eta(t)\|_{L_6(\BR_-^3)}\|\nabla\Bv(t)\|_{L_s(\BR_-^3)} \quad (1/6+1/s=1/\bar{q}) \notag \\
		&\quad\leq
			C\|\nabla \pa_t \eta(t)\|_{L_2(\BR_-^3)}\|\nabla\Bv(t)\|_{L_2(\BR_-^3)}^\beta\|\nabla\Bv(t)\|_{L_q(\BR_-^3)}^{1-\beta}\notag
\end{align}
for $t>0$, where we note that $0<\al,\beta<1$ and
\begin{equation}\label{141224_15}
	3\left(\frac{1}{q}-\frac{1}{r}\right)=\frac{1}{2}<1,\quad
	\frac{1}{3}=\frac{\al}{2}+\frac{1-\al}{q},\quad
	\frac{1}{s}=\frac{\beta}{2}+\frac{1-\beta}{q}.
\end{equation}
By \eqref{141224_1} and \eqref{141224_12}, we obtain
\begin{align}\label{141225_7}
	\|\pa_t \eta D_3\Bv\|_{L_p(\BR_+,L_q(\BR_-^3))} 
			&\leq C\|\nabla \pa_t \eta\|_{L_\infty(\BR_+,L_2(\BR_-^3))}\|\nabla\Bv\|_{L_p(\BR_+,W_q^1(\BR_-^3))}
			 \leq C\|\Bz\|_{X_{q,p}(1/2,3/4)}^2,  \\
		\|\pa_t \eta D_3\Bv\|_{L_p(\BR_+,L_2(\BR_-^3))} 
			&\leq C\|\nabla \pa_t \eta\|_{L_p(\BR_+,L_2(\BR_-^3))}\|\nabla\Bv\|_{L_\infty(\BR_+,L_2(\BR_-^3))}^\al
				\|\nabla\Bv\|_{L_\infty(\BR_+,L_q(\BR_-^3))}^{1-\al} \notag \\
			&\leq C\|\Bz\|_{X_{q,p}(1/2,3/4)}^2. \notag
	\end{align}
In addition, it follows from \eqref{141224_12} that for any $t>0$ 
\begin{align*}
	&\|\pa_t \eta(t)D_3\Bv(t)\|_{L_q(\BR_-^3)}
		\leq
			C(t+2)^{-\Fm(\bar{q},2)-1/2}\|\Bz\|_{X_{q,p}(1/2,3/4)} \\
			&\quad \cdot\Big((t+2)^{-\Fn(\bar{q},q)-1/8}\|\Bz\|_{X_{q,p}(1/2,3/4)} +(t+2)^{-1/2}\big\{(t+2)^{1/2}\|\nabla^2\Bv(t)\|_{L_q(\BR_-^3)}\}\Big), \\
	 &\|\pa_t \eta(t)D_3\Bv(t)\|_{L_{\bar{q}}(\BR_-^3)}
		\leq
			C(t+2)^{-\Fm(\bar{q},2)-1/2}\|\Bz\|_{X_{q,p}(1/2,3/4)}	 \\
			&\quad\cdot\Big((t+2)^{-\Fn(\bar{q},2)-1/8}\|\Bz\|_{X_{q,p}(1/2,3/4)}\Big)^\beta
			\Big((t+2)^{-\Fn(\bar{q},q)-1/8}\|\Bz\|_{X_{q,p}(1/2,3/4)}\Big)^{1-\beta} \\
		&= C(t+2)^{-(2/q+3/8)}\|\Bz\|_{X_{q,p}(1/2,3/4)}^2,
\end{align*}
because $\Fm(\bar{q},2)+1/2=2/q$ and by \eqref{141224_15}
\begin{align*}
	&\beta \Fn\left(\bar{q},2\right)+(1-\beta)\Fn\left(\bar{q},q\right)
		=\frac{3\beta}{2}\left(\frac{2}{q}-\frac{1}{2}\right)+\frac{3(1-\beta)}{2q} \\
		&=\frac{3}{2q}-\frac{3\beta}{2}\left(\frac{1}{2}+\frac{1}{q}-\frac{2}{q}\right)
		=\frac{3}{2q}-\frac{3}{2}\cdot\frac{6-q}{3(q-2)}\cdot\frac{q-2}{2q}=\frac{1}{4}.
\end{align*}
Then, noting that by \eqref{141225_3}
\begin{align*}
	&p\left(\Fm\left(\bar{q},2\right)+\frac{1}{2}+\Fn\left(\bar{q},q\right)+\frac{1}{8}-\frac{3}{4}\right)
	=p\left(\frac{7}{2q}-\frac{5}{8}\right)>p\left(\frac{7}{8}-\frac{5}{8}\right)=\frac{p}{4}>1, \\
	&\Fm\left(\bar{q},2\right)+\frac{1}{2}+\frac{1}{2}-\frac{3}{4}=\frac{2}{q}-\frac{1}{4}>\frac{2}{4}-\frac{1}{4}>0
\end{align*}
because $q<16/5<4$, we see that
\begin{align*}
	&\|\pa_t \eta D_3\Bv\|_{L_p^{3/4}(\BR_+,L_q(\BR_-^3))} \\
		&\leq  C_{p,q}\|\Bz\|_{X_{q,p}(1/2,3/4)} 
		\Big(\|(t+2)^{-\left(\Fm(\bar{q},2)+1/2+\Fn(\bar{q},q)+1/8-3/4\right)}\|_{L_p(\BR_+)}\|\Bz\|_{X_{q,p}(1/2,3/4)} \\
		&\quad+\|(t+2)^{-\left(\Fm(\bar{q},2)+1/2+1/2-3/4\right)}\|_{L_\infty(\BR_+)}\|\nabla^2\Bv\|_{L_p^{1/2}(\BR_+,L_q(\BR_-^3))}\Big) \\
		&\leq
			C_{p,q}\|\Bz\|_{X_{q,p}(1/2,3/4)}^2,  \\
	&\|\pa_t \eta D_3\Bv\|_{L_\infty^{2/q+3/8}(\BR_+,L_{\bar{q}}(\BR_-^3))}\leq C_{p,q}\|\Bz\|_{X_{q,p}(1/2,3/4)}^2.
\end{align*}
Combining this inequality with \eqref{141225_7} furnishes that
\begin{equation}\label{141225_11}
	\|\pa_t \eta D_3\Bv\|_{\BBF_2\cap\widetilde{\BBF}_2(3/4,2/q+3/8)}\leq C_{p,q}\|\Bz\|_{X_{q,p}(1/2,3/4)}^2.
\end{equation}
By \eqref{141224_1} and \eqref{141225_11}, we have
\begin{equation}\label{170108_2}
	\|\BF_2(\Bv,\eta)\|_{\BBF_2\cap\widetilde{\BBF}_2(3/4,2/q+3/8)} 
		\leq 
		C_{p,q}\|\Bz\|_{X_{q,p}(1/2,3/4)}^2.
\end{equation}

{\bf Step 3: Estimates of $\BF_3(\Bv,\eta)$.}
We only consider the term:
\begin{equation*}
	\mu(1+\BM_3(\eta))\sum_{j=1}^3\CD_{jj}(\eta)\Bv,
\end{equation*}
because we can deal with the other terms similarly.
By \eqref{141224_1}, it is clear that for $r\in\{q,2\}$ and $j=1,2,3$
\begin{equation}\label{141225_13}
	\|\CD_{jj}(\eta)\Bv\|_{L_p(\BR_+,L_r(\BR_-^3))}  \leq C_{p,q} \|\Bz\|_{X_{q,p}(1/2,3/4)}^2. \notag
\end{equation}
In addition, it follows from H\"older's inequality and Sobolev's embedding theorem that,
for any $t>0$ and $j=1,2,3$,
\begin{align}\label{141225_19}
	&\|\CD_{jj}(\eta(t))\Bv(t)\|_{L_q(\BR_-^3)} \\
		&\leq
			C_q\Big(\|\nabla \Bv(t)\|_{L_\infty(\BR_-^3)}\|\nabla^2 \eta(t)\|_{L_q(\BR_-^3)}
			+\|\nabla^2\Bv(t)\|_{L_q(\BR_-^3)}\|\nabla \eta(t)\|_{L_\infty(\BR_-^3)} \Big) \notag \\
		&\leq
			C_q\Big(\|\nabla \Bv(t)\|_{W_q^1(\BR_-^3)}\|\nabla^2\eta(t)\|_{L_q(\BR_-^3)}
			+\|\nabla^2\Bv(t)\|_{L_q(\BR_-^3)}\|\nabla \eta(t)\|_{W_q^1(\BR_-^3)} \Big) \notag \\
		&\leq
			C_q(t+2)^{-\Fm(\bar{q},q)-1/4}\|\Bz\|_{X_{q,p}(1/2,3/4)} \Big((t+2)^{-\Fn(\bar{q},q)-1/8}\|\Bz\|_{X_{q,p}(1/2,3/4)} \notag \\
			&\quad+(t+2)^{-1/2}\big\{(t+2)^{1/2}\|\nabla^2\Bv(t)\|_{L_q(\BR_-^3)}\big\}\Big),
			\displaybreak[0] \notag 
\end{align}
We thus observe that
\begin{align}\label{141225_17}
	&\|\CD_{jj}(\eta)\Bv\|_{L_p^{3/4}(\BR_+,L_q(\BR_-^3))}  \\ 
	&\leq C_{p,q}\|\Bz\|_{X_{q,p}(1/2,3/4)}
			\Big(\|(t+2)^{-\left(\Fm(\bar{q},q)+1/4+\Fn(\bar{q},q)+1/8-3/4\right)}\|_{L_p(\BR_+)}\|\Bz\|_{X_{q,p}(1/2,3/4)} \notag \\
			&\quad+\|(t+2)^{-\Fm(\bar{q},q)}\|_{L_\infty(\BR_+)}\|\nabla^2\Bv\|_{L_p^{1/2}(\BR_+,L_q(\BR_-^3))}\Big) \notag \\
		&\leq
			C_{p,q}\|\Bz\|_{X_{q,p}(1/2,3/4)}^2, \displaybreak[0] \notag 
\end{align}
because we know that, by \eqref{141225_3} and the assumption: $3<q<16/5$, 
\begin{align*}
	&p\left(\Fm\left(\bar{q},q\right)+\Fm\left(\bar{q},q\right)+\frac{1}{4}+\Fn\left(\bar{q},q\right)+\frac{1}{8}-\frac{3}{4}\right) \\
		&>p\left(\Fm\left(\bar{q},q\right)+\frac{1}{4}+\Fn\left(\bar{q},q\right)+\frac{1}{8}-\frac{3}{4}\right)
		=p\left(\frac{2}{q}-\frac{1}{8}\right)>\frac{p}{2}>1.
\end{align*} 
Analogously, it holds that  
\begin{equation}\label{141225_21}
	\|\CD_{jj}(\eta)\Bv\|_{L_p^1(\BR_+,L_{\bar{q}}(\BR_-^3))}\leq C_{p,q}\|\Bz\|_{X_{q,p}(1/2,3/4)}^2
\end{equation}
by using inequalities,
which are obtained in a similar way to \eqref{141225_19}, as follows:
For every $t>0$ and $j=1,2,3$,
\begin{align*}
	&\|\CD_{jj}(\eta)\Bv(t)\|_{L_{\bar{q}}(\BR_-^3)}
		\leq C_{p,q}(t+2)^{-\Fm(\bar{q},q)-1/4}\|\Bz\|_{X_{q,p}(1/2,3/4)} \\
			&\cdot \Big((t+2)^{-\Fn(\bar{q},q)-1/8}\|\Bz\|_{X_{q,p}(1/2,3/4)}
			+(t+2)^{-1/2}\big\{(t+2)^{1/2}\|\nabla^2\Bu(t)\|_{L_q(\BR_-^3)}\big\}\Big), 
\end{align*}
and note that $\Fm(\bar{q},q)+3/4-1=1/(2q)>0$ and that by \eqref{141225_3}
\begin{align*}
	&p\left(\Fm\left(\bar{q},q\right)+\Fm\left(\bar{q},q\right)+\frac{1}{4}+\Fn\left(\bar{q},q\right)+\frac{1}{8}-1\right) \\
		&>p\left(\Fm\left(\bar{q},q\right)+\frac{1}{4}+\Fn\left(\bar{q},q\right)+\frac{1}{8}-1\right)
		=p\left(\frac{2}{q}-\frac{3}{8}\right)>\frac{p}{4}>1.
\end{align*}
By combining \eqref{141224_1} with \eqref{141225_13}, \eqref{141225_17}, and \eqref{141225_21},
we obtain 
\begin{equation*}
	\left\|\mu(1+\BM_3(\eta))\sum_{j=1}^3\CD_{jj}(\eta)\Bv\right\|_{\BBF_3\cap\widetilde{\BBF}_3(3/4,1)}
		\leq C_{p,q}\|\Bz\|_{X_{q,p}(1/2,3/4)}^2.
\end{equation*}
It thus holds that
\begin{equation}\label{170108_3}
	\|\BF_3(\Bv,\eta)\|_{\BBF_3\cap\widetilde{\BBF}_3(3/4,1)}\leq C_{p,q}\|\Bz\|_{X_{q,p}(1/2,3/4)}^2.
\end{equation}

{\bf Step 4: Estimates of $\Fg(\Bv,\eta)$.}
By \eqref{141224_1}, it is clear that for $r\in\{q,2\}$
\begin{align}\label{141225_25}
	&\|\BM_1(\eta)\Bv\|_{W_p^1(\BR_+,L_r(\BR_-^3))}
		\leq
			\|\nabla \eta\|_{L_\infty(\BR_+,L_\infty(\BR_-^3))}\|\Bv\|_{L_p(\BR_+,L_r(\BR_-^3))} \\
			&\quad+\|\pa_t \nabla \eta\|_{L_p(\BR_+,L_r(\BR_-^3))}\|\Bv\|_{L_\infty(\BR_+,L_\infty(\BR_-^3))} 
		+\|\nabla \eta\|_{L_\infty(\BR_+,L_\infty(\BR_-^3))}\|\pa_t\Bv\|_{L_p(\BR_+,L_r(\BR_-^3))} \notag \\
		&\leq
			C_{p,q}\|\Bz\|_{X_{q,p}(1/2,3/4)}^2. \notag
\end{align}
On the other hand, it follows from H\"older's inequality and Sobolev's embedding theorem that, for any $t>0$,
\begin{align*}
	&\|\BM_1(\eta(t))\Bv(t)\|_{L_q(\lhs)}
			\leq \|\nabla\eta(t)\|_{L_\infty(\BR_-^3)}\|\Bv(t)\|_{L_q(\lhs)} \\
		&\quad \leq C_q \|\nabla\eta(t)\|_{W_q^1(\lhs)}\|\Bv(t)\|_{L_q(\lhs)} 
		 \leq C_q (t+2)^{-\Fm(\bar{q},q)-1/4-\Fm(\bar{q},q)}\|\Bz\|_{X_{q,p}(1/2,3/4)}^2 \\
	&\|(\pa_t \BM_1(\eta(t)))\Bv(t)\|_{L_q(\BR_-^3)}
		\leq
			\|\pa_t\nabla\eta(t)\|_{L_q(\BR_-^3)}\|\Bv(t)\|_{L_\infty(\BR_-^3)} \\
		&\quad\leq
			C_q\|\pa_t\nabla \eta(t)\|_{L_q(\BR_-^3)}\|\Bv(t)\|_{W_q^1(\BR_-^3)} 
		\leq
			C_q(t+2)^{-\Fm(\bar{q},q)-1/2-\Fm(\bar{q},q)}\|\Bz\|_{X_{q,p}(1/2,3/4)}^2, \displaybreak[0] \\
	&\|\BM_1(\eta(t))\pa_t\Bv(t)\|_{L_q(\BR_-^3)}
		\leq
			\|\nabla \eta(t)\|_{L_\infty(\BR_-^3)}\|\pa_t\Bv(t)\|_{L_q(\BR_-^3)} \\
		&\quad\leq
			C_q\|\nabla\eta(t)\|_{W_q^1(\BR_-^3)}\|\pa_t\Bv(t)\|_{L_q(\BR_-^3)} \\
		&\quad\leq
			C_q(t+2)^{-\Fm(\bar{q},q)-3/4}\|\Bz\|_{X_{q,p}(1/2,3/4)}\big\{(t+2)^{1/2}\|\pa_t\Bv(t)\|_{L_q(\BR_-^3)}\big\},
\end{align*}
and furthermore,
\begin{align*}
	&\|\BM_1(\eta(t))\Bv(t)\|_{L_{\bar{q}}(\lhs)}
			\leq C_q\|\nabla\eta(t)\|_{L_q(\lhs)}\|\Bv(t)\|_{L_q(\lhs)} \\
		&\quad \leq C_q(t+2)^{-\Fm(\bar{q},q)-1/4-\Fm(\bar{q},q)}\|\Bz\|_{X_{q,p}(1/2,3/4)}^2, \\
	&\|(\pa_t\BM_1(\eta(t)))\Bv(t)\|_{L_{\bar{q}}(\BR_-^3)}
		\leq
			\|\pa_t \nabla \eta(t)\|_{L_q(\BR_-^3)}\|\Bv(t)\|_{L_q(\BR_-^3)} \\
		&\quad\leq
			(t+2)^{-\Fm(\bar{q},q)-1/2-\Fm(\bar{q},q)}\|\Bz\|_{X_{q,p}(1/2,3/4)}^2, \\
	&\|\BM_1(\eta(t))\pa_t\Bv(t)\|_{L_{\bar{q}}(\BR_-^3)}
		\leq
			\|\nabla \eta(t)\|_{L_q(\BR_-^3)}\|\pa_t\Bv(t)\|_{L_q(\BR_-^3)} \\
		&\quad\leq
			(t+2)^{-\Fm(\bar{q},q)-3/4}\|\Bz\|_{X_{q,p}(1/2,3/4)}\big\{(t+2)^{1/2}\|\pa_t\Bv(t)\|_{L_q(\BR_-^3)}\big\}.
\end{align*}
Then, noting that by \eqref{141225_3}
\begin{equation*} 
	p\left(\Fm\left(\bar{q},q\right)+\frac{1}{4}+\Fm\left(\bar{q},q\right)-\frac{3}{4}\right)
	=\frac{p}{q}>1,
\end{equation*}
we have
\begin{align*}
	\|\BM_1(\eta)\Bv\|_{L_p^{3/4}(\BR_+,L_q(\lhs))} 
		 &\leq C_{p,q}\|(t+2)^{-\left(\Fm(\bar{q},q)+1/4+\Fm(\bar{q},q)-3/4\right)}\|_{L_p(\BR_+)}\|\Bz\|_{X_{q,p}(1/2,3/4)}^2 \\
		&\leq
			C_{p,q}\|\Bz\|_{X_{q,p}(1/2,3/4)}^2 , \\
	\|(\pa_t\BM_1(\eta))\Bv\|_{L_p^{3/4}(\BR_+,L_q(\BR_-^3))} 
		&\leq
			C_{p,q}\|(t+2)^{-\left(\Fm(\bar{q},q)+1/2+\Fm(\bar{q},q)-3/4\right)}\|_{L_p(\BR_+)}\|\Bz\|_{X_{q,p}(1/2,3/4)}^2 \\
		&\leq
			C_{p,q}\|\Bz\|_{X_{q,p}(1/2,3/4)}^2 , \\
	\|\BM_1(\eta)\pa_t\Bv\|_{L_p^{3/4}(\BR_+,L_q(\BR_-^3))} 
		&\leq
			C_{p,q}\|(t+2)^{-\Fm(\bar{q},q)}\|_{L_\infty(\BR_+)}\|\Bz\|_{X_{q,p}(1/2,3/4)}\|\pa_t\Bv\|_{L_p^{1/2}(\BR_+,L_q(\BR_-^3))} \\
		&\leq
			C_{p,q}\|\Bz\|_{X_{q,p}(1/2,3/4)}^2
\end{align*}
with some positive constant $C_{p,q}$, and besides,
it similarly holds that
\begin{align*}		
	\|\BM_1(\eta)\Bv\|_{L_p^1(\BR_+,L_{\bar{q}}(\BR_-^3))}
		&\leq C_{p,q}\|\Bz\|_{X_{q,p}(1/2,3/4)}^2, \\
	\|(\pa_t\BM_1(\eta))\Bv\|_{L_p^1(\BR_+,L_{\bar{q}}(\BR_-^3))}
		&\leq C_{p,q}\|\Bz\|_{X_{q,p}(1/2,3/4)}^2, \\
	\|\BM_1(\eta)\pa_t\Bv\|_{L_p^1(\BR_+,L_{\bar{q}}(\BR_-^3))}
		&\leq C_{p,q}\|\Bz\|_{X_{q,p}(1/2,3/4)}^2,
\end{align*}
because $\Fm(\bar{q},q)+3/4-1=1/(2q)>0$ and 
\begin{equation*}
	 p\left(\Fm\left(\bar{q},q\right)+\frac{1}{4}+\Fm\left(\bar{q},q\right)-1\right)
	=\frac{p}{q}-\frac{1}{4}>\frac{5p-4}{16}>\frac{p}{16}>1
\end{equation*}
by \eqref{141225_3} and the assumption: $3<q<16/5$.
We thus obtain 
\begin{equation}\label{161222_1}
	\|\BM_1(\eta)\Bv\|_{W_p^{1,3/4}(\BR_+,L_q(\BR_-^3))}+\|\BM_1(\eta)\Bv\|_{W_p^{1,1}(\BR_+,L_{\bar{q}}(\BR_-^3))} 
		\leq C_{p,q}\|\Bz\|_{X_{q,p}(1/2,3/4)}^2. 
\end{equation}
In addition, it is clear that by \eqref{141224_1} and H\"older's inequality
\begin{equation*}
	\|\BM_1(\eta(t))\Bv(t)\|_{L_r(\BR_-^3)}
		\leq 
			\|\nabla \eta(t)\|_{L_r(\BR_-^3)}\|\Bv(t)\|_{L_\infty(\BR_-^3)} 
		\leq
			C_{p,q}(t+2)^{-\Fm(\bar{q},r)-1/4}\|\Bz\|_{X_{q,p}(1/2,3/4)}^2
\end{equation*}
for $r\in\{q,2\}$ and any $t>0$, which furnishes that
\begin{equation*}
\|\BM_1(\eta)\Bv\|_{\BBA_1}\leq C_{p,q}\|\Bz\|_{X_{q,p}(1/2,3/4)}^2.
\end{equation*}
By combining the inequality with \eqref{141225_25} and \eqref{161222_1},
we obtain
\begin{equation}\label{170108_4}
	\|\Fg(\Bv,\eta)\|_{\CG\cap\wt\CG(3/4,1)\cap\BBA_1}\leq C_{p,q}\|\Bz\|_{X_{q,p}(1/2,3/4)}^2.
\end{equation}


{\bf Step 5: Estimates of $\SSG(\Bv,\eta)$.}
By Lemma \ref{lemm:embed} 
 and \eqref{141224_1}, it is clear that for $r\in\{q,2\}$
\begin{equation}\label{141226_11}
	\|\SSG(\Bv,\eta)\|_{\BBG}\leq C_{p,q}\|\Bz\|_{X_{q,p}(1/2,3/4)}^2
\end{equation}
with some positive constant $C_{p,q}$.
On the other hand, it holds that for any $t>0$
\begin{align*}
	&\|\nabla \eta(t)\cdot D_3\Bv(t)\|_{W_q^1(\BR_-^3)}
		\leq
			C_q\|\nabla \eta(t)\|_{W_q^1(\BR_-^3)}\|\nabla \Bv(t)\|_{W_q^1(\BR_-^3)} \\
		&\leq
			C_q(t+2)^{-\Fm(\bar{q},q)-1/4}\|\Bz\|_{X_{q,p}(1/2,3/4)} 
		\Big( (t+2)^{-\Fn(\bar{q},q)-1/8}\|\Bz\|_{X_{q,p}(1/2,3/4)} \\
		&\quad +(t+2)^{-1/2}\big\{(t+2)^{1/2}\|\nabla^2\Bv(t)\|_{L_q(\BR_-^3)}\big\}\Big)
\end{align*}
because $W_q^1(\BR_-^3)$ is a Banach algebra, and also that by H\"older's inequality
\begin{align*}
	&\|\nabla \eta(t)\cdot D_3\Bv(t)\|_{W_{\bar{q}}^1(\BR_-^3)}
		\leq 
			C_q\|\nabla \eta(t)\|_{W_q^1(\BR_-^3)}\|\nabla\Bv(t)\|_{W_q^1(\BR_-^3)} \\
		&\leq
			C_q(t+2)^{-\Fm(\bar{q},q)-1/4}\|\Bz\|_{X_{q,p}(1/2,3/4)}
			\Big((t+2)^{-\Fn(\bar{q},q)-1/8}\|\Bz\|_{X_{q,p}(1/2,3/4)} \\
		&\quad+(t+2)^{-1/2}\big\{(t+2)^{1/2}\|\nabla^2\Bv(t)\|_{L_q(\BR_-^3)}\big\}\Big).
\end{align*}
Then, noting that $\Fm(\bar{q},q)+3/4-1=1/(2q)>0$ and
\begin{equation*}
	p\left(\Fm\left(\bar{q},q\right)+\frac{1}{4}+\Fn\left(\bar{q},q\right)+\frac{1}{8}-1\right)
	=p\left(\frac{2}{q}-\frac{3}{8}\right)>\frac{p}{4}>1
\end{equation*}
by \eqref{141225_3} and the assumption: $3<q<16/5$, we have
\begin{align}\label{141226_13}
	&\|\nabla \eta\cdot D_3\Bv\|_{L_p^1(\BR_+,W_q^1(\BR_-^3))}+ \|\nabla \eta\cdot D_3\Bv\|_{L_p^1(\BR_+,W_{\bar{q}}^1(\BR_-^3))}\\
		&\leq
			C_{p,q}\|\Bz\|_{X_{q,p}(1/2,3/4)} 
			 \Big(\|(t+2)^{-\left(\Fm(\bar{q},q)+1/4+\Fn(\bar{q},q)+1/8-1\right)}\|_{L_p(\BR_+)}\|\Bz\|_{X_{q,p}(1/2,3/4)} \notag \\
		&\quad+\|(t+2)^{-\left(\Fm(\bar{q},q)+3/4-1\right)}\|_{L_\infty(\BR_+)}\|\nabla^2\Bv\|_{L_p^{1/2}(\BR_+,L_q(\BR_-^3))}\Big) \notag \\
		&\leq
			C_{p,q}\|\Bz\|_{X_{q,p}(1/2,3/4)}^2. \notag
\end{align}
In addition, since for $r\in\{q,2\}$ and any $t>0$
\begin{align*}
	&\|\nabla \eta(t)\cdot D_3\Bv(t)\|_{L_r(\BR_-^3)}
		\leq
			\|\nabla \eta(t)\|_{L_\infty(\BR_-^3)}\|\nabla\Bv(t)\|_{L_r(\BR_-^3)} \\
		&\leq
			C_q\|\nabla \eta(t)\|_{W_q^1(\BR_-^3)}\|\nabla\Bv(t)\|_{L_r(\BR_-^3)} \\
		&\leq
				C_q(t+2)^{-\Fm(\bar{q},q)-1/4-\Fn(\bar{q},r)-1/8}\|\Bz\|_{X_{q,p}(1/2,3/4)}^2
\end{align*} 
by H\"older's inequality and Sobolev's embedding theorem, which furnishes that
\begin{equation}\label{141226_17}
	\|\nabla \eta\cdot D_3\Bv\|_{\BBA_2}\leq C_{p,q}\|\Bz\|_{X_{q,p}(1/2,3/4)}^2,
\end{equation}
where we have used
\begin{equation*}
	\Fm\left(\bar{q},q\right)+\frac{1}{4}+\Fn\left(\bar{q},r\right)+\frac{1}{8}
		>\Fm\left(\bar{q},r\right)+\frac{1}{2}.
\end{equation*}
Concerning $\BBA_3$-norm, we shall calculate as follows: First, by H\"older's inequality
\begin{align*}
	&\|\nabla \eta(t)\cdot D_3\Bv(t)\|_{W_2^1(\BR_-^3)}
		\leq
			\|\nabla \eta(t)\|_{L_\infty(\BR_-^3)}\|\nabla \Bv(t)\|_{L_2(\BR_-^3)} \\
			&\quad +\|\nabla^2\eta(t)\|_{L_\infty(\BR_-^3)}\|\nabla\Bv(t)\|_{L_2(\BR_-^3)}
			+ \|\nabla\eta(t)\|_{L_\infty(\BR_-^3)}\|\nabla^2\Bv(t)\|_{L_2(\BR_-^3)}\\
		&\leq
			C_q\Big(\|\nabla \eta(t)\|_{W_q^1(\BR_-^3)}\|\nabla\Bv(t)\|_{W_2^1(\BR_-^3)}
			+\|\nabla^2\eta(t)\|_{W_q^1(\BR_-^3)}\|\nabla\Bv(t)\|_{L_2(\BR_-^3)}\Big) \\
		&\leq
			C_q\Big((t+2)^{-\Fm(\bar{q},q)-1/4-\Fn(\bar{q},2)-1/8}\|\Bz\|_{X_{q,p}(1/2,3/4)}^2 \\
			&\quad+ (t+2)^{-\Fm(\bar{q},q)-1/4}\|\Bz\|_{X_{q,p}(1/2,3/4)}\|\nabla^2\Bv(t)\|_{L_2(\BR_-^3)} \\
			&\quad+ (t+2)^{-3/4-\Fn(\bar{q},2)-1/8}\|\Bz\|_{X_{q,p}(1/2,3/4)}\big\{(t+2)^{3/4}\|\nabla^3 \eta(t)\|_{L_q(\BR_-^3)}\big\}\Big),
\end{align*}
which furnishes that
\begin{align}\label{141229_1}
	&\|\nabla \eta\cdot D_3\Bv\|_{L_p^{\Fm(\bar{q},2)+1/2}(\BR_+,W_2^1(\BR_-^3))}  \leq C_{p,q}\|\Bz\|_{X_{q,p}(1/2,3/4)} \\
		& \quad \cdot\Big(\|(t+2)^{-\left(\Fm(\bar{q},q)+1/4+\Fn(\bar{q},2)+1/8-\Fm(\bar{q},2)-1/2\right)}\|_{L_p(\BR_+)}\|\Bz\|_{X_{q,p}(1/2,3/4)}\notag \\
			&\quad+\|(t+2)^{-\left(\Fm(\bar{q},q)+1/4-\Fm(\bar{q},2)-1/2\right)}\|_{L_\infty(\BR_+)}\|\nabla^2\Bu\|_{L_p(\BR_+,L_2(\BR_-^3))}
			\notag \\
			&\quad+\|(t+2)^{-\left(3/4+\Fn(\bar{q},2)+1/8-\Fm(\bar{q},2)-1/2\right)}\|_{L_\infty(\BR_+)}
				\|\nabla^3 \eta\|_{L_p^{3/4}(\BR_+,L_q(\BR_-^3))}\Big) \notag \\
		&\leq
			C_{p,q}\|\Bz\|_{X_{q,p}(1/2,3/4)}^2, \notag
\end{align}
because by the assumption: $q<3<16/5$ and \eqref{141225_3}
\begin{align*}
	&\Fm\left(\bar{q},2\right)+\frac{1}{4}-\Fm\left(\bar{q},2\right)-\frac{1}{2}=\frac{1}{2}\left(1-\frac{3}{q}\right)>0, \\
	&\frac{3}{4}+\Fn\left(\bar{q},2\right)+\frac{1}{8}-\Fm\left(\bar{q},2\right)-\frac{1}{2}=\frac{1}{q}+\frac{1}{8}>0, \\
	&p\left(\Fm\left(\bar{q},q\right)+\frac{1}{4}+\Fn\left(\bar{q},2\right)+\frac{1}{8}-\Fm\left(\bar{q},2\right)-\frac{1}{2}\right) \\
		&>p\left(\Fm\left(\bar{q},q\right)+\frac{1}{4}+\Fn\left(\bar{q},2\right)+\frac{1}{8}-\Fn\left(\bar{q},2\right)-\frac{1}{2}\right)
		=p\left(\frac{1}{2q}+\frac{1}{8}\right)>1.
\end{align*}

Secondly, noting that $1-\Fm(\bar{q},q)-1/2=1/4-1/(2q)>0$ by $q>3$, we have by \eqref{141226_13} and \eqref{141229_1}
\begin{align}\label{141230_3}
	&\|\nabla \eta\cdot D_3\Bv\|_{\BBA_3}
		=
			\|\nabla\eta\cdot D_3\Bv\|_{L_p^{\Fm(\bar{q},q)+1/2}(\BR_+,W_q^1(\BR_-^3))} 
			+\|\nabla\eta\cdot D_3\Bv\|_{L_p^{\Fm(\bar{q},2)+1/2}(\BR_+,W_2^1(\BR_-^3))} \\
		&\leq
			\|(t+2)^{-(1-\Fm(\bar{q},q)-1/2)}\|_{L_\infty(\BR_+)}\|\nabla\eta\cdot D_3\Bv\|_{L_p^1(\BR_+,W_q^1(\BR_-^3))} 
			+\|\nabla\eta\cdot D_3\Bv\|_{L_p^{\Fm(\bar{q},2)+1/2}(\BR_+,W_2^1(\BR_-^3))} \notag \\
		&\leq C_{p,q}\|\Bz\|_{X_{q,p}(1/2,3/4)}^2. \notag
\end{align}
Summing up, \eqref{141224_1}, \eqref{141226_11}, \eqref{141226_13}, \eqref{141226_17}, and \eqref{141230_3}, we have
\begin{align}\label{170108_5}
	\|\SSG(\Bv,\eta)\|_{\BBG\cap\wt{\BBG}(3/4,1)\cap\BBA_2\cap\BBA_3} 
		&\leq
			C_{p,q}\left\|\frac{1}{1+D_3 \eta}\right\|_{W_\infty^1(\BR_-^3)}
			\|\nabla \eta\cdot D_3\Bv(t)\|_{\BBG\cap\wt{\BBG}(3/4,1)\cap\BBA_2\cap\BBA_3}  \\
		&\leq
			C_{p,q}\|\Bz\|_{X_{q,p}(1/2,3/4)}^2. \notag
\end{align}

{\bf Step 6: Estimates of $\SSH(\Bv,\eta)$.}
In this step, we suppose that
$\CA$ is any of the following terms:
\begin{equation*}
	v_1,\,v_2,\,v_3,\,D_1 \eta,\,D_2 \eta,\,D_3 \eta,
\end{equation*}
while $\CB$ is any of the following terms:
\begin{align*}
	\frac{|\nabla' \eta|^2}{(1+\sqrt{1+|\nabla'\eta|^2})\sqrt{1+|\nabla'\eta|^2}},
	\,\frac{D_i\eta D_j \eta}{(1+|\nabla' \eta|^2)^{3/2}},
	\,D_i \eta,\,D_i \eta D_j \eta
\end{align*}
for $i,j=1,2,3$. Then we have 

\begin{lemm}\label{lemm:taylor}
Let $\CA$ and $\CB$ be as above, and exponents $p,q$ satisfy \eqref{pq}.
Then, there exists a positive number $0<\ep_2<1$ such that
for any $\Bz=(\Bv,\Fq,h,\eta)\in X_{q,p}(\ep_2;1/2,3/4)$ the following assertions hold.
\begin{enumerate}[$(1)$]
	\item\label{lemm:taylor_1}
		There exists a $C_{p,q}>0$, independent of $\Bz$, such that
		\begin{alignat*}{2}
			\|\CA\|_{L_\infty^{\Fm(\bar{q},q)}(\BR_+,L_q(\lhs))}  &\leq C_{p,q}\|\Bz\|_{X_{q,p}(1/2,3/4)}, \quad
			&\|\nabla \CA\|_{L_\infty^{\Fn(\bar{q},q)+1/8}(\BR_+,L_q(\BR_-^3))} &\leq C_{p,q}\|\Bz\|_{X_{q,p}(1/2,3/4)}, \\
			\|\CB\|_{L_\infty(\BR_+,W_\infty^1(\BR_-^3))} &\leq C_{p,q}\|\Bz\|_{{X_{q,p}(1/2,3/4)}}, \quad
			&\|\CB\|_{L_\infty^{\Fm(\bar{q},q)+1/4}(\BR_+,W_q^1(\BR_-^3))}  &\leq C_{p,q}\|\Bz\|_{X_{q,p}(1/2,3/4)}.
		\end{alignat*}
	\item\label{lemm:taylor_3}
		There exists a $C_{p,q}>0$, independent of $\Bz$, such that
		\begin{equation*}
			\|\CA\CB\|_{W_{r,p}^{2,1}(\lhs\times\BR_+)} 
				\leq C_{p,r}\|\Bz\|_{X_{q,p}(1/2,3/4)}^2, \quad
			\|\CA\nabla\CB\|_{W_p^1(\BR_+,L_r(\lhs))}
				\leq C_{p,r}\|\Bz\|_{X_{q,p}(1/2,3/4)}^2.
		\end{equation*}
	\item\label{lemm:taylor_4}
		There exists a $C_{p,q}>0$, independent of $\Bz$, such that
		\begin{equation*}
			\|(t+2)^{3/4}\CA\CB\|_{W_{q,p}^{2,1}(\BR_-^3\times\BR_+)} 
				\leq C_{p,q}\|\Bz\|_{X_{q,p}(1/2,3/4)}^2, \quad
			\|\CA\nabla\CB\|_{W_p^{1,3/4}(\BR_+,L_q(\lhs))} 
				\leq C_{p,q}\|\Bz\|_{X_{q,p}(1/2,3/4)}^2,  
		\end{equation*}
			and also that
		\begin{equation*}
			\|(t+2)\CA\CB\|_{W_{\bar{q},p}^{2,1}(\BR_-^3\times\BR_+)} 
				\leq C_{p,q}\|\Bz\|_{X_{q,p}^{1/2,3/4}}^2, \quad
			\|\CA\nabla\CB\|_{W_p^{1,1}(\BR_+,L_{\bar{q}}(\lhs))} 
				\leq C_{p,q}\|\Bz\|_{X_{q,p}(1/2,3/4)}^2.
		\end{equation*}
	\item\label{lemm:taylor_5}
		There exists a $C_{p,q}>0$, independent of $\Bz$, such that	
		\begin{align*}
			\|(\nabla\CA)\CB\|_{L_p^{\Fm(\bar{q},q)+1/2}(\BR_+,W_q^1(\lhs))}
				&\leq C_{p,q}\|\Bz\|_{X_{q,p}(1/2,3/4)}^2, \\
			\|(\nabla\CA)\CB\|_{L_p^{\Fm(\bar{q},2)+1/2}(\BR_+,W_2^1(\lhs))}
				&\leq C_{p,q}\|\Bz\|_{X_{q,p}(1/2,3/4)}^2.
		\end{align*}
\end{enumerate}
\end{lemm}

\begin{proof}
(1). The first and second inequalities are clear,
so that we here prove the other inequalities.
We use the following expansions:
\begin{align*}
	&\frac{1}{\sqrt{1+x}}=1-\frac{x}{2}+O(x^2)\quad \text{as $x \to 0$}, \quad
	\frac{1}{1+\sqrt{1+x}}=\frac{1}{2}-\frac{x}{8}+O(x^2)\quad\text{as $x\to 0$}, \displaybreak[0] \\
	&\frac{1}{(1+x)^{3/2}}=1-\frac{3}{2}x+O(x^2)\quad \text{as $x\to0$}.
\end{align*}
Combining the above expansions with \eqref{141224_1}, we see that
there exist positive constants $C$ and $0<\ep_3<1$ such that
\begin{align}\label{141231_1}
	&\left\|\frac{1}{\sqrt{1+|\nabla' \eta|^2}}\right\|_{L_\infty(\BR_+,W_\infty^1(\BR_-^3))} \leq C, \quad
	\left\|\frac{1}{1+\sqrt{1+|\nabla' \eta|^2}}\right\|_{L_\infty(\BR_+,W_\infty^1(\BR_-^3))} \leq C, \\
	&\left\|\frac{1}{(1+|\nabla' \eta|^2)^{3/2}}\right\|_{L_\infty(\BR_+,W_\infty^1(\BR_-^3))} \leq C, \notag
\end{align}
if $\eta(x,t)$ satisfies $\|\nabla \eta\|_{L_\infty(\BR_+,W_\infty^1(\BR_-^3))}\leq \ep_3$.
We thus obtain, by using \eqref{141224_1}
and choosing $\ep_2$ small enough,
\begin{align*}
	&\left\|\frac{|\nabla' \eta|^2}{(1+\sqrt{1+|\nabla' \eta|^2})\sqrt{1+|\nabla' \eta|^2}}\right\|_{L_\infty(\BR_+,W_\infty^1(\BR_-^3))}
	\displaybreak[0] \\
	&\leq
		\left\|\frac{1}{1+\sqrt{1+|\nabla' \eta|^2}}\right\|_{L_\infty(\BR_+,W_\infty^1(\BR_-^3))}
		\left\|\frac{1}{\sqrt{1+|\nabla' \eta|^2}}\right\|_{L_\infty(\BR_+,W_\infty^1(\BR_-^3))} 
		\|\nabla \eta\|_{L_\infty(\BR_+,W_\infty^1(\BR_-^3))}^2 \\
	&\leq
		C(M\|\Bz\|_{X_{q,p}(1/2,3/4)})^2\leq CM^2\ep_2\|\Bz\|_{X_{q,p}(1/2,3/4)}
\end{align*}
for any $\Bz=(\Bv,\Fq,h,\eta)\in X_{q,p}(\ep_2;1/2,3/4)$, and also it similarly holds that
\begin{align*}
	\left\|\left(\frac{D_ i\eta D_j \eta}{(1+|\nabla' \eta|^2)^{3/2}},D_i \eta, D_i \eta D_j \eta\right)\right\|_{L_\infty(\BR_+,W_\infty^1(\BR_-^3))}
	\leq C_{p,q}\|\Bz\|_{X_{q,p}(1/2,3/4)},
\end{align*}
which completes the third inequality of (1).

Next, we show the last inequality. By using \eqref{141224_1} and \eqref{141231_1},
we have, for any $t>0$ and $\Bz=(\Bv,\Fq,h,\eta)\in X_{q,p}(\ep_2;1/2,3/4)$,
\begin{align*}
	&\left\|\frac{|\nabla' \eta(t)|^2}{(1+\sqrt{1+|\nabla' \eta(t)|^2})\sqrt{1+|\nabla' \eta(t)|^2}}\right\|_{W_q^1(\BR_-^3)}
	\displaybreak[0] \\
		&\leq
		\left\|\frac{1}{1+\sqrt{1+|\nabla' \eta|^2}}\right\|_{L_\infty(\BR_+,W_\infty^1(\BR_-^3))}
		\left\|\frac{1}{\sqrt{1+|\nabla' \eta|^2}}\right\|_{L_\infty(\BR_+,W_\infty^1(\BR_-^3))} 
		\|\nabla \eta\|_{L_\infty(\BR_+,W_\infty^1(\BR_-^3))}\|\nabla \eta(t)\|_{W_q^1(\BR_-^3)} \\
		&\leq
		CM\|\Bz\|_{X_{q,p}(1/2,3/4)}(t+2)^{-\Fm(\bar{q},q)-1/4}\big\{(t+2)^{\Fm(\bar{q},q)+1/4}\|\nabla \eta(t)\|_{W_q^1(\BR_-^3)}\big\} \\
		&\leq
			C_{p,q}(t+2)^{-\Fm(\bar{q},q)-1/4}\|\Bz\|_{X_{q,p}(1/2,3/4)},
\end{align*}
which furnishes that
\begin{equation*}
	\left\|\frac{|\nabla' \eta|^2}{(1+\sqrt{1+|\nabla' \eta|^2})\sqrt{1+|\nabla' \eta|^2}}
	\right\|_{L_\infty^{\Fm(\bar{q},q)+1/4}(\BR_+,W_q^1(\BR_-^3))}
	\leq C_{p,q}\|\Bz\|_{X_{q,p}(1/2,3/4)}.
\end{equation*}
Analogously, we have
\begin{equation*}
		\left\|\left(\frac{D_ i\eta D_j \eta}{(1+|\nabla' \eta|^2)^{3/2}},D_i \eta, D_i \eta D_j \eta\right)
		\right\|_{L_\infty^{\Fm(\bar{q},q)+1/4}(\BR_+,W_q^1(\BR_-^3))} 
			\leq C_{p,q}\|\Bz\|_{X_{q,p}(1/2,3/4)}
\end{equation*}
for $i,j=1,2,3$, 
which completes the last inequality of (1).

(2). 
By direct calculations, we can prove the required estimates, so that
we may omit the proof here.

(3), (4). 
By the inequalities obtained in (1), we can prove that
the required inequalities hold in the same manner as in the case of  Step 5, 
so that we omit the detailed proof. 
\end{proof}

In what follows, we additionally assume that $\ep_1\in(0,\ep_2)$
in order to use Lemma \ref{lemm:taylor}.
Since it holds that $(\nabla \CA)\CB = \nabla (\CA\CB)-\CA(\nabla \CB)$,
we observe that
\begin{equation*}
	\|(\nabla \CA)\CB\|_{H_p^{1/2}(\BR_+,L_s(\lhs))} 
		 \leq \|\CA\CB\|_{H_p^{1/2}(\BR_+,W_s^1(\lhs))}
		+\|\CA(\nabla \CB)\|_{H_p^{1/2}(\BR_+,L_s(\lhs))},
\end{equation*}
where $s\in\{q,2,\bar{q}\}$.
Combining this inequality with Lemma \ref{lemm:embed} and
\begin{equation*}\label{161223_7}
	W_p^1(\BR_+,L_s(\lhs))  \hookrightarrow
	H_p^{1/2}(\BR_+,L_s(\lhs))
\end{equation*}
furnishes that
\begin{equation*}
	\|(\nabla \CA)\CB\|_{H_p^{1/2}(\BR_+,L_r(\lhs))} 
		\leq C_{p,q}
		\left(\|\CA\CB\|_{W_{q,p}^{2,1}(\BR_-^3\times\BR_+)}+\|\CA(\nabla \CB)\|_{W_p^1(\BR_+,L_r(\lhs))}\right). \notag
\end{equation*}
Thus, we have, by Lemma \ref{lemm:taylor} \eqref{lemm:taylor_3},
\begin{equation*}
	\|(\nabla \CA)\CB\|_{H_p^{1/2}(\BR_+,L_r(\lhs))}
		\leq C_{p,q}\|\Bz\|_{X_{q,p}(1/2,3/4)}^2,
\end{equation*}
while we have, by Lemma \ref{lemm:taylor} \eqref{lemm:taylor_1},
\begin{equation*}
	\|(\nabla \CA)\CB\|_{L_p(\BR_+,W_r^1(\BR_-^3))} 
	\leq
		C_{p,q}\|\nabla \CA\|_{L_p(\BR_+,W_r^1(\BR_-^3))}\|\CB\|_{L_\infty(\BR_+,W_\infty^1(\BR_-^3))} 
	\leq
		C_{p,q}\|\Bz\|_{{X_{q,p}(1/2,3/4)}}^2.
\end{equation*}
These two inequalities imply that $\|\SSH(\Bv,\eta)\|_{\BBH}\leq C_{p,q}\|\Bz\|_{X_{q,p}(1/2,3/4)}^2$.

Next, we consider the estimates of $\wt{\BBH}(3/4,1)$-norm and $\BBA_3$-norm.
As mentioned above, we see that
\begin{align*}
	&\|(t+2)^{3/4}(\nabla \CA)\CB\|_{H_p^{1/2}(\BR_+,L_q(\lhs))} \\
		&\leq C_{p,q}\left(\|(t+2)^{3/4}\CA\CB\|_{W_{q,p}^{2,1}(\lhs\times\BR_+)}
		+\|(t+2)^{3/4}\CA(\nabla\CB)\|_{W_p^1(\BR_+,L_q(\lhs))}\right),
\end{align*}
which, combined with Lemma \ref{lemm:taylor} \eqref{lemm:taylor_4}, furnishes that
\begin{equation*}
	\|(t+2)^{3/4}(\nabla \CA)\CB\|_{H_p^{1/2}(\BR_+,L_q(\lhs))}\leq C_{p,q}\|\Bz\|_{X_{q,p}(1/2,3/4)}^2.
\end{equation*}
It similarly holds that
\begin{equation*}
	\|(t+2)(\nabla\CA)\CB\|_{H_p^{1/2}(\BR_+,L_{\bar{q}}(\lhs))} \leq C_{p,q}\|\Bz\|_{X_{q,p}(1/2,3/4)}^2.
\end{equation*}
On the other hand, by Lemma \ref{lemm:taylor} \eqref{lemm:taylor_1},
\begin{align*}
	\|(t+2)^{3/4}(\nabla\CA)\CB\|_{L_p(\BR_+,W_q^1(\lhs))}
		&\leq C_{p,q}\|\Bz\|_{X_{q,p}(1/2,3/4)}^2, \\	
	\|(t+2)(\nabla\CA)\CB\|_{L_p(\BR_+,W_{\bar{q}}^1(\lhs))}
		&\leq C_{p,q}\|\Bz\|_{X_{q,p}(1/2,3/4)}^2. 
\end{align*}
We thus see that $\|\SSH(\Bv,\eta)\|_{\wt\BBH(3/4,1)} \leq C_{p,q}\|\Bz\|_{X_{q,p}(1/2,3/4)}^2$.
Concerning $\BBA_3$-norm, it follows from Lemma \ref{lemm:taylor} \eqref{lemm:taylor_5} that
$\|(\nabla\CA)\CB\|_{\BBA_3}\leq C_{p,q}\|\Bz\|_{X_{q,p}(1/2,3/4)}^2$. 
Hence,
\begin{equation}\label{170108_6}
	\|\SSH(\Bv,\eta)\|_{\BBH\cap\wt\BBH(3/4,1)\cap \BBA_3}\leq C_{p,q}\|\Bz\|_{X_{q,p}(1/2,3/4)}^2.
\end{equation}

{\bf Step 7: Estimates of $K(\Bv,\eta)$.}
Let $\Bv'\cdot\nabla'\eta=\sum_{j=1}^2v_j\pa_j \eta$.
By \eqref{141224_1}, it is clear that for $r\in \{q,2\}$
\begin{align*}
	\|\Bv'\cdot \nabla'\eta\|_{W_{r,p}^{2,1}(\BR_-^3\times\BR_+)}
		\leq C_{p,q}\|\Bz\|_{X_{q,p}(1/2,3/4)}^2.
\end{align*}
On the other hand, it follows from H\"older's inequality and Sobolev's embedding theorem that for any $t>0$
\begin{align*}
	&\|\Bv'(t)\cdot\nabla' \eta(t)\|_{W_q^2(\BR_-^3)}
		\leq
			C_q\Big(\|\Bv(t)\|_{L_q(\BR_-^3)}\|\nabla \eta(t)\|_{L_\infty(\BR_-^3)} \\
			&\quad+\|\nabla\Bv(t)\|_{L_q(\BR_-^3)}\|\nabla\eta(t)\|_{L_\infty(\BR_-^3)}
			+\|\nabla \Bv(t)\|_{L_q(\BR_-^3)}\|\nabla^2\eta(t)\|_{L_\infty(\BR_-^3)} \\
			&\quad+\|\nabla^2\Bv(t)\|_{L_q(\BR_-^3)}\|\nabla\eta(t)\|_{L_\infty(\BR_-^3)}
			+\|\Bv(t)\|_{L_\infty(\BR_-^3)}\|\nabla^3\eta(t)\|_{L_q(\BR_-^3)}\Big) \displaybreak[0] \\
		&\leq
			C_q\Big(\|\Bv(t)\|_{W_q^1(\BR_-^3)}\|\nabla\eta(t)\|_{W_q^1(\BR_-^3)}
			+\|\Bv(t)\|_{W_q^1(\BR_-^3)}\|\nabla^3\eta(t)\|_{L_q(\BR_-^3)} \\
			&\quad+\|\nabla^2\Bv(t)\|_{L_q(\BR_-^3)}\|\nabla\eta(t)\|_{W_q^1(\BR_-^3)}\Big) \displaybreak[0] \\
		&\leq
			C_q\Big((t+2)^{-\Fm(\bar{q},q)-\Fm(\bar{q},q)-1/4}\|\Bz\|_{X_{q,p}(1/2,3/4)}^2 \\
			&\quad+(t+2)^{-\Fm(\bar{q},q)-3/4}\|\Bz\|_{X_{q,p}(1/2,3/4)}\big\{(t+2)^{1/2}\|\nabla^2\Bu(t)\|_{L_q(\BR_-^3)}\big\} \\
			&\quad+(t+2)^{-\Fm(\bar{q},q)-3/4}\|\Bz\|_{X_{q,p}(1/2,3/4)}\big\{(t+2)^{3/4}\|\nabla^3 \eta(t)\|_{L_q(\BR_-^3)}\big\}\Big),
\end{align*}		
and furthermore, it similarly holds that for any $t>0$
\begin{align*}
	&\|\Bv'(t)\cdot\nabla'\eta(t)\|_{W_{\bar{q}}^2(\BR_-^3)}
		\leq
			C\Big((t+2)^{-\Fm(\bar{q},q)-\Fm(\bar{q},q)-1/4}\|\Bz\|_{X_{q,p}(1/2,3/4)}^2 \\
			&\quad+(t+2)^{-\Fm(\bar{q},q)-3/4}\|\Bz\|_{X_{q,p}(1/2,3/4)}\big\{(t+2)^{1/2}\|\nabla^2\Bu(t)\|_{L_q(\BR_-^3)}\big\} \\
			&\quad+(t+2)^{-\Fm(\bar{q},q)-3/4}\|\Bz\|_{X_{q,p}(1/2,3/4)}\big\{(t+2)^{3/4}\|\nabla^3 \eta(t)\|_{L_q(\BR_-^3)}\big\}\Big).
\end{align*}
Then, noting that by \eqref{141225_3} and the assumption: $3<q<16/5$
\begin{equation*}
	p\left(m\left(\frac{q}{2},q\right)+m\left(\frac{q}{2},q\right)+\frac{1}{4}-\frac{3}{4}\right)
		=\frac{p}{q}>\frac{5p}{16}>1,
\end{equation*}
we have
\begin{align*}
	\|\Bv'\cdot\nabla'\eta\|_{L_p^{3/4}(\BR_+,W_q^2(\BR_-^3))} 
		&\leq
			C\|\Bz\|_{X_{q,p}(1/2,3/4)}\Big(\|(t+2)^{-\left(\Fm(\bar{q},q)+\Fm(\bar{q},q)+1/4-3/4\right)}\|_{L_p(\BR_+)}
			\|\Bz\|_{X_{q,p}(1/2,3/4)} \displaybreak[0] \\
			&\quad+\|(t+2)^{-\Fm(\bar{q},q)}\|_{L_\infty(\BR_+)}\|\nabla^2\Bu\|_{L_p^{1/2}(\BR_+,L_q(\BR_-^3))}
			\displaybreak[0] \\
			&\quad+\|(t+2)^{-\Fm(\bar{q},q)}\|_{L_\infty(\BR_+)}\|\nabla^3 \eta\|_{L_p^{3/4}(\BR_+,L_q(\BR_-^3))}\Big) \\
		&\leq C\|\Bz\|_{X_{q,p}(1/2,3/4)}^2.
\end{align*}
Analogously, we have 
\begin{equation*}
	\|\Bv'\cdot\nabla'\eta\|_{L_p^{1}(\BR_+,W_{\bar{q}}^2(\BR_-^3))}\leq C\|\Bz\|_{X_{q,p}(1/2,3/4)}^2,
\end{equation*}
because $\Fm(\bar{q},q)+3/4-1=1/(2q)>0$ and 
\begin{equation*}
	p\left(\Fm\left(\bar{q},q\right)+\Fm\left(\bar{q},q\right)+\frac{1}{4}-1\right)=p\left(\frac{1}{q}-\frac{1}{4}\right)
	>\frac{p}{16}>1
\end{equation*}
by \eqref{141225_3} and the assumption: $3<q<16/5$.

Next, we consider $\BBA_4$ and $\BBB_1$-norm.
Since it holds that 
%
%
\begin{align*}
	&\|\Bv'\cdot\nabla'\eta\|_{\BBA_4} 
		= \|\Bv'\cdot\nabla' \eta\|_{L_\infty^{\Fm(\bar{q},q)+1/2}(\BR_+,\wh W_q^1(\lhs))\cap L_\infty^{\Fm(\bar{q},2)+1/2}(\BR_+,\wh W_2^1(\lhs))} \\
		&\leq \|\Bv'\cdot\nabla' \eta\|_{L_\infty^{\Fm(\bar{q},q)+1/2}(\BR_+,W_q^1(\lhs))\cap L_\infty^{\Fm(\bar{q},2)+1/2}(\BR_+,W_2^1(\lhs))}
\end{align*}
and that
\begin{align*}
	&\|\Bv'\cdot\nabla'\eta\|_{\BBB_1} = \|\Bv'\cdot\nabla'\eta\|_{L_\infty^{\Fm(\bar{q},q)}(\BR_+,L_q(\BR_0^3))\cap L_\infty^{\Fm(\bar{q},2)}(\BR_+,L_2(\BR_0^3))} \\
		&\leq 
			C_q\|\Bv'\cdot\nabla'\eta\|_{L_\infty^{\Fm(\bar{q},q)+1/2}(\BR_+,W_q^1(\BR_-^3))\cap L_\infty^{\Fm(\bar{q},2)+1/2}(\BR_+,W_2^1(\BR_-^3))},
\end{align*}
it is enough to show the following estimate:
\begin{equation}\label{161226_1}
	\|\Bv'\cdot\nabla'\eta\|_{L_\infty^{\Fm(\bar{q},r)+1/2}(\BR_+,W_r^1(\lhs))}
		\leq C_{p,q,r}\|\Bz\|_{X_{q,p}(1/2,3/4)}^2.
\end{equation}
We observe that, by Sobolev's embedding theorem,
\begin{align*}
	&\|\Bv'(t)\cdot\nabla'\eta(t)\|_{W_r^1(\lhs)} \\
		&\leq \|\Bv\|_{L_\infty(\lhs)}\|\nabla\eta\|_{L_r(\lhs)}+\|\nabla \Bv\|_{L_r(\lhs)}\|\nabla\eta\|_{L_\infty(\lhs)}
		+\|\Bv\|_{L_\infty(\lhs)}\|\nabla^2\eta\|_{L_r(\lhs)} \\
		&\leq C_r\left(\|\nabla \Bv\|_{L_r(\lhs)}\|\nabla\eta\|_{W_q^1(\lhs)}+\|\Bv\|_{W_q^1(\lhs)}\|\nabla^2\eta\|_{L_r(\lhs)}\right) \\
		&\leq C_{q,r}\left((t+2)^{-\Fn(\bar{q},r)-1/8-\Fm(\bar{q},q)-1/4}+(t+2)^{-\Fm(\bar{q},q)-\Fm(\bar{q},r)-1/4}\right)\|\Bz\|_{X_{q,p}(1/2,3/4)}^2 \\
		&\leq C_{q,r}(t+2)^{-\Fm(\bar{q},q)-\Fm(\bar{q},r)-1/4}\|\Bz\|_{X_{q,p}(1/2,3/4)}^2,
\end{align*}
which, combined with
\begin{equation*}
	\Fm(\bar{q},q)+\Fm(\bar{q},r)+\frac{1}{4} =\Fm(\bar{q},r)+\frac{1}{2}+\frac{1}{2q},
\end{equation*}
implies \eqref{161226_1}. Thus, we obtain
$\|\Bv'\cdot\nabla'\eta\|_{\BBA_4\cap \BBB_1}\leq C_{p,q}\|\Bz\|_{X_{q,p}(1/2,3/4)}^2$.

Finally, we consider $\BBB_2(\te)$-norm with $0< \te< 1$.
It is clear that, by the trace theorem and H\"older's inequality,
\begin{equation}\label{161226_5}
	\|\Bv'(t)\cdot\nabla'\eta(t)\|_{L_{q(\te)}(\BR_0^3)} \leq \|\Bv'(t)\cdot\nabla'\eta(t)\|_{W_{q(\te)}^1(\lhs)} 
	\leq C\|\Bv(t)\|_{W_{2q(\te)}^1(\lhs)}\|\nabla\eta(t)\|_{W_{2q(\te)}^1(\lhs)}. 
\end{equation}
Since $2< 2 q(\te) <q$, there exists a positive number $\al_\te\in(0,1)$ such that
$2q(\te) = 2\al_\te + q(1-\al_\te)$. 
Thus, 
\begin{equation*}
	\|\Bv(t)\|_{W_{2q(\te)}^1(\lhs)} \leq \|\Bv(t)\|_{W_2^1(\lhs)}^{\al_\te}\|\Bv(t)\|_{W_q^1(\lhs)}^{1-\al_\te}
		\leq C_{p,q}\|\Bz\|_{X_{q,p}(1/2,3/4)},
\end{equation*}
and also
	$\|\nabla\eta(t)\|_{W_{2q(\te)}^1(\lhs)} \leq C_{p,q}\|\Bz\|_{X_{q,p}(1/2,3/4)}$.
Combining these two inequalities with \eqref{161226_5} furnishes
\begin{equation*}
	\|\Bv'\cdot\nabla\eta\|_{L_\infty(\BR_+,L_{q(\te)}(\BR_0^3))} \leq C_{p,q}\|\Bz\|_{X_{q,p}(1/2,3/4)}^2.
\end{equation*}
On the other hand, it is clear that
\begin{equation*}
	\|\Bv'\cdot\nabla\eta\|_{L_\infty(\BR_+,L_{2}(\BR_0^3))} \leq C_{p,q}\|\Bz\|_{X_{q,p}(1/2,3/4)}^2,
\end{equation*}
and thus
\begin{equation*}
	\|\Bv'\cdot\nabla\eta\|_{\BBB_2(\te)} \leq C_{p,q}\|\Bz\|_{X_{q,p}(1/2,3/4)}^2.
\end{equation*}
Summing up the above estimates, we have
\begin{equation}\label{170108_7}
	\|\SSK(\Bv,\eta)\|_{\BBK\cap\wt\BBK(3/4,1)\cap\BBA_4\cap\BBB_1\cap\BBB(\te)}\leq C_{p,q}\|\Bz\|_{X_{q,p}(1/2,3/4)}^2.
\end{equation}


Thus, \eqref{170108_1}, \eqref{170108_2}, \eqref{170108_3},
\eqref{170108_4}, \eqref{170108_5}, \eqref{170108_6}, and \eqref{170108_7}
imply \eqref{160902_10}.
For given $\bar{\Bz}=(\bar{\Bv},\bar{\Fq},\bar{h},\bar{\eta})\in {}_0X_{q,p}(\ep_1;1/2,3/4)$,
we have, by \eqref{141224_10}, \eqref{141224_11}, and Theorem \ref{theo:main1},
a unique solution $\Bz=(\Bv,\Fq,h,\eta)$ to
\begin{equation*}
	\left\{\begin{aligned}
		\pa_t\Bv -\Di\BT(\Bv,\Fq) =\SSF(\bar\Bv,\bar\eta)&
			&& \text{in $\BR_-^3$, $t>0$,}  \\
		\di\Bv = \SSG(\bar\Bv,\bar\eta) = \di\Fg(\bar\Bv,\bar\eta)&
			&& \text{in $\BR_-^3$, $t>0$,} \\
		\BT(\Bv,\Fq)\Be_3-(c_g-c_\si\De')h\Be_3 =\SSH(\bar\Bv,\bar \eta)&
			&& \text{on $\BR_0^3$, $t>0$,} \\
		\pa_t h -  v_3 =\SSK(\bar\Bv,\bar\eta)&
			&& \text{on $\BR_0^3$, $t>0$,} \\
		\Bv|_{t=0} =0&
 			&& \text{in $\BR_-^3$,} \\
		 h|_{t=0} =0 &
			&& \text{on $\BR^2$,}
	\end{aligned}\right.	
\end{equation*} 
together with
\begin{equation*}
	\left\{\begin{aligned}
		\De\eta &=0 && \text{in $\lhs$, $t>0$,} \\
		\eta &= h && \text{on $\bdry$, $t>0$.}
	\end{aligned}\right.
\end{equation*}
It is then possible to define a map $\Phi$ as
\begin{equation*}
	\Phi: {}_0X_{q,p}(\ep_1;1/2,3/4)\ni \bar\Bz \mapsto \Bz\in {}_0X_{q,p}(\ep_1;1/2,3/4),
\end{equation*}
because we can choose $\de_0$, $\ep_1$ so small that, by \eqref{141224_11},
\begin{equation*}
	\|\Phi(\bar{z})\|_{X_{q,p}(1/2,3/4)}
		\leq C_{p,q}(\|(\Bv_0,h_0)\|_{\BBI_1(\te)\times\BBI_2}+\|\bar\Bz\|^2_{X_{q,p}(1/2,3/4)}) 
		\leq C_{p,q}(\de_0+\ep_1^2) \leq \ep_1.
\end{equation*}
In addition, similarly to Step 1-Step 7, we see that
\begin{equation*}
	\|\Phi(\bar{\Bz}_1)-\Phi(\bar{\Bz}_2)\|_{X_{q,p}(1/2,3/4)}
		\leq \frac{1}{2}\|\bar{\Bz}_1-\bar{\Bz}_2\|_{X_{q,p}(1/2,3/4)}
\end{equation*}
for $\bar{\Bz}_i=(\Bv_i,\Fq_i,h_i,\eta_i)\in {}_0X_{q,p}(\ep_1;1/2,3/4)$ $(i=1,2)$.
The contraction mapping principle then implies that
$\Phi$ has a fixed point $\Bz'=(\Bv',\Fq',h',\eta')$,
which furnishes that System \eqref{eq:comp} admits a solution $(\Bv',\Fq',h')$.
Setting $\Bz=(\Bv,\Fq,h,\eta)$ as
\begin{equation*}
	\Bz=\Bz'+\Bz^* =
		(\Bv'+\Bv^*,\Fq'+\Fq^*,h'+h^*,\eta'+\CE(h^*)),
\end{equation*}
we see that $(\Bv,\Fq,h)$ is a global-in-time solution of the equations \eqref{NS3_1}-\eqref{NS3_5}
and satisfies the estimate: $\|\Bz\|_{X_{q,p}(1/2,3/4)}\leq \ep_0$ with $\ep_0=\ep_1$.
This  completes the proof of Theorem \ref{theo:main2}.

\section{Proof of Theorem \ref{theo:main3} and Theorem \ref{theo:main4}}\label{sec7}
In this section, we prove Theorem \ref{theo:main3} and Theorem \ref{theo:main4}
by using Theorem \ref{theo:main2}.

\subsection{Proof of Theorem \ref{theo:main3}}
(1).
%
%
We here show that $\Te_0$ is a $C^2$-diffeomorphism from $\BR_-^3$ onto 
$\Om_0=\{(x',x_3) \mid x'=(x_1,x_2)\in\BR^2, x_3<h_0(x')\}$.
Since it holds that
\begin{equation*}
	\|\CE(h_0)\|_{W_q^m(G)} \leq C_{q,G}\|h_0\|_{W_q^{m-1/q}(\BR^2)} \quad (m=2,3), \quad
	\|\nabla\CE(h_0)\|_{W_q^n(\BR_-^3)} \leq C_q\|h_0\|_{W_q^{n+1-1/q}(\BR^2)} \quad (n=1,2)
\end{equation*}
for any compact set $G$ in $\BR_-^3$, we obtain,  by the real interpolation method,
\begin{equation*}
	\|\CE(h_0)\|_{B_{q,p}^{3-1/p}(G)} \leq M_{p,q,G}\|h_0\|_{B_{q,p}^{3-1/q-1/p}(\BR^2)}, \quad \\ 
	\|\nabla\CE(h_0)\|_{B_{q,p}^{2-1/p}(\BR_-^3)} \leq M_{p,q}\|h_0\|_{B_{q,p}^{3-1/q-1/p}(\BR^2)} \notag 
\end{equation*}
with positive constants $M_{p,q.G}$, $M_{p,q}$. 
Combining these inequalities with Sobolev's embedding theorem:
\begin{equation*}
	B_{q,p}^{3-1/p}(G) \hookrightarrow BUC^2(G), \quad 
	B_{q,p}^{2-1/p}(\BR_-^3) \hookrightarrow BUC^1(\BR_-^3) \quad (1/p+3/q<1)
\end{equation*}
furnishes that $\CE(h_0)$ is a function of class $C^2$ and that
\begin{equation*}
	\frac{\pa x_3}{\pa y_3} = 1 + \frac{\pa}{\pa y_3}\CE(h_0) \geq 1- \|\nabla\CE(h_0)\|_{L_\infty(\BR_-^3)} \geq 1- M_{p,q}r_0. 
\end{equation*}
If we choose $r_0\in(0,1)$ so that $M_{p,q}r_0\in(0,1/2)$,
then we observe that $\Te_0$ is a $C^2$-diffeomorphism from $\BR_-^3$ onto $\Om_0$ as was discussed in Subsection \ref{subsec2_2}.



(2). We prove a smallness condition of $\Bv_0=\Bu_0\circ\Te_0$ and $h_0$ as follows:
\begin{equation}\label{170104_5}
	\|\Bv_0\|_{\BBI_1(\te)}+\|h_0\|_{\BBI_2} \leq \de_0,
\end{equation}
where $\de_0$ is the same positive number as in Theorem \ref{theo:main2}.
By direct calculations,
\begin{equation*}
	\|\Bv_0\|_{W_r^{l+1}(\BR_-^3)} \leq C_{p,q}F(\|\nabla\CE(h_0)\|_{W_\infty^l(\lhs)})\|\Bu_0\|_{W_r^{l+1}(\Om_0)}
\end{equation*}
for $r\in \{q,q(\te)\}$ and $l=0,1$, 
where $F:[0,\infty)\to\BR$ is a continuous function with $F(0)=1$.
Combining this inequality with the real interpolation yields that
$\|\Bv_0\|_{\BBI_1(\te)} \leq c_1\|\Bu_0\|_{\BBJ_{q,p,\te}(h_0)}$.
If necessary, we choose $r_0$ so small that
\begin{equation*}
	\|\Bv_0\|_{\BBI_1(\te)}+\|h_0\|_{\BBI_2}\leq  (c_1+1)r_0\leq \de_0,
\end{equation*}
which implies \eqref{170104_5}.
On the other hand, the condition \eqref{comp:2} implies \eqref{comp:1}.
Thus we obtain


(3).  Noting that
\begin{align*}
	&W_p^1(\BR_+,W_q^2(G)) \cap L_p(\BR_+,W_q^3(G)) \hookrightarrow BUC([0,\infty),B_{q,p}^{3-1/p}(G)), \\
	&W_p^1(\BR_+,W_q^1(\lhs)) \cap L_p(\BR_+,W_q^2(\lhs))\hookrightarrow BUC([0,\infty),B_{q,p}^{2-1/p}(\lhs)),
\end{align*}
we can prove the required properties in the same manner as in (1).
This completes the proof of Theorem \ref{theo:main3}.

\subsection{Proof of Theorem \ref{theo:main4}}
Let $(\Bu,\Fp,h)$ be the solution, obtained in Theorem \ref{theo:main3}, to System \eqref{NS1}.
Then, by the change of variables,
\begin{equation*}
	\|\Bu(t)\|_{L_q(\Om_t)}=\left(\int_{\BR_-^3}|\Bv(y,t)|^q\left|\frac{\pa x}{\pa y}\right|\intd y\right)^{1/q}
	\leq C_{p,q}\|\Bv(t)\|_{L_q(\lhs)},
\end{equation*}
where $\pa x/\pa y$ is the Jacobian matrix defined as in \eqref{Jacobian}.
Thus, by Theorem \ref{theo:main3}, we observe that
\begin{equation*}
	\|\Bu(t)\|_{L_r(\Om_t)} \leq \ep_0 C_{p,q}(t+2)^{-\Fm(\bar{q},r)} \quad \text{for $r\in\{q,2\}$,}
\end{equation*}
which, combined with the interpolation inequality:
\begin{equation*}
	\|f\|_{L_{s_1}(\BR_-^3)}\leq \|f\|_{L_{s_2}(\BR_-^3)}^{\al}\|f\|_{L_{s_3}(\BR_-^3)}^{1-\al} \quad 
	\text{for } \frac{1}{s_1} = \frac{\al}{s_2}+\frac{1-\al}{s_3}, \text{ $0\leq \al \leq 1$,} 
\end{equation*}
furnishes that  $\|\Bu(t)\|_{L_r(\Om_t)}=O(t^{-\Fm(\bar{q},r)})$ as $t\to\infty$ for $2\leq r \leq q$.
Analogously, the asymptotic behavior of the other terms can be proved.
This completes the proof of Theorem \ref{theo:main4}.

\def\thesection{A}
\renewcommand{\theequation}{A.\arabic{equation}}
\section{}

In the appendix, we consider the $N$-dimensional case for $N\geq 2$.
Let
\begin{align*}
	\BR_-^N=\{(x',x_N) \mid x'=(x_1,\dots,x_{N-1})\in\BR^{N-1}, x_N<0\}, \\
	\BR_0^N=\{(x',x_N) \mid x'=(x_1,\dots,x_{N-1})\in\BR^{N-1}, x_N=0\}.
\end{align*}

The aim of this appendix is to prove the existence of extension
used in Subsections \ref{subsec:3_2} and \ref{subsec:3_3}, i.e. we prove

\begin{lemm}\label{lemm:A1}
Let $2<p<\infty$ and $1<q<\infty$, and suppose that $2/p+1/q<1$.
Then, for any $u\in H_{q,p}^{1,1/2}(\BR_-^N\times\BR_+)$ with $u|_{t=0}=0$ on $\BR_0^N$,
there exists $\wt u\in H_{q,p}^{1,1/2}(\BR_-^N\times\BR)$ such that
\begin{equation*}
	\wt u=u \quad \text{in $\BR_-^N$ $(t>0)$,} \quad \wt u=0 \quad \text{on $\BR_0^N$ $(t<0)$}
\end{equation*}
and that
\begin{equation*}
	\|\wt u\|_{H_{q,p}^{1,1/2}(\BR_-^N\times\BR)} \leq C_{N,p,q}\|u\|_{H_{q,p}^{1,1/2}(\BR_-^N\times\BR_+)}
\end{equation*}
with some positive constant $C_{N,p,q}$ independent of $u$ and $\wt u$.
\end{lemm}

\begin{proof}
Let $u^e$ be the even extension of $u$ with respect to $x_N$.
Then 
$$
u^e\in H_{q,p}^{1,1/2}(\BR^N\times \BR_+) \text{\ \  with \ \  } 
\|u^e\|_{H_{q,p}^{1,1/2}(\BR^N\times \BR_+)}
\leq C\|u\|_{H_{q,p}^{1,1/2}(\BR_-^N\times \BR_+)}.
$$
By the time trace theorem (cf. \cite[Theorem 3.4.8 and Example 3.4.9]{PS16}),
we have 
$$
u^e|_{t=0}\in B_{q,p}^{1-2/p}(\BR^N) \text{\ \ with \ \ }
\|u^e|_{t=0}\|_{B_{q,p}^{1-2/p}(\BR^N)}\leq C\|u^e\|_{H_{q,p}^{1,1/2}(\BR^N\times \BR_+)}.
$$
Thus, setting $u(0)=u|_{t=0}$ in $\BR_-^N$, we see that 
\begin{equation}\label{timetrace}
\|u(0)\|_{B_{q,p}^{1-2/p}(\BR_-^N)} 
\leq \|u^e|_{t=0}\|_{B_{q,p}^{1-2/p}(\BR^N)} 
\leq C \|u^e\|_{H_{q,p}^{1,1/2}(\BR^N\times \BR_+)} 
\leq C \|u\|_{H_{q,p}^{1,1/2}(\BR_-^N\times \BR_+)}. 
\end{equation}

We first consider in $X=L_q(\BR_-^N)$ the following system:
\begin{equation*}
	\left\{\begin{aligned}
		\pa_t v+  A v =0,& \quad \text{ $t>0$,} \\
		v|_{t=0} = u(0),& 
	\end{aligned}\right.
\end{equation*}
where we have set $A v= (1-\De_D)v$ for $v\in D(A)$ with domain
\begin{equation*}
	D(A)  = \{u \in W_q^2(\BR_-^N) \mid u=0 \text{ on $\BR_0^N$}\}.
\end{equation*}
Then, by \cite[Proposition 3.4.3]{PS16}, we have
\begin{equation}\label{170214_5}
	\|v\|_{H_p^{1/2}(\BR_+,X)}+\|v\|_{L_p(\BR_+,D(A^{1/2}))}\leq C_{p,q}\|u(0)\|_{(X,D(A))_{1/2-1/p,p}.}
\end{equation}
Since $-\De_D$ admits a bounded $H^\infty$-calculus
on $X$ on the sector $\Sigma_\varepsilon$ for each $0<\ep<\pi/2$ (cf. \cite[Corollary 7.3]{DHP01},
and \cite[Section 6]{DHP01} for the definition),
the operator $A=1-\De_D$ admits bounded imaginary powers on $X$ by \cite[Proposition 3.3.8, page 125 (3.62)]{PS16}.
Thus, by \cite[Theorem 3.3.7]{PS16}, we have
the characterization of $D(A^{1/2})$ as $D(A^{1/2}) = [X,D(A)]_{1/2}$.
In addition, by \cite[Subsection 2.3]{NS03},
\begin{equation}\label{170218_1}
 [X,D(A)]_{1/2}=\{u\in W_q^1(\BR_-^N) \mid u=0 \text{ on $\BR_0^N$}\}.
\end{equation}
On the other hand, by \cite[Theorem 4.9.1]{Amann09},
we observe that
\begin{equation}\label{170219_1}
	(X,D(A))_{1/2-1/p,p} = \{u\in B_{q,p}^{1-2/p}(\BR_-^N) \mid u=0 \text{ on $\BR_0^N$}\}.
\end{equation}
Summing up \eqref{timetrace}-\eqref{170219_1}, we have
\begin{align}
&C_1\|v\|_{H_{q,p}^{1,1/2}(\BR_-^N\times\BR_+)} 
\leq \|u(0)\|_{B_{q,p}^{1-2/p}(\BR_+^N)}\leq C_2\|u\|_{H_{q,p}^{1,1/2}(\BR_-^N\times\BR_+)}, \label{170219_2}\\
	&v=0 \quad \text{on $\BR_0^N$ $(t>0)$,} \label{170219_7}
\end{align}
with some positive constants $C_1$ and $C_2$.

We now set $w=w(x,t)$ as follows: $w=u(t)$ when $t>0$ and $w=v(-t)$ when $t<0$.
It clearly holds that, by \eqref{170219_2},
\begin{equation}\label{170214_7}
	\|w\|_{L_p(\BR,W_q^1(\BR_-^N))} \leq C_{p,q}\|u\|_{H_{q,p}^{1,1/2}(\BR_-^N\times\BR_+)}.
\end{equation}
In what follows, we consider the time regularity of $w$. 
By using the even extension with respect to time $t$,
we extend $u(t)$ and $v(t)$ to $U(t)$ and $V(t)$ defined on $\BR$, respectively, such that
\begin{align}\label{170219_3}
	&u(t)=U(t) \quad (t>0), \quad v(t) = V(t) \quad (t>0), \\
	&\|U\|_{H_{q,p}^{1,1/2}(\BR_-^N\times\BR)} \leq C\|u\|_{H_{q,p}^{1,1/2}(\BR_-^N\times\BR_+)}, \label{est:ext_U} \\
	&\|V\|_{H_{q,p}^{1,1/2}(\BR_-^N\times\BR)} \leq C\|v\|_{H_{q,p}^{1,1/2}(\BR_-^N\times\BR_+)}. \label{est:ext_V}
\end{align}
Then, we have
\begin{align*}
	D_t^{1/2}w = \CF_\tau^{-1}[(1+|\tau|^2)^{1/4}\CF_t[\chi_{\BR_+}(t)U(t)+\chi_{\BR_-}(t)V(-t)](\tau)](t),
\end{align*}
where $D_t^{1/2}$ is defined in Subsection \ref{subsec2_1}. Since 
\begin{equation*}
	(1+|\tau|^2)^{1/4} = \frac{1+|\tau|^2}{(1+|\tau|^2)^{3/4}}
		=\frac{1}{(1+|\tau|^2)^{3/4}}-\frac{i\tau}{(1+|\tau|^2)^{3/4}}i\tau
\end{equation*}
and 
\begin{equation*}
	\pa_t\left(\chi_{\BR_+}(t)U(t)+\chi_{\BR_-}(t)V(-t)\right)
		=\chi_{\BR_+}(t)(\pa_t U)(t)-\chi_{\BR_-}(t)(\pa_t V)(-t)
\end{equation*}
we see that
\begin{align*}
	D_t^{1/2}w
		&=\CF_\tau^{-1}\left[\frac{1}{(1+|\tau|^2)^{3/4}}\CF_t\left[\chi_{\BR_+}(t)U(t)+\chi_{\BR_-}(t)V(-t)\right](\tau)\right](t) \\
		&-\CF_\tau^{-1}\left[\frac{i\tau}{(1+|\tau|^2)^{3/4}}\CF_t\left[\chi_{\BR_+}(t)(\pa_t U)(t) -\chi_{\BR_-}(t)(\pa_t V)(-t)\right](\tau)\right](t) \\
		&=:I_1(t)+I_2(t).
\end{align*}
Combining the vector-valued Fourier multiplier theorem of Zimmermann \cite[Proposition 3]{Zimmermann89} with
\eqref{170219_2}, \eqref{est:ext_U}, and \eqref{est:ext_V}
furnishes that 
\begin{equation*}
	\|I_1\|_{L_p(\BR,L_q(\BR_-^N))}
		\leq C\|(U,V)\|_{L_p(\BR,L_q(\BR_-^N))}
		\leq C_{p,q}\|u\|_{H_{q,p}^{1,1/2}(\BR_-^N\times\BR_+)}.
\end{equation*}

We next estimate $I_2(t)$ that can be written as
\begin{align*}
	I_2(t) 
		&=-\CF_\tau^{-1}\left[\frac{i\tau}{(1+|\tau|^2)^{1/2}}\CF_t\left[D_t^{-1/2}\left(\chi_{\BR_+}(t)(\pa_t U)(t)\right)\right](\tau)\right](t) \\
		&+\CF_\tau^{-1}\left[\frac{i\tau}{(1+|\tau|^2)^{1/2}}\CF_t\left[D_t^{-1/2}\left(\chi_{\BR_-}(t)(\pa_t V)(-t)\right)\right](\tau)\right](t) \\
		&=:J_1(t)+J_2(t).
\end{align*} 
Then, by the vector-valued Fourier multiplier theorem again,
we have
\begin{equation*}
	\|J_1\|_{L_p(\BR,L_q(\BR_-^N))} \leq C_{p,q}\|\chi_{\BR_+}\pa_t U\|_{H_p^{-1/2}(\BR,L_q(\BR_-^N))}.
\end{equation*}
Combining this inequality with \cite[Theorem 1.1]{MV15} and \eqref{est:ext_U} furnishes that
\begin{equation*}
	\|J_1\|_{L_p(\BR,L_q(\BR_-^N))} 
		\leq C\|\pa_t U\|_{H_p^{-1/2}(\BR,L_q(\BR_-^N))}
		\leq C\|U\|_{H_p^{1/2}(\BR,L_q(\BR_-^N))} 
		\leq C\|u\|_{H_{q,p}^{1,1/2}(\BR_-^N\times\BR_+)}.
\end{equation*}
Similarly, we have
	$\|J_2\|_{L_p(\BR,L_q(\BR_-^N))}
		\leq C\|u\|_{H_{q,p}^{1,1/2}(\BR_-^N\times\BR_+)}$.
It thus holds that
\begin{equation*}
	\|w\|_{H_p^{1/2}(\BR,L_q(\BR_-^N))} \leq C_{p,q}\|u\|_{H_{q,p}^{1,1/2}(\BR_-^N\times\BR_+)},
\end{equation*}
which, combined with \eqref{170219_7} and \eqref{170214_7},
implies that $w$ is the desired extension of Lemma \ref{lemm:A1}.
This completes the proof of the lemma.
\end{proof}

\end{document}